\documentclass[a4paper]{amsbook}%
\usepackage{amssymb}
\usepackage{amsmath}
\usepackage{amsfonts}
\usepackage{graphicx}%
\providecommand{\U}[1]{\protect\rule{.1in}{.1in}}
\theoremstyle{plain}
\newtheorem{acknowledgement}{Acknowledgement}

\newtheorem{corollary}{Corollary}

\newtheorem{definition}{Definition}
\newtheorem{example}{Example}

\newtheorem{lemma}{Lemma}

\newtheorem{proposition}{Proposition}
\newtheorem{remark}{Remark}

\newtheorem{theorem}{Theorem}
\numberwithin{equation}{section}
\begin{document}
\frontmatter
\title[Laws of Large Numbers]{Extracting information from random data. Applications of laws of large numbers
in technical sciences and statistics.}
\author[Pawe\l \ J. Szab\l owski]{Pawe\l \ J. Szab\l owski}
\address{Department of Mathematics and Information Sciences\\
Warsaw University of Technology\\
ul. Koszykowa, 75, 00-662 Warszawa, Poland}
\email{pawel.szablowski@gmail.com;paweljs@mini.pw.edu.pl}

\begin{abstract}
We formulate conditions for convergence of Laws of Large Numbers and show its
links with of the parts of mathematical analysis such as summation theory,
convergence of orthogonal series. We present also applications of the Law of
Large Numbers such as Stochastic Approximation, Density and Regression
Estimation, Identification.

\end{abstract}
\maketitle
\tableofcontents

\pagebreak

\begin{acknowledgement}
The author would like to thank two of his students and co-workers dr. dr.
Marcin Dudzi\'{n}ski and Wojciech Matysiak for correction of the manuscript.
\end{acknowledgement}

%

%TCIMACRO{\TeXButton{Place Index Here}{\printindex}}%
%BeginExpansion
\printindex
%EndExpansion

\chapter*{List of symbols and denotations}

$x$, $y$, $\ldots$, $\alpha$, $\beta$, $\ldots-$ real numbers,

$\mathbf{x}$, $\mathbf{y}$, $\ldots$ -vectors, always column

$\mathbf{A}$, $\mathbf{B}$, $\ldots$ -matrices,

$\mathbf{x}^{T}$, $\mathbf{A}^{T},\ldots$ -transposition of vector, matrix

$A$, $B$, $\ldots$ -events, subsets of some space of elementary events
$\Omega$,

$A^{c}$, $\bar{A}$ -composition of event $A,$

$A\cap B$, $A\cup B$, $A\vartriangle B$ respectively product, union and the
symmetric difference of two sets $A$ and $B,$

$\omega$ element of the space of elementary events - elementary event,
$\omega\in\Omega$

$P(A)$ -probability of an event $A,$

$I(A)(\omega)=\left\{
\begin{array}
[c]{ccc}%
1, & if & \omega\in A\\
0, & if & \omega\notin A
\end{array}
\right.  $ -indicator function of the event $A$, often denoted simply by
$I(A),$

$X$, $Y$, $Z$, $\ldots$ -random variables,

$\mathbf{X,Y,\ldots}$ -random vectors,

$EX$, $EY$, $\ldots$ -expectations of the random variables,

$\sigma(X),\sigma(\mathbf{Y),}\sigma(X,Y,Z)$ -$\sigma-$ fields generated
respectively by the random variable $X$, random vector $\mathbf{Y}$, random
variables $X,Y,Z,$

$\mathcal{A}$, $\mathcal{B},\mathcal{C}$, $\ldots$ -denotations of $\sigma- $ fields,

$E(X|\mathcal{F})$ -conditional expectation with respect to the $\sigma-$
field $\mathcal{F},$

$P(A|\mathcal{G)}$ -conditional probability with respect to the $\sigma-$
field $\mathcal{G},$

$\left\{  x_{i}\right\}  _{i\geq1}$, $\left\{  Y_{i}\right\}  _{i\geq0}$,
$\left\{  A_{i}\right\}  _{i\geq2}$ -sequences respectively, of real numbers,
random variables, events,

$%
%TCIMACRO{\U{2115} }%
%BeginExpansion
\mathbb{N}
%EndExpansion
,%
%TCIMACRO{\U{211d} }%
%BeginExpansion
\mathbb{R}
%EndExpansion
,%
%TCIMACRO{\U{2124} }%
%BeginExpansion
\mathbb{Z}
%EndExpansion
,%
%TCIMACRO{\U{2102} }%
%BeginExpansion
\mathbb{C}
%EndExpansion
$ --sets respectively, of natural numbers, real numbers, integers and complex numbers,

$\#A$ -cardinality of the set $A,$

$\left\vert A\right\vert $ - Lebesgue measure of the set $A.$

$\left\{  X<x\right\}  $, $\left\{  \sum X_{i}\;\text{converges }\right\}  $
-shortened denotation of the events $\left\{  \omega:X(\omega)<x\right\}
$,\newline$\left\{  \omega:\sum X_{i}(\omega)\;\text{converges}\right\}  ,$

$x^{+}=\left\{
\begin{array}
[c]{ccc}%
x & if & x\geq0\\
0 & if & x<0
\end{array}
\right.  $ -other words positive part of $x.$

a.e., a.s. - respectively almost everywhere, almost surely. Refers to events
that have zero measures or probabilities.

\chapter*{Preface}

By iterative we mean those random phenomena that can be presented in the
following form:
\[%
\begin{tabular}
[c]{|ccccc|}\hline
\emph{Next observation} & = & \emph{function of (the present observation)} &
+ & \emph{correction,}\\\hline
\end{tabular}
\
\]
It means that the information based on the knowledge collected so far,
complemented by the actually observed "correction" is the base of the
knowledge about the future behavior of the examined random phenomenon.

A typical example of such \textquotedblright iterative\textquotedblright%
\ approach is the so-called law of large numbers\emph{,} or behavior of the
averages of observed measurements. More precisely, if the measured quantities
are denoted by $x_{i}$, $i=\allowbreak1,\allowbreak2,\allowbreak\ldots$, then
their averages are $y_{n}=\frac{x_{1}+\cdots+x_{n}}{n}$. In an obvious way the
sequence $\left\{  y_{n}\right\}  _{n\geq1}$ can be presented in one of the
following forms:
\begin{align*}
y_{n+1}  &  =y_{n}+\frac{1}{n+1}(x_{n+1}-y_{n}),\\
y_{n+1}  &  =(1-\frac{1}{n+1})y_{n}+\frac{1}{n+1}x_{n+1}.
\end{align*}

In probability theory, such iterative forms appear quite often, although we do
not always use them, or even are not aware that a given quantity can be
presented in such, iterative way.

On the other hand, it is well known that some Markov processes can be
presented in an 'iterative' way. In particular, processes connected with
filtration problems (corrections are here called innovation processes), or
some processes appearing in the analysis of queuing systems can be naturally
presented in an iterative way. Such Markov processes not necessarily converge
(or more generally stabilize their behavior) as the number of iterations tends
to infinity. The behavior of such processes for large values of indices can be
very complex and sometimes exceeds the scope of this monograph. We will not
consider such general situations.

Instead, we will concentrate on iterative procedures converging to nonrandom
constants. The number of random phenomena, that can be described by such
procedures is so large that not all of them will be analyzed here. We will
concentrate here on:

\begin{itemize}
\item \emph{laws of large numbers (LLN) and their connections to mathematical
analysis and In particular, to the theory of summability and the theory of
orthogonal series,}

\item \emph{some procedures of stochastic approximation,}

\item \emph{some procedures of density estimation,}

\item \emph{some procedures of identification.}
\end{itemize}

The title refers to the laws of large numbers and their applications in
technology and statistics. Typical applications of the laws of large numbers
in technology or physics are applications in the theory of measurements.
Suppose that we are given a series of measurements of some quantity. Then, if
some relatively mild conditions, under which the measurements were performed,
are satisfied, the arithmetic means of the measured values can be considered
as good approximations of the measured quantity. It is worth noticing that
this approximation is getting better if the number of measurements is greater.

The typical applications of LLN in statistics are estimators. Usually, we
observe, that the larger the sample the estimator is based upon, the closer
the value of an estimator is to the theoretical value of the parameter. This
basic property is called consistency (of an estimator). In fact, the fact that
LLN can be applied is the base on which one states that the given estimator is
strongly consistent or not.

Other rather typical applications of LLN are the so-called Monte Carlo methods
and based on them, simulations. Monte Carlo methods are used in numerical
methods to estimate values of some difficult to find constants (given by, say,
hard to compute integrals) and also in physics to estimate, hard to get
directly, constants used in the description of physical phenomena, mainly, but
not necessarily, in statistical physics.

Variants of LLN are used in estimation theory, measurement theory, Monte Carlo
methods or statistical physics described above are simple. In most cases, one
assumes that random variables used there are independent identically
distributed (i.i.d.). Convergence problem then is very simple. Necessary and
sufficient conditions guaranteeing, that a random variable in question satisfy
LLN, are known and simple. They will be discussed in sections \ref{sspwl_iid}
and \ref{smpwl_iid} of chapter \ref{simpwl}. Difficulties associated with say
Monte Carlo methods lie elsewhere. Namely, they lie in defining estimator with
good properties or finding such physical experiment that could be easily
simulated and in which the estimated constant would appear. Discussion of
these problems would lead us too far from probability theory and would require
a separate book. The reader interested in Monte Carlo methods or stochastic
simulations, we refer to the monographs of R. Zieli\'{n}ski \cite{Zielinski75}
and D. W. Heermann \cite{Heermann97}.

To be consistent with the title we decided to present three applications of
LLN important in technology (identification, density estimation) or stochastic
optimization (stochastic approximation). These applications can be also
considered as parts of mathematical statistics, less known and less obviously
associated with LLN. Moreover, we were able to indicate formal similarities in
the description and formulation of these problems and convergence problems
appearing in LLN. It turns out that the methods developed in chapter
\ref{zbiez}, can be applied in chapter \ref{simpwl} dedicated to the laws of
large numbers as well as in chapters \ref{aproksymacja}, \ref{metody_jadrowe},
\ref{identyf} dedicated respectively to stochastic approximation, kernel
methods of density estimation or identification methods. On the other hand,
each of the mentioned applications contains dozens of cases. Each of these
applications is extensively described in the literature. Thus, it is
impossible to present it exhaustively. On should write the separate volumes to
make such presentation. Besides, it is not the aim of this book.

As it was already stated above the aim of the author was to present problems
connected with LLN and indicate their connection with classical parts of
mathematical analysis such as summation theory, convergence theory of
orthogonal series. As it was mentioned before the aim of the author was
relatively extensive presenting of problems associated with LLN and indicate
strong bonds with classical sections of mathematical analysis and at the same
time indicate that laws of large numbers are the base for intensively
developing sections of statistics such as stochastic optimization
nonparametric estimation or adaptive identification. We are convinced that
many statisticians working in stochastic optimization or nonparametric
estimation are not aware of how closely they are in their research to
classical problems of analysis. Similarly, mathematicians working in the
theory of summability or orthogonal series are not aware that their results
can have practical applications. The author wanted to visualize those facts to
both groups of researchers. To do this, one must not be mired in the details.

Basically, the book was written for students of mathematics or physics or for
the engineers applying mathematics. The author assumes that the reader knows
the basic course of probability and elements of mathematical statistics.
Nevertheless, some important notions and facts that are important to the logic
of the argument were recalled. The book is written as a mathematical text that
is facts are presented in the form of theorems. Proofs of the majority of
theorems are presented in the main bulk of the book. Some of the proofs that
are less important or are longer are shifted to the appendix. The facts that
came from deeper or more complicated theories are recalled without proofs.

The aims of the book are different and depend on the reader. Students are
exposed here to interesting applications of mathematics that make them aware
that many issues coming from different sections of mathematics can be treated
by the same methods. The book makes mathematicians or statisticians realize
that the methods developed specially for one section of mathematics can be
useful in the other. Finally, those readers that do not work in stochastic
optimization or nonparametric estimation are acquainted with those sections of statistics.

It should be underlined that neither of the topics raised in the book is
exhausted. What seems to be the book's advantage is that it presents a unified
approach to different, at first sight, applications. We use, in fact the same
basic theorems to prove convergence of some orthogonal series, procedures of
stochastic approximation, iterative procedures of density estimation or
iterative procedures of identification.

Another advantage of the book seems to be great number and variety of examples
illustrated by drawings made basically by MathCad and Mathematica. Looking at
these examples one can get an idea of how effective are the discussed methods
or how quick is the convergence in the described random phenomenon.

\mainmatter

\chapter{Overview of the most important random phenomena.}

Instead of an introduction, we will present the most important random
phenomena called sometimes \emph{pearls of probability} \emph{\ }(see, e.g.
Hoffman-Jorgensen \cite{hoffman}). By pearls of probability, we mean
\emph{laws of large numbers} (LLN), \emph{central limit theorem }%
(CLT)\emph{\ }and\emph{\ }the\emph{ law of iterated logarithm }(LIL)\emph{.
}In the sequel, we will show that these phenomena can be presented in an
iterative form so that problems appearing in their analysis lie naturally
within the scope of this monograph. Not all of these problems could be solved
by the simple methods developed in this book. Sometimes one should refer to
more advanced means. Mathematical problems appearing in the analysis of these
pearls of probability are connected mainly with convergence. The types of
convergence considered in probability are recalled in appendix \ref{rzbiez}.

In the three subsequent sections, we will present the three random phenomena
mentioned above, point out analogies and differences between them and present
some of the related open problems. As it will turn out that the forms of these
phenomena are very similar. The differences concern properties of some of the
parameters and the types of convergence that these phenomena obey. So first we
will present these random phenomena and later we will return to general questions.

\section{Laws of Large numbers}

Let $\{X_{n}\}_{n\geq1}$ be a sequence of the random variables having
expectations. Let us denote $m_{n}=EX_{n};n\geq1.$

\begin{definition}
\label{LLN}%
\index{Law!of large numbers}%
We say that the sequence $\{X_{n}\}_{n\geq1}$ satisfies \emph{weak (strong)
}law of large numbers\emph{\ }(briefly WLLN (SLLN)), \emph{if}
\[
Y_{N}=\frac{\sum_{n=1}^{N}(X_{n}-m_{n})}{N}\rightarrow0;\;\text{when}%
\;N\rightarrow\infty,
\]
where convergence is either in probability (for the WLLN) or with probability
$1$ (for the SLLN).
\end{definition}

\begin{remark}
On considers also the so-called generalized LLN, that is, sequences of the
random variables that are summable by the so-called Riesz method. More
precisely, we say that the sequence $\left\{  X_{n}\right\}  _{n\geq1}$
satisfies \emph{weak (strong) }generalized law of large numbers (briefly
WGLLN, SGLLN) with weights $\left\{  \alpha_{i}\right\}  _{i\geq0}$, if
\[
Y_{N}=\frac{\sum_{i=0}^{N-1}\alpha_{i}\left(  X_{i+1}-m_{i+1}\right)  }%
{\sum_{i=0}^{N-1}\alpha_{i}}\longrightarrow0,\text{when }N\longrightarrow
\infty,
\]
where, as before, convergence is in probability for the WGLLN and with
probability $1$ for the SGLLN. We will return to this definition in section
\ref{Riesz}
\end{remark}

\begin{remark}
Following intentions of this book we will present the random sequence
$\left\{  Y_{N}\right\}  _{N\geq1}$ in a recursive (iterative) form. Namely,
we have : \newline$Y_{N+1}=\allowbreak\frac{N}{N+1}Y_{N}+\allowbreak\frac
{1}{N+1}(X_{N+1}-m_{N+1})$, or equivalently in slightly different more general
forms:
\begin{align}
Y_{N+1}  &  =\allowbreak\left(  1-\mu_{N}\right)  Y_{N}+\allowbreak\mu
_{N}(X_{N+1}-m_{N+1}),\label{forma1}\\
Y_{N+1}  &  =\allowbreak Y_{N}+\allowbreak\mu_{N}\left(  (X_{N+1}%
-m_{N+1})-Y_{N}\right)  , \label{forma2}%
\end{align}
where $\mu_{N}=\frac{1}{N+1}$. In the sequel, we will be interested in the
convergence of the sequence $\left\{  Y_{N}\right\}  $ to zero, and also a
convergence of the series $\sum_{N\geq1}\mu_{N}(X_{N+1}-m_{N+1})$. The form
(\ref{forma1}) will be more useful in examining convergence, while
(\ref{forma2}) will be more useful in analyzing stochastic approximation
procedures since due to it one can easily notice the connections between the
stochastic approximation and laws of large numbers.
\end{remark}

\begin{example}
As the first example of the application of the law of large numbers, let us
consider the problem of measuring the unknown quantity $m$. As the result of
independent, and performed in the same conditions, measurements, we obtain
observations: $x_{1},\allowbreak x_{2},\allowbreak\ldots,\allowbreak x_{n}$.
We assume the following model of taking measurement:
\[
x_{i}=m+\varepsilon_{i},\;E\varepsilon_{i}=0,\;i=1,2,\ldots,n.
\]
If one can assume that the sequence of measurements $x_{i}$, $i=1,2,\ldots,n$
satisfies SLLN, then the sequence of quantities $\left\{  \frac{x_{1}%
+\cdots+x_{n}}{n}\right\}  _{n\geq1}$ converges almost surely to $Ex_{i}=m$.
Hence, the postulate to approximate the measured quantity by the mean of the
measurements makes sense.
\end{example}

\begin{example}
Another more spectacular example of the application of LLN concerns estimating
the number of fish in the pond. Suppose that we would like to get information
on the number of fish without emptying the pond which would inevitably kill
the fish. To this end, we release $N$ marked fish (those can be fish of the
other species) to the pond. Next, we perform $n$ catches with return. Each
time we note if the caught fish was marked or not. Let $M$ be the unknown
number of fish in the pond. Let us denote:
\[
X_{i}=\left\{
\begin{array}
[c]{ll}%
1 & \text{if in }i\text{-th catch there was a marked fish}\\
0 & \text{otherwise.}%
\end{array}
.\right.
\]
Notice that $EX_{i}=P(X_{i}=1)=\frac{N}{N+M}$. If one can assume that the
sequence $\left\{  X_{i}\right\}  _{i\geq1}$ satisfies LLN, then for
sufficiently large $n$ we have approximate equality:
\begin{align*}
\frac{\sum_{i=1}^{n}X_{i}}{n}  &  =\allowbreak\frac{\text{number of caught
marked fish}}{n}\\
&  =\text{fraction of marked fish}\allowbreak\approx\frac{N}{N+M}.
\end{align*}
Now it is elementary to solve this equality for $M$.
\end{example}

\begin{example}
[identification]\label{ident}In the last example, let us consider the
following time series (i.e. solution of the following recursive equation):
\[
X_{i+1}=\alpha X_{i}+\zeta_{i+1}\allowbreak,\;x_{0}=xo\allowbreak,\;i\geq0.
\]
We assume that random variables $\left\{  \zeta_{i}\right\}  _{i\geq1}$ form
an i.i.d. (independent identically distributed) sequence with zero
expectations. Let us suppose that we are given observations $\left\{
X_{i}\right\}  _{i\geq1}$ and using only them, we would like to estimate the
value of parameter $\alpha$. Can one find a sequence of functions of these
observations that would converge to $\alpha$ ? It turns out that if one
assumes that the sequences $\left\{  X_{i}\zeta_{i+1}\right\}  _{i\geq1}$ and
$\left\{  X_{i}^{2}\right\}  _{i\geq1}$ satisfy strong law of large numbers
and moreover , that $EX_{i}^{2}\neq0$, then such a sequence is defined by the
following formula:
\[
a_{n}=\frac{\sum_{i=1}^{n}X_{i}X_{i+1}}{\sum_{i=1}^{n}X_{i}^{2}}.
\]
To be convinced let us notice that $X_{i+1}X_{i}=\alpha X_{i}^{2}+\zeta
_{i+1}X_{i}$. Moreover, as it can be easily noticed, for every $i\geq1$,
$X_{i}$ is a function of $\zeta_{j}$ for $j\leq i$ (other words is
$\sigma(\zeta_{1},\ldots,\zeta_{i})$ measurable), hence $EX_{i}\zeta_{i+1}=0$.
As a result we have
\[
a_{n}=\frac{\alpha\frac{1}{n}\sum_{i=1}^{n}X_{i}^{2}+\frac{1}{n}\sum_{i=1}%
^{n}X_{i}\zeta_{i+1}}{\frac{1}{n}\sum_{i=1}^{n}X_{i}^{2}}%
\underset{n\rightarrow\infty}{\longrightarrow}\alpha,
\]
with probability $1$.
\end{example}

The above-mentioned examples underline how important is to be sure that a
given sequence satisfies or not a version of the LLN.

Under what assumptions a given sequence of the random variables $\left\{
X_{i}\right\}  _{i\geq1}$ satisfies a version of the law of large numbers will
be presented in detail in chapter \ref{simpwl}.

In this part we will present only one simulation illustrating the law of large
numbers. It will illustrate the example \ref{ident}. First, there were
generated $N=300000$ observations of $\left\{  X_{i}\right\}  _{i\geq1}$ with
$\alpha=.99$ and the sequence of $\left\{  \zeta_{i}\right\}  _{i\geq1}$
consisting an i.i.d. sequence with normal distribution $N(0,3)$. Then one
created sequence $\left\{  a_{n}\right\}  _{n\geq1}$ as in the example
\ref{ident}. The behavior of this sequence is presented below where however,
only the sequence $\left\{  a_{100\ast j}\right\}  _{j=1}^{3000}$ is
presented.%
%TCIMACRO{\FRAME{ftbpF}{3.2136in}{1.9865in}{0pt}{}{}{ident2000.eps}%
%width 3.2136in;  height 1.9865in;  depth 0pt;  original-width 5.7795in;
%original-height 3.5613in;  cropleft "0";  croptop "1";  cropright "1";
%cropbottom "0";  filename 'ident2000.eps';file-properties "XNPEU";}}}%
%BeginExpansion
\begin{figure}[ptb]%
\centering
\includegraphics[
height=1.9865in,
width=3.2136in
]%
{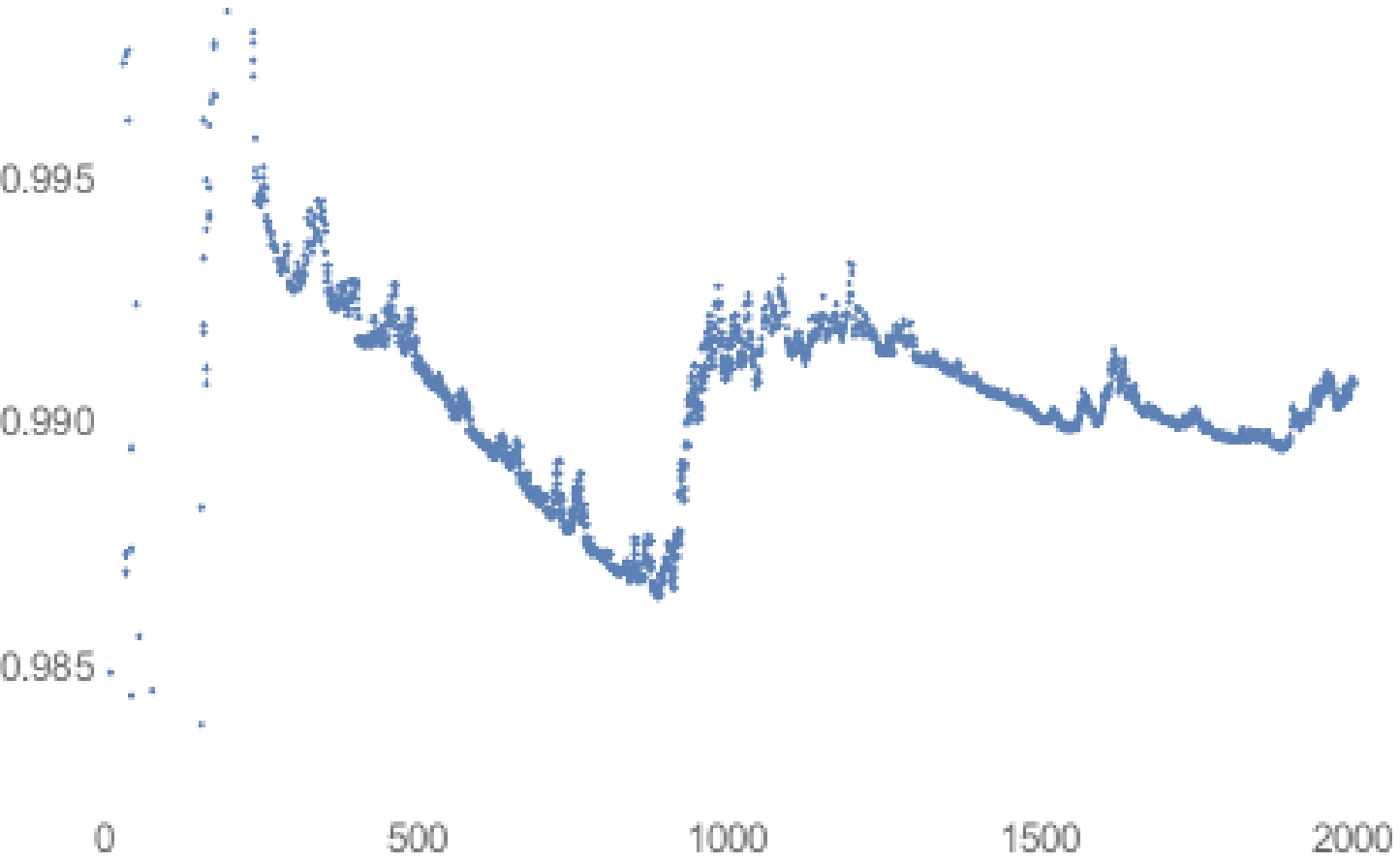}%
\end{figure}
%EndExpansion

\section{Central limit theorem}

\label{sek_ctg}Let $\{X_{n}\}_{n\geq1}$ be a sequence of the random variables
with the finite second moments. Let us denote $m_{n}=EX_{n}$, $v_{n}%
=\operatorname*{var}(X_{n}).$

\begin{definition}%
\index{Central limit theorem}%
We say that the sequence $\{X_{n}\}_{n\geq1}$ satisfies central limit theorem
(CLT), if the following auxiliary sequence :
\[
Y_{N}=\frac{\sum_{n=1}^{N}(X_{n}-m_{n})}{\sqrt{\operatorname*{var}(\sum
_{n=1}^{N}X_{n})}},
\]
converges in distribution to a random variable $N(0,1)$ (normal with zero mean
and variance equal $1$).
\end{definition}

\begin{remark}
$\forall N\geq1:EY_{N}=0,\;\operatorname*{var}(Y_{N})=1.$
\end{remark}

\begin{remark}
Again as before we can present elements of the sequence $\left\{
Y_{N}\right\}  $ in an iterative way. Namely:
\[
Y_{N+1}=(1-\mu_{N})Y_{N}+\mu_{N}^{^{\prime}}(X_{N+1}-m_{N+1}),
\]
where we denoted : $\mu_{N}=1-\sqrt{\frac{\operatorname*{var}(\sum_{n=1}%
^{N}X_{n})}{\operatorname*{var}(\sum_{n=1}^{N+1}X_{n})}}$, $\mu_{N}^{^{\prime
}}=\sqrt{\frac{1}{\operatorname*{var}(\sum_{n=1}^{N+1}X_{n})}}$. Notice that
if random variables $\left\{  X_{n}\right\}  $ are independent, then :
$Y_{N}=\frac{\sum_{n=1}^{N}(X_{n}-m_{n})}{\sqrt{\sum_{n=1}^{N}v_{n}}}$. It is
easy then to notice that under some additional technical assumptions
concerning variances $v_{i}$ of the random variables $X_{i}$ we get: $\mu
_{n}\approx\frac{v_{n+1}}{2\sum_{i=1}^{n+1}v_{i}}$, and $\mu_{n}^{^{\prime}%
}=\frac{1}{\sqrt{\sum_{i=1}^{n+1}v_{i}}}.$
\end{remark}

\begin{remark}
\label{szykosc_CTG}Notice also that if the random variables $\{X_{n}%
\}_{n\geq1}$ posses variances and are not correlated, then
$\operatorname*{var}(\sum_{n=1}^{N}X_{n})=N\operatorname*{var}(X_{1})$.
Moreover, we have:
\[
Y_{N}=\sqrt{\frac{N}{\nu_{1}}}\left[  \frac{1}{N}\sum_{n=1}^{N}(X_{n}%
-m_{n})\right]  .
\]
Assuming that the sequence $\{X_{n}\}_{n\geq1}$, satisfies CLT or equivalently
that the sequence $\left\{  Y_{N}\right\}  _{N\geq1}$ converges weakly to
normal ($N(0,1))$ random variable and remembering that, $\frac{1}{N}\sum
_{n=1}^{N}(X_{n}-m_{n})\rightarrow0$, as $N\rightarrow\infty$ (i.e. LLN is
satisfied) we see that, the fact that CLT\ is satisfied to tell us something
about the speed of convergence in LLN.
\end{remark}

\subsubsection{Criteria for the sequence of independent random variables to
satisfy CLT}

\begin{proposition}
\label{ctg_iid}Let $\{X_{n}\}_{n\geq1}$ be a sequence of i.i.d. random
variables having greater than zero variance. Then the sequence $\{X_{n}%
\}_{n\geq1}$ satisfies CLT.
\end{proposition}

\begin{proof}
Let us denote $EX_{1}=m$ and $\operatorname{var}(X_{1})=\sigma^{2}$. Let
$\varphi(t)$ be a characteristic function of the random variable $\frac
{X_{1}-EX_{1}}{\sigma}$. Since the variance of this random variable is equal
to $1$ and its expectation is equal to $0$ we have $\varphi(t)=1-\frac{t^{2}%
}{2}+o(t^{2})$. Let $\psi_{n}(t)$ be the characteristic function of the random
variable: $Y_{n}=\frac{\sum_{i=1}^{n}(X_{i}-m)}{\sigma\sqrt{n}}$. Obviously it
is related to function $\varphi$ in the following way: $\psi_{n}%
(t)=\allowbreak\varphi^{n}(\frac{t}{\sqrt{n}})$. Let us consider logarithm of
this function and recall that $\log\left(  1-x\right)  =-x+o(x)$, we get:
\[
\log\psi_{n}(t)\allowbreak=n\log(1-\frac{t^{2}}{2n}+o(\frac{t^{2}}%
{n}))\allowbreak=-\frac{t^{2}}{2}+no_{1}(\frac{t^{2}}{n}).
\]
Hence it can be easily seen that for fixed $t$ we have:
$\underset{n\rightarrow\infty}{\lim}\psi_{n}(t)\allowbreak=\allowbreak
\exp(-\frac{t^{2}}{2})$. Now it is enough to recall that $\exp(-\frac{t^{2}%
}{2})$ is the characteristic function of the normal distribution with zero
mean and variance equal to $1$.
\end{proof}

\begin{theorem}
[Lindeberg]%
\index{Theorem!Lindenberg-Feller}%
Let $\{X_{n}\}_{n\geq1}$ be a sequence of independent random variables having
finite second moments. Let us denote: $m_{n}=EX_{n}$, $\sigma_{n}%
^{2}=\operatorname*{var}(X_{n})$, $s_{n}^{2}=\sum_{i=1}^{n}\sigma_{i}^{2}$. If
the sequence $\{X_{n}\}_{n\geq1}$ satisfies the following condition: \newline%
\[
\forall\epsilon>0:\underset{N\rightarrow\infty}{\lim}\frac{1}{s_{N}^{2}}%
\sum_{n=1}^{N}E\left(  X_{n}-m_{n}\right)  ^{2}I(\left\vert X_{n}%
-m_{n}\right\vert >\epsilon s_{N})=0,
\]
then the sequence $\{X_{n}\}_{n\geq1}$ satisfies CLT.
\end{theorem}

\begin{proof}
Proof of this theorem is somewhat complicated and is present in every more
detailed textbook on probability. In particular, one can find it in e.g.
\cite{feller1}
\end{proof}

We will illustrate the Central Limit Theorem with the help of the following example.

\begin{example}
We consider a sequence of independent observations drawn from exponential
distribution $Exp(1)$ i.e. having density $I(x\geq0)\exp(-x)$. Let us denote
these observations by $\left\{  X_{i}\right\}  _{i\geq1}$. It is elementary to
notice that $EX=1$ and $\operatorname*{var}(X)=1$. We constructed histograms
of the random variables : let $T_{i}=X_{i}-1,i\geq1$, $\eta4_{i}=\frac
{T_{4i}+T_{4i+1}+T_{4i+2}+T_{4i+3}}{2}$, $\eta10_{i}=\frac{1}{\sqrt{10}}%
\sum_{k=0}^{9}T_{10i+k}$, $\eta25_{i}=\frac{1}{5}\sum_{k=0}^{24}T_{25i+k}$,
$\eta100_{i}=\frac{1}{10}\sum_{k=0}^{99}T_{100i+k}$, where $i$ ranges in the
first case between $1,\ldots,25000$. In the second case between $1,\ldots
,10000$, in the third case between $1,\ldots,4000$ and between $1,\ldots,1000$
in the last case. Those of the readers who are not familiar with the notion of
the histogram we refer to the beginning of chapter \ref{metody_jadrowe}. The
results were divided by respectively $25000$, $10000$, $4000$ and $1000$. The
following figures were obtained for random variables respectively: $\eta4_{i}$
$\eta10_{i}$, $\eta25_{i}$ and $\eta100_{i}$%

%TCIMACRO{\FRAME{itbpFU}{2.4059in}{1.4987in}{0in}{\Qcb{$\eta4$}}{}%
%{symctg4.eps}{\special{ language "Scientific Word";  type "GRAPHIC";
%width 2.4059in;  height 1.4987in;  depth 0in;  original-width 5.7795in;
%original-height 3.589in;  cropleft "0";  croptop "1";  cropright "1";
%cropbottom "0";  filename 'symCTG4.eps';file-properties "XNPEU";}}}%
%BeginExpansion
{\parbox[b]{2.4059in}{\begin{center}
\includegraphics[
height=1.4987in,
width=2.4059in
]%
{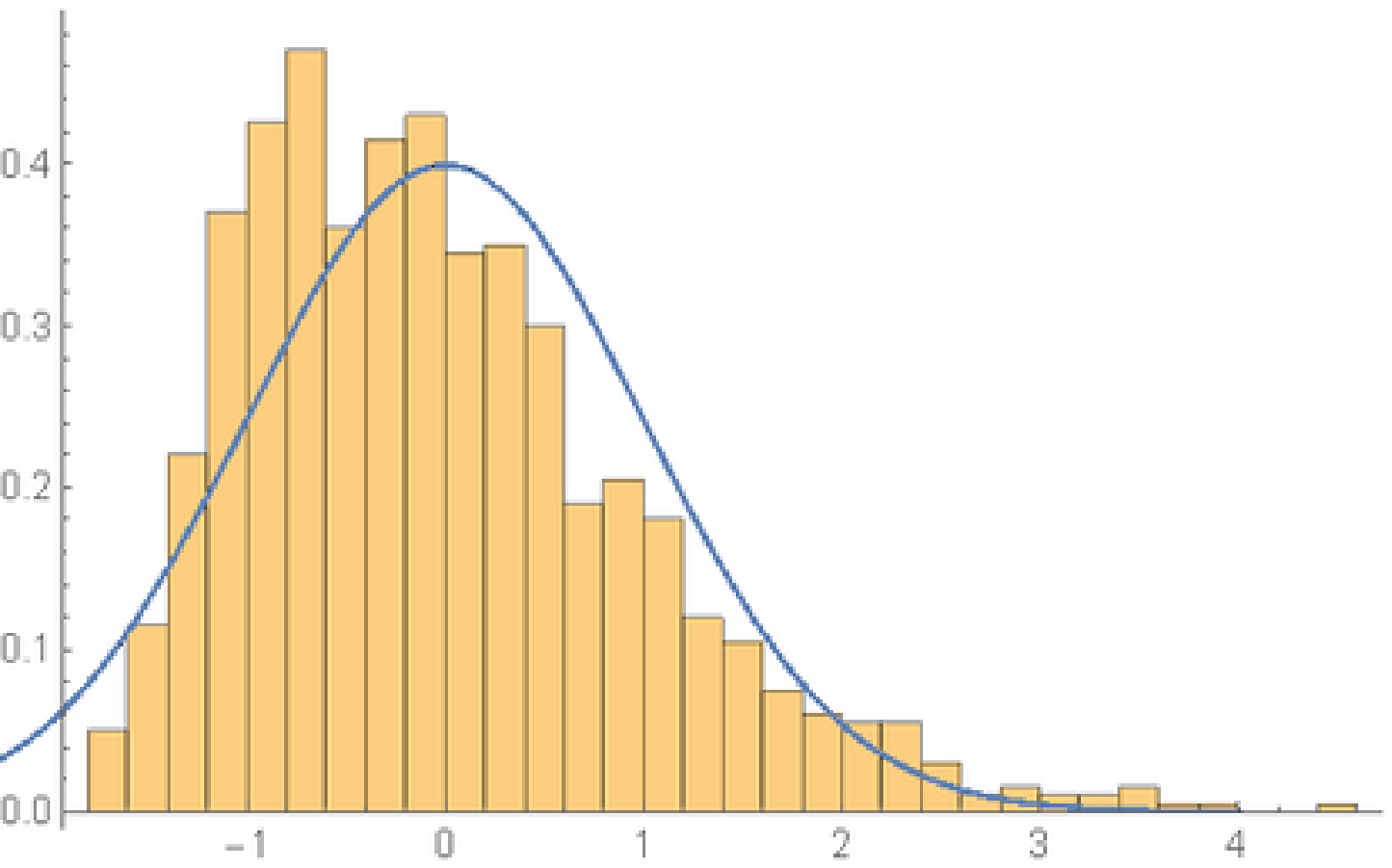}%
\\
$\eta4$%
\end{center}}}%
%EndExpansion%
%TCIMACRO{\FRAME{itbpFU}{2.1715in}{1.3534in}{0in}{\Qcb{$\eta10$}}%
%{}{symctg10.eps}{\special{ language "Scientific Word";  type "GRAPHIC";
%width 2.1715in;  height 1.3534in;  depth 0in;  original-width 5.7795in;
%original-height 3.589in;  cropleft "0";  croptop "1";  cropright "1";
%cropbottom "0";  filename 'symCTG10.eps';file-properties "XNPEU";}}}%
%BeginExpansion
{\parbox[b]{2.1715in}{\begin{center}
\includegraphics[
height=1.3534in,
width=2.1715in
]%
{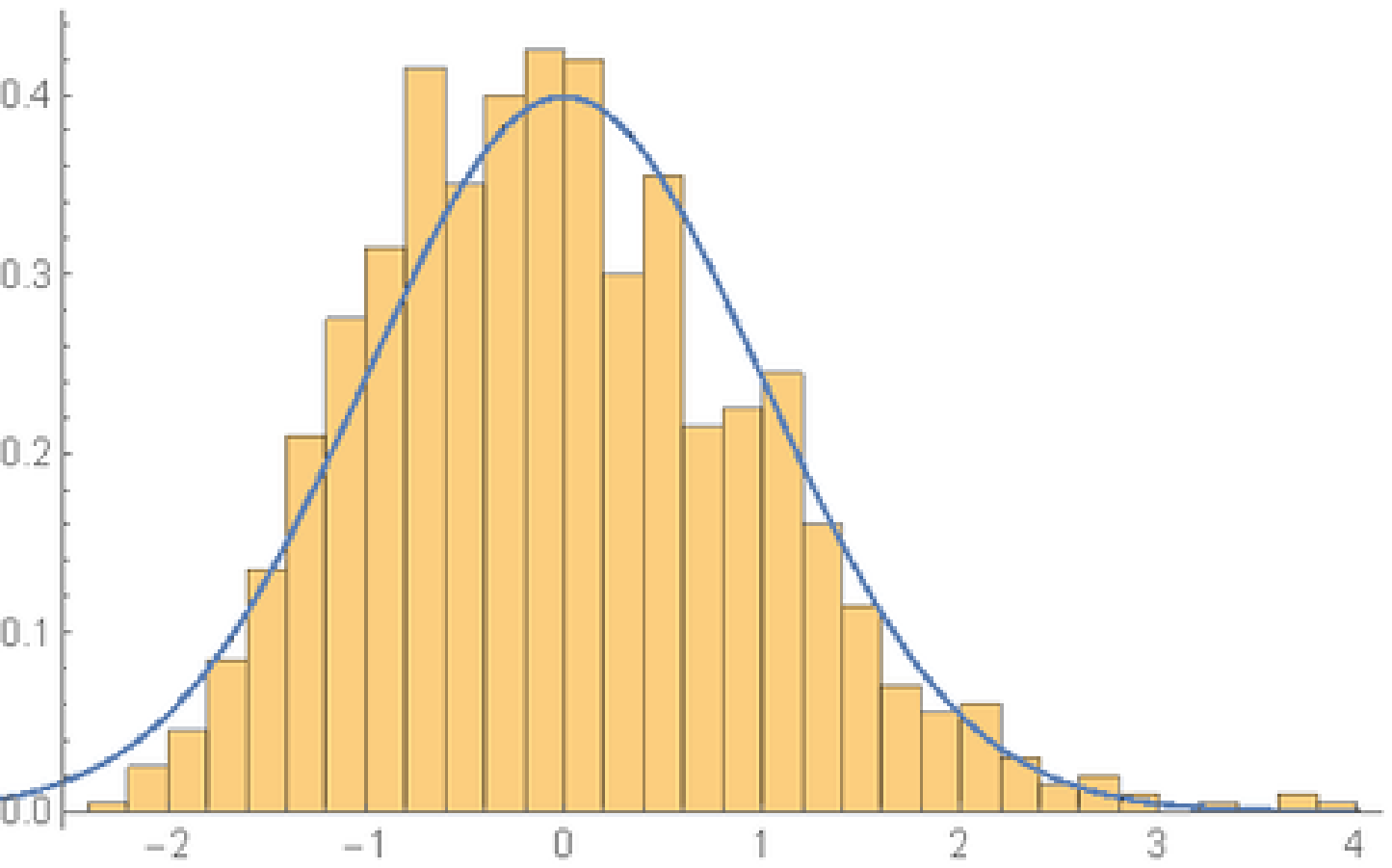}%
\\
$\eta10$%
\end{center}}}%
%EndExpansion
%

%TCIMACRO{\FRAME{itbpFU}{2.4431in}{1.5212in}{0in}{\Qcb{$\eta25$}}%
%{}{symctg25.eps}{\special{ language "Scientific Word";  type "GRAPHIC";
%width 2.4431in;  height 1.5212in;  depth 0in;  original-width 5.7795in;
%original-height 3.589in;  cropleft "0";  croptop "1";  cropright "1";
%cropbottom "0";  filename 'symCTG25.eps';file-properties "XNPEU";}}}%
%BeginExpansion
{\parbox[b]{2.4431in}{\begin{center}
\includegraphics[
height=1.5212in,
width=2.4431in
]%
{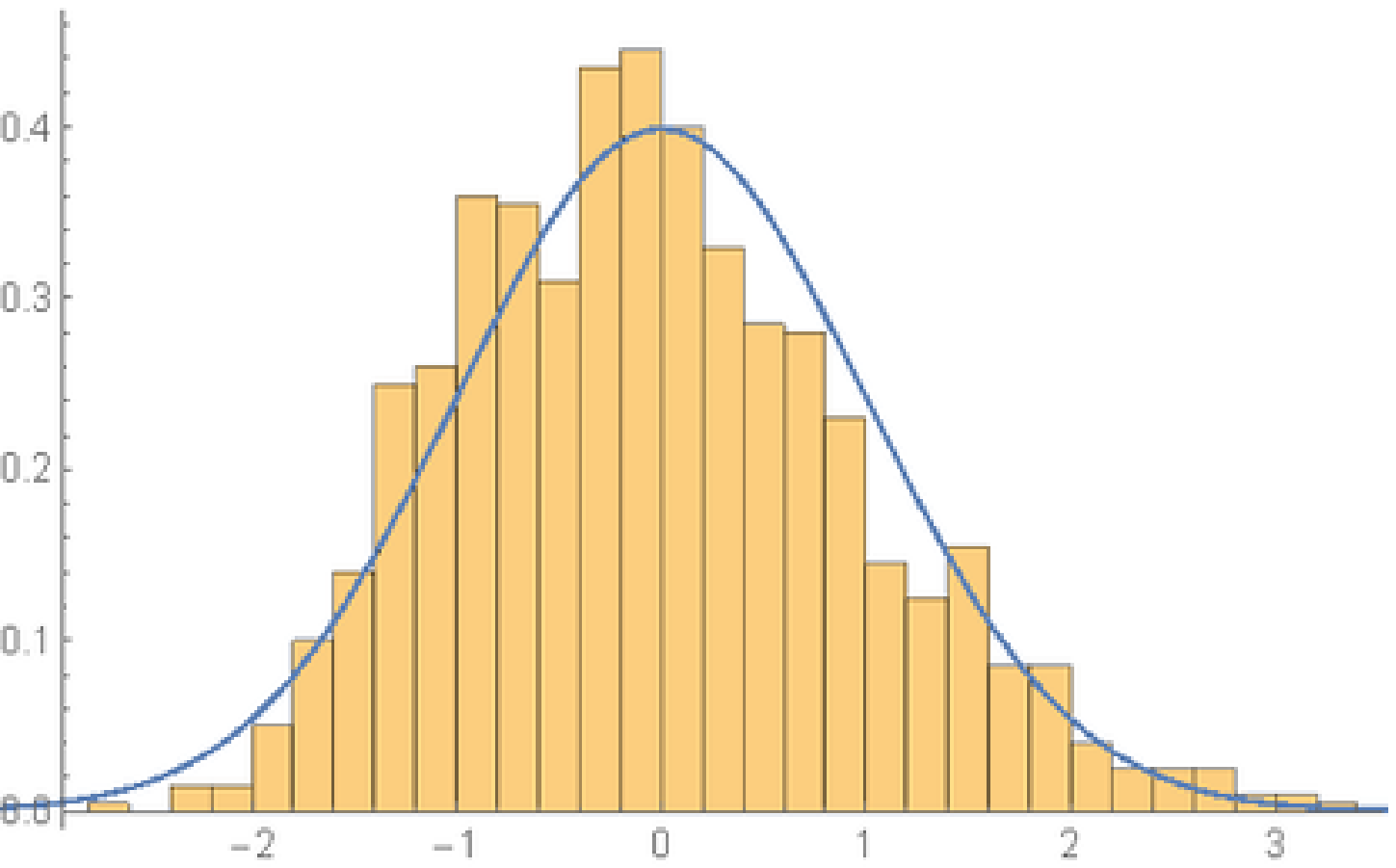}%
\\
$\eta25$%
\end{center}}}%
%EndExpansion%
%TCIMACRO{\FRAME{itbpFU}{2.1612in}{1.3465in}{0in}{\Qcb{$\eta100$}}{\Qlb{h}%
%}{symctg100.eps}{\special{ language "Scientific Word";  type "GRAPHIC";
%width 2.1612in;  height 1.3465in;  depth 0in;  original-width 5.7795in;
%original-height 3.589in;  cropleft "0";  croptop "1";  cropright "1";
%cropbottom "0";  filename 'symCTG100.eps';file-properties "XNPEU";}}}%
%BeginExpansion
{\parbox[b]{2.1612in}{\begin{center}
\includegraphics[
height=1.3465in,
width=2.1612in
]%
{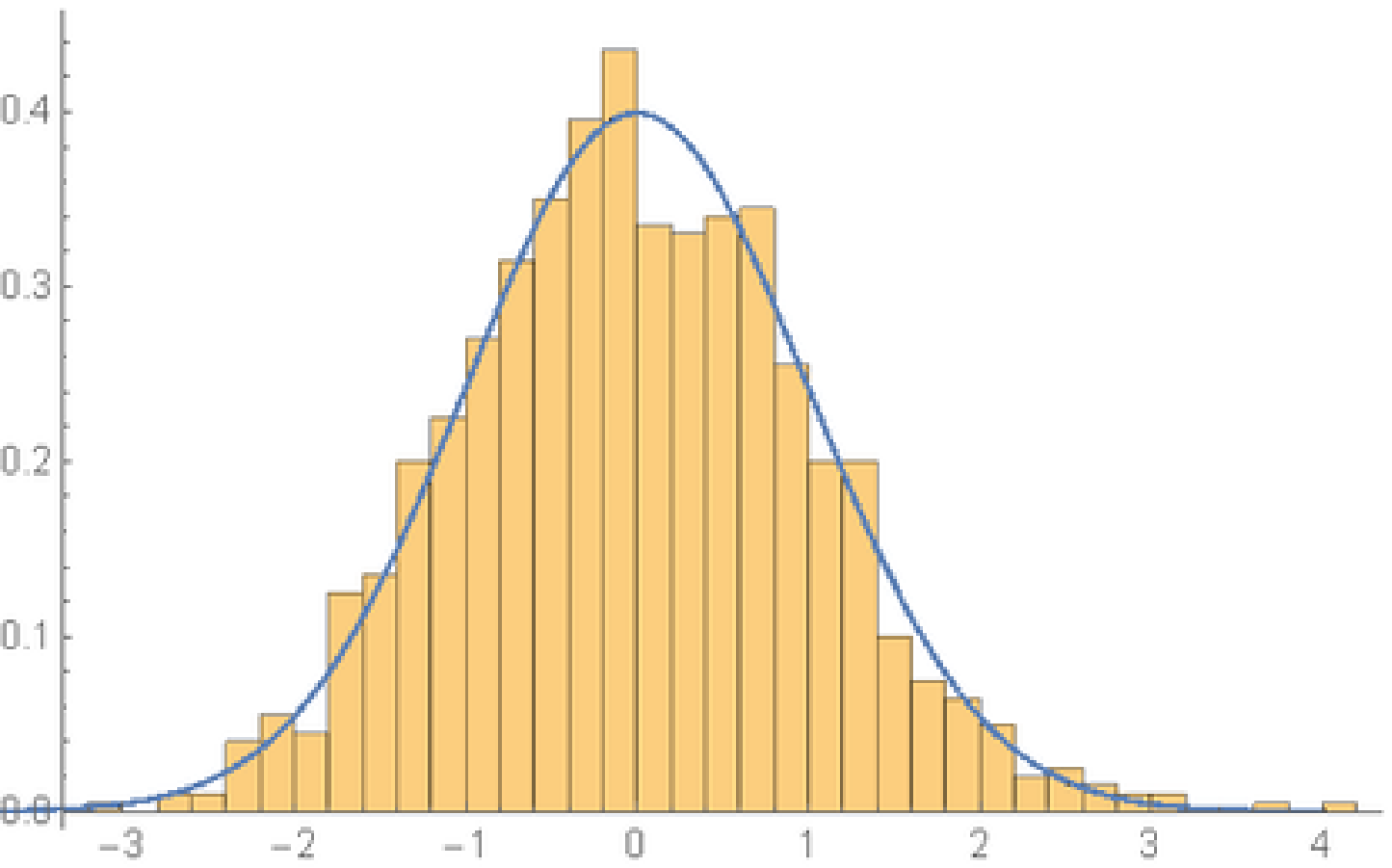}%
\\
$\eta100$%
\end{center}}}%
%EndExpansion

\end{example}

\section{Law of iterated logarithm}

Let $\{X_{n}\}_{n\geq1}$ be a sequence of the random variables with finite
second moments. Let us denote $m_{n}=EX_{n}$, $v_{n}=\operatorname*{var}%
(X_{n})$ and $s_{n}^{2}=\operatorname*{var}(\sum_{i=1}^{n}X_{i})$

\begin{definition}%
\index{Law!of Iterated logarthm}%
We say that the sequence $\left\{  X_{i}\right\}  _{i\geq1}$ satisfies Law of
Iterated Logarithm (LIL), if:
\begin{align*}
\underset{n\rightarrow\infty}{\lim\inf}\,\frac{\sum_{i=1}^{n}(X_{i}-m_{i}%
)}{\sqrt{2s_{n}^{2}\log\log s_{n}^{2}}}  &  =1,\\
\underset{n\rightarrow\infty}{\lim\inf}\,\frac{\sum_{i=1}^{n}(X_{i}-m_{i}%
)}{\sqrt{2s_{n}^{2}\log\log s_{n}^{2}}}  &  =-1.
\end{align*}

\end{definition}

Law of iterated logarithm is in fact a statement about the speed of
convergence in LLN. This time one can estimate this speed quite precisely
(compare remarks concerning CTG in particular \ref{szykosc_CTG}). It can be
clearly seen if one assumes, that $\{X_{n}\}_{n\geq1}$ is a sequence of
uncorrelated random variables with zero mean a identical variances equal
$\sigma^{2}$. If for this sequence LIL is satisfied then we have:
\[
\underset{n\rightarrow\infty}{\lim\inf}\,\frac{\sum_{i=1}^{n}X_{i}}{\sigma
n}\sqrt{\frac{n}{2\log\log n}}=1\text{ and }\underset{n\rightarrow\infty
}{\lim\inf}\,\frac{\sum_{i=1}^{n}X_{i}}{\sigma n}\sqrt{\frac{n}{2\log\log n}%
}=-1,
\]
since $\frac{\log\log n}{\log\log(n\sigma^{2})}\cong1$ for any $\sigma^{2}>0$ .

LIL is not satisfied by any sequences of the random variables. Majority of
results concern sequences of independent random variables (see e.g. papers of
\cite{Hartman41}, \cite{Strassen65}, \cite{Strass65}, were known earlier
results were generalized). There exist also results concerning sequences of
dependent random variables e.g. so-called martingale differences (for the
definition of martingales see Appendix \ref{martyngaly}
page\pageref{martyngaly}).

Moreover, one can present random variables
\[
Z_{n}=\frac{\sum_{i=1}^{n}(X_{i}-m_{i})}{\sqrt{2s_{n}^{2}\log\log s_{n}^{2}}%
},
\]
in an iterative way. Namely, we have:
\[
Z_{n+1}=\left(  1-\frac{d_{n+1}-d_{n}}{d_{n+1}}\right)  Z_{n}+\frac{1}%
{d_{n+1}}(X_{n+1}-m_{n+1}),
\]
where we denoted%
\[
d_{n}=\sqrt{2s_{n}^{2}\log\log s_{n}^{2}},
\]
similarly, as it was done in the previous section when we discussed CLT.
Unfortunately, , as before, the iterative form helps in the analysis, only a
little. Methods presented in chapter \ref{zbiez} should be modified and
improved in order to be applied in the analysis of LIL or CLT. It is a
challenge for the astute reader. To analyze LIL and CLT other methods were
developed that not necessarily utilize iterative forms. These methods are not
in the main course of this book hence we will present them briefly just to
give the readers the scent of the difficulties associated with examining these
two random phenomena.

As it was mentioned the majority of papers dedicated to LIL concern the case
of independent random variables. This group of papers again can be divided on
the group when the case of identical distributions is concerned. One should
mention here in this group the following Hartman Wintnera Theorem
\cite{Hartman41}

\begin{theorem}%
\index{Thorem!Hartman-Wintner}%
Let $\{X_{n}\}_{n\geq1}$ be a sequence of independent random variables having
identical distributions and such that Then $EX_{1}=0$, $EX_{1}^{2}=\sigma
^{2}<\infty$. Then this sequence satisfies LIL.
\end{theorem}

To give a foretaste of the difficulties that appear while proving LIL we
present proof of the simplified version of the law of iterated logarithm for
the i.i.d. sequence of the random variables having Normal $\mathcal{N}(0,1)$
distribution in Appendix \ref{PIL} .

The figure below presents simulation connected with LIL. Sequence marked
\emph{green} denotes the sequence of partial sums of independent identically
distributed random variables having zero means and positive finite variances.
The sequence marked \emph{blue} denotes partial sums of independent
identically distributed random variables having zero means and having no
variances. More precisely, we took random variables having distribution as
$\operatorname*{sgn}(C)\sqrt{\left\vert C\right\vert }$, where $C$ has Cauchy
distribution. Let us notice that the law of iterated logarithm can be
interpreted in the following way. Let: $S_{n}=\sum_{i=1}^{n}X_{i}%
\allowbreak:\allowbreak n\geq1$. If the sequence $\{X_{n}\}_{n\geq1}$ satisfy
LIL, then for any $\varepsilon>0$, the events
\[
G_{n}=\left\{  S_{n}\notin(-(1+\varepsilon)\sigma\sqrt{2n\log\log
n},(1+\varepsilon)\sigma\sqrt{2n\log\log n})\right\}  ,
\]
will occur an only finite number of times with probability $1$. Moreover, also
with probability $1$ infinite number of times we will have:
\[
S_{n}>(1-\varepsilon)\sigma\sqrt{2n\log\log n}%
\]
and
\[
S_{n}<-(1-\varepsilon)\sigma\sqrt{2n\log\log n}.
\]
Hence the first of these sequences (green one) rather satisfies LIL ( by
Hartman-Wintner Theorem we know that it satisfies). However the second
sequence (blue) rather does nor satisfy. This is so since the sequence of
partial sums reaches far beyond the area $(-\sigma\sqrt{2n\log\log
n},\allowbreak\sigma\sqrt{2n\log\log n})$, despite very large number of
observations (approximately $8\ast10^{6})$.%
%TCIMACRO{\FRAME{ftbpF}{3.4506in}{1.977in}{0pt}{}{}{pil.eps}%
%width 3.4506in;  height 1.977in;  depth 0pt;  original-width 6.2379in;
%original-height 3.5613in;  cropleft "0";  croptop "1";  cropright "1";
%cropbottom "0";  filename 'PIL.eps';file-properties "XNPEU";}}}%
%BeginExpansion
\begin{figure}[ptb]%
\centering
\includegraphics[
height=1.977in,
width=3.4506in
]%
{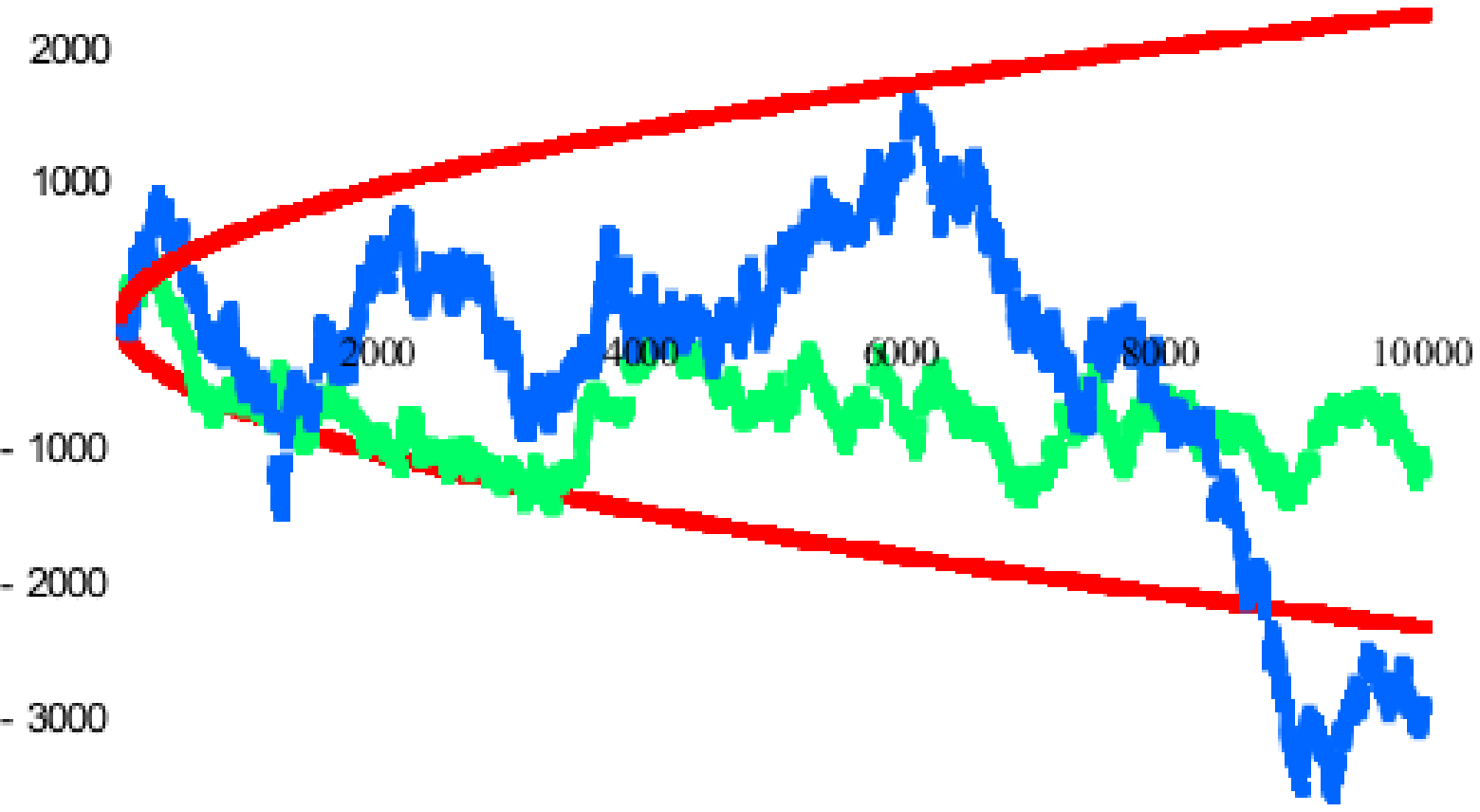}%
\end{figure}
%EndExpansion

Results if this simulation as far as the 'blue' sequence is concerned can be
justified and supported by the following theorem that is, in fact, a reverse
of the law of iterated logarithm.

\begin{theorem}
[V. Strassen]Let $\{X_{n}\}_{n\geq1}$ be a sequence of independent random
variables having identical distributions. If with a positive probability we
have:
\begin{equation}
\underset{n\rightarrow\infty}{\lim\inf}\,\frac{\left\vert \sum_{i=1}^{n}%
X_{i}\right\vert }{\sqrt{2n\log\log n}}<\infty, \label{odwr_pil}%
\end{equation}
then $EX_{1}^{2}<\infty$ and $EX_{1}=0.$
\end{theorem}

Proof of this theorem is placed in Appendix \ref{strassen}.

\section{Iterative form of random phenomena}

Let us sum up the problems presented in previous sections. Let be given a
sequence of the random variables $\left\{  X_{i}\right\}  _{i\geq1}$. In the
case of LLN one has to find conditions under which the sequence $\left\{
\bar{X}_{i}\right\}  _{i\geq1}$ generated by the iterative procedure:
\begin{equation}
\bar{X}_{n+1}=(1-\mu_{n})\bar{X}_{n}+\nu_{n}(X_{n+1}-EX_{n+1}),\;n\geq0,
\label{og_iteracja}%
\end{equation}
and initial condition $\bar{X}_{0}=0$ \textbf{converges almost surely to
zero}. For the LLN we have $\nu_{n}=\mu_{n},\;n\geq0$. Sequence $\left\{
\mu_{n}\right\}  _{n\geq0}$ is defined for the LLN as $\mu_{n}=\frac{1}%
{n+1},\;n\geq0$, while for the generalized LLN as any sequence of positive
numbers such that
\begin{equation}
\mu_{0}=1,\mu_{n}\in(0,1)\text{, when }n\geq1\;\text{and\ }\sum_{n\geq0}%
\mu_{n}=\infty. \label{og_ciag}%
\end{equation}

As far as the law of iterated logarithm is concerned, we have to give
condition under which the sequence of the random variables generated by the
procedure (\ref{og_iteracja}) is bounded with probability one. One has to find
also the limits: $\underset{n\rightarrow\infty}{\lim\sup\,}\bar{X}_{n}$ and
$\underset{n\,\longrightarrow\infty}{\lim\inf}\,\bar{X}_{n}$. Number sequences
$\left\{  \mu_{n}\right\}  _{n\geq0}$ and $\left\{  \nu_{n}\right\}  _{n\geq
0}$ are in this case the following: $\mu_{n}=\frac{d_{n+1}-d_{n}}{d_{n+1}%
}\equiv O(\frac{1}{2(n+1)})$, $\nu_{n}=\frac{1}{d_{n+1}}$, where $d_{n}%
=\sqrt{2s_{n}^{2}\log\log s_{n}^{2}}$, $s_{n}^{2}=\operatorname*{var}%
(\sum_{i=1}^{n}X_{i})$, when $n\geq3$ and $d_{n}=1$, $n=0,1,2.$

In the case of CLT one has to give\emph{\ }conditions under which sequence of
the random variables generated by the iterative procedure (\ref{og_iteracja})
converges in distribution to the Normal one. Number sequences $\left\{
\mu_{n}\right\}  _{n\geq0}$ and $\left\{  \nu_{n}\right\}  _{n\geq0}$ are in
this case given by $\mu_{n}=\frac{s_{n+1}-s_{n}}{s_{n+1}}$, $\nu_{n}=\frac
{1}{s_{n}}$, where $s_{n}^{2}$ is defined in the same way as in the case of
law of iterated logarithm.

Let us notice that the differences between these problems can be reduced to
considering number sequences $\left\{  \mu_{i}\right\}  _{i\geq0}$ and
$\left\{  \nu_{i}\right\}  _{i\geq0}$, (different sets for a different
problem) and also considering a different type of convergence. The form of the
recursive equation is the same in all these three cases. As it will turn out
in chapter \ref{simpwl} due to such a general approach and getting acquainted
with the general properties of iterative procedures, that we were able to
depart from the traditional assumptions traditionally assumed in the case of
LLN (independence of elements of the sequence $\left\{  X_{i}\right\}
_{i\geq1}$, considering more general number sequences $\left\{  \mu
_{i}\right\}  _{i\geq0}$ than the traditional $\mu_{i}=\frac{1}{i+1})$. Can
one do the same in the case of CLT or LIL? It is not known. If it can be done,
then it is very difficult and the methods developed in this book are not sufficient.

\chapter{Convergence of iterative procedures\label{zbiez}}

In this chapter, we have included facts, methods, tools and mental schemes
that will be used in the sequel. It is essential, very important for the
farther parts of the book.

\section{Auxiliary facts\label{pom_zbiez}}

\begin{proposition}
\label{w-ocz} Let $X$ be nonnegative and integrable random variable and let
$F$, be its cumulative distribution function (cdf). Then,
\[
EX=\int_{0}^{\infty}P(X>t)dt=\int_{0}^{\infty}(1-F(t))dt.
\]

\end{proposition}

\begin{proof}
Integrating by parts we have for $T>0$: $\int_{0}^{T}xdF(x)=TF(T)-\int_{0}%
^{T}F(x)dx$\allowbreak$=-T[1-F(T)]+\int_{0}^{T}[1-F(x)]dx$. However
$T[1-F(T)]=T\int_{T}^{\infty}dF(x)<\int_{T}^{\infty}xdF(x)\rightarrow0$, as
$T\rightarrow\infty$, since $\int_{0}^{\infty}xdF(x)<\infty.$
\end{proof}

\begin{remark}
\label{calkowalnosc}The above mentioned proposition has its discrete version.
Namely, notice that $\sum_{i=1}^{\infty}I(X\geq i)$ $\allowbreak=\allowbreak$
integer part of nonnegative random variable $X$ ( for random variables
assuming nonnegative integer values we have $X\allowbreak=\allowbreak
\sum_{i=1}^{\infty}I(X\geq i)$ with probability $1$). Hence, we have for all
elementary events:
\begin{equation}
\sum_{i=1}^{\infty}I(X\geq i)\leq X<1+\sum_{i=1}^{\infty}I(X\geq i).
\label{EX_dod}%
\end{equation}
Thus, we see that nonnegative random variable $X$ is integrable if and only
if
\[
\sum_{i\geq0}P(X\geq i)<\infty.
\]
Moreover, for random variables assuming nonnegative integer values we have:
\begin{equation}
EX=\sum_{i\geq1}P(X\geq i). \label{EX_dod_calk}%
\end{equation}

\end{remark}

As an immediate application of this remark, we have the following interesting proposition.

\begin{proposition}
\label{wlasnosc_iid}Let $\{X_{n}\}_{n\geq1}$ be a sequence of independent
random variables having identical distributions. Moreover, let us assume that
$E\left\vert X_{1}\right\vert =\infty$. Then
\[
\forall k>0:\sum_{i\geq1}P(\left\vert X_{i}\right\vert \geq ki)=\infty,
\]
and
\[
\underset{n\rightarrow\infty}{\lim\sup\,}\frac{\left\vert X_{n}\right\vert
}{n}=\infty.
\]

\end{proposition}

\begin{proof}
We see that for any $k>0$ random variable $\left\vert X_{1}\right\vert /k$
also is not integrable. Hence, we have:
\[
\sum_{i\geq1}P(\left\vert X_{1}\right\vert \geq ki)=\infty,
\]
basing on the remark \ref{calkowalnosc}. Since $\{X_{n}\}_{n\geq1}$ have the
same distributions we have
\[
\sum_{i\geq1}P(\left\vert X_{i}\right\vert \geq ki)=\infty.
\]
Now we apply assertion $ii)$ of Borel-Cantelli lemma (see Appendix
\ref{Borel-Cantelli}) and deduce that the events $\left\{  \left\vert
X_{i}\right\vert \geq ki\right\}  ,\allowbreak i\geq1$ occur infinite number
of times for any $k$. This means however that $\underset{n\rightarrow
\infty}{\lim\sup\,}\frac{\left\vert X_{n}\right\vert }{n}=\infty.$
\end{proof}

It turns out that in the case of any random variables having finite
expectations we have:
\[
EX=\int_{0}^{\infty}[1-F(x)+F(-x)]dx,
\]
and that the necessary condition for the existence of expectation is the
following one:
\[
\underset{x\rightarrow-\infty}{\lim}xF(x)=\underset{x\rightarrow\infty}{\lim
}x[1-F(x)]=0.
\]

There exist a generalization of the statement \ref{w-ocz}, namely:

\begin{proposition}
If $X$ is a nonnegative random variable possessing moment of order $\alpha
\geq1$, then :
\[
EX^{\,\alpha}=\alpha\int_{0}^{\infty}x^{\alpha-1}[1-F(x)]dx.
\]

\end{proposition}

\begin{proof}
Proof will be omitted. It is very similar to the proof of the proposition
\ref{w-ocz}.
\end{proof}

In order to formulate theorems concerning almost sure convergence of the
sequences of the random variables, and formulate conditions in terms of
moments of these random variables it is useful to remember about the following
simple$\allowbreak$ facts :

\begin{lemma}
[Fatou]\label{Fatou}%
\index{Lemma!Fatou}%
Let $\left\{  X_{i}\right\}  _{i\geq1}$ be a sequence of nonnegative,
integrable random variables. Then
\[
E\,\underset{i\rightarrow\infty}{\lim\inf}\,X_{i}\leq\underset{i\,\rightarrow
\infty}{\lim\inf}\,EX_{i}.
\]

\end{lemma}

\begin{theorem}
[Lebesgue'a]%
\index{Theorem!Lebesgue}%
Let $\left\{  X_{i}\right\}  _{i\geq1}$ be a monotone, nonnegative sequence of
the random variables ( i.e. $X_{i}\uparrow X$ or $X_{i}\downarrow X$ a.s.),
then
\[
EX_{n}\underset{n\rightarrow\infty}{\longrightarrow}EX.
\]

\end{theorem}

\begin{proof}
Proof of the Lebesgue theorem as well as of Fatou's lemma one can find in any
book on analysis containing a theory of integration. One can find it in e.g.
\cite{lojasiewicz}.
\end{proof}

\begin{corollary}
\label{zbieznoscbezwzgledna} If the series $\sum_{i\geq1}E\left\vert
X_{i}\right\vert $ converges, then the series $\sum_{i\geq1}X_{i}$ converges
almost surely.
\end{corollary}

\begin{proof}
A sequence of the random variables $\left\{  \sum_{i=1}^{n}\left\vert
X_{i}\right\vert \right\}  _{n\geq1}$ is increasing almost surely, hence by
the Lebesgue theorem, if only the sequence of its expectations converges to a
finite limit, then the sequence converges almost surely to an integrable
limit, that is obviously finite almost surely. Hence, the series $\sum
_{i\geq1}X_{i}$ converges absolutely for almost every $\omega$. In particular,
it converges also conditionally.
\end{proof}

It will turn out that the following notion of \emph{uniform integrability of
the family of the random variables} and its properties are of use.

\begin{proposition}
\label{alfa_calk}Let us assume that the sequence $\{X_{n}\}_{n\geq1}$
converges in probability to $X$, and moreover , that for some $\alpha>1$
\[
\underset{n}{\sup}E\left\vert X_{n}\right\vert ^{\alpha}<\infty,
\]
then
\[
E\underset{n\rightarrow\infty}{\lim}X_{n}=\underset{n\rightarrow\infty}{\lim
}EX_{n},
\]
and
\[
E\left\vert X_{n}-X\right\vert \underset{n\rightarrow\infty}{\rightarrow}0.
\]

\end{proposition}

\begin{proof}
Can be found in Appendix \ref{UnInt}.
\end{proof}

\section{A few numerical lemmas}

\label{lemliczbowe} The lemmas presented below come mainly from papers
\cite{[SZA22]}, \cite{[SZA320]} and \cite{Sza872}.

We will start by recalling some basic facts.

\begin{proposition}
\label{o_exp}$\forall$ $x\in%
%TCIMACRO{\U{211d} }%
%BeginExpansion
\mathbb{R}
%EndExpansion
:\exp(-x)\geq1-x.$
\end{proposition}

\begin{proof}
follows directly convexity of the function $\exp(-x).$
\end{proof}

\begin{proposition}
\label{o_iloczynach}Let $\left\{  a_{n}\right\}  _{n\geq1}$ be a number sequence.

$i)$ If
\[
\exists N\,\forall n\geq N:a_{n}\geq0\text{ or }a_{n}\leq0,
\]
then an infinite product $\prod_{i\geq1}(1+a_{n})$ converges if and only if,
the series $\sum_{i\geq1}a_{n}$ converges.

$ii)$ If the series $\sum_{i\geq1}a_{n}$ and $\sum_{i\geq1}a_{n}^{2}$ are
convergent then convergent is also the infinite product $\prod_{i\geq
1}(1+a_{n}).$
\end{proposition}

\begin{proof}
Can be found in e.g. second volume of \cite{Fihtenholtz64}
\end{proof}

\begin{lemma}
\label{zmaxem} Let $\left\{  d_{i}\right\}  _{i\geq1}$, $\left\{  \epsilon
_{i}\right\}  _{i\geq1}$, $\left\{  \lambda_{i}\right\}  _{i\geq1}$ be three
nonnegative number sequences such that:
\[
{\large \exists}N{\large \,\,\,\forall}n{\large \geq}N\,\,\,d_{n+1}\leq
\lambda_{n}\max(d_{n},\epsilon_{n}).
\]
If only
\[
\underset{n,k}{\sup}\prod_{i=k}^{n}\lambda_{i}<\infty\text{, and
}{\large \forall}k:\prod_{i=k}^{n}\lambda_{i}\underset{n\rightarrow
\infty}{\rightarrow}0,
\]
then:
\[
\underset{k\rightarrow\infty}{\lim\inf}\,\;d_{k}\underset{n\geq k}{\sup}%
\prod_{i=k}^{n}\lambda_{i}\leq\underset{k\rightarrow\infty}{\lim\sup
}\;\epsilon_{k}\underset{n\geq k}{\sup}\prod_{i=k}^{n}\lambda_{i}.
\]

\end{lemma}

\begin{proof}
Let us denote $J_{n,k}=\prod_{i=k}^{n}\lambda_{i}$, $q_{k}\allowbreak
=\allowbreak\sup J_{n,k}$. From assumptions \emph{\ }it follows that $\forall
k\,\,\,\,J_{n,k}$\emph{\ }$\underset{n\rightarrow\infty}{\rightarrow}0$, and
that $\underset{k}{\sup\,}q_{k}<\infty$. If $\underset{k\rightarrow
\infty}{\lim\inf}\,\allowbreak\epsilon_{k}\allowbreak\underset{n}{\sup}%
\prod_{i=k}^{n}\lambda_{i}\allowbreak=\infty$ then, the lemma is true. Let us
assume that this quantity is finite. Let us suppose also that the lemma is not
true. Then there exists such constant $\theta$ and a sequence $\left\{
k_{i}\right\}  $ of naturals that $d_{k_{i}}q_{k_{i}}\geq\theta
>\underset{k\rightarrow\infty}{\lim\inf}\,\epsilon_{k}q_{k}$. Let us denote
$\mathcal{M\allowbreak=\allowbreak}\left\{  i:d_{i}q_{i}\geq\theta\right\}  $.
By definition of the upper bound we have $\exists j$ $\forall i\geq j:$
$q_{i}\epsilon_{i}<\theta$. Let us set $\mathcal{K=M}\cap\allowbreak\left\{
i:i\geq j\right\}  \cap\allowbreak\left\{  n:n\geq N\right\}  $,
$k=\inf\mathcal{K}$. $k$ exists and belongs to $\mathcal{K}$, since
$\mathcal{K}$ is a subset of natural numbers. Let us take any $m\in
\mathcal{K}$ such that $m>k$. Then we have $m>j$ and
\begin{align*}
q_{m-1}\epsilon_{m-1}  &  <\theta\leq q_{m}d_{m}\leq q_{m}\lambda_{m-1}%
\max(d_{m-1},\epsilon_{m-1})\\
&  \leq q_{m-1}\max(d_{m-1},\epsilon_{m-1}).
\end{align*}
Hence $q_{m-1}d_{m-1}\geq\theta$. It means that $m-1\in\mathcal{K}$. Similarly
one can show that $m-2,\ldots,m-(m-k-1)\in\mathcal{K}$. Since $m$ was selected
to any member of $\mathcal{K}$ greater than any $k$ we see that $\forall m\geq
k$, $m\in\mathcal{K}$. This is however, impossible since taking $n$ big
enough, to satisfy $J_{n,k}d_{k}\sup q_{m}<\theta$ (our assumptions assure
that it is possible) and using definition of the set $\mathcal{K}$ we get:
\begin{align*}
\theta &  \leq q_{n+1}d_{n+1}\leq\max(q_{n+1}J_{n,k}d_{k},q_{n+1}%
J_{n,k}\epsilon_{k},\ldots,q_{n+1}J_{n,n}\epsilon_{n})\\
&  \leq\max(q_{n+1}J_{n,k}d_{k},\underset{m\geq k}{\sup}\epsilon_{m}%
q_{m})<\theta.
\end{align*}
since obviously by the definition of $J_{n,k}$ and $q_{k}$, $n\geq k\geq1$ we
have
\[
q_{n+1}J_{n,i}=J_{n+1,i}\leq q_{i}\text{ for }1\leq i\leq n.
\]
Thus, $\mathcal{K}$ has to be finite.
\end{proof}

\begin{lemma}
\label{iteracje} Let $\left\{  \mu_{n}\right\}  _{n\geq0}$ be a number
sequence such that \newline\emph{i) }$\mu_{0}=1,\mu_{n}\in(0,1),$ $n\geq1$,
$\sum_{n\geq0}\mu_{n}=\infty,\,\mu_{n}\underset{n\rightarrow\infty
}{\rightarrow}0$, \newline Let further $\left\{  x_{n}\right\}  _{n\geq0}$ and
$\left\{  b_{n}\right\}  _{n\geq0}$ be such nonnegative number sequences, that
\newline\emph{ii)}
\[
x_{n+1}=(1-\mu_{n})x_{n}+\mu_{n}b_{n},\,\,n=0,1,\ldots\,.
\]
Then:
\[
\underset{n\rightarrow\infty}{\lim\inf}\,b_{n}\leq\underset{n\rightarrow
\infty}{\lim\inf}\,x_{n}\leq\underset{n\rightarrow\infty}{\lim\inf}\,x_{n}%
\leq\underset{n\rightarrow\infty}{\lim\inf}\,b_{n}.
\]

\end{lemma}

\begin{proof}
We will prove first inequality $\underset{n\rightarrow\infty}{\lim\inf}%
\,x_{n}\leq\underset{n\rightarrow\infty}{\lim\inf}\,b_{n}$. Let us take
$\varepsilon\in(0,1)$ and consider inequality
\[
x_{n+1}\leq(1-\varepsilon\mu_{n})x_{n}.
\]
It is true, when $(\varepsilon-1)\mu_{n}x_{n}+\mu_{n}b_{n}\leq0$ or
equivalently, when
\[
x_{n}\geq\frac{b_{n}}{1-\varepsilon}\overset{df}{=}\epsilon_{n}.
\]
Let us now suppose, that $x_{n}<\epsilon_{n}$. We have then
\[
x_{n+1}=(1-\mu_{n})x_{n}+\mu_{n}b_{n}<(1-\mu_{n})\epsilon_{n}+\mu_{n}%
b_{n}=(1-\varepsilon\mu_{n})\epsilon_{n}.
\]
Hence in both cases for any $\varepsilon\in(0,1)$ we have:
\[
x_{n+1}\leq(1-\varepsilon\mu_{n})\max(x_{n},\epsilon_{n}).
\]
Since $\sum_{n\geq0}\mu_{n}=\infty$ and $\mu_{n}\underset{n\rightarrow
\infty}{\longrightarrow}0$, we have $\underset{n\rightarrow\infty}{\lim\inf
}(1-\varepsilon\mu_{n})>0$, and moreover , $\forall k\in%
%TCIMACRO{\U{2115} }%
%BeginExpansion
\mathbb{N}
%EndExpansion
:$ $\prod_{n=k}^{N}(1-\varepsilon\mu_{n})\underset{N\rightarrow\infty
}{\rightarrow}0$ and
\[
\underset{N\geq k}{\sup}\prod_{n=k}^{N}(1-\varepsilon\mu_{n})\allowbreak
=\allowbreak(1-\varepsilon\mu_{k}).
\]
Now we apply Lemma \ref{zmaxem} and get
\[
\underset{n\rightarrow\infty}{\forall\varepsilon\in(0,1):\lim\sup
}(1-\varepsilon\mu_{n})x_{n}\leq\underset{n\rightarrow\infty}{\lim\sup
}(1-\varepsilon\mu_{n})\epsilon_{n}.
\]
Since, that $\mu_{n}\rightarrow0$ it is easy to get desired inequality.

We will prove now inequality $\underset{n\rightarrow\infty}{\lim\inf}%
\,b_{n}\leq\underset{n\rightarrow\infty}{\lim\inf}\,x_{n}$. Firstly, let us
notice that if $\underset{n\rightarrow\infty}{\lim\inf}\,b_{n}=0$, then there
is nothing to prove. Hence, let us assume that $\underset{n\rightarrow
\infty}{\lim\inf}\,b_{n}>0$ and let $j$ be the smallest index for which
$b_{j}>0$. Then let us notice that $\forall n>j$ $:x_{n}>0$. For $n>j$ let us
denote $z_{n}=1/x_{n}$. We have then:
\[
z_{n+1}=\frac{z_{n}}{(1-\mu_{n})+\mu_{n}b_{n}z_{n}}.
\]
Let us consider inequality
\[
z_{n+1}=\frac{z_{n}}{(1-\mu_{n})+\mu_{n}b_{n}z_{n}}\leq(1-\varepsilon\mu
_{n})z_{n},
\]
for some $1>\varepsilon>0$. It is true, when
\[
z_{n}\geq\frac{1+\varepsilon-\varepsilon\mu_{n}}{b_{n}(1-\varepsilon\mu_{n}%
)}\overset{df}{=}\epsilon_{n}.
\]
Further for $z_{n}<\epsilon_{n}$ we have
\[
z_{n+1}<\frac{\epsilon_{n}}{(1-\mu_{n})+\mu_{n}b_{n}\epsilon_{n}%
}=(1-\varepsilon\mu_{n})\epsilon_{n},
\]
since function $f(x)=\frac{x}{A+Bx}$ is decreasing for $x>0$ when $A,B>0$, and
Moreover, $\epsilon_{n}$ was defined in such way as to satisfy equation:
\[
\frac{1}{(1-\mu_{n})+\mu_{n}b_{n}\epsilon_{n}}=(1-\varepsilon\mu_{n}).
\]
Hence, in both cases
\[
z_{n+1}\leq(1-\varepsilon\mu_{n})\max(z_{n},\epsilon_{n}).
\]
Thus, utilizing Lemma \ref{zmaxem}, we get
\[
\underset{n\rightarrow\infty}{\lim\sup\,}z_{n}=\allowbreak\frac{1}%
{\underset{n\,\longrightarrow\infty}{\lim\inf}\,x_{n}}\leq
\underset{n\rightarrow\infty}{\,\lim\sup}\epsilon_{n}=\allowbreak
\frac{1+\varepsilon}{\underset{n\,\longrightarrow\infty}{\lim\inf}\,b_{n}}.
\]
Since $\varepsilon$ was any number this proves our inequality.
\end{proof}

\begin{corollary}
\label{nieiter} Let $\left\{  \mu_{n}\right\}  _{n\geq0}$ be a sequence
considered in lemma \ref{iteracje}. Let us assume that
\[
\exists N\,\,\,\,\forall n\geq N:b_{n}\geq0\,\,\,\text{and }\,0\leq
x_{n+1}\leq(1-\mu_{n})x_{n}+\mu_{n}b_{n}.
\]
Then
\[
\underset{n\rightarrow\infty}{\lim\inf}\,x_{n}\leq\underset{n\rightarrow
\infty}{\,\lim\sup}\,b_{n}.
\]

\end{corollary}

\begin{proof}
Let us denote by $t_{n}$, $n\geq N$ a solution of the iterative equation
\[
t_{n+1}=(1-\mu_{n})t_{n}+\mu_{n}b_{n},
\]
with an initial condition $t_{N}=x_{N}$. Let us suppose further, that for
$k\allowbreak=\allowbreak N,N+1,\ldots,n$ we have $t_{k}\geq x_{k}$. Of
course, we have also
\[
t_{n+1}=(1-\mu_{n})x_{n}+\mu_{n}b_{n}\geq x_{n+1}.
\]
Thus, by the indiction assumption we deduce that $\forall n\geq N$ $:x_{n}\leq
t_{n}$. Hence, using lemma \ref{iteracje} we get:
\[
\underset{n\rightarrow\infty}{\lim\inf}\,x_{n}\leq\underset{n\rightarrow
\infty}{\lim\inf}\,t_{n}\leq\underset{n\rightarrow\infty}{\lim\inf}\,b_{n}.
\]

\end{proof}

\begin{definition}%
\index{Sequence!Normal}%
Positive number sequence $\left\{  \mu_{i}\right\}  _{i\geq0}$, satisfying
assumption \emph{i) }of the lemma \ref{iteracje} we will call \emph{normal}.
\end{definition}

\begin{lemma}
\label{podstawowy} Let the sequence $\left\{  \mu_{n}\right\}  _{n\geq0}$ be
normal. Let us assume that the sequences $\left\{  x_{n}\right\}  _{n\geq0}$
and $\left\{  b_{n}\right\}  _{n\geq0}$ are such that
\[
{\large \exists}N\,\,\,\,{\large \forall}n\geq N:x_{n+1}=(1-\mu_{n})x_{n}%
+\mu_{n}b_{n}.
\]
Then the following statements are equivalent:
\[
\sum_{n\geq0}\mu_{n}b_{n}\,\,\,\,\,\text{is convergent}\Longleftrightarrow
x_{n}\underset{n\rightarrow\infty}{\rightarrow}0\,\,\,\text{and }\sum_{n\geq
0}\mu_{n}x_{n}\text{ is convergent.}%
\]

\end{lemma}

\begin{proof}
First let us add side by side the equality $x_{n+1}=(1-\mu_{n})x_{n}+\mu
_{n}b_{n}$ for $n=N,\ldots,M$. We get then $\sum_{n=N}^{M}x_{n+1}=\sum
_{n=N}^{M}x_{n}-\sum_{n=N}^{M}\mu_{n}x_{n}+\sum_{n=N}^{M}\mu_{n}b_{n}$, which
after little algebra reduces to the following equality:
\begin{equation}
x_{M+1}-x_{N}=-\sum_{n=N}^{M}\mu_{n}x_{n}+\sum_{n=N}^{M}\mu_{n}b_{n}.
\label{tozsamosc_podstawowa}%
\end{equation}
Proof of the implication $\Longleftarrow$. Taking $N=0$ in the above identity,
passing with $M$ to infinity and taking into account assumptions, we get
convergence of the series $\sum_{n\geq0}\mu_{n}b_{n}$ .

Proof of the implication $\Longrightarrow$. Let us denote $G_{n}=\sum_{i\geq
n}\mu_{i}b_{i}\,$ and $D_{n}=x_{n}+G_{n}$. For $n\geq N$ let us add $G_{n+1}$
to both sides of equality $x_{n+1}=(1-\mu_{n})x_{n}+\mu_{n}b_{n}$. We obtain
then:
\[
D_{n+1}=(1-\mu_{n})x_{n}+G_{n}\allowbreak=(1-\mu_{n})D_{n}+\mu_{n}G_{n}.
\]
Let us denote further $d_{n}=\left\vert D_{n}\right\vert $ and $g_{n}%
=\left\vert G_{n}\right\vert $. We have then
\[
d_{n+1}\leq(1-\mu_{n})d_{n}+\mu_{n}g_{n}.
\]
Now we apply corollary \ref{nieiter} and deduce that $\underset{n\rightarrow
\infty}{\lim}D_{n}=0$, since of course $\underset{n\rightarrow\infty}{\lim
}G_{n}\allowbreak=\allowbreak0$, consequently we have $\underset{n\rightarrow
\infty}{\lim}g_{n}=0$. Further, since $\underset{n\rightarrow\infty}{\lim
}G_{n}=0$, we get $\underset{n\rightarrow\infty}{\lim}x_{n}=0$. To prove
convergence of the series $\sum_{n\geq0}\mu_{n}x_{n}$ we utilize identity
(\ref{tozsamosc_podstawowa}), assumed convergence of the series $\sum_{n\geq
0}\mu_{n}b_{n}$ and proved convergence to zero of the sequence $\left\{
x_{n}\right\}  .$
\end{proof}

The following corollary can be deduced from the above-mentioned lemma.

\begin{corollary}
\label{podstawowy1} Let $\left\{  \mu_{n}\right\}  _{n\geq0}$ be a normal
sequence i.e. considered in lemma \ref{iteracje}. Let us assume that
\[
\exists N\,\,\,\,\forall n\geq N:b_{n}\geq0\,\,\text{\thinspace\ and }\,0\leq
x_{n+1}\leq(1-\mu_{n})x_{n}+\mu_{n}b_{n},
\]
then the following implication is true:
\[
\sum_{n\geq0}\mu_{n}b_{n}\,\,\,\,\,\text{is convergent}\Rightarrow
x_{n}\underset{n\rightarrow\infty}{\rightarrow}0\,\,\,\text{and }\sum_{n\geq
0}\mu_{n}x_{n}\text{ is convergent.}%
\]

\end{corollary}

\begin{proof}
At the beginning, we argue as in the proof of the Corollary \ref{nieiter}
introducing sequence $\left\{  t_{n}\right\}  $ such that $x_{n}\leq t_{n},$
$n\geq0$. Next we apply Lemma \ref{podstawowy} to iterative equality defining
sequence $\left\{  t_{n}\right\}  _{n\geq1}$. We infer that convergence of the
series $\sum_{n\geq1}\mu_{n}b_{n}$ implies convergence of the sequence
$\left\{  t_{n}\right\}  _{n\geq1}$ to zero and convergence of the series
$\sum_{n\geq1}\mu_{n}t_{n}$. Remembering that the sequences $\left\{
x_{n}\right\}  _{n\geq1}$ and $\left\{  b_{n}\right\}  _{n\geq1}$ are
nonnegative it is now elementary to get the assertion.
\end{proof}

\begin{corollary}
\label{podobne}Let sequences $\left\{  \mu_{n}\right\}  _{n\geq0}$, $\left\{
\mu_{n}^{^{\prime}}\right\}  $ be normal. Let us assume that the sequences
$\left\{  x_{n}\right\}  _{n\geq0}$, $\left\{  x_{n}^{^{\prime}}\right\}
_{n\geq0}$and $\left\{  b_{n}\right\}  _{n\geq0}$ are such that:
\begin{align}
{\large \exists}N\,\,\,\,{\large \forall}n  &  \geq N:x_{n+1}=(1-\mu_{n}%
)x_{n}+\mu_{n}b_{n+1},\label{rown1}\\
x_{n+1}^{^{\prime}}  &  =(1-\mu_{n}^{^{\prime}})x_{n}^{^{\prime}}+\mu
_{n}^{^{\prime}}b_{n+1}. \label{rown2}%
\end{align}
If for some positive constant $M$ the series
\[
\sum_{n\geq0}\mu_{n}b_{n+1}\text{ and }\sum_{n\geq0}(\mu_{n}-M\mu
_{n}^{^{\prime}})b_{n+1},
\]
are convergent, then the sequences $\left\{  x_{n}\right\}  _{n\geq0}$ and
$\left\{  x_{n}^{^{\prime}}\right\}  _{n\geq0}$ are convergent to zero, while
series $\sum_{n\geq0}\mu_{n}x_{n}$ and $\sum_{n\geq0}\mu_{n}^{^{\prime}}%
x_{n}^{^{\prime}}$ are convergent.
\end{corollary}

\begin{proof}
Since the series $\sum_{n\geq0}\mu_{n}b_{n+1}$ is convergent and the sequence
$\left\{  \mu_{n}\right\}  _{n\geq0}$ is normal, then the sequence $\left\{
x_{n}\right\}  $ converges to zero, and the series $\sum_{n\geq0}\mu_{n}x_{n}$
is convergent on the base of Lemma \ref{podstawowy}. Further, since together
with the series $\sum_{n\geq0}\mu_{n}b_{n+1}$ converges the series
$\sum_{n\geq0}(\mu_{n}-M\mu_{n}^{^{\prime}})b_{n+1}$, hence by Lemma
\ref{podstawowy} and equality (\ref{rown2}) we deduce that the sequence
$\left\{  x_{n}^{^{\prime}}\right\}  $ converges to zero, and the series
$\sum_{n\geq0}\mu_{n}^{^{\prime}}x_{n}^{^{\prime}}$ converges.
\end{proof}

We will show that from Lemma \ref{podstawowy} follows well known Kronecker's
Lemma. Let $\left\{  x_{i}\right\}  _{i\geq1}$ be a sequence real numbers,
while $\left\{  a_{i}\right\}  _{i\geq1}$ increasing to infinity sequence
positive numbers. Let us denote
\[
m_{n}=\frac{\sum_{i=1}^{n}x_{i}}{a_{n}}.
\]
Let us notice that the sequence $\left\{  m_{n}\right\}  _{n\geq1}$ satisfies
the following recurrent relationship:
\[
m_{n+1}=(1-(a_{n+1}-a_{n})/a_{n+1})m_{n}+x_{n+1}/a_{n+1}.
\]
Let us denote $\mu_{n}=(a_{n+1}-a_{n})/a_{n+1}$. Let us notice also that
$1>\mu_{n}>0$ and $\prod_{i=1}^{n}(1-\mu_{i})=\frac{a_{1}}{a_{n+1}}%
\rightarrow0$, as $n\rightarrow\infty$. Hence, $\sum_{i\geq1}\mu_{i}=\infty$.
Thus, one can apply Lemma \ref{podstawowy} and get the following lemma, that
is in fact a generalization of the Kronecker's Lemma
\index{Lemma!Kronecker}%
.

\begin{lemma}
Series $\sum_{n\geq1}x_{n}/a_{n}$ is convergent if and only if, sequence
$\left\{  \frac{\sum_{i=1}^{n}x_{i}}{a_{n}}\right\}  _{n\geq1}$ converges to
zero and the series $\sum_{n\geq1}\left(  \frac{1}{a_{n}}-\frac{1}{a_{n+1}%
}\right)  \sum_{i=1}^{n}x_{i}$ is convergent.
\end{lemma}

\begin{remark}
Kronecker's Lemma is in fact the following corollary: \newline If the series
$\sum_{n\geq1}x_{n}/a_{n}$ is convergent, then the sequence $\left\{
\frac{\sum_{i=1}^{n}x_{i}}{a_{n}}\right\}  _{n\geq1}$ converges to zero.
\end{remark}

\section{Summability}

\label{Riesz}Summability theory, it is a part of the analysis that assigns
some numbers or functions (if one deals with a function sequence) to divergent
sequences. These numbers (or functions) are called their \emph{limits}
or\emph{ sums} (in the case of a series). Recall that an infinite series can
be understood as a sequence of its partial sums. There exist in the literature
many different methods (i.e. ways to assign these numbers or functions) of
summing sequences (summing of series it is nothing else than summing of the
sequence of its partial sums). Of course, every reasonable method of
summability should have the following property:
\[
\text{\emph{sequences\thinspace\thinspace or \thinspace\thinspace
series\thinspace\thinspace\thinspace that converge\thinspace\thinspace should
be summed to its limits.}}%
\]
Summability methods satisfying this condition will be called\emph{\ regular}.
The following theorem of Toeplitz is true:

\begin{theorem}
[Toeplitz]\label{toeplitz}%
\index{Theorem!Toeplitz}%
Let $T=[t_{ij}]_{i,j\geq0}$ be an infinite matrix having nonnegative entries.
Let us consider the following summation method of the $\left\{  q_{n}\right\}
_{n\geq0}$.
\[
Q_{n}=\sum_{k=0}^{\infty}t_{nk}q_{k};n\geq0.
\]
This method is regular if and only if:
\begin{gather}
\underset{n\rightarrow\infty}{\lim}\sum_{k=0}^{\infty}t_{nk}=1,\label{_1}\\
\forall k\geq0\underset{n\rightarrow\infty}{:\lim}t_{nk}=0. \label{_2}%
\end{gather}

\end{theorem}

\begin{proof}
Necessity. Let us take constant sequence i.e. $q_{n}=q$, $n=0,1,2,\ldots$.
Then we have $Q_{n}=q\sum_{k=0}^{\infty}t_{nk}$. The method is regular i.e.
$Q_{n}\rightarrow q$, for $n\rightarrow\infty$, hence the condition (\ref{_1})
must be satisfied. To show the necessity of (\ref{_2}), let us take the
following sequence: let us fix $k$, then $q_{n}=0$ for $n\neq k$ and $q_{k}$
$=q\neq0$. Then of course $Q_{n}=qt_{nk}$. Regularity implies that
$Q_{n}\underset{n\rightarrow\infty}{\longrightarrow}0$. Hence, (\ref{_2}) must
also be satisfied.

Sufficiency. Let $q_{n}\underset{n\rightarrow\infty}{\longrightarrow}q$. Let
us take any $\varepsilon$. By $K$ let us denote such index that for $k>K:$
$\left\vert q_{k}-q\right\vert <\varepsilon$. We have then
\[
Q_{n}=q\sum_{k=0}^{\infty}t_{nk}+\sum_{k=0}^{K}t_{nk}(q_{k}-q)+\sum
_{k>K}t_{nk}(q_{k}-q).
\]
Moreover,
\[
\left\vert \sum_{k=0}^{K}t_{nk}(q_{k}-q)\right\vert \leq\underset{0\leq k\leq
K}{\max}\left\vert q_{k}-q\right\vert \sum_{k=0}^{K}t_{nk}\rightarrow0,
\]
as $n\rightarrow\infty$, since we have (\ref{_2}) and
\[
\left\vert \sum_{k>K}t_{nk}(q_{k}-q)\right\vert \leq\varepsilon\sum_{k\geq
K}t_{nk}\rightarrow\varepsilon,
\]
for $n\rightarrow\infty$, by (\ref{_1}).
\end{proof}

\subsection{Ces\`{a}ro methods of summation
\index{Summability!Cesaro}%
}

Ces\`{a}ro methods are the very popular methods of summing divergent
sequences. We will be concerned mostly with the so-called Riesz methods of
summation mainly because of its strong connection with laws of large numbers.
It turns out that the Ces\`{a}ro method of order $1$ is the same as the
Riesz's method with weights equal to $1$. Moreover, as it will turn out due to
some properties of the Ces\`{a}ro methods it will be possible to prove in a
simple way a basic inequality for the orthogonal series (see Lemma
\ref{fundamental_inequality}).

Let us denote
\begin{equation}
A_{n}^{\alpha}=\binom{n+\alpha}{n},\;\alpha\neq-1,-2,\ldots\;.
\label{wsp_Cesaro}%
\end{equation}
As it can be easily shown coefficient $A_{n}^{\alpha}$ is equal to the
coefficient by the $n-$th power of $x$ in the power series expansion of
$\left(  1-x\right)  ^{-1-\alpha}$. Let $\left\{  q_{n}\right\}  _{n\geq0}$ be
a number sequence. Let us define sequence $\left\{  q_{n}^{\alpha}\right\}
_{n\geq0}$ using relationship:
\begin{equation}
\sum_{n\geq0}q_{n}^{\alpha}x^{n}\overset{df}{=}\frac{\sum_{n\geq0}q_{n}x^{n}%
}{\left(  1-x\right)  ^{\alpha}}. \label{definicjaCesaro}%
\end{equation}
The following quantity
\begin{equation}
Q_{n}^{\alpha}=\frac{q_{n}^{\alpha}}{A_{n}^{\alpha}}, \label{sredniaCesaro}%
\end{equation}
is called $n-$th Ces\`{a}ro mean of order $\alpha$ of the sequence $\left\{
q_{n}\right\}  _{n\geq0}$ , briefly $n-$th $(C,\alpha)-$mean. Let us consider
Ces\`{a}ro summation methods, that is $(C,\alpha)$ summation methods for
$\alpha>-1$. Their most important features are collected in the lemma below.

\begin{lemma}
\label{Cesaro}Let be given sequence $\left\{  q_{n}\right\}  _{n\geq0}$. For
all $\alpha>-1$ we have:

i) $Q_{n}^{\alpha}=\frac{1}{A_{n}^{\alpha}}\sum_{k=0}^{n}A_{n-k}^{\alpha
-1}q_{k}.$

ii) $\sum_{k=0}^{n}A_{k}^{\alpha}=A_{n}^{\alpha+1}.$

iii) $\exists$ $0<K_{1}<K_{2},\forall n\geq1:K_{1}<\frac{A_{n}^{\alpha}%
}{n^{\alpha}}<K_{2},\frac{A_{n-k}^{\alpha-1}}{A_{n}^{\alpha}}=O(\frac{1}{n}).$

iv) $\forall\beta>0:Q_{n}^{\alpha+\beta}=\frac{1}{A_{n}^{\alpha+\beta}}%
\sum_{k=0}^{n}A_{n-k}^{\beta-1}A_{k}^{\alpha}Q_{k}^{\alpha}$. \newline In
particular, $Q_{n}^{\alpha+1}=\frac{1}{A_{n}^{\alpha+1}}\sum_{k=0}^{n}%
A_{k}^{\alpha}Q_{k}^{\alpha}.$

v) Methods $(C,\alpha)$ are regular for all $\alpha>0.$

vi) If the sequence $\left\{  q_{n}\right\}  _{n\geq0}$ is $(C,\alpha)$
summable for some $\alpha>-1$, then it is also $(C,\alpha+\beta)$ summable for
$\beta>0.$

vii) Let $s_{n}=\sum_{k=0}^{n}q_{k},$ $n\geq0,$ be the sequence of partial
sums of the series $\sum_{k\geq0}q_{k}$. Let $S_{n}^{\alpha}$ be the $n-$th
Ces\`{a}ro of order $\alpha$ mean of the sequence $\left\{  s_{n}\right\}
_{n\geq0}$ . Then:
\begin{equation}
S_{n}^{\alpha}=\frac{1}{A_{n}^{\alpha}}\sum_{k=0}^{n}A_{n-k}^{\alpha}q_{k}.
\label{sr_szeregu}%
\end{equation}
In particular, $S_{n}^{0}=s_{n};n\geq0.$
\end{lemma}

\begin{proof}
Using well known formula for the product of two series and formula
(\ref{sredniaCesaro}) we get assertion i). Using formula
\[
\sum_{k\geq0}q_{k}^{\alpha+\beta}x^{k}=\frac{\sum_{k\geq0}q_{k}x^{k}}{\left(
1-x\right)  ^{\alpha}}\frac{1}{\left(  1-x\right)  ^{\beta}}=\sum_{k\geq
0}q_{k}^{\alpha}x^{k}\sum_{j\geq0}A_{j}^{\beta-1}x^{j},
\]
and then using formula for the product of power series we get assertion iv).
Assertion ii) we get by the straightforward, easy algebra. Assertion iii) we
get with the help of the following estimation:
\begin{align*}
\log A_{n}^{\alpha}  &  =\sum_{k=1}^{n}\log(1+\frac{\alpha}{k})=\sum_{k=1}%
^{n}\frac{\alpha}{k}+\sum_{k=1}^{n}O(\frac{1}{k^{2}})\\
&  =\alpha\log n+\alpha C+o(1)+\sum_{k=1}^{n}O(\frac{1}{k^{2}}),
\end{align*}
where $C$ denotes Euler's constant. Hence,
\[
\left\vert \log A_{n}^{\alpha}-\alpha\log n\right\vert \leq\left\vert
\alpha\right\vert C+o(1)+\left\vert \sum_{k=1}^{\infty}O(\frac{1}{k^{2}%
})\right\vert \leq K,
\]
where $K$ denotes some positive constant. Hence, we have the first assertion
of iii). In order to get the second one, let us notice that:
\[
A_{n-k}^{\alpha-1}/A_{n}^{\alpha}=O\left(  \frac{\left(  n-k\right)
^{\alpha-1}}{n^{\alpha}}\right)  =O\left(  \frac{1}{n}\right)  .
\]
Assertions v) and vi) follow straightforwardly (since we have $t_{nk}%
=\frac{A_{n-k}^{\alpha-1}}{A_{n}^{\alpha}}$ for $1\leq k\leq n$ and $0$ for
the remaining $k)$ from the properties ii), iii) and iv), and also from
Toeplitz's theorem \ref{toeplitz}. Thus, it remained to prove assertion vii).
We have:
\begin{align*}
A_{n}^{\alpha}S_{n}^{\alpha}  &  =\sum_{k=0}^{n}A_{n-k}^{\alpha-1}\sum
_{j=0}^{k}q_{j}=\\
&  =\sum_{j=0}^{n}q_{j}\sum_{k=j}^{n}A_{n-k}^{\alpha-1}=\\
&  =\sum_{j=0}^{n}q_{j}A_{n-j}^{\alpha},
\end{align*}
by the property ii).
\end{proof}

\begin{corollary}
Let $\left\{  q_{n}\right\}  _{n\geq0}$ be a number sequence. We have then:
\begin{gather}
Q_{n}^{\alpha+1}=\frac{\sum_{k=0}^{n}A_{k}^{\alpha}q_{k}}{A_{n}^{\alpha+1}%
},\label{alpha+1}\\
A_{n}^{0}=1,\;A_{n}^{1}=n+1,\;A_{n}^{2}=\frac{(n+2)(n+1)}{2},\nonumber\\
q_{n}^{1}=\sum_{i=0}^{n}q_{i}\;\text{ and }\;Q_{n}^{1}=q_{n}^{1}%
/(n+1),\nonumber\\
q_{n}^{2}=\sum_{i=0}^{n}(n+1-i)q_{i},\;Q_{n}^{2}=\frac{2}{n+2}\sum_{i=0}%
^{n}\left(  1-\frac{i}{n+1}\right)  q_{i}. \label{n1in2}%
\end{gather}

\end{corollary}

\begin{proof}
It follows directly from formula (\ref{wsp_Cesaro}) and assertion $i)$ of
Lemma \ref{Cesaro}.
\end{proof}

\subsection{Riesz's summation method
\index{Summability!Riesz}%
}

Among different summation methods, the Riesz's method is interesting from the
point of view of this book, since the sequence of Riesz's means can be
presented in a recursive form.

\begin{definition}
\label{Sum_Riesz}We say that the sequence $\left\{  q_{i}\right\}  _{i\geq0}$
is summable by the Riesz's method with the sequence of (nonnegative) weights
$\left\{  \alpha_{i}\right\}  _{i\geq0}$, if the sequence \newline$\left\{
\sum_{j=0}^{i}\alpha_{j}q_{j}/\sum_{j=0}^{i}\alpha_{i}\right\}  _{i\geq1}$ is convergent.
\end{definition}

\begin{remark}
If the sequence of weights $\left\{  \alpha_{i}\right\}  $ consists of $1$,
then, as it can be easily seen, Riesz's method is equivalent in this case to
the Ces\`{a}ro method of order $1$.
\end{remark}

Below we will give a sufficient condition of the regularity of Riesz's method,
and also will present a useful lemma exposing essential features of this method.

Let $\left\{  \alpha_{i}\right\}  _{i\geq0}$ be a nonnegative number sequence,
such that
\[
\alpha_{0}=1,\,\,\sum_{i\geq0}\alpha_{i}=\infty.
\]
For every such sequence we will define a sequence $\left\{  \mu_{i}\right\}
_{i\geq0}$ in the following way:
\begin{equation}
\mu_{0}=1;\mu_{i}=\frac{\alpha_{i}}{\sum_{k=0}^{i}\alpha_{k}},i\geq1.
\label{sprzez}%
\end{equation}

\begin{proposition}
\label{o_ciagach_normalnych}\emph{i) }$\forall i\geq1$ $\mu_{i}\in(0,1)$.

\emph{ii) }Every sequence $\left\{  \mu_{i}\right\}  _{i\geq0}$ satisfying
\emph{i), }uniquely defines the sequence $\left\{  \alpha_{i}\right\}  $ and
$\alpha_{0}=1.$

\emph{iii) }$\sum_{i\geq0}\alpha_{i}=\infty\Longleftrightarrow\sum_{i\geq0}%
\mu_{i}=\infty.$

\emph{iv) }$\sum_{i\geq0}\mu_{i}=\infty$, $\mu_{n}\underset{n\rightarrow
\infty}{\rightarrow}0\Longleftrightarrow\underset{1\leq i\leq n}{\max}%
\alpha_{i}/\sum_{i=0}^{n}\alpha_{i}\underset{n\rightarrow\infty}{\rightarrow
}0.$
\end{proposition}

\begin{proof}
Assertion \textit{i) } follows directly from equality (\ref{sprzez}).
Assertion \textit{ii):} Solving sequentially equalities (\ref{sprzez}) with
respect $\alpha_{i}$, we get: $\alpha_{0}=1;$ $\alpha_{i}=\mu_{i}/\prod
_{j=1}^{i}(1-\mu_{j});$ $i\geq1$. Assertion \textit{iii): }from \textit{ii)}
we have
\[
\sum_{i=0}^{n}\alpha_{i}=\allowbreak\frac{\alpha_{n}}{\mu_{n}}=\allowbreak
1/\prod_{i=1}^{n}(1-\mu_{i})\allowbreak\geq\exp(\sum_{i=1}^{n}\mu_{i}).
\]
Hence, if $\sum_{i\geq0}\mu_{i}=\infty$, then $\sum_{i\geq0}\alpha_{i}=\infty
$. On the other hand if $\sum_{i\geq0}\alpha_{i}=\infty$, then $\prod_{i\geq
1}(1-\mu_{i})\allowbreak=0$, that implies condition $\sum_{i\geq0}\mu
_{i}=\infty$. Assertion \textit{iv):}\emph{\ }implication $\Leftarrow$ is
obvious. Implication $\Rightarrow$. If\thinspace%
\[
\underset{n\rightarrow\infty}{\lim\inf}\,\underset{1\leq i\leq n}{\max}%
\alpha_{i}/\sum_{i=0}^{n}\alpha_{i}=\delta>0,
\]
then there exists such sequence of indices $k$, such that :
\[
k>k_{0}\;\,\alpha_{i_{n_{k}}}/\sum_{i=0}^{n_{k}}\alpha_{i}>\delta/2>0.
\]
However, if $\sup i_{n_{k}}<\infty$, then it is impossible by assertion $iii)$
and condition $\sum_{i\geq0}\alpha_{i}=\infty$, if however $i_{n_{k}%
}\underset{k\rightarrow\infty}{\rightarrow}\infty$, then
\[
\alpha_{i_{n_{k}}}/\sum_{i=0}^{n_{k}}\alpha_{i}\leq\alpha_{i_{n_{k}}}%
/\sum_{i=0}^{i_{n_{k}}}\alpha_{i}=\mu_{i_{n_{k}}}%
\]
and $\mu_{i_{n_{k}}}>\delta/2>0$ that is also impossible by the fact that the
condition $\mu_{n}\underset{n\rightarrow\infty}{\rightarrow}0.$
\end{proof}

\begin{remark}
Let us notice that the conditions: $\mu_{n}\geq0$ and $\sum_{i\geq0}\mu
_{i}=\infty$ are necessary and sufficient for the regularity of Riesz's
method. It easily follows Toeplitz' theorem and the lemma
\ref{o_ciagach_normalnych}.
\end{remark}

The relationship between $\left\{  \alpha_{i}\right\}  _{i\geq0}$ and
$\left\{  \mu_{i}\right\}  _{i\geq0}$ will be denoted in the following way:
$\left\{  \mu_{i}\right\}  \allowbreak=\allowbreak\overline{\left\{
\alpha_{i}\right\}  },\,$ and $\widehat{\left\{  \mu_{i}\right\}  }=\left\{
\alpha_{i}\right\}  $. Sequence$\left\{  \alpha_{i}\right\}  _{i\geq0}$ will
be called conjugate
\index{Sequence!conjugate}
with respect to the sequence $\left\{  \mu_{i}\right\}  _{i\geq0}$ .

One can easily calculate using formulae (\ref{sprzez}), that
\begin{align*}
\overline{\left\{  1\right\}  }\allowbreak &  =\left\{  \frac{1}{i+1}\right\}
,\allowbreak\overline{\left\{  i+1\right\}  }=\left\{  \frac{2}{i+2}\right\}
,\\
\overline{\left\{  (i+1)^{2}\right\}  }  &  =\allowbreak\left\{
\allowbreak\frac{(i+1)^{2}}{\sum_{j=0}^{i}(j+1)^{2}}\right\}  \allowbreak
=\allowbreak\left\{  \frac{6(i+1)}{(i+2)(2i+3)}\right\}  ,\\
\overline{\left\{  \exp(\alpha i)\right\}  }\allowbreak &  =\allowbreak
\left\{  \frac{\exp(\alpha)-1}{\exp(\alpha)-\exp(-i\alpha)}\right\}
;\allowbreak\alpha>0,\\
\overline{\left\{  1,q/(1-q),q/(1-q)^{2},\ldots\right\}  }\allowbreak &
=\left\{  q\right\}  ;q\in(0,1)\;\;\;\text{and so on}.
\end{align*}

Let $\left\{  x_{i}\right\}  _{i\geq1}$ be a sequence of real numbers. Having
given sequence $\left\{  \alpha_{i}\right\}  _{i\geq0}$, we define the
following sequences :
\begin{align}
\overline{x}_{0}  &  =0,\,\,\overline{x}_{i}=\frac{\sum_{j=0}^{i-1}\alpha
_{j}x_{j+1}}{\sum_{j=0}^{i-1}\alpha_{j}},i\geq1,\label{srednie}\\
s_{0}  &  =0,\,\,s_{i}=\sum_{j=0}^{i-1}\mu_{j}x_{j+1},i\geq1,\label{sumy}\\
\overline{s}_{i}  &  =\frac{\sum_{j=0}^{i}\alpha_{j}s_{j}}{\sum_{j=0}%
^{i}\alpha_{j}},i\geq0,\label{srsumy}\\
\widehat{s}_{i}  &  =\sum_{j=0}^{i}\mu_{j}\overline{x}_{j},.i\geq0
\label{sumpom}%
\end{align}

Sequence $\left\{  \overline{x}_{i}\right\}  _{i\geq0}$ it is, as it can be
seen, the sequence of Riesz's means of numbers $\left\{  x_{i}\right\}
_{i\geq1}$ with respect to the sequence $\left\{  \alpha_{i}\right\}
_{i\geq0}$, sequence $\left\{  s_{i}\right\}  _{i\geq0}$ it is the sequence of
partial sums of the series $\sum_{i\geq0}\mu_{i}x_{i+1}$, sequence $\left\{
\overline{s}_{i}\right\}  _{i\geq0}$ it is the sequence of Riesz's means of
the sequence $\left\{  s_{i}\right\}  _{i\geq0}$, while the sequence $\left\{
\widehat{s}_{i}\right\}  _{i\geq0}$ it is the sequence partial partial sums of
the series $\sum_{i\geq0}\mu_{i}\overline{x}_{i}$. The mutual relationships
between those means are exposed in the Lemma below.

\begin{lemma}
\label{lemosr}

\emph{i) }$\forall i\geq0:\widehat{s}_{i}=\overline{s}_{i},$

\emph{ii)} let us assume additionally, that $\sum_{i\geq0}\mu_{i}^{2}%
x_{i+1}^{2}<\infty$, $\sum_{i\geq0}\mu_{i}\overline{x}_{i}^{2}<\infty$, then
we have: \newline\hspace*{2cm}\emph{iia) }$\overline{x}_{i}\rightarrow0$ as
$i\rightarrow\infty$ if and only if, when series $\sum_{i\geq0}\mu
_{i}\overline{x}_{i}x_{i+1}$ converges,\newline\hspace*{2cm}\emph{iib) }%
$\frac{\sum_{j=0}^{n}\alpha_{j}\overline{x}_{j}^{2}}{\sum_{j=0}^{n}\alpha_{j}%
}\underset{n\rightarrow\infty}{\longrightarrow}0$.\newline\hspace
*{0pt}0pt2cm\emph{iic) }Let\emph{\ }$v_{n}\overset{df}{=}\allowbreak\frac
{\sum_{j=0}^{n}\alpha_{j}(s_{j}-\overline{s}_{j})^{2}}{\sum_{j=0}^{n}%
\alpha_{j}}\allowbreak$ and $\kappa_{n}\overset{df}{=}\allowbreak\frac
{\sum_{j=0}^{n}\alpha_{j}s_{j}^{2}}{\sum_{j=0}^{n}\alpha_{j}}-\overline{s}%
_{n}^{2}$. Then $\forall n\geq1$ $\kappa_{n}\geq0$, and Moreover, both
$\nu_{n}\underset{n\rightarrow\infty}{\longrightarrow}0$ and $\sum_{n\geq1}%
\mu_{n+1}v_{n}<\infty$, as well as and $\kappa_{n}\underset{n\rightarrow
\infty}{\rightarrow}0$ and $\sum_{n\geq1}\mu_{n+1}\kappa_{n}<\infty.$
\end{lemma}

\begin{proof}
$\emph{i)}$ Let us notice that the sequences $\left\{  \overline{x}%
_{n}\right\}  $ and $\left\{  \overline{s}_{i}\right\}  $ satisfy the
following recursive equations:
\begin{equation}
\emph{\ }\overline{x}_{i+1}=(1-\mu_{i})\overline{x}_{i}+\mu_{i}x_{i+1};i\geq1,
\label{postaciter}%
\end{equation}%
\begin{equation}
\overline{s}_{i+1}=(1-\mu_{i+1})\overline{s}_{i}+\mu_{i+1}s_{i+1};i\geq0.
\label{sumiter}%
\end{equation}
Let us add side by side equality (\ref{postaciter}) for $i=0,1,\ldots,n$. We
will get then: $\overline{x}_{n+1}=\bar{x}_{0}-\hat{s}_{n}+s_{n+1}$. Taking
into account definition of the sequence $\left\{  \bar{x}_{i}\right\}  $, we
get finally:
\begin{equation}
s_{n+1}=\hat{s}_{n}+\bar{x}_{n+1}. \label{waznazal}%
\end{equation}
Let us now perform the same operation on the equality (\ref{sumiter}). We will
get then:
\begin{equation}
\bar{s}_{n+1}=-\sum_{i=0}^{n}\mu_{i+1}\bar{s}_{i}+\sum_{i=0}^{n}\mu
_{i+1}s_{i+1}. \label{tozsamosc}%
\end{equation}
Now notice that $\hat{s}_{0}=\bar{s}_{0}$ and let us make an induction
assumption, that $\hat{s}_{i}=\bar{s}_{i}$for $i\leq n$. Now let us put in
(\ref{tozsamosc}) instead $s_{i}$, the value that follows from (\ref{waznazal}%
). We get then:
\[
\bar{s}_{n+1}=-\sum_{i=0}^{n}\mu_{i+1}\bar{s}_{i}+\sum_{i=0}^{n}\mu_{i+1}%
\hat{s}_{i}+\sum_{i=0}^{n}\mu_{i+1}\bar{x}_{i+1}=\hat{s}_{n+1}.
\]

\emph{ii)} Let us calculate squares of both sided of (\ref{postaciter})
\[
\emph{\ }\overline{x}_{i+1}^{2}=(1-2\mu_{i}+\mu_{i}^{2})\overline{x}_{i}%
^{2}+2\mu_{i}(1-\mu_{i})\overline{x}_{i}x_{i+1}+\mu_{i}^{2}x_{i+1}^{2}.
\]
We apply now Lemma \ref{podstawowy} except that the r\^{o}le of the sequence
$\left\{  \mu_{n}\right\}  _{n\geq0}$ will now be played by the $\left\{
\mu_{i}(2-\mu_{i})\right\}  $. Let us notice that the series $\sum_{i\geq0}%
\mu_{i}^{2}\overline{x}_{i}x_{i+1}$ is convergent, since we have
\[
\left\vert \sum_{i\geq0}\mu_{i}^{2}\overline{x}_{i}x_{i+1}\right\vert
\leq\sqrt{\sum_{i\geq0}\mu_{i}\left\vert \overline{x}\right\vert ^{2}}%
\sqrt{\sum_{i\geq0}\mu_{i}^{3}x_{i+1}^{2}}.
\]
Hence the lemma can be used. Assertion \emph{iia) }is a simple consequence of
the lemma and the assumptions. In order to prove assertion \emph{iib) }let us
present $y_{n}=\allowbreak\frac{\sum_{j=0}^{n}\alpha_{j}\overline{x}_{j}^{2}%
}{\sum_{j=0}^{n}\alpha_{j}}$ in a recursive form. We have:
\[
y_{n+1}=(1-\mu_{n+1})y_{n}+\mu_{n+1}\bar{x}_{n+1}^{2}.
\]
Using assumed convergence of the series $\sum_{n\geq1}\mu_{n}\bar{x}_{n}^{2}$
and using Lemma \ref{podstawowy} we get immediately the assertion. In order to
prove \emph{iic) } let us present $v_{n}$ again in a recursive form:
\begin{equation}
v_{n+1}=(1-\mu_{n+1})v_{n}+\mu_{n+1}(s_{n+1}-\bar{s}_{n+1})^{2}. \label{vn}%
\end{equation}
Remembering about relationship (\ref{waznazal}) and the relationship in
\emph{i) }we see that
\[
s_{n+1}-\bar{s}_{n+1}=\bar{x}_{n+1}-\mu_{n+1}\bar{x}_{n+1}.
\]
And again using Lemma \ref{podstawowy} and assumptions \emph{ii)} we get
convergence of the sequence $\left\{  v_{n}\right\}  _{n\geq1}$ to zero and
also the convergence of the series $\sum_{n\geq1}\mu_{n+1}v_{n}$. Let us
concentrate now on the sequence $\left\{  \kappa_{n}\right\}  $. Let us
denote: $K_{n}=\frac{\sum_{j=0}^{n}\alpha_{j}s_{j}^{2}}{\sum_{j=0}^{n}%
\alpha_{j}}$. We have:
\begin{equation}
K_{n+1}=(1-\mu_{n+1})K_{n}+\mu_{n+1}s_{n+1}^{2}. \label{1*}%
\end{equation}
Calculating squares on both sides of the identity (\ref{sumiter}) we get:
\begin{equation}
\bar{s}_{n+1}^{2}=(1-\mu_{n+1})^{2}\bar{s}_{n}^{2}+2\mu_{n+1}(1-\mu_{n+1}%
)\bar{s}_{n}s_{n+1}+\mu_{n+1}^{2}s_{n+1}^{2}. \label{2*}%
\end{equation}
Subtracting side by side (\ref{2*}) from (\ref{1*}) we get:
\begin{align}
\kappa_{n+1}  &  =K_{n+1}-\bar{s}_{n+1}^{2}=\label{vn1}\\
&  =(1-\mu_{n+1})\kappa_{n}+\mu_{n+1}(1-\mu_{n+1})(\bar{s}_{n}^{2}-2\bar
{s}_{n}s_{n+1}+s_{n+1}^{2})=\\
&  =(1-\mu_{n+1})\kappa_{n}+\mu_{n+1}(1-\mu_{n+1})\bar{x}_{n+1}^{2}.
\end{align}
Now it is easy to get the assertion using again Lemma \ref{podstawowy} and
convergence of the series $\sum_{i\geq1}\mu_{i}\bar{x}_{i}^{2}$.
\end{proof}

Let us recall that in probability theory, we often meet the problem of almost
sure convergence of the sequences of the random variables of the form
$\left\{  Y_{n}=\frac{\sum_{i=1}^{n}X_{i}}{n}\right\}  $, where $\left\{
X_{i}\right\}  _{i\geq1}$ is the sequence of some random variables. We say
then that the strong law of large numbers is satisfied by the sequence
$\left\{  X_{i}\right\}  _{i\geq1}$. However the sequence $\left\{
Y_{n}\right\}  $ can be viewed as the sequence of Riesz's means of the
sequence of the random variables $\left\{  X_{i}\right\}  _{i\geq1}$ with
respect to the weight sequence $\left\{  1\right\}  $. Let us now recall
Remark following definition \ref{LLN}. We extend the notion of LLN in\emph{\ }
the following way:

\begin{definition}%
\index{Law!Large Numbers!Generalized}%
\label{uog_pwl}Let $\left\{  X_{i}\right\}  _{i\geq1}$ be a sequence random
variables such that $\forall i\geq1:$ $E\left\vert X_{i}\right\vert <\infty$.
For for some sequence positive numbers $\left\{  \alpha_{i}\right\}  _{i\geq
0}$ the following sequence: \newline%
\[
\left\{  \frac{\sum_{i=0}^{n-1}\alpha_{i}\left(  X_{i+1}-EX_{i+1}\right)
}{\sum_{i=0}^{n-1}\alpha_{i}}\right\}  _{n\geq1},
\]
is almost surely (in probability) convergent, then we say that the sequence
$\left\{  X_{i}\right\}  _{i\geq1}$ satisfies generalized strong (weak) law of
large numbers with respect to the sequence\emph{\ }$\left\{  \alpha
_{i}\right\}  _{i\geq0}.$
\end{definition}

Hence the generalized strong laws of large numbers are nothing else than
summing of some sequences of the random variables by the Riesz's method with
some weights. Let us notice that from the Lemma \ref{podstawowy} it follows
that the fact that SLLN is satisfied is strictly connected with the almost
sure convergence of some series composed of the random variables. Examining
the almost sure convergence of a series under very general assumptions
concerning random variables $\{X_{n}\}_{n\geq1}$ is very difficult and there
are not many results concerning this question. There exist, however many
results stating strong convergence of such series under some additional
assumptions concerning this sequence, such as independence, or lack of
correlation. There exists, as it turns out one more extremely important class
of sequences $\{X_{n}\}_{n\geq1}$ constituting the intermediate case between
independence, and a lack of correlation. Namely, the class of
'\emph{martingale differences}'. In the sequel, we will present series of
results concerning almost sure convergence of a series of the random
variables, under the assumption, that the random variables $\{X_{n}\}_{n\geq
1}$ are either martingale differences or are uncorrelated (that is orthogonal
in other terminology). Let us recall by the way, that the problem of
convergence of the so-called \emph{orthogonal series }is\emph{\ }sometimes
presented in more general, not only probabilistic context. We will present its
partial solution. By the way, we will try the methods presented above, by
examining the almost sure convergence of orthogonal and others, connected with
them, functional series. The notions of martingale and martingale difference
we discuss in Appendix \ref{martyngaly}.

\section{Convergence of series of the random variables}

In this section we will present a few results concerning almost sure
convergence of the following series
\begin{equation}
\sum_{i\geq1}X_{i}, \label{szer_podst}%
\end{equation}
where $\{X_{n}\}_{n\geq1}$ is the sequence of martingale differences with
respect to filtration $\left\{  \mathcal{G}_{n}\right\}  _{n\geq1}$. In
particular, sequence $\{X_{n}\}_{n\geq1}$ can consist of independent random
variables. When random variables have variances we have immediately:

\begin{theorem}
If the sequence $\{X_{n}\}_{n\geq1}$ consists of martingale differences with
respect to $\left\{  \mathcal{G}_{n}\right\}  _{n\geq1}$ and $\sum_{i\geq
1}\operatorname*{var}(X_{i})<\infty$, then the series (\ref{szer_podst})
converges almost surely.
\end{theorem}

\begin{proof}
It is enough to notice, that the sequence of partial sums of the series
(\ref{szer_podst}) is a martingale with respect to filtration $\left\{
\sigma(X_{1},\ldots,X_{n})\right\}  _{n\geq1}$, bounded in $L_{2}$, hence a.s. convergent.
\end{proof}

There exists an extension of this theorem that is coming from Doob.

\begin{theorem}
\label{ozbwL2}If the sequence $\{X_{n}\}_{n\geq1}$ consists of martingale
differences with respect to $\left\{  \mathcal{G}_{n}\right\}  _{n\geq1}$,
then for almost every elementary event $\omega$ we have:
\[
\sum_{i\geq1}E(X_{i}^{2}|\mathcal{G}_{i-1})<\infty\;\;\;\Rightarrow
\;\;\;\text{series}\;\sum_{i\geq1}X_{i}\;\text{converges.}%
\]

If additionally we assume that $E\left(  \underset{n}{\sup}\left\vert
X_{n}\right\vert ^{2}\right)  <\infty$, then we have also the following
implication that is satisfied for almost all $\omega$:
\[
\text{series}\;\sum_{i\geq1}X_{i}\;\text{converges.\ \ }\Rightarrow
\;\;\sum_{i\geq1}E(X_{i}^{2}|\mathcal{G}_{i-1})<\infty\;\text{.}%
\]

\end{theorem}

\begin{proof}
Let us assume that $X_{1}=0$ (one can always assume so, it will not affect
convergence). Let us fix $K>0$. Let $T_{K}$ be the smallest natural number $n$
such that $\sum_{i=1}^{n+1}E(X_{i}^{2}|\mathcal{G}_{i-1})>K$, if such $n$
exists and $T_{K}=\infty$, if there is not such $n$. $T_{K}$ is a stopping
time (see Appendix \ref{czas}), since the event $\left\{  T_{K}\leq n\right\}
$ depends only on random variables $E(X_{i}^{2}|\mathcal{G}_{i-1})$ for
$i=1,\ldots,n$. Let $S_{n}^{(T_{K})}=\sum_{i=1}^{n}X_{i}I(T_{K}\geq i)$.
$\left\{  S_{n}^{(T_{K})}\right\}  _{n\geq1}$ is a martingale and the sequence
$\left\{  X_{i}I(T_{K}\geq i)\right\}  _{i\geq1}$ consist of martingale
differences, since random variable $I(T_{K}\geq n)=1-I(T_{K}\leq n-1)$ is
$\mathcal{G}_{n-1}-$measurable and we have:
\[
E\left(  X_{n}I(T_{K}\geq n)|\mathcal{G}_{n-1}\right)  =I(T_{K}\geq n)E\left(
X_{n}|\mathcal{G}_{n-1}\right)  =0\;a.s.\,~.
\]
Because there is no correlation between the variables $\left\{  X_{i}%
I(T_{K}\geq i)\right\}  _{i\geq1}$ we have:
\begin{align*}
E\left(  S_{n}^{(T_{K})}\right)  ^{2}  &  =E\sum_{i=1}^{n}X_{i}^{2}I(T_{K}\geq
i)=E\sum_{i=1}^{n}E\left(  X_{i}^{2}I(T_{K}\geq i)|\mathcal{G}_{i-1}\right) \\
&  =E\sum_{i=1}^{n}I(T_{K}\geq i)E\left(  X_{i}^{2}|\mathcal{G}_{i-1}\right)
=E\sum_{i=1}^{\min(T_{K},n)}E\left(  X_{i}^{2}|\mathcal{G}_{i-1}\right)  \leq
K,
\end{align*}
since we have not reached yet the moment when the sum under the expectation
exceeds $K$. Martingale\allowbreak\ $\left\{  S_{n}^{(T_{K})}\right\}
_{n\geq1}$ is bounded in $L_{2}$, hence convergent. If $T_{K}=\infty$, notice,
that then the series $\sum_{i\geq1}X_{i}$ is convergent.\ Further, we have
$\left\{  \sum_{i\geq1}E(X_{i}^{2}|\mathcal{G}_{i-1})<\infty\right\}
\allowbreak=\allowbreak\bigcup_{K=1}^{\infty}\left\{  T_{K}=\infty\right\}  $,
hence indeed on the event \newline$\left\{  \sum_{i\geq1}E(X_{i}%
^{2}|\mathcal{G}_{i-1})<\infty\right\}  $ series $\sum_{i\geq1}X_{i}$ is convergent.

In order to get the second assertion for the fixed $K>1$, let us consider
random variable $T_{K}$ defined in the following way: $T_{K}$ is the smallest
natural number $n$ such that $\left\vert \sum_{i=1}^{n}X_{i}\right\vert >K$ or
$0$, if such natural number does not exist. $T_{K}$ is a stopping time. Let us
denote: $S_{n}^{T_{K}}=\sum_{i=1}^{n}X_{i}I(T_{K}\geq i)$. If $T_{K}>n$, then
of course $\left(  S_{n}^{T_{K}}\right)  ^{2}\leq K^{2}$, if $T_{K}\leq n$,
then
\begin{align*}
\left(  S_{n}^{T_{K}}\right)  ^{2}  &  =\left(  S_{T_{K}-1}+X_{T_{K}}\right)
^{2}\leq2\left(  S_{T_{K}-1}\right)  ^{2}+2\underset{n}{\sup}\left(
X_{n}\right)  ^{2}\\
&  \leq2K^{2}+2\underset{n}{\sup}\left(  X_{n}\right)  ^{2},
\end{align*}
since $T_{K}$ is the first number $n$ such that $\left\vert S_{n}\right\vert
>K$, hence earlier, that is e.g. at $T_{K}-1$ we had $\left\vert S_{T_{K}%
-1}\right\vert \leq K.$

Because of assumptions $E\underset{n}{\sup}X_{n}^{2}<\infty$ and random
variables $\left\{  X_{i}I(T_{K}\geq i)\right\}  _{i\geq1}$ are martingale
differences we have
\[
\infty>\underset{n}{\sup}E(S_{n}^{T_{K}})^{2}=\sum_{i\geq1}EX_{i}^{2}%
I(T_{K}\geq i).
\]
Moreover, we have:
\[
\sum_{i\geq1}E\left(  X_{i}^{2}I(T_{K}\geq i)\right)  \allowbreak=\allowbreak
E\sum_{i\geq1}I(T_{K}\geq i)E(X_{i}^{2}|\mathcal{G}_{i-1}).
\]
This means that series $\sum_{i\geq1}I(T_{K}\geq i)E(X_{i}^{2}|\mathcal{G}%
_{i-1})$ is almost surely convergent. In particular, the event $\left\{
T_{K}=\infty\right\}  $ implies convergence of the series $\sum_{i\geq
1}E(X_{i}^{2}|\mathcal{G}_{i-1})$. Finally, lest us notice, that the event
$\left\{  \text{series }\sum_{i\geq1}X_{i}\text{ is convergent}\right\}
\allowbreak$\newline$=\allowbreak\bigcup_{K=1}^{\infty}\left\{  T_{K}%
=\infty\right\}  .$
\end{proof}

Now we will apply this theorem to special random variables, namely variables
of the form $I(B_{i})$, where $\left\{  B_{i}\right\}  _{i\geq1}$ is some
sequence of events such that $B_{i}\in\mathcal{G}_{i}$. Let us notice that the
variables $\left\{  I(B_{i})-P\left(  B_{i}|\mathcal{G}_{i-1}\right)
\right\}  _{i\geq2}$ are martingale differences. We have the following
statement being generalization of assertion $i)$ of the Borel- Cantelli' Lemma
(see Appendix \ref{Borel-Cantelli}):

\begin{proposition}
\label{uog_Borel} $i)$ Event $\sum_{i\geq2}P\left(  B_{i}|\mathcal{G}%
_{i-1}\right)  <\infty$ implies \newline$\sum_{i\geq1}I(B_{i})<\infty$, or
equivalently $\left\{  B_{i}:f.o.\right\}  .$

$ii)$ Event $\sum_{i\geq2}P\left(  B_{i}|\mathcal{G}_{i-1}\right)  =\infty$
implies then $\underset{n\rightarrow\infty}{\lim}\frac{\sum_{i=1}^{n}I(B_{i}%
)}{\sum_{i=2}^{n}P\left(  B_{i}|\mathcal{G}_{i-1}\right)  }=1.$
\end{proposition}

\begin{proof}
Since the sequence $M_{n}=\sum_{i=2}^{n}\left[  I(B_{i})-P\left(
B_{i}|\mathcal{G}_{i-1}\right)  \right]  $, $n\geq2$ is a martingale, then
from the beginning of the previous theorem it follows that it converges, if
only series $\sum_{i\geq2}E\left[  \left(  I(B_{i})-P(B_{i}|\mathcal{G}%
_{i-1})\right)  ^{2}|\mathcal{G}_{i-1}\right]  $ converges. But we have:
\[
E\left[  \left(  I(B_{i})-P\left(  B_{i}|\mathcal{G}_{i-1}\right)  \right)
^{2}|\mathcal{G}_{i-1}\right]  =\allowbreak P\left(  B_{i}|\mathcal{G}%
_{i-1}\right)  (1-P\left(  B_{i}|\mathcal{G}_{i-1}\right)  ).
\]
Let us denote for brevity:
\begin{align*}
Y_{n}  &  =\sum_{i=2}^{n}P\left(  B_{i}|\mathcal{G}_{i-1}\right)  ,\\
Z_{n}  &  =\sum_{i=1}^{n}I(B_{i}),\\
A_{n}  &  =\sum_{i=2}^{n}P\left(  B_{i}|\mathcal{G}_{i-1}\right)  (1-P\left(
B_{i}|\mathcal{G}_{i-1}\right)  ),\;n=2,3,\ldots\,.
\end{align*}
Of course, convergence of the sequence $\left\{  Y_{n}\right\}  _{n\geq2}$
implies convergence of the sequence $\left\{  A_{n}\right\}  _{n\geq2}$. In
other words convergence of the sequence $\left\{  Y_{n}\right\}  _{n\geq2}$
implies convergence of the series $\sum_{i\geq1}I(B_{i})$. \newline Let us
suppose now, that $\sum_{i\geq2}P\left(  B_{i}|\mathcal{G}_{i-1}\right)
=\infty$. There are the following possibilities. Either sequence $\left\{
A_{n}\right\}  _{n\geq2}$ is convergent, then martingale\ $\left\{
M_{n}\right\}  _{n\geq2}$ is convergent and now it is easy to get the
assertion. However, if the sequence $\left\{  A_{n}\right\}  _{n\geq2}$ is
divergent, then we argue in the following way. Let us consider the sequence
\[
W_{n}=\sum_{i=1}^{n}\frac{M_{i}-M_{i-1}}{1+A_{i}}.
\]
It is martingale, since random variable $A_{i}$ is $\mathcal{G}_{i-1}$
measurable. We have
\[
E\left(  (W_{n}-W_{n-1})^{2}|\mathcal{G}_{n-1}\right)  \allowbreak
=\allowbreak(1+A_{n})^{-2}(A_{n}-A_{n-1})
\]
and
\[
(1+A_{n})^{-2}(A_{n}-A_{n-1})\leq(1+A_{n-1})^{-1}-(1+A_{n})^{-1}.
\]
Hence the series
\[
\sum_{n\geq2}E(W_{n}-W_{n-1})^{2}|\mathcal{G}_{n-1})
\]
is convergent. It means that the series
\[
\sum_{i=1}^{\infty}\frac{M_{i}-M_{i-1}}{1+A_{i}}%
\]
is a convergent martingale. For every elementary event belonging to the event
\[
\left\{  \sum_{i\geq1}P\left(  B_{i}|\mathcal{G}_{i-1}\right)  (1-P\left(
B_{i}|\mathcal{G}_{i-1}\right)  )=\infty\right\}  ,
\]
we apply Kronecker's Lemma, getting $\frac{M_{n}}{A_{n}}\rightarrow0$, when
$n\rightarrow\infty$. Consequently remembering that $A_{n}\leq\sum_{i=2}%
^{n}P\left(  B_{i}|\mathcal{G}_{i-1}\right)  $. we see that
\[
\frac{M_{n}}{\sum_{i=2}^{n}P\left(  B_{i}|\mathcal{G}_{i-1}\right)
}\rightarrow0,\text{that is~}\frac{\sum_{i=1}^{n}I(B_{i})}{\sum_{i=2}%
^{n}P\left(  B_{i}|\mathcal{G}_{i-1}\right)  }\rightarrow1,
\]
when $n\rightarrow\infty.$
\end{proof}

We have also the following theorem:

\begin{theorem}
\label{3_ciagi_mart}Let $\{X_{n}\}_{n\geq1}$ be a sequence adapted to the
filtration $\left\{  \mathcal{G}_{n}\right\}  _{n\geq1}$ (i.e. $X_{n}$ jest
$\mathcal{G}_{n}$ measurable). Then the series $\sum_{i\geq1}X_{i}$ converges
almost surely on an event such that for some constant $C>0$:
\begin{gather}
\sum_{i\geq1}P\left(  \left\vert X_{i}\right\vert >C|\mathcal{G}_{i-1}\right)
<\infty,\label{tw_3ciagi1}\\
\text{series }\sum_{i\geq1}E\left(  X_{i}I(\left\vert X_{i}\right\vert \leq
C)|\mathcal{G}_{i-1}\right)  \;\text{converges,}\label{tw_3ciagi2}\\
\text{and }\sum_{i\geq1}\operatorname*{var}\left(  X_{i}I(\left\vert
X_{i}\right\vert >C)|\mathcal{G}_{i-1}\right)  <\infty. \label{tw_3ciagi3}%
\end{gather}

\end{theorem}

\begin{proof}
Let $A$ denote an event defined by the relationships (\ref{tw_3ciagi1}),
(\ref{tw_3ciagi2}), (\ref{tw_3ciagi3}). Since (\ref{tw_3ciagi1}) is true, then
using Proposition \ref{uog_Borel} we see that events $\left\{  \left\vert
X_{i}\right\vert >C\right\}  _{i\geq1}$ will happen only a finite number of
times, hence the series
\[
\sum_{i\geq1}X_{i}I(\left\vert X_{i}\right\vert >C)
\]
is convergent.\ Further, it means that events $\left\{  \sum_{i\geq1}%
X_{i},\;\text{converges}\right\}  $ and \newline$\left\{  \sum_{i\geq1}%
X_{i}I(\left\vert X_{i}\right\vert \leq C)\;\text{converges}\right\}  $ are
identical on $A$. Since we have (\ref{tw_3ciagi2}), then of course we have
also
\[
\left\{  \sum_{i\geq1}X_{i},\;\text{converges}\right\}  =\left\{  \sum
_{i\geq1}X_{i}I(\left\vert X_{i}\right\vert \leq C)-E\left(  X_{i}I(\left\vert
X_{i}\right\vert \leq C)|\mathcal{G}_{i-1}\right)  \;\text{converges}\right\}
.
\]
Series
\[
\sum_{i\geq1}X_{i}I(\left\vert X_{i}\right\vert \leq C)-E\left(
X_{i}I(\left\vert X_{i}\right\vert \leq C)|\mathcal{G}_{i-1}\right)
\]
is a martingale, that converges by Theorem \ref{ozbwL2}, since we have
(\ref{tw_3ciagi3}).
\end{proof}

When we deal with random variables that are independent the theorem can be
reversed. Namely, we have:

\begin{theorem}
[Ko\l mogorov's three series]%
\index{Theorem!Kolmogorov!3 series}%
Let $\{X_{n}\}_{n\geq1}$ be a sequence of independent random variables. Series
$\sum_{n\geq1}X_{n}$ converges if and only if, for some $K>0$ the following
three series are convergent:
\begin{subequations}
\begin{gather}
\sum_{n\geq1}P(\left\vert X_{n}\right\vert >K),\label{3s1}\\
\sum_{n\geq1}EX_{n}^{K},\label{3s2}\\
\sum_{n\geq1}\operatorname*{var}(X_{n}^{K}), \label{3s3}%
\end{gather}
where we denoted
\end{subequations}
\[
X_{n}^{K}=\left\{
\begin{array}
[c]{ccc}%
X_{n} & gdy & \left\vert X_{n}\right\vert \leq K\\
0 & gdy & \left\vert X_{n}\right\vert >K
\end{array}
\right.  .
\]

\end{theorem}

\begin{proof}
Implication $\Leftarrow$ is obvious. We apply Theorem \ref{3_ciagi_mart} and
remember, that for independent random variables one has to substitute
conditional expectations by unconditional ones.

Implication $\Rightarrow$, that is, let us assume that the series $\sum
_{n\geq1}X_{n}$ is convergent. It means, in particular, that
$\underset{n\rightarrow\infty}{\lim}X_{n}=0$ almost surely. This fact on its
side, implies that the events $\left\{  \left(  \left\vert X_{n}\right\vert
>K\right)  \right\}  _{n\geq1}$ will happen only a finite number of times.
Independence and assertion \emph{iii) }of\emph{\ }the\emph{ }Borel-Cantelli'
Lemma give convergence of the series (\ref{3s1}). In order to show the
convergence of the remaining series let us consider symmetrization of the
random variables $X_{n}^{K}$, $n\geq1$, i.e. let us consider their independent
copies $\left(  X_{n}^{K}\right)  ^{\prime}$, $n\geq1$ and random variables
$X_{n}^{s}=X_{n}^{K}-\left(  X_{n}^{K}\right)  ^{\prime}$. Of course,
convergence of the series $\sum_{n\geq1}X_{n}^{K}$ implies convergence of the
series $\sum_{n\geq1}X_{n}^{s}$. We have also $\left\vert X_{n}^{s}\right\vert
\leq2K$ for all $n\in%
%TCIMACRO{\U{2115} }%
%BeginExpansion
\mathbb{N}
%EndExpansion
$. Now we apply the second part of Theorem \ref{ozbwL2} and deduce that
$\sum_{i\geq1}\operatorname*{var}(X_{i}^{K})<\infty$, since
\[
E(\left(  X_{n}^{s}\right)  ^{2}|\sigma(X_{1},\ldots,X_{n-1}%
))=2\operatorname*{var}(X_{n}^{K}).
\]
Further convergence of the series $\sum_{i\geq1}\operatorname*{var}(X_{i}%
^{K})$ implies convergence of the series $\sum_{i\geq1}(X_{i}^{K}-EX_{i}^{K}%
)$, which connected with the convergence of the series $\sum_{i\geq1}X_{i}%
^{K}$ gives convergence of the series (\ref{3s2}).
\end{proof}

\subsection{Orthogonal series}

\label{sz_ort}Orthogonal series it is an interesting class of functional
series. It was intensively examined in 1920-60 by many excellent
mathematicians such as Menchoff, Steinhaus, Kaczmarz, Zygmund, Riesz, Hardy
and Littlewood. Some of their results will be possible to get directly from
the presented above lemmas and theorems concerning convergence number
sequences. The present chapter\ can be viewed as the "test for the usefulness
of methods developed above".

Since there exist strong links of the present subsection with the mathematical
analysis we will present first the problem of convergence of orthogonal series
generally using terminology accepted in the analysis. Later we shall confine
ourselves to probabilistic terminology.

Let on the measure space $([a,b],\mathcal{B}\emph{(}[a,b]),\mu(.))$, (where
$\mathcal{B}\emph{(}[a,b])$ is Borel $\sigma$-field of the segment $[a,b]$,
and $\mu$ some finite measure on \emph{B)} be defined the following
functions:
\begin{gather*}
\phi_{i}:[a,b]\rightarrow%
%TCIMACRO{\U{211d} }%
%BeginExpansion
\mathbb{R}
%EndExpansion
;\;\int_{[a,b]}\left\vert \phi_{i}(x)\right\vert ^{2}\mu\left(  dx\right)
=1,\\
\int_{\lbrack a,b]}\phi_{i}(x)\phi_{j}(x)\mu(dx)=0;\;i,j=1,\ldots;\;i\neq j.
\end{gather*}
The cases $a\allowbreak=\allowbreak-\infty$ and $b\allowbreak=\allowbreak
\infty$ are allowed.

Such sequence of functions is called \emph{orthonormal system }%
\index{System!orthonormal}%
. For any of functions $f\in L^{2}([a,b],\mathcal{B}([a,b]),\mu(.))$ we define
series
\[
\mathcal{S}_{f}=\sum_{i\geq1}c_{i}\phi_{i},
\]
where $c_{i}=\int_{[a,b]}f(x)\phi_{i}(x)\mu(dx)$. Does the series
$\mathcal{S}_{f}$ has any connection with the function $f$ ? It turns out that
it converges in $L^{2}$ to $f$ if and only if, the following Parseval's
identity is satisfied:
\[
\int_{\lbrack a,b]}f^{2}(x)\mu(dx)=\sum_{i\geq1}c_{i}^{2}.
\]
Does it converge almost everywhere to $f\,?$ It turns out that not always.
Moreover, it turns out, that the answer depends:

\begin{enumerate}
\item on coefficients $\left\{  c_{i}\right\}  _{i\geq1};$ more precisely, on
the speed, with which they converge to zero

\item on the form of the functions $\left\{  \phi_{i}\right\}  _{i\geq1}$
constituting the orthonormal system.
\end{enumerate}

We will present now two theorems concerning those two points.

On the way we will use the following conventions and notation:

\begin{itemize}
\item all considered below logarithms will be with base $2,$

\item $\log_{+}x=\max(\log x,1)$, for $x>0,$

\item $\log_{+}\frac{a}{b}=\log(a/b)$ when $(a/b)\geq2$ and $1$ if
$(a/b)\in\lbrack0,2)$ or $b=0.$
\end{itemize}

As far as the first property, we have the following result.

\begin{theorem}
[Rademacher-Menchoff's]%
\index{Theorem!Rademacher-Menchoff}%
\label{menshov} Let be given an orthonormal system $\left\{  \phi_{i}\right\}
_{i\geq1}$. If the real sequence $\left\{  c_{i}\right\}  _{i\geq1}$ is such
that
\begin{equation}
\sum_{i\geq1}c_{i}^{2}\log^{2}i<\infty, \label{rademacher-menchoff}%
\end{equation}
then the functional series $\sum_{i\geq1}c_{i}\phi_{i}(x)$ converges for
almost every (mod $\mu)$ $x\in\lbrack a,b].$
\end{theorem}

Proof of this theorem is elementary, although not simple. It is based on the
following lemma.

\begin{lemma}
\label{fundamental_inequality}Let $\left\{  \theta_{i}(x)\right\}
_{i=1}^{\infty}$ be a sequence of mutually orthogonal functions defined on
$([a,b],\mathcal{B}\emph{(}[a,b]),\mu(.))$. Let $S_{i}=\sum_{j=1}^{i}%
\theta_{j}$. Then:
\begin{equation}
\int_{a}^{b}\left(  \underset{1\leq i\leq n}{\max}S_{i}^{2}\right)
d\mu(x)\leq O(\log^{2}n)\sum_{i=1}^{n}\int_{a}^{b}\theta_{i}^{2}d\mu(x).
\label{nier_fund}%
\end{equation}

\end{lemma}

\begin{proof}
Let us set $S_{0}=0$. By $\nu_{n}(x)$ let us denote an index (possibly
depending on $x)$, not greater than $n$, such that.:
\[
\underset{0\leq i\leq n}{\max}\left\vert S_{i}\right\vert =|S_{\nu_{n}}|.
\]
Let us denote by $S_{k}^{\alpha}$ the $k-$ th $(C,\alpha)$ mean, $\alpha>-1$.
Let us notice also, that from formula (\ref{sr_szeregu}) it follows that
$S_{k}$ is equal to the $k-$th $(C,0)$ mean of our series. Let us apply
assertion $iv)$ of Lemma \ref{Cesaro}. We will get:
\begin{align*}
\underset{0\leq i\leq n}{\max}\left\vert S_{i}\right\vert  &  =|S_{\nu_{n}%
}^{o}|\leq\sum_{k=0}^{\nu_{n}}A_{\nu_{n}-k}^{-1/2}A_{k}^{-1/2}\left\vert
S_{k}^{-1/2}\right\vert \leq\\
&  \leq\sqrt{\sum_{k=0}^{\nu_{n}}\left(  A_{\nu_{n}-k}^{-1/2}\right)  ^{2}%
\sum_{k=0}^{n}\left(  A_{k}^{-1/2}S_{k}^{-1/2}\right)  ^{2}}\overset{df}{=}%
\delta_{n}(x).
\end{align*}
Taking advantage of assertion $iii)$ of Lemma \ref{Cesaro} we get
$A_{k}^{-1/2}=O\left(  k^{-1/2}\right)  $ and further:
\[
\sum_{k=0}^{\nu_{n}}\left(  A_{\nu_{n}-k}^{-1/2}\right)  ^{2}=1+\sum
_{k=0}^{\nu_{n}-1}O(\frac{1}{\nu_{n}-k})=O(\log(\nu_{n}))\leq O(\log n).
\]
Further, using assertion $vii)$ of Lemma \ref{Cesaro} and the above mentioned
estimation we get:
\begin{align*}
\int_{a}^{b}\delta_{n}^{2}(x)d\mu(x)  &  \leq O(\log n)\sum_{k=0}^{n}\int%
_{a}^{b}\left(  \sum_{j=1}^{k}A_{k-j}^{-1/2}\theta_{j}(x)\right)  ^{2}%
d\mu(x)=\\
&  =O(\log n)\sum_{k=0}^{n}\sum_{j=1}^{k}\left(  A_{k-j}^{-1/2}\right)
^{2}\int_{a}^{b}\theta_{j}^{2}(x)d\mu(x)=\\
&  =O(\log n)\sum_{j=1}^{n}\int_{a}^{b}\theta_{j}^{2}(x)d\mu(x)\left[
1+\sum_{k=j+1}^{n}O\left(  \frac{1}{n-k}\right)  \right]  =\\
&  =O(\log^{2}n)\sum_{j=1}^{n}\int_{a}^{b}\theta_{j}^{2}(x)d\mu(x).
\end{align*}

\end{proof}

\begin{proof}
[Proof of the Rademacher-Menchoff theorem]In order to prove Theorem
\ref{menshov} let us denote $r_{n}=\sum_{i=n}^{\infty}c_{i}\phi_{i}$. We
have:
\begin{gather*}
\int_{a}^{b}r_{2^{n}}^{2}d\mu(x)\allowbreak=\sum_{i=2^{n}}^{\infty}\int%
_{a}^{b}c_{i}^{2}\phi_{i}^{2}d\mu(x)=\\
\frac{1}{\left(  \log2^{n}\right)  ^{2}}\sum_{i=2^{n}}^{\infty}(\log2^{n}%
)^{2}c_{i}^{2}\int_{a}^{b}\phi_{i}^{2}d\mu(x)\leq\\
\frac{1}{\left(  \log2^{n}\right)  ^{2}}\sum_{i=2^{n}}^{\infty}(\log
i)^{2}c_{i}^{2}\int_{a}^{b}\phi_{i}^{2}d\mu(x)\leq\frac{C}{n^{2}},\allowbreak
\end{gather*}
where $C=\sum_{i\geq2}\left(  \log i\right)  ^{2}c_{i}^{2}\int_{a}^{b}\phi
_{i}^{2}(x)d\mu(x).$Thus, the series $\sum_{n\geq1}\int_{a}^{b}r_{2^{n}}%
^{2}d\mu(x)$ is convergent, and consequently sequence $\left\{  r_{2^{n}%
}\right\}  _{n\geq1}$ converges almost everywhere to zero. This means that
also the subsequence $\left\{  S_{2^{n}}\right\}  _{n\geq1}$ of the sequence
of partial sums $\left\{  S_{i}\right\}  _{i\geq1}$ of the series $\sum
_{i\geq1}c_{i}\phi_{i}(x)$ converges almost everywhere. In order to show, that
the sequence $\left\{  S_{i}\right\}  _{i\geq1}$ converges almost everywhere,
it is enough to show, that the functional sequence :
\[
\underset{2^{n}<i\leq2^{n+1}}{\max}(S_{i}-S_{2^{n}})^{2}%
\]
converges to zero almost everywhere. We have however on the basis of Lemma
\ref{fundamental_inequality}:
\begin{align*}
\int_{a}^{b}\underset{2^{n}<i\leq2^{n+1}}{\max}(S_{i}-S_{2^{n}})^{2}d\mu(x)
&  \leq\left[  \log(2^{n+1}-2^{n})\right]  ^{2}\sum_{j=2^{n}+1}^{2^{n+1}}%
c_{j}^{2}\int_{a}^{b}\phi_{j}^{2}d\mu(x)\\
\allowbreak &  =n^{2}\sum_{j=2^{n}+1}^{2^{n+1}}c_{j}^{2}\int_{a}^{b}\phi
_{j}^{2}d\mu(x).
\end{align*}
Moreover, we have:
\begin{align*}
\sum_{n\geq1}(\log2^{n})^{2}\sum_{j=2^{n}+1}^{2^{n+1}}c_{j}^{2}\int_{a}%
^{b}\phi_{j}^{2}d\mu(x)\allowbreak &  \leq\sum_{n\geq1}\sum_{j=2^{n}%
+1}^{2^{n+1}}\left(  \log j\right)  ^{2}c_{j}^{2}\int_{a}^{b}\phi_{j}^{2}%
d\mu(x)\\
&  =\sum_{j\geq3}\left(  \log j\right)  ^{2}c_{j}^{2}\int_{a}^{b}\phi_{j}%
^{2}d\mu(x)\allowbreak<\infty.
\end{align*}
A hence sequence $\underset{2^{n}\leq i<2^{n+1}}{\max}(S_{i}-S_{2^{n}})^{2}$
converges almost everywhere to zero.
\end{proof}

In order to illustrate the second point, we quote the following second
Menchoff's Theorem:

\begin{theorem}
[Menchoff]%
\index{Theorem!Menchoff}%
For every non-increasing number sequence $\left\{  c_{i}^{2}\right\}
_{i\geq1}$, and satisfying conditions $\sum_{i\geq1}c_{i}^{2}<\infty$ and
$\sum_{i\geq1}c_{i}^{2}\log^{2}i=\infty$ it is possible to construct such
orthonormal system $\left\{  \phi_{i}\right\}  _{i\geq1}$, that the series
$\sum_{i\geq1}c_{i}\phi_{i}(x)$ is almost everywhere divergent!
\end{theorem}

Proof of this theorem is very complex. It can be found e.g. in the book of
Alexits \cite{Alexits}.

Above mentioned theorems state, that if only sequence of coefficients
$\left\{  c_{i}^{2}\right\}  _{i\geq1}$ is monotone, then condition
(\ref{rademacher-menchoff}) guaranteeing convergence of the series
$\sum_{i\geq1}\allowbreak c_{i}\phi_{i}(x)$, cannot be improved. It turns out,
however, that when the sequence $\left\{  c_{i}^{2}\right\}  _{i\geq1}$ is not
monotone, then this condition can be improved. In 1965 Tandori in the paper
\cite{Tandori65} replaced\ condition (\ref{rademacher-menchoff}) with the
condition
\begin{equation}
\sum_{k=3}c_{k}^{2}\log k\log_{+}\frac{1}{c_{k}^{2}}<\infty. \label{tandori}%
\end{equation}
It turns out that this condition and (\ref{rademacher-menchoff}) are
equivalent, when the sequence $\left\{  c_{i}^{2}\right\}  _{i\geq1}$ is
non-increasing M\'{o}ricz and Tandori have improved slightly condition
(\ref{tandori}) for the first time in the paper \cite{Moricz94} and then in
the paper \cite{Moricz96}, namely it turned out, that if only $\exists
\varepsilon\in(0,2]:$
\[
\sum_{n\geq0}\sum_{k\in Z(n)}c_{k}^{2}\left(  \log k\right)  ^{\varepsilon
}\left(  \log_{+}\frac{2A_{n}}{c_{k}^{2}}\right)  ^{2-\varepsilon}<\infty,
\]
where
\[
Z(n)=\left\{  2^{n}+1,2^{n}+2,\ldots,2^{n+1}\right\}  ,~A_{n}=\sum_{k\in
Z(n)}c_{k}^{2},
\]
then the orthogonal series $\sum_{k\geq1}c_{k}\phi_{k}(x)$ is convergent.

Hence only for some orthogonal systems, one can expect equivalence of
convergence in $L^{2}$ and almost sure convergence. What are those systems? A
great achievement of mathematical analysis of the $60-$ties was
\emph{Carleson's Theorem}%
\index{Theorem!Carleson}
stating, that \emph{system of trigonometric functions} has this property. And
what about other, broader classes of such orthogonal systems?

Let us notice that the fact that we have considered so far space
$([a,b],\allowbreak\mathcal{B}\emph{(}[a,b]),\allowbreak\mu(.))$ is not very
important. Orthogonality can be defined on any finite measure space. It is
also not important that the measure $\mu$ could have been not normalized.
Hence, one can consider some probability space $(\Omega,\mathcal{F},P)$ and
the above mentioned problems express in probabilistic terms. Namely, the
r\^{o}le of functions $\left\{  \phi_{i}\right\}  _{i\geq1}$ satisfy sequences
$\left\{  X_{i}\right\}  $ of uncorrelated random variables having variances
equal to $1$ and zero (for $i\geq2)$ expectations. R\^{o}le of functions $f$
would be played by the sums $\sum_{i\geq1}c_{i}X_{i}$ such that $\sum_{i\geq
1}c_{i}^{2}<\infty$. The question about almost everywhere convergence of the
orthogonal series would concern classes of sequences $\left\{  X_{i}\right\}
_{i\geq1}$, for which convergence in $L^{2}$ of the series $\sum_{i\geq1}%
c_{i}X_{i}$ implies almost sure convergence.

Finally, let us notice, that there exists a strict connection between
orthogonal series, and generalized, strong laws of large numbers for
uncorrelated random variables. Namely, let $\left\{  X_{i}\right\}  _{i\geq1}$
be a sequence uncorrelated random variables, and $\sum_{i\geq1}c_{i}X_{i}$ let
be any orthogonal series, constructed with the help those random variables.
Let further $\left\{  \mu_{i}\right\}  _{i\geq0}$ be any sequence positive
numbers, satisfying assumption \emph{i) }of\emph{\ }Lemma \ref{iteracje}, i.e.
normal sequence. Let $\overline{\left\{  \alpha_{i}\right\}  }_{i\geq
0}=\left\{  \mu_{i}\right\}  _{i\geq0}$. Let us denote:
\begin{subequations}
\begin{align}
T_{0}  &  =0;~T_{n}=\frac{\sum_{i=0}^{n-1}\alpha_{i}c_{i+1}X_{i+1}/\mu_{i}%
}{\sum_{i=0}^{n-1}\alpha_{i}};n\geq1,\label{defti}\\
S_{0}  &  =0;~S_{n}=\sum_{i=1}^{n}c_{i}X_{i};n\geq1,\label{sumaszer}\\
\bar{S}_{n}  &  =\frac{\sum_{i=0}^{n}\alpha_{i}S_{i}}{\sum_{i=0}^{n}\alpha
_{i}};n\geq0. \label{srszumszer}%
\end{align}

In view of the above mentioned considerations it is clear, that the sequence
$\left\{  T_{i}\right\}  _{i\geq0}$ is a sequence of Riesz's means of the
sequence $\left\{  c_{i+1}X_{i+1}/\mu_{i}\right\}  _{i\geq0}$ of uncorrelated
random variables with respect to the sequence of weights $\left\{  \alpha
_{i}\right\}  _{i\geq0}$ and satisfies the following iterative equation:
\end{subequations}
\begin{equation}
T_{n+1}=(1-\mu_{n})T_{n}+c_{n+1}X_{n+1}. \label{rekurencyjneti}%
\end{equation}

As it follows from the auxiliary lemmas presented in sections
\ref{lemliczbowe} and \ref{Riesz} there exists a strict connection between
almost surely convergence of the series $\sum_{i\geq1}c_{i}X_{i}$, and almost
surely convergence to zero of the sequence $\left\{  T_{i}\right\}  _{i\geq0}$.

Conversely, having given a sequence of Riesz's means $\left\{  \bar{X}%
_{n}=\frac{\sum_{i=0}^{n-1}\alpha_{i}X_{i}}{\sum_{i=0}^{n-1}\alpha_{i}%
}\right\}  _{n\geq1}$ of the sequence of uncorrelated random variables
$\left\{  X_{i}\right\}  _{i\geq1}$ with respect to sequence $\left\{
\alpha_{i}\right\}  _{i\geq1}$, we can present it in a recursive form:
\[
\bar{X}_{n+1}=(1-\mu_{n})\bar{X}_{n}+\mu_{n}X_{n+1}.
\]
And again, there appears orthogonal series $\sum_{i\geq0}\mu_{i}X_{i+1}.$

As it follows from lemmas presented in sections \ref{lemliczbowe} and
\ref{Riesz}, examining of convergence of Riesz's means requires examining of
the convergence of some series, and examining of convergence of the series is
connected with examining the convergence of some Riesz's means.

In order to briefly describe properties of Riesz's means of orthogonal series,
let us introduce also the following sequence of indices $\left\{
n_{k}\right\}  _{k\geq1}$ defined in the following way:
\begin{equation}
\sum_{j=n_{k}+1}^{n_{k+1}}\mu_{j}=O(1);k\geq1. \label{indeksy}%
\end{equation}

We have the following simple, general lemma.

\begin{lemma}
\label{prostypomocniczy} Let be given converging in $L_{2}$ orthogonal series
$\sum_{i\geq1}c_{i}X_{i}$ and normal number sequence $\left\{  \mu
_{i}\right\}  _{i\geq0}$. Let sequences of the random variables $\left\{
T_{i}\right\}  _{i\geq0}$, $\left\{  S_{i}\right\}  _{i\geq0}$, $\left\{
\bar{S}_{i}\right\}  _{i\geq0}$ be defined relationships respectively
(\ref{defti}), (\ref{sumaszer}), (\ref{srszumszer}). Then:

\begin{enumerate}
\item \label{pierwsza} series $\sum_{i\geq0}\mu_{i}ET_{i}^{2}$ is convergent
and series $\sum_{i\geq0}\mu_{i}T_{i}^{2}$ is convergent a.s.,

\item \label{sr_sumcz} $\bar{S}_{n}=\sum_{i=0}^{n}\mu_{i}T_{i}$ a.s. for
$n=0,1,\ldots$ ,

\item \label{zb_sred}$\frac{\sum_{i=0}^{n}\alpha_{i}T_{i}^{2}}{\sum_{i=0}%
^{n}\alpha_{i}}\underset{n\rightarrow\infty}{\longrightarrow}0$ a.s.,

\item \label{zb_podciagu} $T_{n_{k}}\underset{k\rightarrow\infty
}{\longrightarrow}0$ a.s.,

\item \label{zb_wariancji}Let $V_{n}=\frac{\sum_{i=0}^{n}\alpha_{i}S_{i}^{2}%
}{\sum_{i=0}^{n}\alpha_{i}}-\left(  \bar{S}_{n}\right)  ^{2}$. Then almost
surely $V_{n}\underset{n\rightarrow\infty}{\longrightarrow}0$ and the series
$\sum_{i\geq1}\mu_{i}V_{i}$ is convergent,

\item \label{zb_podciagu2} Subsequence $\left\{  S_{n_{k}}\right\}  _{k\geq1}$
converges almost surely to some square integrable random variable if and only
if, the series $\sum_{i\geq0}\mu_{i}T_{i}$ converges almost surely,

\item \label{zb_kwadrat}$T_{n}\underset{n\rightarrow\infty}{\longrightarrow}0$
a.s. if and only if the series $\sum_{i\geq1}\mu_{i}T_{i}X_{i+1}$ converges
almost surely,

\item \label{zb_podciagu3} If almost surely $T_{n}\underset{n\rightarrow
\infty}{\longrightarrow}0$, then almost sure convergence of the sequence
$\left\{  S_{n}\right\}  _{n\geq0}$ to some square integrable random variable
is equivalent to the a.s. convergence of the subsequence $\left\{  S_{n_{k}%
}\right\}  _{k\geq1}$ to the same random variable.
\end{enumerate}
\end{lemma}

Before we will present proof of this lemma, we will make a few remarks.

\begin{remark}
Let us notice that assertions \ref{zb_podciagu} and \ref{zb_podciagu2} remain
true, if the subsequence $\left\{  n_{k}\right\}  _{k\geq1}$ was defined in
the following way: $1/\sum_{j=n_{k}+1}^{n_{k+1}}\mu_{j}\underset{k\rightarrow
\infty}{\longrightarrow}0$. On the other hand assertion \ref{zb_podciagu3}
remain true, if the subsequence $\left\{  n_{k}\right\}  _{k\geq1}$ was
defined by the relationship: $\sum_{j=n_{k}+1}^{n_{k+1}}\mu_{j}%
\underset{k\rightarrow\infty}{\longrightarrow}0.$
\end{remark}

\begin{remark}%
\index{Theorem!Zygmund}%
Assertion \ref{zb_podciagu2} together with assertion \ref{sr_sumcz} are
strictly connected with Zygmund's theorem concerning Riesz summability of
orthogonal series (see \cite{Alexits}, th.. 2.8.7). Let us recall that this
theorem states, that Riesz summability of the orthogonal series with some
weights that converges in $L_{2}$ is equivalent to convergence of some
subsequences (defined by the system of weights) of the sequence of partial
sums. Strictly speaking, Zygmund understands Riesz summability of series in a
slightly different way, namely he defines summability to $s$ of the series
$\sum_{i\geq0}u_{i}$ with respect to some increasing weight sequence $\left\{
\lambda_{n}\right\}  _{n\geq0}$ as the convergence to $s$ of the sequence
\[
\left\{  \sum_{k=0}^{n}(1-\frac{\lambda_{k}}{\lambda_{n+1}})u_{k}\right\}
_{n\geq1}.
\]
We leave it to the reader as a simple exercise to check, that this definition
and considered above definition \ref{Sum_Riesz} are equivalent as far as the
series are concerned, when one takes $\lambda_{n}=\sum_{i=0}^{n-1}\alpha_{i}$.
Using Zygmund's terminology, Zygmund's Theorem states, that orthogonal series
$\sum_{i\geq0}c_{i}\phi_{i}(x)$, whose coefficients satisfy condition
$\sum_{i\geq0}c_{i}^{2}<\infty$, is summable almost everywhere with respect to
sequence $\left\{  \lambda_{n}\right\}  _{n\geq0}$ if and only if, the
following subsequence sequence of partial sums $\left\{  \sum_{i=\geq0}%
^{n_{k}}c_{i}\phi_{i}(x)\right\}  _{k\geq1}$ is convergent almost surely, here
sequence of indices $\left\{  n_{k}\right\}  $ is defined with the help of the
following condition:
\[
1<q\leq\frac{\lambda_{n_{k+1}}}{\lambda_{n_{k}}}\leq r,
\]
where $1<q\leq r$ are some real numbers. We will show that this theorem is in
fact equivalent to assertions \ref{zb_podciagu2} and \ref{sr_sumcz} of the
lemma. It can be deduced arguing in the following way. Firstly, from assertion
\ref{sr_sumcz} we know, that the sequence of partial sums of the series
$\sum_{i\geq1}\mu_{i}T_{i}$ constitutes also a sequence respective Riesz's
means of partial sums of the orthogonal series. Assertion \ref{zb_podciagu2}
gives an equivalence of summability of the series $\sum_{i\geq1}\mu_{i}T_{i}$
and the convergence of the respective subsequence of the sequence of partial
sums of the orthogonal series. Thus, it remained to check, if the subsequence
defined in assertion \ref{zb_podciagu2} is the same, as in Zygmund's theorem.
Let us denote $\lambda_{n}=\sum_{i=0}^{n-1}\alpha_{i}$. Let $\overline
{\left\{  \alpha_{i}\right\}  }=\left\{  \mu_{i}\right\}  $. Then, as it
follows from the proof of proposition \ref{o_ciagach_normalnych} we have:
\[
\lambda_{n}=\prod_{i=1}^{n-1}(1-\mu_{i})^{-1}\geq\exp(\sum_{i=0}^{n-1}\mu
_{i}).
\]
Hence, using inequality $\frac{1}{1-x}\leq1+x+3x^{2}$ that is true for all
$x\leq2/3$ and taking $k$ big enough that $\mu_{n_{k}}\leq2/3$ (it is
possible, since the sequence $\left\{  \mu_{n}\right\}  $ converges to zero),
we get:
\begin{align*}
r  &  \geq\exp\left(  \sum_{j=n_{k}+1}^{n_{k+1}-1}\mu_{j}+3\sum_{j=n_{k}%
+1}^{n_{k+1}-1}\mu_{j}^{2}\right)  \geq\prod_{i=n_{k}+1}^{n_{k+1}-1}(1+\mu
_{i}+3\mu_{i}^{2})\geq\\
&  \geq\prod_{i=n_{k}+1}^{n_{k+1}-1}(1-\mu_{i})^{-1}=\lambda_{n_{k+1}}%
/\lambda_{n_{k}}\geq\exp(\sum_{j=n_{k}+1}^{n_{k+1}-1}\mu_{j})\geq q>1,
\end{align*}
where
\[
q=\exp(\underset{k\rightarrow\infty}{\lim\inf}\,\sum_{j=n_{k}+1}^{n_{k+1}%
-1}\mu_{j}),~r=\exp(\underset{k\rightarrow\infty}{\lim\inf}\,\sum_{j=n_{k}%
+1}^{n_{k+1}-1}\mu_{j}).
\]
Sequence$\left\{  n_{k}\right\}  $ defined by (\ref{indeksy}) satisfies
conditions of Zygmund's theorem.
\end{remark}

\begin{remark}
It is easy to get the following observation basing on assertion \ref{sr_sumcz}%
: \emph{orthogonal series is absolutely summable by the Riesz's method if and
only if, series }$\sum_{i\geq1}\mu_{i}\left\vert T_{i}\right\vert $\emph{ is
convergent almost surely.} This statement and a few corollaries following it
constitute the main subject of the paper \cite{Okuyama81}.
\end{remark}

\begin{remark}%
\index{Theorem!Hardy-Littlewood}%
Let us take $\mu_{i}=\frac{1}{i+1};$ $i\geq0$. Then the assertion
\ref{zb_wariancji} states, that
\[
V_{n}=\frac{\sum_{i=0}^{n}S_{i}^{2}}{n+1}-\left(  \bar{S}_{n}\right)  ^{2}%
\]
converges almost surely to zero. Let us transform a bit this quantity. It is
not difficult to notice, that
\[
V_{n}=\frac{\sum_{i=0}^{n}\left(  S_{i}-S\right)  ^{2}}{n+1}-\left(  \bar
{S}_{n}-S\right)  ^{2},
\]
where by $S$ we denoted the limit in $L_{2}$ of the our orthogonal series.
Hence, one can notice, that if $\left(  \bar{S}_{n}-S\right)  ^{2}%
\underset{n\rightarrow\infty}{\longrightarrow}0$ almost surely, then and
\[
\frac{\sum_{i=0}^{n}\left(  S_{i}-S\right)  ^{2}}{n+1}\underset{n\rightarrow
\infty}{\longrightarrow}0
\]
almost surely. This observation means, that Ces\`{a}ro summability of order
$1$ of the orthogonal series is equivalent to its strong summability. For
Fourier series, this theorem was formulated by Hardy and Littlewood., and
later generalized by Zygmund. (See comments and observations in \cite{Alexits}
p. 111).
\end{remark}

\begin{remark}
Continuing analysis of the case $\mu_{i}=\frac{1}{i+1};$ $i\geq0$ let us
consider a sequence $\left\{  X_{i}\right\}  _{i\geq1}$ of the random
variables with zero expectations and finite variances. More precisely, let us
assume: $\operatorname*{var}(X_{i})=\sigma_{i}^{2}$. Let us set also
$c_{i}=\frac{1}{i+1};$ $i\geq1$. Let us consider also sequence $\left\{
\widetilde{X}_{n}\right\}  _{n\geq1}$ of Ces\`{a}ro means of order $2$ created
from variables $\left\{  X_{i}\right\}  _{i\geq1}$. It is not difficult to
notice, that
\[
\widetilde{X}_{n}=\frac{2}{n(n+1)}\sum_{i=1}^{n}(n-i)X_{i}%
\]
and moreover, that the sequence $\left\{  \widetilde{X}_{n}\right\}  _{n\geq
1}$ satisfies recurrent relationship:
\[
\widetilde{X}_{n+1}=(1-\frac{2}{n+2})\widetilde{X}_{n}+\frac{2}{\left(
n+2\right)  }\bar{X}_{n}.
\]
Hence we have
\[
\widetilde{X}_{n+1}^{2}\leq(1-\frac{2}{n+2})\widetilde{X}_{n}^{2}+\frac
{2}{\left(  n+2\right)  }\bar{X}_{n}^{2}.
\]
Consequently , if the series $\sum_{n\geq1}\frac{1}{n}\bar{X}_{n}^{2}$ is
convergent with probability $1$, then the sequence $\left\{  \widetilde{X}%
_{n}^{2}\right\}  _{n\geq1}$ converges with probability $1$ to zero. Since
$E\bar{X}_{n}^{2}=\frac{1}{n^{2}}\sum_{i=1}^{n}\sigma_{i}^{2}$ we see that if
only series $\sum_{n\geq0}\frac{1}{n^{3}}\sum_{i=1}^{n}\sigma_{i}^{2}$ is
convergent, then the sequence $\{X_{n}\}_{n\geq1}$ is $(C,2)$ summable. This
corollary will be generalized in Theorem \ref{(C,alfa)}.
\end{remark}

\begin{remark}
From assertion \ref{zb_sred} it follows that e.g.
\[
\bar{T}_{n}=\allowbreak\sum_{i=0}^{n}\alpha_{i}T_{i}/\sum_{i=0}^{n}\alpha
_{i}\underset{n\rightarrow\infty}{\longrightarrow}0,
\]
almost surely hence, that e.g. $\bar{S}_{n}$\thinspace$\underset{n\rightarrow
\infty}{\longrightarrow}S$ a.s. if and only if,
\[
\sum_{i=0}^{n}\alpha_{i}\bar{S}_{i}/\sum_{i=0}^{n}\alpha_{i}%
\,\underset{n\rightarrow\infty}{\longrightarrow}S.
\]
Simple proofs of these facts we leave to the reader as an exercise.
\end{remark}

\begin{proof}
In the proof we will use Lemma \ref{lemosr}. In order to prove assertion
\ref{pierwsza} let us notice that the sequence $\left\{  T_{n}\right\}
_{n\geq0}$ satisfies a recurrent relationship (\ref{rekurencyjneti}). Let us
calculate squares of both sides of this equation and let us calculate the
expectation of both sides. We get then
\[
ET_{n+1}^{2}=(1-\mu_{n})^{2}ET_{n}^{2}+c_{n+1}^{2}EX_{n+1}^{2}.
\]
Since that orthogonal series $\sum_{i\geq1}c_{i}X_{i}$ is convergent in
$L_{2}$, the series $\sum_{i\geq1}c_{i}^{2}EX_{i}^{2}$ is convergent. On the
base of Lemma \ref{podstawowy} we deduce that the series
\[
\sum_{n\geq0}(2\mu_{n}-\mu_{n}^{2})ET_{n}^{2}\allowbreak=\allowbreak
\sum_{n\geq0}(1-(1-\mu_{n})^{2})ET_{n}^{2}\allowbreak=\allowbreak\sum_{n\geq
0}\mu_{n}(1+1-\mu_{n})ET_{n}^{2}%
\]
is convergent. This series has positive summands and that it is a sum of two
series also having positive summands. Hence, we deduce that series
$\sum_{n\geq0}\mu_{n}ET_{n}^{2}$ converges. Further, on the basis of corollary
\ref{zbieznoscbezwzgledna} we deduce almost sure convergence of the series
$\sum_{n\geq0}\mu_{n}T_{n}^{2}.$

Assertion \ref{zb_sred} is a simple consequence of the assertion
\ref{pierwsza} and Lemma \ref{podstawowy}.

In order to prove assertion \ref{zb_podciagu} let us notice that the sequence
$\left\{  T_{n_{k}}\right\}  _{k\geq1}$ can be presented in the following
recursive form:
\begin{equation}
T_{n_{k+1}}=(1-\eta_{k})T_{n_{k}}+\eta_{k}W_{k}, \label{podciag2}%
\end{equation}
where we have defined:
\[
1-\eta_{k}=\sum_{i=0}^{n_{k}}\alpha_{i}/\sum_{i=0}^{n_{k+1}}\alpha_{i}%
\]
and
\[
W_{k}=\frac{\sum_{i=n_{k}+1}^{n_{k+1}}\alpha_{i}c_{i+1}X_{i+1}/\mu_{i}}%
{\sum_{i=n_{k}+1}^{n_{k+1}}\alpha_{i}};\;k=1,2,\ldots\;.
\]
Since, quadratic function is convex we have:
\[
T_{n_{k+1}}^{2}\leq(1-\eta_{k})T_{n_{k}}^{2}+\eta_{k}W_{k}^{2}.
\]
Let us apply now Corollary \ref{podstawowy1} of Lemma \ref{podstawowy} and
Corollary \ref{zbieznoscbezwzgledna} of Lebesgue' Theorem that from the
convergence of the number series $\sum_{k\geq1}\eta_{k}EW_{k}^{2}$ will follow
the convergence of the sequence $\left\{  T_{n_{k}}\right\}  $ to zero with
probability $1$. Checking of the convergence of this sequence we leave to the
reader as an exercise.

Assertion \ref{zb_podciagu2} follows, firstly from observation, that
\[
S_{n_{k}}-T_{n_{k}}=\sum_{i=0}^{n_{k}-1}\mu_{i}T_{i},
\]
from which it follows, in the light of the assertion \ref{zb_podciagu}, that
$\left\{  S_{n_{k}}\right\}  $ converges if and only if, the sequence
$\left\{  \sum_{i=0}^{n_{k}-1}\mu_{i}T_{i}\right\}  $ converges almost surely.
But on the other hand, we have :
\[
\underset{n_{k}+1\leq i<n_{k+1}}{\sup}\left\vert \sum_{j=n_{k}+1}^{i}\mu
_{j}T_{j}\right\vert ^{2}\leq\left(  \sum_{j=n_{k}+1}^{n_{k+1}}\mu_{j}\right)
\left(  \sum_{j=n_{k}+1}^{n_{k+1}}\mu_{j}T_{j}^{2}\right)
\underset{k\rightarrow\infty}{\longrightarrow}0
\]
almost surely because of the assertion \ref{pierwsza} and the definition of
the sequence $\left\{  n_{k}\right\}  $. Hence, we have showed, that the
sequence $\left\{  \sum_{i=0}^{n_{k}-1}\mu_{i}T_{i}\right\}  $ converges with
probability $1$ if and only if, the series $\sum_{i=0}^{\infty}\mu_{i}T_{i}$
converges$.$

Assertion \ref{zb_podciagu3} follows directly assertion \ref{zb_podciagu2} and
the identity: $S_{n+1}=T_{n+1}+\bar{S}_{n}.$

Finally assertion \ref{sr_sumcz}, \ref{zb_wariancji} and \ref{zb_kwadrat} are
repetitions of respective assertions of Lemma \ref{lemosr}.
\end{proof}

To end this dedicated to the orthogonal series section we will impose some
additional conditions to be satisfied by the orthogonal system and we will
show, that one can substantially weaken Menchoff's condition, in order to
guarantee almost sure convergence of the orthogonal series. Namely, we will
assume, that orthogonal system, i.e. the sequence $\{X_{n}\}_{n\geq1}$ of
uncorrelated, standardized random variables is weakly multiplicative
\index{System!Multiplicative}%
.

\begin{definition}
Sequence of the random variables $\{X_{n}\}_{n\geq1}$ is called $q$ weakly
multiplicative, if :
\[
\forall1\leq i_{1}<i_{2}<,\ldots,<i_{q}:EX_{i_{1}}X_{i_{2}}\cdots X_{i_{q}%
}=0.
\]

\end{definition}

It turns out that for weakly multiplicative systems one can weaken in a sense
the inequality (\ref{nier_fund}). More precisely, we have the following
theorem presented in the paper \cite{Gaposhkin72}:

\begin{theorem}
[Gaposzkin]Let $2<p\leq r$, where $r$ is even number. If the sequence
$\{X_{n}\}_{n\geq1}$ of the random variables is $r$ weakly multiplicative and
if $p\neq r$, then additionally it is orthogonal, and Moreover, if :
\[
\forall k\geq1:E\left\vert X_{k}\right\vert ^{p}\leq M,
\]
for some $M>0$, then
\begin{equation}
\forall n\geq1:E\left\vert \sum_{i=1}^{n}c_{i}X_{i}\right\vert ^{p}\leq
A_{p}\left(  \sum_{i=1}^{n}c_{i}^{2}\right)  ^{p/2}, \label{multi}%
\end{equation}
where $A_{p}$ is constant depending only on $p.$
\end{theorem}

We present this Theorem without proof. It is important since it turns out,
that the condition of convergence in $L_{2}$ of the orthogonal series, i.e.
the condition $\sum_{i\geq1}c_{i}^{2}<\infty$ implies almost sure convergence
of this series. During the last 30 years, there appeared a few papers where
orthogonal series with a system of functions weakly multiplicative were
examined. The papers: \cite{Revesz66}, \cite{Gaposkin67}, \cite{Longnecker78},
\cite{Borwein81}, \cite{Moricz83}, \cite{Moricz76} should be mentioned in the
first place. We will not discuss these papers in detail. We refer to them,
astute readers. Let us only notice, that from the conditions defined
Gaposhkin's theorem it follows that the smallest possible number $r$ is $4$.
Let us notice also, that be able to use a theory based on the Gaposhkin's
theorem one has to assume the existence of moments of order greater than $2$
of elements of the sequence $\{X_{n}\}_{n\geq1}$.

Below we will present the class of orthogonal systems, for which one does not
have to assume the existence of moments of order greater than $2$. It will be
the system slightly 'more than $4$ weakly multiplicative'. Unfortunately, one
does not get almost sure convergence of the respective orthogonal series,
under assumed $L_{2}$ convergence. However, sufficient condition, assuring
convergence is substantially weaker than the condition
(\ref{rademacher-menchoff}). The orthogonal system that we will analyze will
be called systems PSO. It consists of orthonormal random variables
$\{X_{n}\}_{n\geq1}$, satisfying two additional conditions. It will turn out,
that analysis of the convergence of series with PSO can be performed with the
help of methods that are already developed. It is simple and constitutes a
good exercise of application of methods presented in the previous section.

In order to briefly present these conditions let us introduce the following denotations:%

\[
H_{n,m}=span(X_{n+1},\ldots,X_{m})\text{ for }0\leq n<m\leq\infty.
\]
Set $H_{n,m}$ is a set of the random variables, that one can present as linear
combinations (when $m<\infty)$ or limits of such combinations, in $L_{2}$
(when $m=\infty)$ of the random variables $X_{i}$ for $i\in\lbrack n+1,m]$.
Conditions that were mentioned above are the following:
\begin{equation}
\exists C>0\;\forall n\in%
%TCIMACRO{\U{2115} }%
%BeginExpansion
\mathbb{N}
%EndExpansion
,U\in H_{0,n},Z\in H_{n,\infty}:EU^{2}Z^{2}\leq CEU^{2}EZ^{2}, \tag{S}%
\label{s}%
\end{equation}

\begin{align}
\forall n,k  &  \in%
%TCIMACRO{\U{2115} }%
%BeginExpansion
\mathbb{N}
%EndExpansion
,1\leq k<n,\;Z\in H_{0,k},T\in H_{k,n},U\in H_{0,n},\nonumber\\
W  &  \in H_{n,\infty},EW^{2}<\infty:EZTUW=0 \tag{O}\label{o}%
\end{align}

Orthogonal system, satisfying conditions (S) and (O) will be called
pseudo-square orthogonal
\index{System!Pseudo-square-orthonormal}%
(briefly PSO). \label{system PSO}

\begin{remark}
Let us notice that random variable $U\in H_{0,n}$ one can decompose on
$U_{1}\in H_{0,k}$ and $U_{2}\in H_{k,n}$ and the condition
\begin{align}
\forall n,k  &  \in%
%TCIMACRO{\U{2115} }%
%BeginExpansion
\mathbb{N}
%EndExpansion
,1\leq k<n,\;Z\in H_{0,k},T\in H_{k,n},U_{1}\in H_{0,k},U_{2}\in
H_{k,n}\nonumber\\
W  &  \in H_{n,\infty},\text{~}EW^{2}<\infty:EZTU_{1}W=0,EZTU_{2}W=0
\tag{O}\label{o3}%
\end{align}
implies condition (\ref{o}).
\end{remark}

\begin{remark}
Let us notice that the following conditions:
\begin{equation}
\forall i,j,k,l\in%
%TCIMACRO{\U{2115} }%
%BeginExpansion
\mathbb{N}
%EndExpansion
,\mathbb{\;}i\neq j,i\neq k,i\neq l,k\neq l:EX_{i}X_{j}X_{k}X_{l}=0
\tag{O1}\label{o1}%
\end{equation}%
\begin{equation}
\exists C>0\;\forall i\neq j:EX_{i}^{2}X_{j}^{2}\leq CEX_{i}^{2}EX_{j}^{2}.
\tag{S1}\label{s1}%
\end{equation}
imply conditions (\ref{o}) and (\ref{s}).

It is so, since firstly for $Z\in H_{0,k},\allowbreak T\in H_{k,n},\allowbreak
U_{1}\in H_{0,k}$ the product $ZTU_{1}$ is a linear combination of products of
the form $X_{i}X_{j}X_{l}$ where indices $i,j,l$ exclude equality $i=j=l$.
Hence, condition (\ref{o1}) implies that $EZTUX_{m}\allowbreak=\allowbreak0$
for $m>n$, which leads to (\ref{o}). It remained to show that the condition
(\ref{s}) is implied by the conditions (\ref{o1}) and (\ref{s1}). Let us
notice that the condition (\ref{o1}) causes, that the expression $EU^{2}Z^{2}$
will contain only monomials of the form of the form $\alpha_{i}^{2}\beta
_{j}^{2}EX_{i}^{2}X_{j}^{2}$, if we assumed that $U=\sum_{i=1}^{n}\alpha
_{i}X_{i}$, $Z=\sum_{j\geq n+1}\beta_{j}X_{j}$. Now we apply condition
(\ref{s1}) and present $EU^{2}Z^{2}$ as a product of $\sum_{i=1}^{n}\alpha
_{i}^{2}EX_{i}^{2}$ times $\sum_{j\geq n+1}\beta_{j}^{2}EX_{j}^{2}$.
\end{remark}

\begin{remark}
Systems consisting of standardized, independent random, or standardized
martingale differences (see section \ref{martyngaly}) are PSO.
\end{remark}

\begin{remark}
System of trigonometric functions defined on the space $(<0,1>,\mathcal{B}%
(<0,1>),|.|)$, where $|.|$ denotes Lebesgue's measure is not PSO, since
condition (\ref{o1}) is not satisfied by this sequence.
\end{remark}

Let $\{X_{n}\}_{n\geq1}$ be, as usual, orthonormal system and let $S_{n}%
=\sum_{i=1}^{n}c_{i}X_{i}$ be $n-$th partial sum of some orthogonal series
$\sum_{i\geq1}c_{i}X_{i}$. Moreover, let us denote :
\[
2^{(1)}(n)=2^{n};~n\geq0,~2^{(k)}(n)=2^{2^{(k-1)}(n)};~k\geq2,n\geq0.
\]
We have the following two lemmas:

\begin{lemma}
\label{ss}Let us assume that system orthonormal $\{X_{n}\}_{n\geq1}$ satisfies
condition (\ref{s}). Then:
\[
\forall k\in%
%TCIMACRO{\U{2115} }%
%BeginExpansion
\mathbb{N}
%EndExpansion
:S_{n}\rightarrow S\text{ if and only if, when }S_{n^{k}}\rightarrow S,
\]
for some random variable $S$ possessing variance.
\end{lemma}

\begin{lemma}
\label{oo}Let us assume that the sequence random variables $\{X_{n}\}_{n\geq
1}$ is PSO and Moreover, let us assume, that $\underset{i\geq1}{\sup
}\left\vert c_{i}\right\vert \ln i<\infty$. Then
\[
\forall k\in%
%TCIMACRO{\U{2115} }%
%BeginExpansion
\mathbb{N}
%EndExpansion
:S_{n}\rightarrow S\text{ if and only if, when }S_{2^{(k)}(n)}\rightarrow S,
\]
for some random variable $S$ possessing variance.
\end{lemma}

We will present common proof of those two lemmas.

\begin{proof}
In both cases we have to assume convergence in $L_{2}$ of the considered
series, i.e. to assume convergence of the series $\sum_{i\geq1}c_{i}^{2}$. If
the sequence of coefficients $\left\{  c_{i}\right\}  $ would satisfy the
condition of Rademacher-Menchoff Theorem, then our lemmas would be true.
Hence, let us suppose, that $\sum_{i\geq1}c_{i}^{2}\ln^{2}i=\infty$. Moreover,
let $c_{1}=1$. It will not influence the convergence of the series. Let us
denote
\begin{equation}
\mu_{0}=1,\mu_{i}=\left\vert c_{i+1}\right\vert /(1+\underset{i\geq1}{\sup
\,}\left\vert c_{i}\right\vert ),\;i\geq1, \label{mi_ce}%
\end{equation}
in the case of the proof of Lemma \ref{ss} and
\[
\mu_{0}=1,\mu_{i}=c_{i+1}^{2}\ln^{2}(i+1)/(1+\underset{i\geq1}{\sup\,}%
c_{i}^{2}\ln^{2}i),\;i\geq1,
\]
in the case proof of Lemma \ref{oo}. Given such sequences $\left\{  \mu
_{i}\right\}  _{i\geq0}$ let us define sequences $\left\{  T_{i}\right\}
_{i\geq0}$ using formula (\ref{defti}). By the way we take $0/0$ as $0$. Let
$\overline{\left\{  \alpha_{i}\right\}  }_{i\geq0}=\left\{  \mu_{i}\right\}
_{i\geq0}$. We will apply assertion \ref{zb_kwadrat} Lemma
\ref{prostypomocniczy}, in order to show, that $T_{n}\underset{n\rightarrow
\infty}{\rightarrow}0$ a.s. In both cases we have to prove, that series:
\begin{equation}
\sum_{i\geq0}c_{i+1}T_{i}X_{i+1} \label{ctx}%
\end{equation}
is convergent. Let us consider the situation from Lemma \ref{oo}. Let $n>k$ be
two natural numbers. $X_{n+1}\in H_{n+1,\infty}$, $T_{n}\in H_{0,n}$
$X_{k+1}\in H_{k+1,n}$, $T_{k}\in H_{0,k}$. Hence, $ET_{k}X_{k+1}T_{n}%
X_{n+1}=0$ by the condition (\ref{o}). Moreover, $ET_{n}^{2}X_{n+1}^{2}\leq
CET_{n}^{2}EX_{n+1}^{2}<\infty$ by condition (\ref{s}). Hence, random
variables $\left\{  T_{i}X_{i+1}\right\}  _{i\geq1}$ are orthogonal. Hence,
one can apply Rademacher-Menchoff Theorem to the series (\ref{ctx}). This
series will be converging a.s. if the number series
\[
\sum_{i\geq0}c_{i+1}^{2}\ln^{2}(i+1)ET_{i}^{2}X_{i+1}^{2}%
\]
will be convergent. We have, however $\mu_{i}\approx O(1)c_{i+1}^{2}\ln
^{2}(i+1)$ since $\underset{i\geq1}{\sup}\left\vert c_{i}\right\vert \ln
i<\infty$. Hence, on the basis of assumptions (\ref{s}) we deduce that the
considered series is convergent, since the series $\sum_{i\geq0}\mu_{i}%
ET_{i}^{2}$ is convergent. The series $\sum_{i\geq0}\mu_{i}ET_{i}^{2}$ is,
however convergent on the basis of assertion \ref{pierwsza} of Lemma
\ref{prostypomocniczy}. In order to prove, that $T_{n}\underset{n\rightarrow
\infty}{\rightarrow}0$ a.e. under the assumptions of the Lemma \ref{ss}, let
us consider the following sequence random variables:
\[
Z_{n}=\sum_{i=0}^{n-1}\alpha_{i}(c_{i+1}T_{i}X_{i+1}/\mu_{i})/\sum_{i=0}%
^{n-1}\alpha_{i},\,n\geq1.
\]
Let us consider now recursive forms of the sequences $\left\{  T_{i}%
^{2}\right\}  _{i\geq0}$ and $\left\{  T_{i}^{2}-2Z_{i}\right\}  _{i\geq0}$.
We have:
\begin{align*}
T_{n+1}^{2}  &  =(1-\mu_{n})^{2}T_{n}^{2}+2c_{n+1}(1-\mu_{n})T_{n}%
X_{n+1}+c_{n+1}^{2}X_{n+1}^{2},\\
Z_{n+1}  &  =(1-\mu_{n})Z_{n}+c_{n+1}T_{n}X_{n+1},\\
T_{n+1}^{2}-2Z_{n+1}  &  =(1-\mu_{n})(T_{n}^{2}-2Z_{n})+c_{n+1}^{2}X_{n+1}%
^{2}-\\
&  \mu_{n}(1-\mu_{n})T_{n}^{2}-2c_{n+1}\mu_{n}T_{n}X_{n+1}.
\end{align*}
Series
\[
\sum_{n\geq0}c_{n+1}^{2}X_{n+1}^{2},~\sum_{n\geq0}\mu_{n}(1-\mu_{n})T_{n}%
^{2},~\sum_{n\geq0}c_{n+1}\mu_{n}T_{n}X_{n+1}%
\]
are convergent on the basis of assumptions, definition of the sequence
$\left\{  \mu_{n}\right\}  _{n\geq0}$ and assertion \ref{pierwsza} of Lemma
\ref{prostypomocniczy}. Hence, sequence $\left\{  T_{n}^{2}-2Z_{n}\right\}
_{n\geq0}$ converges to zero a.s.. Moreover, we have for the sequence
$\left\{  Z_{n}\right\}  :$%
\[
Z_{n+1}^{2}\leq(1-\mu_{n})Z_{n}^{2}+\mu_{n}(c_{n+1}^{2}T_{n}^{2}X_{n+1}%
^{2}/\mu_{n}^{2}).
\]
Remembering, that the sequence $\left\{  \mu_{n}\right\}  $ is defined in this
case by the formula (\ref{mi_ce}), we deduce, that the series
\[
\sum_{n\geq0}c_{n+1}^{2}T_{n}^{2}X_{n+1}^{2}/\mu_{n}\allowbreak=\allowbreak
O(1)\sum_{n\geq0}\mu_{n}T_{n}^{2}X_{n+1}^{2},
\]
is almost surely convergent on the basis of assertion \ref{pierwsza} of Lemma
\ref{prostypomocniczy} and assumptions (\ref{s}). Hence, the sequence
$\left\{  Z_{n}\right\}  _{n\geq1}$, and consequently sequence $\left\{
T_{n}\right\}  _{n\geq1}$, converge a.s. to zero. Having proved convergence
$T_{n}\underset{n\rightarrow\infty}{\rightarrow}0$ a.s., we use now assertion
\ref{zb_podciagu3} of Lemma \ref{prostypomocniczy}. Hence, let us examine
subsequences of the indices $\left\{  n_{k}\right\}  _{k\geq1}$ in both
situations. In the case of Lemma \ref{ss} we have:
\[
\sum_{j=n_{k}+1}^{n_{k+1}+1}\left\vert c_{j}\right\vert =O(1),
\]
or in particular
\[
O(1)k=\sum_{i=0}^{n_{k}}\left\vert c_{i}\right\vert \leq\sqrt{n_{k}}\sqrt
{\sum_{j=1}^{n_{k}}c_{j}^{2}}.
\]
Hence $O(1)k^{2}\leq n_{k}$, $k\geq1$, since $\sum_{j\geq1}c_{j}^{2}<\infty$.
In the case of Lemma \ref{oo} we have:
\[
\sum_{i=n_{k}+1}^{n_{k+1}+1}c_{i}^{2}\ln^{2}i=O(1).
\]
Thus, we have:
\[
\ln^{2}n_{k}\sum_{j=n_{k}+1}^{n_{k+1}+1}c_{j}^{2}\leq\sum_{j=n_{k}+1}%
^{n_{k+1}+1}c_{j}^{2}\ln^{2}j=O(1)\leq\ln^{2}n_{k+1}\sum_{j=n_{k}+1}%
^{n_{k+1}+1}c_{j}^{2}.
\]
On the base of this inequality we deduce that $\sum_{k\geq1}1/\ln^{2}%
n_{k}<\infty$, or $1/\ln^{2}n_{k}=o(1/k)$, since the sequence $\left\{
n_{k}\right\}  $ is increasing. Thus, $n_{k}\geq o(1)[2^{\sqrt{k}}].$

Now one can define new orthogonal series $\sum_{j\geq0}c_{j}^{\prime}%
X_{j}^{\prime}$, by putting:
\[
(c_{k}^{\prime})^{2}=\sum_{j=n_{k}+1}^{n_{k+1}+1}c_{j}^{2},\;\;\;\;\;X_{j}%
^{\prime}=\frac{1}{c_{k}^{\prime}}\sum_{j=n_{k}+1}^{n_{k+1}+1}c_{j}X_{j}.
\]
Let us notice that a new orthonormal sequence $\left\{  X_{j}^{\prime
}\right\}  $ satisfies condition (\ref{s}), when the sequence $\{X_{n}%
\}_{n\geq1}$ satisfies this condition. Similarly the new orthonormal sequence
$\left\{  X_{j}^{\prime}\right\}  $ is PSO, when the sequence $\{X_{n}%
\}_{n\geq1}$ is PSO. It remained to check, if $\underset{n\geq1}{\sup
}\left\vert c_{n}^{^{\prime}}\right\vert \ln n<\infty$. We have, however:
\[
\left(  c_{k}^{^{\prime}}\right)  ^{2}\ln^{2}k\leq\left(  c_{k}^{^{\prime}%
}\right)  ^{2}k\leq\left(  c_{k}^{^{\prime}}\right)  ^{2}\ln^{2}n_{k}\leq
\sum_{j=n_{k}+1}^{n_{k+1}}c_{j}^{2}\ln^{2}j=O(1).
\]
Moreover, let us notice, that $k-$th partial sum of the series $\sum_{i\geq
0}c_{i}^{\prime}X_{i}^{\prime}$ is $n_{k}-$th partial sum of the series
$\sum_{i\geq1}c_{i}X_{i}$. Hence, we can apply previous considerations and
deduce, that $S_{n}\underset{n\rightarrow\infty}{\rightarrow}S$ a.s. if and
only if, $S_{n_{n_{k}}}\underset{k\rightarrow\infty}{\rightarrow}S$ a.s. In
particular, we have in the case of Lemma \ref{ss}: $S_{n}%
\underset{n\rightarrow\infty}{\rightarrow}S$ a.s. if and only if,
$S_{(n^{2})^{2}}\underset{n\rightarrow\infty}{\rightarrow}S$ , while in the
case of Lemma \ref{oo}: $S_{n}\underset{n\rightarrow\infty}{\rightarrow}S$
a.s. if and only if $S_{2^{n}}\underset{n\rightarrow\infty}{\rightarrow}S$
a.s. since $\sqrt{n^{2}}=n$. Further, we can again introduce new orthogonal
series and repeat the same argument,. and so on any, a finite number of times.
\end{proof}

These lemmas are the base of the theorem, whose proof will not be presented
because of lack of space and the fact, that it is not difficult however not
too short, and not probabilistic. Its probabilistic essence is contained in
lemmas \ref{oo} and \ref{ss}. The main idea of the proof can be reduced to the
decomposition of the series on the so-called lacunary series, that is partial
ones. For every of such lacunary series, we prove convergence, making use of
Lemma \ref{oo} and respective version of the Rademacher-Menchoff's Theorem.
Proof can be found in the paper \cite{szab4}. In order to formulate briefly
this theorem, let us introduce the following notation: $\ln^{(1)}n=\log_{+}n;$
$\ln^{(j)}n=\log_{+}(\ln^{(j-1)}n)$, $n,j\in%
%TCIMACRO{\U{2115} }%
%BeginExpansion
\mathbb{N}
%EndExpansion
.$

\begin{theorem}
\label{o PSO}Let $\{X_{n}\}_{n\geq1}$ be PSO system. Then, if for some $k\in%
%TCIMACRO{\U{2115} }%
%BeginExpansion
\mathbb{N}
%EndExpansion
:$ $\sum_{i\geq1}c_{i}^{2}(\ln^{(k)}i)^{2}<\infty$, then orthogonal series
$\sum_{i\geq1}c_{i}X_{i}$ is convergent almost surely.
\end{theorem}

\begin{remark}
The above-mentioned theorem was formatted and proved by P. R\'{e}v\'{e}sz in
1966 (see \cite{Revesz66}) under somewhat different but equivalent condition
imposed on a multiplicative system of orthogonal functions. We quote it to
show that in fact, it follows from the two Lemmas mentioned above just
providing another proof of this Theorem based on methods developed in this book.
\end{remark}

\begin{remark}
Gaposhkin in \cite{Gaposkin67} showed that assertion of theorem \ref{o PSO}
can be strengthened. Namely, one can drop condition for some $k\in%
%TCIMACRO{\U{2115} }%
%BeginExpansion
\mathbb{N}
%EndExpansion
:$ $\sum_{i\geq1}c_{i}^{2}(\ln^{(k)}i)^{2}<\infty.$ More precisely, Gaposhkin
showed that every PSO system satisfies condition (\ref{multi}) with $p=4.$
\end{remark}

\chapter{Laws of Large Numbers\label{simpwl}}

In this chapter, we will give a few criteria for LLN weak and strong to be
satisfied. Generally speaking, we will consider generalized laws of large
numbers in the sense of definition \ref{uog_pwl}. Since, the classical case,
i.e., constant weights $\left\{  \alpha_{i}\right\}  _{i\geq0}$ is the most
important, sometimes we will present only those versions of laws of large
numbers, leaving to the reader formulation and proving more general version.
In any case, if we will talk about LLN without mentioning weights we will mean
constant weights equal to $1$.

Methods and results on which we will base proofs respective theorems were
presented in the previous chapter

It is worth to mention, that in 1967 appeared a classical book entitled
\textquotedblright The Laws of Large Numbers\textquotedblright\ by P\'{a}l
R\'{e}v\'{e}sz \cite{Revesz67}. We will not, of course, quote all theorems
from this book. For completeness, we will quote only a few the most important ones.

\section{Necessary condition}

We will start with the following simple sufficient condition.

\begin{proposition}
\label{war_kon}If the sequence $\{X_{n}\}_{n\geq1}$ satisfies SLLN (resp.
MLLN), then the sequence $\left\{  X_{n}/n\right\}  _{n\geq1}$ converges to
zero in probability (resp. with probability $1$).
\end{proposition}

\begin{proof}
Let us denote $S_{n}=\sum_{i=1}^{n}X_{i}$. To fix notations let us consider
the case of SLLN. Since $\left\{  S_{n}/n\right\}  _{n\geq1}$ converges in
probability do finite limit, then also $\left\{  S_{n+1}/n\right\}  _{n\geq1}$
converges to the same limit. Hence, and sequence $\left\{  (S_{n}%
-S_{n+1})/n\right\}  _{n\geq1}$ converges to zero. But of course we have
$\left(  S_{n+1}-S_{n}\right)  /n=X_{n}/n$. Similarly we argue in the case of MLLN.
\end{proof}

\section{Weak laws of large numbers}%

\index{Law!Large numbers!weak}%
In this section, we will prove a few criteria concerning weak laws of large
numbers under different assumptions concerning sequence random variables
$\left\{  X_{n}\right\}  _{n\geq1}$.

\subsection{For independent random variables}

First, we will assume that random variables in question may not have variances.

\subsubsection{Have identical distributions\label{sspwl_iid}}

Let us start with the results presented in the paper \cite{Jamison65}.

\begin{theorem}
\label{spwl_iid}If only $\frac{\alpha_{n}}{\sum_{i=0}^{n}\alpha_{i}%
}\longrightarrow0$ and $\sum_{i\geq1}^{n}\alpha_{i}\rightarrow\infty$, when
$n\,\longrightarrow\infty$ (i.e. when sequence $\left\{  \frac{\alpha_{n}%
}{\sum_{i=0}^{n}\alpha_{i}}\right\}  _{n\geq0}$ is normal), then the sequence
$\left\{  X_{n}\right\}  _{n\geq1}$ of independent random variables having
identical distributions satisfies WLLN if and only if,
\begin{equation}
\underset{T\rightarrow\infty}{\lim}TP\left\{  \left\vert X_{1}\right\vert \geq
T\right\}  =0\text{ and }\underset{T\rightarrow\infty}{\lim}\int_{\left\vert
x\right\vert \leq T}xdF\text{ exists,} \label{w_spwl_iid}%
\end{equation}
\footnote{It follows from Pitman's theorem that these conditions are
equivalent to the existence of the first derivative if the characteristic
function of the random variable $X_{1}.$} where $F$ is cumulative distribution
function (cdf) of the random variable $X_{1}.$
\end{theorem}

\begin{proof}
Necessity condition is proved in the book \cite{Loeve55}, hence we will not
prove it. Now let us assume that the condition (\ref{w_spwl_iid}) is
satisfied. Let us consider $X_{k,n}=X_{k}I(\left\vert X_{k}\right\vert
\leq\sum_{i=0}^{n-1}\alpha_{i}/\alpha_{k-1})$, that is random variable $X_{k}$
cut at the level $\sum_{i=0}^{n-1}\alpha_{i}/\alpha_{k-1}$. Moreover, let us
denote: $S_{n}\allowbreak=\allowbreak\sum_{i=0}^{n}\alpha_{i}X_{i+1}$,
$S_{nn}=\sum_{i=0}^{n-1}\alpha_{i}X_{i+1,n}$. Let us recall that the condition
$\alpha_{n}\allowbreak/\sum_{i=0}^{n}\alpha_{i}\allowbreak\rightarrow0$,
\ $n\rightarrow\infty$ implies that $\underset{0\leq i\leq n}{\max}\alpha
_{i}/\sum_{i=0}^{n}\alpha_{i}\rightarrow0,\;n\rightarrow\infty$ (assertion
\emph{iv) }of\emph{\ }Proposition \ref{o_ciagach_normalnych}). Thus, in
particular, taking into account conditions (\ref{w_spwl_iid}) we get:
\begin{equation}
\underset{n\rightarrow\infty}{\lim}\,\underset{0\leq i\leq n}{\max}\left(
\frac{\sum_{j=0}^{n-1}\alpha_{j}}{\alpha_{i}}P\left(  |X|>\frac{\sum
_{j=0}^{n-1}\alpha_{j}}{\alpha_{i}}\right)  \right)  =0. \label{do0}%
\end{equation}
Since $\frac{\alpha_{n}}{\sum_{i=0}^{n}\alpha_{i}}\longrightarrow0$,
$n\,\longrightarrow\infty$ and including identity of distributions of the
sequence $\left\{  X_{i}\right\}  _{i\geq1}$ we get for sufficiently large
$n$:
\begin{align*}
P(S_{nn}  &  \neq S_{n})\leq\sum_{i=1}^{n}P(X_{i+1,n}\neq X_{i})\\
&  =\sum_{i=1}^{n}P(\left\vert X_{1}\right\vert \geq\sum_{j=0}^{n-1}\alpha
_{j}/\alpha_{i-1})\leq\varepsilon\sum_{k=0}^{n-1}\frac{\alpha_{k}}{\sum
_{i=0}^{n-1}\alpha_{i}}=\varepsilon,
\end{align*}
where $\varepsilon$ is such number, that
\[
\underset{0\leq i\leq n}{\max}\left(  \frac{\sum_{j=0}^{n-1}\alpha_{j}}%
{\alpha_{i}}P(|X_{1}|>\frac{\sum_{j=0}^{n-1}\alpha_{j}}{\alpha_{i}})\right)
<\varepsilon.
\]
This particular choice of $\varepsilon$ is possible since we have (\ref{do0}).
Hence, it is enough to consider $S_{nn}$ instead $S_{n}$. We have:
\[
E\frac{S_{nn}}{\sum_{i=0}^{n-1}\alpha_{i}}=\frac{1}{\sum_{i=0}^{n-1}\alpha
_{i}}\sum_{k=0}^{n-1}\alpha_{k}\int_{\left\vert x\right\vert <\sum_{i=0}%
^{n-1}\alpha_{i}/\alpha_{k}}xdF\longrightarrow\kappa,
\]
where $\kappa$ denotes the second of the limits in (\ref{w_spwl_iid}).
Moreover, by integrating by parts we get:
\[
\frac{1}{T}\int_{\left\vert x\right\vert <T}x^{2}dF=\frac{1}{T}\left[
-T^{2}P(\left\vert X\right\vert \geq T)+2\int_{0\leq x<T}xP(\left\vert
X\right\vert \geq x)dx\right]  \longrightarrow0,
\]
when $T\longrightarrow\infty$. In particular, we have:
\begin{equation}
\underset{n\rightarrow\infty}{\lim}\,\underset{0\leq i\leq n}{\max}\left(
\frac{\alpha_{i}}{\sum_{j=0}^{n-1}\alpha_{j}}\int_{\left\vert x\right\vert
\leq\sum_{j=0}^{n-1}\alpha_{j}/\alpha_{i}}x^{2}dF\right)  =0. \label{do01}%
\end{equation}
Thus:
\begin{align*}
\operatorname{var}\frac{S_{nn}}{\sum_{i=0}^{n-1}\alpha_{i}}  &  =\frac
{1}{\left(  \sum_{i=0}^{n-1}\alpha_{i}\right)  ^{2}}\sum_{i=0}^{n-1}\alpha
_{i}^{2}\operatorname{var}(X_{i+1,n})\\
&  \leq\frac{1}{\left(  \sum_{i=0}^{n-1}\alpha_{i}\right)  ^{2}}\sum
_{i=0}^{n-1}\alpha_{i}^{2}\int_{\left\vert x\right\vert \leq\sum_{j=0}%
^{n-1}\alpha_{j}/\alpha_{i}}x^{2}dF\\
&  \leq\frac{1}{\left(  \sum_{i=0}^{n-1}\alpha_{i}\right)  ^{2}}\sum
_{i=0}^{n-1}\alpha_{i}^{2}\varepsilon\frac{\sum_{j=0}^{n-1}\alpha_{j}}%
{\alpha_{i}}=\varepsilon,
\end{align*}
where $\varepsilon$ is such number, that%
\[
\underset{0\leq i\leq n}{\max}\left(  \frac{\alpha_{i}}{\sum_{j=0}^{n-1}%
\alpha_{j}}\int_{\left\vert x\right\vert \leq\sum_{j=0}^{n-1}\alpha_{j}%
/\alpha_{i}}x^{2}dF\right)  \leq\varepsilon.
\]
Now it remains to apply Chebyshev inequality (see Appendix \ref{czebyszew}),
in order to get assertion.
\end{proof}

\begin{example}
We will illustrate this theorem by the following example. One took $N=1000000$
observations of the random variables $\left\{  \xi_{i}\right\}  _{i\geq1}$ of
the form $\xi_{i}\allowbreak=\allowbreak S_{i}\zeta_{i},$ where the random
variables $S_{i}$ and $\zeta_{i}$ are independent, having identical
distributions and $P(S_{1}=-1)\allowbreak=\allowbreak P(S_{1}=1)\allowbreak
=\allowbreak1/2$, $\zeta_{1}$ has cdf $F_{\gamma}(x)$, where
\begin{equation}
F_{\gamma}(x)=\left\{
\begin{array}
[c]{ccc}%
0, & gdy & x\leq1\\
1-\frac{1}{x^{\gamma}}, & gdy & x>1
\end{array}
\right.  . \label{pareto}%
\end{equation}
We have assumed $\gamma=1.05$. Hence, $E\left\vert \xi_{1}\right\vert =\infty
$. Moreover, we took $\alpha_{i}=(i+1)^{2};i\geq0$, that is one examined
convergence of the sequence $\left\{  X_{n}=\sum_{i=1}^{n}i^{2}\xi_{i}%
/\sum_{i=1}^{n}i^{2}\right\}  _{n\geq1}$. The conditions given in theorem
\ref{spwl_iid} are satisfied, since distribution $\xi_{1}$ is symmetric hence
the second of the limits in condition (\ref{spwl_iid}) exists,\ Further,
$TP(\left\vert \xi_{1}\right\vert >T)\allowbreak=\allowbreak TP(\zeta
_{1}>T)\allowbreak=\allowbreak T\frac{1}{T^{1.05}}\allowbreak=\allowbreak
\frac{1}{T^{.05}}\allowbreak\rightarrow0$, when $T\rightarrow\infty$. One
obtained the following plot where the values of variable $X_{n}$ were marked
every $100$ observations.
%TCIMACRO{\FRAME{ftbpFU}{3.0338in}{1.8187in}{0pt}{\Qcb{\QTR{tiny}{Law of large
%numbers for iid random variables not possesing expactations}}}{}%
%{spwl1.eps}{\special{ language "Scientific Word";  type "GRAPHIC";
%width 3.0338in;  height 1.8187in;  depth 0pt;  original-width 6.2379in;
%original-height 3.7291in;  cropleft "0";  croptop "1";  cropright "1";
%cropbottom "0";  filename 'spwl1.eps';file-properties "XNPEU";}}}%
%BeginExpansion
\begin{figure}[ptb]%
\centering
\includegraphics[
height=1.8187in,
width=3.0338in
]%
{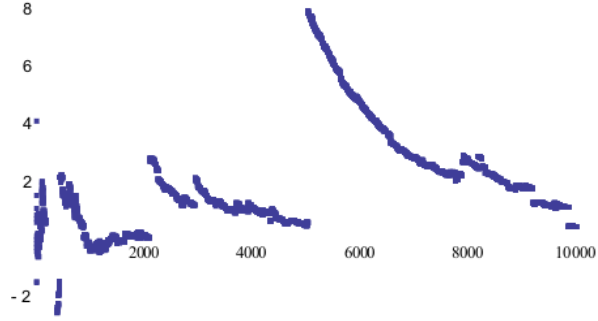}%
\caption{{\protect\tiny Law of large numbers for iid random variables not
possesing expactations}}%
\end{figure}
%EndExpansion

\end{example}

\begin{example}
As far as speed of convergence is concerned true is the following Katz'
Theorem being generalization earlier Erd\"{o}s Theorem.
\end{example}

\begin{theorem}
[Katz]Let $\left\{  X_{i}\right\}  _{i\geq1}$ be a sequence of independent
random variables having identical distributions such that $EX_{1}=0$. Then
$E\left\vert X_{1}\right\vert ^{t}<\infty$ for some $t\geq1$ if and only if,:
\[
\sum_{n\geq1}n^{t-2}P\left(  \left\vert \frac{\sum_{i=1}^{n}X_{i}}%
{n}\right\vert \geq\varepsilon\right)  <\infty
\]
for any $\varepsilon>0.$
\end{theorem}

\begin{proof}
see \cite{Revesz67}.
\end{proof}

\subsubsection{Having different distributions}

For simplicity, we will be concerned with only classical case, i.e. when
weights are equal to $1$. In this situation, we have classical
Gniedenko\&Kolmogorov Theorem \cite{Gnedenko54}

\begin{theorem}
[Gniedenko-Kolmogorov]Sequence $\left\{  X_{i}\right\}  _{i\geq1}$ satisfies
SLLN if and only if, :
\begin{align*}
&  \sum_{k\geq1}^{n}P\left(  \left\vert X_{k}-m(X_{k})\right\vert \geq
n\right)  \underset{n\rightarrow\infty}{\rightarrow}0,\\
&  \frac{1}{n^{2}}\sum_{k\geq1}^{n}\int_{\left\vert x\right\vert \leq
n}\left(  x-m(X_{k})\right)  ^{2}dF_{k}(x)\underset{n\rightarrow
\infty}{\rightarrow}0,
\end{align*}
where $F_{k}(x)$ is cdf of a random variable $X_{k}$ and $m(X_{k})$ is its median.
\end{theorem}

\begin{proof}
Sketch of the proof. Proof of necessity of the above-mentioned conditions is
somewhat arduous and strongly uses the properties of characteristic functions
hardly mentioned in this book. We will not present here this proof of
necessity, since, one would have to present in detail needed properties of
characteristic functions. It is presented, e.g. in \cite{Revesz67}.

Proof of sufficiency of the conditions defined in this theorem is somewhat
typical, it utilizes the truncation method, used already e.g. in the proof of
theorems \ref{spwl_iid}.

Let us denote:%

\begin{align*}
X_{k}^{\prime}  &  =X_{k}-m(X_{k}),\\
F_{k}^{\prime}(x)  &  =P(X_{k}^{\prime}<x)=F_{k}(x+m(X_{k})),\\
X_{k,n}^{\prime\prime}  &  =\left\{
\begin{array}
[c]{ccc}%
X_{k}^{\prime} & gdy & \left\vert X_{k}^{\prime}\right\vert \leq n\\
0 & gdy & \left\vert X_{k}^{\prime}\right\vert >n
\end{array}
\right.  ,\\
A_{n}  &  =\frac{1}{n}\sum_{i=1}^{n}\left(  m(X_{i})+EX_{i,n}^{\prime\prime
}\right)  ,\\
\zeta_{n}^{\prime}  &  =\frac{1}{n}\sum_{i=1}^{n}X_{i}^{\prime},\;\zeta
_{n}^{\prime\prime}=\frac{1}{n}\sum_{i=1}^{n}X_{i}^{\prime\prime},\\
B_{n}  &  =\left\{  \omega:\zeta_{n}^{\prime}(\omega)=\zeta_{n}^{\prime\prime
}(\omega)\right\}  .
\end{align*}
We have of course
\[
P(\overline{B}_{n})\leq\sum_{i=1}^{n}P(\left\vert X_{i}^{\prime}\right\vert
>n)=\sum_{i=1}^{n}\int_{\left\vert x\right\vert >n}dF_{i}^{\prime}(x).
\]
For any $\varepsilon$ we have
\[
P(\left\vert \zeta_{n}-A_{n}\right\vert \geq\varepsilon)=P(B_{n})P(\left\vert
\zeta_{n}-A_{n}\right\vert \geq\varepsilon|B_{n})+P(\overline{B}%
_{n})P(\left\vert \zeta_{n}-A_{n}\right\vert \geq\varepsilon|\overline{B}%
_{n}).
\]
Further we have
\begin{align*}
P(B_{n})P(\left\vert \zeta_{n}-A_{n}\right\vert  &  \geq\varepsilon|B_{n})\leq
P\left(  \left\vert \zeta_{n}^{\prime\prime}-E\zeta_{n}^{\prime\prime
}\right\vert \geq\varepsilon\right)  \leq\frac{\operatorname*{var}(\zeta
_{n}^{\prime\prime})}{\varepsilon^{2}}\\
&  \leq\frac{1}{\varepsilon^{2}}\sum_{i=1}^{n}E\left(  X_{i}^{\prime\prime
}\right)  ^{2}.
\end{align*}
Hence on the basis of assumptions, for any $\varepsilon>0$ we have:
\[
P(\left\vert \zeta_{n}-A_{n}\right\vert \geq\varepsilon)\leq\frac
{1}{\varepsilon^{2}n^{2}}\sum_{i=1}^{n}\int_{\left\vert x\right\vert \leq
n}x^{2}dF_{k}^{\prime}(x)+\sum_{i=1}^{n}\int_{\left\vert x\right\vert
>n}dF_{i}^{\prime}(x)\rightarrow0,
\]
when $n\rightarrow\infty$ .
\end{proof}

As far as speed convergence is concerned in this the case, the following Katz'
Theorem is true

\begin{theorem}
Let $\left\{  X_{i}\right\}  _{i\geq1}$ be a sequence of independent random
variables with zero expectations, such that:
\[
\exists t\in\{3,4,\ldots\}\,\exists C>0\,\forall k\in\{1,2,\ldots
\}:E\left\vert X_{k}\right\vert ^{t}\leq C,
\]
then:
\[
P\left(  \left\vert \frac{\sum_{i=1}^{n}X_{i}}{n}\right\vert \geq
\varepsilon\right)  =O\left(  \frac{1}{n^{t-1}}\right)  ,
\]
for any $\varepsilon>0.$
\end{theorem}

\begin{proof}
Proof of this theorem is presented in \cite{Revesz67}. It is somewhat tedious
and that is why we will not present it here.
\end{proof}

\subsection{For random variables possessing variances}

\begin{proposition}
If
\[
\operatorname*{var}\left(  \frac{\sum_{n=1}^{N}\alpha_{n}X_{n}}{\sum_{n=1}%
^{N}\alpha_{n}}\right)  \rightarrow0;N\rightarrow\infty,
\]
for some sequence $\left\{  \alpha_{n}\right\}  _{n\geq1}$ satisfying
conditions Theorem \ref{spwl_iid}, then $\{X_{n}\}_{n\geq1}$ satisfies
\emph{WGLLN} $.$
\end{proposition}

\begin{proof}
By Chebyshev's inequality we have:
\begin{align*}
P\left(  \left\vert \frac{\sum_{n=1}^{N}\alpha_{n}(X_{n}-EX_{n})}{\sum
_{n=1}^{N}\alpha_{n}}\right\vert >\varepsilon\right)  \allowbreak &  \leq
E\left(  \frac{\sum_{n=1}^{N}\alpha_{n}(X_{n}-EX_{n})}{\sum_{n=1}^{N}%
\alpha_{n}}\right)  ^{2}/\varepsilon^{2}\allowbreak\\
&  =\operatorname*{var}\left(  \frac{\sum_{n=1}^{N}\alpha_{n}X_{n}}{\sum
_{n=1}^{N}\alpha_{n}}\right)  /\varepsilon^{2}.
\end{align*}
Hence, if \newline$\operatorname*{var}\left(  \frac{\sum_{n=1}^{N}\alpha
_{n}X_{n}}{\sum_{n=1}^{N}\alpha_{n}}\right)  \underset{N\rightarrow
\infty}{\longrightarrow}0$, then
\[
\forall\varepsilon>0:\allowbreak P\left(  \left\vert \frac{\sum_{n=1}%
^{N}\alpha_{n}(X_{n}-EX_{n})}{\sum_{n=1}^{N}\alpha_{n}}\right\vert
>\varepsilon\right)  \allowbreak\underset{N\rightarrow\infty}{\longrightarrow
}0.
\]

\end{proof}

\begin{example}
In particular, if e.g. $\left\{  X_{n}\right\}  $ are \emph{uncorrelated} and
have identical variances, then the sequence $\{X_{n}\}_{n\geq1}$ satisfies the
weak law of large numbers (SLLN). Since, we have, then $\operatorname{var}%
\left(  \sum_{i=1}^{N}X_{i}\right)  \allowbreak=\allowbreak\sum_{i=1}%
^{N}\operatorname{var}(X_{i})$, hence $\operatorname{var}\left(  \frac
{\sum_{i=1}^{N}X_{i}}{N}\right)  \allowbreak=\frac{\operatorname{var}(X_{1}%
)}{N}\allowbreak\underset{N\rightarrow\infty}{\longrightarrow}0.$
\end{example}

\begin{example}
\label{slabe_pwl}When $\left\{  X_{n}\right\}  $ are \emph{uncorrelated} and
$\operatorname*{var}(X_{n})\approx n^{\alpha};$ $\alpha<1$, then the sequence
$\{X_{n}\}_{n\geq1}$ satisfies the weak law of large numbers (WLLN). Since
arguing in the similar fashion, we have:
\[
\operatorname{var}\left(  \frac{\sum_{n=1}^{N}X_{n}}{N}\right)  \allowbreak
\approx\frac{\sum_{n=1}^{N}n^{\,\alpha}}{N^{2}}\allowbreak\approx
\frac{N^{\alpha+1}}{(\alpha+1)N^{2}}\underset{N\rightarrow\infty
}{\longrightarrow}0,
\]
if only $\alpha<1$. This example we will illustrate by the following
simulation. One took $N=4000000$ observations $\left\{  \tau_{i}\right\}
_{i=1}^{N}$ of the random variables of the form $\tau_{i}=i^{\beta}(\xi
_{i}-E\xi_{i})$, where the random variables $\left\{  \xi_{i}\right\}  $ are
independent and have the same distributions with cdf given by (\ref{pareto})
for $\gamma=\frac{7}{3}$. One took $\beta=\frac{6}{14}$. Let us notice that
then $\operatorname*{var}(\tau_{i})=\frac{\gamma}{(\gamma-2)(\gamma-1)^{2}%
}i^{2\beta}$, hence $\alpha=2\beta=\frac{6}{7}$. One obtained the following
course of averages $Y_{n}=\frac{1}{n}\sum_{i=1}^{n}\tau_{i}$, $n\geq1$. Again,
as before sampling at every $K\allowbreak=\allowbreak200$, i.e. number of
iteration $=$ number on the plot times $200$.
%TCIMACRO{\FRAME{ftbpFU}{3.2474in}{1.9043in}{0pt}{\Qcb{\QTR{tiny}{Weak law of
%large numbers for independent random variables having increasing
%variances}}}{}{spwl2.eps}{\special{ language "Scientific Word";
%type "GRAPHIC";  maintain-aspect-ratio TRUE;  display "USEDEF";
%valid_file "F";  width 3.2474in;  height 1.9043in;  depth 0pt;
%original-width 6.2379in;  original-height 3.6469in;  cropleft "0";
%filename 'spwl2.eps';file-properties "XNPEU";}}}%
%BeginExpansion
\begin{figure}[ptb]%
\centering
\includegraphics[
height=1.9043in,
width=3.2474in
]%
{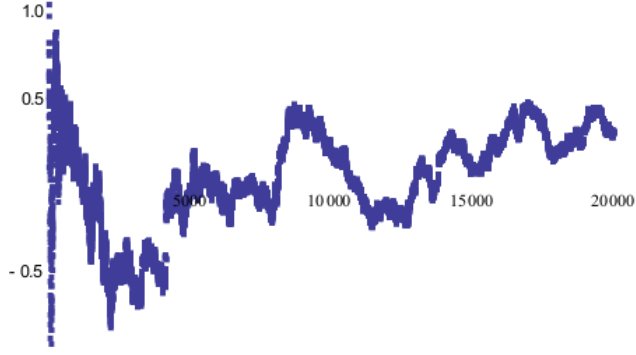}%
\caption{{\protect\tiny Weak law of large numbers for independent random
variables having increasing variances}}%
\end{figure}
%EndExpansion
As one can see the convergence is rather very slow.
\end{example}

\section{Strong laws of large numbers%
\index{Law!Large numbers!strong}%
}

\label{pwl}Let us notice that Lemmas \ref{podstawowy} and \ref{lemosr} supply
tools to examine the conditions under which generalized, strong laws of large
numbers is satisfied. As we mentioned before, the sequence of Riesz's means
$\left\{  \bar{X}_{n}=\frac{\sum_{i=0}^{n-1}\alpha_{i}X_{i+1}}{\sum
_{i=0}^{n-1}\alpha_{i}}\right\}  _{n\geq1}$ of the sequence $\left\{
X_{i}\right\}  _{i\geq1}$ with respect to the sequence $\left\{  \alpha
_{i}\right\}  _{i\geq0}$ satisfies the following recursive equation:
\[
\bar{X}_{n+1}=(1-\mu_{n})\bar{X}_{n}+\mu_{n}X_{n+1},
\]
where $\left\{  \mu_{i}\right\}  _{i\geq0}$ $=\allowbreak$ $\overline{\left\{
\alpha_{i}\right\}  _{i\geq0}}.$

\begin{remark}
Let us notice also, that for sequences of the random variables $\left\{
X_{i}\right\}  _{i\geq1}$, having second moments and satisfying conditions:
\begin{align*}
\sum_{i\geq0}\mu_{i}^{2}X_{i+1}^{2}  &  <\infty\,\,\;a.s.~,\\
\sum_{i\geq1}\mu_{i}\bar{X}_{i}^{2}  &  <\infty\,\,\;a.s.~,
\end{align*}
convergence of the sequence $\left\{  \bar{X}_{n}\right\}  _{n\geq1}$ to zero
one can prove in two ways.\newline\emph{Firstly,} one can prove convergence
a.s. of the series $\sum_{n\geq1}\mu_{n}\bar{X}_{n}X_{n+1}$, which in the
light of Lemma \ref{lemosr}, is equivalent to proving convergence a.s. of the
sequence $\left\{  \bar{X}_{n}\right\}  $ to zero.\newline\emph{Secondly, }one
can prove convergence a.s. of the series $\sum_{i\geq0}\mu_{i}X_{i+1}$, which
in light of Lemma \ref{podstawowy} implies convergence a.s. sequence $\left\{
\bar{X}_{n}\right\}  $ to zero.
\end{remark}

\begin{remark}
If we choose the second method proposed in the previous remark, to have almost
sure convergence of the series $\sum_{i\geq0}\mu_{i}X_{i+1}$, we have also
almost sure convergence series of the form $\sum_{i\geq0}\mu_{i}^{\prime
}X_{i+1}$, where the sequence is normal $\left\{  \mu_{i}^{\prime}\right\}
_{i\geq0}$ is selected, that e.g. number series \newline$\sum_{i\geq
1}\left\vert \mu_{i}-\beta\mu_{i}^{\prime}\right\vert E\left\vert
X_{i+1}\right\vert $ is convergent for some $\beta$ (compare corollary
\ref{podobne}). Let us recall that a bit different sequences $\left\{  \mu
_{i}\right\}  _{i\geq0}$ and $\left\{  \mu_{i}^{\prime}\right\}  $ may imply
very different conjugate $\left\{  \alpha_{i}\right\}  $ and $\left\{
\alpha_{i}^{\prime}\right\}  $ such that $\overline{\left\{  \alpha
_{i}\right\}  }\allowbreak=\allowbreak\left\{  \mu_{i}\right\}  $ and
$\overline{\left\{  \alpha_{i}^{\prime}\right\}  }\allowbreak=\allowbreak
\left\{  \mu_{i}^{\prime}\right\}  $. Hence, e.g. if we have proven
convergence of the series $\sum_{i\geq1}\frac{X_{i}}{i}$, then of course
sequence $\left\{  \frac{\sum_{i=1}^{n}X_{i}}{n}\right\}  _{n\geq1}$ converges
to zero. On the other hand, since series $\sum_{i\geq1}\frac{X_{i}}{i}$
converges, then converges also series $\sum_{i\geq1}\frac{2X_{i}}{i+1}$ (why?
and what else has to be assumed about the moments? we leave it as an exercise
to the reader), hence converges to zero also sequence $\left\{  \frac
{\sum_{i=1}^{n}iX_{i}}{\sum_{i=1}^{n}i}\right\}  _{n\geq1}$ (we have
$\overline{\left\{  i+1\right\}  }=\left\{  \frac{2}{i+2}\right\}  )$.
Similarly, one can show that the sequence $\left\{  \frac{\sum_{i=1}^{n}%
i^{2}X_{i}}{\sum_{i=1}^{n}i^{2}}\right\}  _{n\geq1}$ converges a.s. to zero.
Why? and what additional technical assumptions have to be made? Formulation
and justification of respective simple fact again we leave to the reader.
\end{remark}

\subsection{For independent random variables}

\subsubsection{Having identical distributions\label{smpwl_iid}}

We will start with the classical Kolmogorov's result. Proof of this theorem
will not be however classical in the sense, that it is different from the
original Kolmogorov's proof and uses theorems on reverse martingale
convergence and $0-1$ Hewitt-Savege' law;

\begin{theorem}
[Kolmogorow]%
\index{Theorem!Kolmogorov}%
\label{kolmogor}If $\{X_{n}\}_{n\geq1}$ are independent random variables
having identical distributions, in order that it satisfied SLLN it is
necessary and sufficient that $E\left\vert X_{1}\right\vert <\infty.$
\end{theorem}

\begin{proof}
Let us denote $S_{n}=\sum_{i=1}^{n}X_{i}$, $\mathcal{B}_{-n}=\sigma
(S_{n},S_{n+1},\ldots)\;,n=1,2,\ldots\;$. Let us notice that $S_{n}%
=\allowbreak E(S_{n}|\mathcal{B}_{-n})=\allowbreak\sum_{i=1}^{n}%
E(X_{i}|\mathcal{B}_{-n})=\allowbreak nE(X_{1}|\mathcal{B}_{n})\,\;a.s.$,
since $E(X_{1}|\mathcal{B}_{n})=E(X_{i}|\mathcal{B}_{n})$ a.s. for $i\leq n$
because of symmetry. Hence,
\[
E(X_{1}|\mathcal{B}_{-n})=S_{n}/n\;a.s.~.
\]
Moreover, we have of course
\[
E(E(X_{1}|\mathcal{B}_{-n+1})|\mathcal{B}_{-n})\allowbreak=E(\frac{S_{n-1}%
}{n-1}|\mathcal{B}_{-n})\allowbreak=\frac{(n-1)E(X_{1}|\mathcal{B}_{-n})}%
{n-1}\allowbreak=E(X_{1}|\mathcal{B}_{-n}),
\]
since of course $\mathcal{B}_{-n}\supseteq\mathcal{B}_{-n-1},\;n=1,2,\ldots
$.\ Further, we have $E\left\vert E(X_{1}|\mathcal{B}_{-n})\right\vert \leq
E\left\vert X_{1}\right\vert $. Hence, one can make use of Theorem
\ref{odwrotny_martyngal}. Now we deduce, that the sequence $\left\{
S_{n}/n\right\}  $ converges with probability $1$. Let $L$ denote this the
limit. Events $\left\{  \omega:L(\omega)<x\right\}  $ are symmetric in the
sense of definition \ref{sym3}. Now by $0-1$ Hewitt-Savege' law (see Appendix
\ref{prawo01}) it follows that the probability of this event is $0$ or $1$. In
other words cdf of a random variable $L$ is a jump function having one jump,
that is the random variable $L$ has degenerated distribution. Consequently
there exists such constant $c$, that $P(L=c)=1$. Being aware that we have also
convergence in $L_{1}$ (family of the random variables $\left\{
S_{n}/n\right\}  _{n\geq1}$ is uniformly integrable see Appendix \ref{UnInt}),
we have: $c=EL=\underset{n\rightarrow\infty}{\lim}E\frac{S_{n}}{n}=EX_{1}.$

Let us suppose now, that the sequence $\left\{  S_{n}/n\right\}  $ converges
almost surely to a finite limit In such a case this sequence on the basis of
Proposition \ref{war_kon}, as well as the sequence $\left\{  X_{n}/n\right\}
$ converges to a finite limit (equal to zero). If $E\left\vert X_{1}%
\right\vert =\infty$, then as we know by Proposition \ref{wlasnosc_iid} we
would have $\underset{n\rightarrow\infty}{\lim\sup\,}\frac{\left\vert
X_{n}\right\vert }{n}=\infty$. Hence, we must have $E\left\vert X_{1}%
\right\vert <\infty.$
\end{proof}

In the sequel, we will consider necessary and sufficient conditions for the
SGLLN to be satisfied under the assumption, that $E\left\vert X_{1}\right\vert
\allowbreak<\allowbreak\infty$. It will be the result of B. Jamison, S. Orey
and W. Pruitt from the paper \cite{Jamison65} .

Let $\left\{  \alpha_{i}\right\}  _{i\geq0}$, $\alpha_{0}=1$ be a sequence
weights, a $\left\{  X_{i}\right\}  _{i\geq1}$ sequence independent random
variables having identical distributions. We will consider a sequence:
\[
M_{n}=\frac{\sum_{i=0}^{n-1}\alpha_{i}X_{i+1}}{\sum_{i=0}^{n-1}\alpha_{i}}%
\]
and examine its convergence with probability $1$ to a constant. Let us recall
that the sequence $\left\{  M_{n}\right\}  _{n\geq1}$ satisfies the following
recurrent relationship:
\begin{equation}
M_{n+1}=(1-\mu_{n})M_{n}+\mu_{n}X_{n+1}. \label{tozsam}%
\end{equation}
Let us denote $\overline{\left\{  \alpha_{i}\right\}  }_{i\geq0}=\left\{
\mu_{i}\right\}  _{i\geq0}$, i.e. $\mu_{i}=\alpha_{i}/\sum_{k=0}^{i}\alpha
_{k}$, $i\geq1.$

\begin{remark}
If $\sum_{i\geq0}\alpha_{i}<\infty$ or equivalently $\sum_{i\geq0}\mu
_{i}<\infty$, then the convergence of the sequence $\left\{  M_{n}\right\}  $
is equivalent to the convergence of the series $\sum_{i\geq0}\alpha_{i}%
X_{i+1}$. This series except for the trivial case of degenerated distribution
of the random variable $X_{1}$ cannot converge to a constant. Hence, if
$\sum_{i\geq0}\alpha_{i}<\infty$, then LLN is not satisfied.
\end{remark}

\begin{remark}
Thus, let us assume that $\sum_{i\geq0}\mu_{i}=\infty$. If LLN\ is satisfied,
i.e. the sequence of the random variables $\left\{  M_{n}\right\}  _{n\geq1}$
converges to a constant, then from the identity (\ref{tozsam}) it follows that
$\mu_{n}\rightarrow0$, as $n\rightarrow\infty$ and moreover, that $\mu
_{n}X_{n+1}\rightarrow0$ with probability $1$, as $n\rightarrow\infty.$
\end{remark}

Let $N(x)$, $x>0$ denote the number of those $n$, for which $1/x\leq\mu_{n-1}%
$, i.e. $N(x)=\#\left\{  n:\frac{1}{\mu_{n-1}}\leq x\right\}  .$

\begin{proposition}
\label{o_N(X)}\emph{ }Let\emph{ }
\begin{equation}
N(x)=\sum_{n\geq1}I\left(  x\geq1/\mu_{n-1}\right)  , \label{def_Nx}%
\end{equation}
\emph{ }then the\emph{ }function $N(x)$ is a nondecreasing step function, with
jumps of size $1$, at values of the sequence $\left\{  \frac{1}{\mu_{n-1}%
}\right\}  _{n\geq1}.$
\end{proposition}

\begin{proof}
Obvious. It follows directly from the definition.
\end{proof}

We have also the following lemma.

\begin{lemma}
\label{lem_oNx}If the sequence $\left\{  M_{n}\right\}  _{n\geq1}$ defined by
the relationship (\ref{tozsam}) converges to a constant, then
\[
\forall c>0:EN\left(  c\left\vert X_{1}\right\vert \right)  <\infty.
\]

\end{lemma}

\begin{proof}
We saw already that, convergence of the sequence $\left\{  M_{n}\right\}  $ to
a constant implies convergence of the sequence $\left\{  \mu_{n}%
X_{n+1}\right\}  _{n\geq1}$ to zero with probability $1$. In other words, the
event $\left\{  \mu_{n}\left\vert X_{n+1}\right\vert >\varepsilon\right\}  $
occurs a finite number of times, for every $\varepsilon>0$. From assertion
$iii)$ of the Borel-Cantelli Lemma (see Appendix \ref{Borel-Cantelli}) it
follows that taking into account independence of the random variables $X_{i}$,
$i>0$ we have
\[
\sum_{i\geq0}P\left(  \mu_{i}\left\vert X_{i+1}\right\vert \geq\varepsilon
\right)  <\infty.
\]
Due to the assumption of the same distributions of the random variables
$\{X_{n}\}_{n\geq1}$ the last condition can be presented the following way,
due to formula(\ref{def_Nx}), :
\[
\sum_{i\geq0}\int_{\left\vert x\right\vert \geq\varepsilon/\mu_{i}}%
dF(x)=\int\sum_{i\geq1}I(\left\vert x\right\vert \geq\varepsilon/\mu
_{i-1})dF(x)=\int N\left(  \frac{\left\vert x\right\vert }{\varepsilon
}\right)  dF(x),
\]
where $F$ denotes cdf of the random variable $X_{1}.$
\end{proof}

\begin{theorem}
\label{mpwl_iid}Let $\left\{  X_{n}\right\}  _{n\geq1}$ be a sequence of
independent random variables having identical distributions, and such that
$E\left\vert X_{1}\right\vert <\infty$, and let $\left\{  \alpha_{i}\right\}
_{i\geq0}$ be a sequence of positive weights. Sequence$\{X_{n}\}_{n\geq1}$
satisfies generalized MLLN with weights $\left\{  \alpha_{i}\right\}
_{i\geq0}$ if and only if,
\begin{equation}
\text{ }\underset{x\rightarrow\infty}{\lim\inf}\,\frac{N(x)}{x}<\infty,
\label{ogr_nx}%
\end{equation}
where we denoted:
\[
N(x)=\#\{n:\frac{\sum_{i=0}^{n-1}\alpha_{i}}{\alpha_{n-1}}\leq x\},
\]
for positive $x.$
\end{theorem}

\begin{proof}
Without loss of generality, let us assume that the random variables
$\{X_{n}\}_{n\geq1}$ have zero expectations. The main idea of the proof
consists on considering \textquotedblright truncated\textquotedblright\ random
variables $\left\{  Y_{n}\right\}  _{n\geq1}$ defined in the following way:
\[
Y_{n}\allowbreak=\allowbreak X_{n}I\left(  \left\vert X_{n}\right\vert
<\frac{\sum_{i=0}^{n-1}\alpha_{i}}{\alpha_{n-1}}\right)  .
\]
We have:
\begin{gather*}
\sum_{n\geq1}P(Y_{n}\neq X_{n})=\sum_{n\geq1}P\left(  \left\vert
X_{n}\right\vert \geq\frac{\sum_{i=0}^{n-1}\alpha_{i}}{\alpha_{n-1}}\right)
=\\
=\sum_{n\geq1}P\left(  \left\vert X_{1}\right\vert \geq\frac{\sum_{i=0}%
^{n-1}\alpha_{i}}{\alpha_{n-1}}\right)  =E\sum_{n\geq1}I\left(  \left\vert
X_{1}\right\vert \geq\frac{\sum_{i=0}^{n-1}\alpha_{i}}{\alpha_{n-1}}\right)
=EN(\left\vert X_{1}\right\vert ).
\end{gather*}
We utilized here identity of distributions of the variables $\{X_{n}%
\}_{n\geq1}$, part \emph{ii)} of Proposition \ref{o_N(X)} and the formula
\ref{EX_dod_calk}. Hence, if $EN(\left\vert X_{1}\right\vert )<\infty$, then
from Borel-Cantelli Lemma, it will follow, that it is enough to examine random
variables $\left\{  Y_{n}\right\}  _{n\geq1}$. However condition
(\ref{ogr_nx}) guarantees, that $E\left\vert X_{1}\right\vert <\infty$
$\Rightarrow$ $EN(\left\vert X_{1}\right\vert )<\infty$. Let us denote
\[
T_{n}=\sum_{i=0}^{n-1}\alpha_{i}Y_{i+1}/\sum_{i=0}^{n-1}\alpha_{i}.
\]
Let us also notice, that since the events $\left\{  Y_{i}\neq X_{i}\right\}
_{i\geq1}$ have occured only a finite number of times, we must have
$EY_{n}\rightarrow EX_{1}$, when $n\rightarrow\infty$. Thus, if we will show,
that
\begin{equation}
\sum_{i=0}^{n-1}\alpha_{i}(Y_{i+1}-EY_{i+1})/\sum_{i=0}^{n-1}\alpha
_{i}\rightarrow0\;a.s. \label{zb_yi}%
\end{equation}
as $n\rightarrow\infty$, then the generalized strong law of large numbers will
be satisfied, i.e. we will have the following convergence
\[
\sum_{i=0}^{n-1}\alpha_{i}(X_{i+1}-EX_{i+1})/\sum_{i=0}^{n-1}\alpha
_{i}\rightarrow0\;a.s.
\]
as $n\rightarrow\infty$. From Lemma \ref{podstawowy} it follows that for the
condition (\ref{zb_yi}) to be satisfied it is enough that the series
\begin{equation}
\sum_{i\geq1}\mu_{i}(Y_{i+1}-EY_{i+1}), \label{szer}%
\end{equation}
converge almost surely, whereas usually we denoted $\mu_{n}=\alpha_{n}%
/\sum_{i=0}^{n}\alpha_{i}$. The sequence of partial sums of this series is
(taking into account independence of the random variables $\left\{
Y_{i}\right\}  _{i\geq1})$ a martingale. Hence, it is enough to, e.g., that
the series
\begin{equation}
\sum_{i\geq1}\mu_{i}^{2}\operatorname*{var}(Y_{i+1}) \label{szer_var}%
\end{equation}
and the respective martingale\ are being convergent almost surely and
consequently strong law of large numbers is being satisfied. Let us examine
the condition (\ref{szer_var}). We have:
\begin{gather*}
\sum_{i\geq1}\mu_{i}^{2}\int_{%
%TCIMACRO{\U{211d} }%
%BeginExpansion
\mathbb{R}
%EndExpansion
}x^{2}I(\left\vert x\right\vert <1/\mu_{i})dF(x)\\
=\int_{%
%TCIMACRO{\U{211d} }%
%BeginExpansion
\mathbb{R}
%EndExpansion
}x^{2}\left(  \sum_{i\geq1}\mu_{i}^{2}I(\left\vert x\right\vert <1/\mu
_{i})\right)  dF(x),
\end{gather*}
where we denoted by $F(x)$ the cdf of random the variable $X_{1}$. We have
further:
\[
\sum_{i\geq1}\mu_{i}^{2}I(\left\vert x\right\vert <1/\mu_{i})=\sum_{\left\{
i:\left\vert x\right\vert <1/\mu_{i}\right\}  }\mu_{i}^{2}.
\]
From the remark concerning jumps of the function $N(x)$, it follows that:
\[
\sum_{\left\{  i:\left\vert x\right\vert <1/\mu_{i}\leq z\right\}  }\mu
_{i}^{2}=\int_{\left\vert x\right\vert <y\leq z}\frac{dN(y)}{y^{2}}.
\]
In the last integral let us integrate by parts. We get, then:
\[
\int_{\left\vert x\right\vert <y\leq z}\frac{dN(y)}{y^{2}}=\frac{N(z)}{z^{2}%
}-\frac{N(\left\vert x\right\vert )}{x^{2}}+2\int_{\left\vert x\right\vert
<y\leq z}\frac{N(y)}{y^{3}}dy.
\]
Let us notice also, that
\begin{align*}
\frac{N(z)}{z^{2}}  &  =-\int_{z}^{\infty}d(\frac{N(t)}{t^{2}})\\
&  =-\int_{z}^{\infty}\frac{dN(t)}{t^{2}}+2\int_{z}^{\infty}\frac{N(t)}{t^{3}%
}dt\leq2\int_{z}^{\infty}\frac{N(t)}{t^{3}}dt.
\end{align*}
Hence
\begin{align*}
\int_{\left\vert x\right\vert <y\leq z}\frac{dN(y)}{y^{2}}  &  \leq2\left(
\int_{z}\frac{N(t)}{t^{3}}dt+\int_{\left\vert x\right\vert <y\leq z}%
\frac{N(y)}{y^{3}}dy\right) \\
&  =2\int_{\left\vert x\right\vert }^{\infty}\frac{N(y)}{y^{3}}dy.
\end{align*}
Thus, we have:
\begin{gather*}
\sum_{i\geq1}\mu_{i}^{2}\int_{%
%TCIMACRO{\U{211d} }%
%BeginExpansion
\mathbb{R}
%EndExpansion
}x^{2}I(\left\vert x\right\vert <1/\mu_{i})dF(x)\leq2\int x^{2}\int%
_{\left\vert x\right\vert }^{\infty}\frac{N(y)}{y^{3}}dydF(x)\leq\\
2\int_{%
%TCIMACRO{\U{211d} }%
%BeginExpansion
\mathbb{R}
%EndExpansion
}x^{2}\int_{\left\vert x\right\vert }^{\infty}\frac{M}{y^{2}}dydF(x)=2\int_{%
%TCIMACRO{\U{211d} }%
%BeginExpansion
\mathbb{R}
%EndExpansion
}x^{2}\frac{M}{\left\vert x\right\vert }dF(x)=2ME\left\vert X_{1}\right\vert ,
\end{gather*}
where we denoted $\underset{x>0}{\sup}\frac{N(x)}{x}=M$. Hence, we have shown,
that series (\ref{szer}) converges, consequently, that the generalized strong
law of large numbers is satisfied.

Let us concentrate now on the sufficient condition. We will prove indirectly.
Let us assume that $\underset{x\rightarrow\infty}{\lim\sup\,}\frac{N(x)}%
{x}\allowbreak=\allowbreak\infty$. It means that there exists such number
sequence $\left\{  x_{i}\right\}  _{i\geq1}$, that $\frac{N(x_{k})}{x_{k}%
}\rightarrow\infty$, as $k\rightarrow\infty$. Hence, we can select such a
sequence $\left\{  f_{k}\right\}  _{k\geq1}$ of positive numbers, summing to
one such that $\sum_{i\geq1}x_{i}f_{i}<\infty$ and $\sum_{i\geq1}f_{i}%
N(x_{i})\allowbreak=\allowbreak\infty$. Treating sequence $\left\{
f_{i}\right\}  $ as step sizes of some some random variable $\left\vert
X_{1}\right\vert $, we see that $E\left\vert X_{1}\right\vert <\infty$,
however that $EN\left(  \left\vert X_{1}\right\vert \right)  =\infty$. The
last condition does not allow that SLLN be satisfied in the light of Lemma
\ref{lem_oNx}.
\end{proof}

\begin{remark}
Notice that in the classical case, i.e. when $\mu_{n}\allowbreak
=\allowbreak1/(n+1)$, $n\geq0$ we have $N(x)\allowbreak=\allowbreak
\left\lfloor x\right\rfloor $, i.e. $N(x)$ is equal to the largest integer not
exceeding $x$. We have, then of course $\underset{x\rightarrow\infty}{\lim
\sup}\frac{N(x)}{x}\leq1.$
\end{remark}

\begin{example}
The above mentioned Theorem will also be illustrated by the following example.
\newline Let $\left\{  \xi_{i}\right\}  _{i\geq1}$ be a sequence of
independent random variables having identical distributions. Let us assume
that $\xi_{1}\sim F_{\gamma}(x)$, $\gamma>1$, where $F_{\gamma}$ is defined by
(\ref{pareto}). Then of course we have $E\xi_{1}=\frac{\gamma}{\gamma-1}%
$.\ Further, let $\mu_{i}=\frac{4(i+1)}{(i+2)^{2}};i\geq0$, i.e. $\left\{
\frac{4(i+1)}{(i+2)^{2}}\right\}  =\overline{\left\{  (i+1)^{2}\right\}  }$
and $X_{n}=\sum_{i=1}^{n}i^{2}\xi_{i}/\sum_{i=1}^{n}i^{2}$. In this example we
took $\gamma=\frac{5}{4}$, i.e. $E\xi_{1}=5$. Moreover, let us notice that
$\xi_{1}$ does not have a variance.
%TCIMACRO{\FRAME{ftbpFU}{2.6013in}{1.5264in}{0pt}{\Qcb{\QTR{tiny}{Strong law of
%large numbers for independent random variables not possesing variances}}}%
%{}{mpwl1.eps}{\special{ language "Scientific Word";  type "GRAPHIC";
%width 2.6013in;  height 1.5264in;  depth 0pt;  original-width 6.2379in;
%original-height 3.6469in;  cropleft "0";  croptop "1";  cropright "1";
%cropbottom "0";  filename 'mpwl1.eps';file-properties "XNPEU";}}}%
%BeginExpansion
\begin{figure}[ptb]%
\centering
\includegraphics[
height=1.5264in,
width=2.6013in
]%
{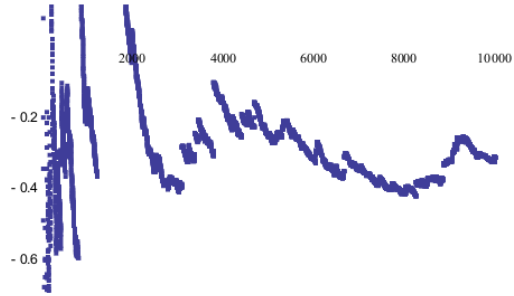}%
\caption{{\protect\tiny Strong law of large numbers for independent random
variables not possesing variances}}%
\end{figure}
%EndExpansion
In order to omit technical difficulties we presented on the plot only every
$K=200$'th average (i.e. $X_{200n}$; $n\allowbreak=\allowbreak1,...$). Hence,
the plot is based on $N=2000000$ observations. Let us notice also, that
convergence is rather slow. It follows from the fact, that distribution of the
random variable $\xi_{1}$ has the so-called \textquotedblright fat
tails\textquotedblright, i.e. function $1-F_{\gamma}(x)$ slowly decreases as
$x\rightarrow\infty$. It means simply that relatively often we get very large
values of $\xi_{i}$. To illustrate this, below we present a table of numbers
of records of observations $\left\{  \xi_{i}\right\}  $ and values of these
records that is the table of numbers $\left(  m,M_{m}\right)  $, where these
numbers are defined recursively $m(0)\allowbreak=0$, $\allowbreak
m(1)\allowbreak=\allowbreak1$, $m(n+1)\allowbreak=\allowbreak\min\{j:\xi
_{j}>M_{m(n)}\}$ and $\xi_{n}\allowbreak=\allowbreak M_{m(n)}$. \newline%
\begin{tabular}
[c]{llllllllll}%
${\small 3}$ & ${\small 14}$ & $15$ & $17$ & $42$ & $189$ & $327$ & $99184$ &
$101942$ & $461831$\\
$-2.20$ & $-0.27$ & $0.68$ & $1.21$ & $39.55$ & $102.11$ & $4602.51$ &
$9292.62$ & $73342.5$ & $84457$%
\end{tabular}
\newline As one can see the first record happened to observation $3$ and had
valued $-2.2$, the second in the $14$-th observation and had value $-.27$ and
so on. Finally, the $11$-th record happened to observation $13338828$ and had
value equal $87444$ .
\end{example}

\subsubsection{Having different distributions}

Recall that in the case of identical distributions of the sequence of
independent random variables, it was sufficient to assume the existence of
expectations to get \ MLLN (compare theorem \ref{kolmogor}) to be satisfied by
$\{X_{n}\}_{n\geq1}$. The case of different distributions is different. It
turns out that even the existence of second moments is not enough. It turns
out that these moments must satisfy some additional conditions. This case will
be treated in a whole in the next section together with the case of dependent
elements of the sequence $\{X_{n}\}_{n\geq1}$. Whereas here we will present
examples of sequences of independent random variables, not satisfying MLLN.
Interesting and instructive examples presented below, are modifications of
examples taken from the book of R\'{e}v\'{e}sz \cite{Revesz67}.

\begin{example}
Let sequence $\left\{  \sigma_{n}^{2}\right\}  _{n\geq1}$ will be such
sequence of positive numbers, that
\begin{equation}
\sum_{n\geq1}\frac{\sigma_{n}^{2}}{n^{2}}=\infty. \label{war_rozbieznosci}%
\end{equation}
Let sequence $\{X_{n}\}_{n\geq1}$ be a sequence independent random variables
having the following distributions:
\begin{align*}
P(X_{n}  &  =n)=P(X_{n}=-n)=\frac{\sigma_{n}^{2}}{2n^{2}},\\
P(X_{n}  &  =0)=1-\frac{\sigma_{n}^{2}}{n^{2}},
\end{align*}
when $\sigma_{n}^{2}\leq n^{2}$ and
\[
P(X_{n}=\sigma_{n})=P\left(  X_{n}=-\sigma_{n}\right)  =\frac{1}{2},
\]
when $\sigma_{n}^{2}>n^{2}$. Let us notice that $EX_{n}=0$,
$\operatorname*{var}(X_{n})=\sigma_{n}^{2,}$ and Moreover, $P(\left\vert
X_{n}\right\vert \geq n)\allowbreak\geq\allowbreak\sigma_{n}^{2}/n^{2}$. From
the last estimation, it follows that $P(\underset{n\rightarrow\infty}{\lim
}\left\vert \frac{X_{n}}{n}\right\vert =0)\allowbreak=\allowbreak0$, in other
words necessary condition for the SLLN to be satisfied is not satisfied (see
Proposition \ref{war_kon}). Thus, we have constructed a sequence of
independent random variables with given variances, that is not satisfying
SLLN. Condition (\ref{war_rozbieznosci}) indicates how quickly variances of
elements of the sequence of independent random variables have to increase, so
that MLLN is not satisfied for this sequence.
\end{example}

In connection with the previous example, there appears a question, can one
select the sequence of weights in such a way, as to make a sequence of
independent random variables having increasing variances satisfy a generalized
SLLN. It turns out that one can give a similar condition to
(\ref{war_rozbieznosci}), imposed on the speed with which variances of
elements of the sequence $\{X_{n}\}_{n\geq1}$ increase, so that for any system
of weights generalized SLLN is not satisfied.

To describe this situation more precisely, let us consider the sequence
$\{X_{n}\}_{n\geq1}$ of independent random variables such that $EX_{n}=0$,
$\operatorname*{var}(X_{n})=\sigma_{n}^{2}$. We have the following simple lemma.

\begin{lemma}
$\min\operatorname*{var}(\sum_{i=1}^{n}a_{i}X_{i})=1/\sum_{i=1}^{n}%
1/\sigma_{i}^{2}$, where the minimum is taken over all systems of positive
numbers $\left\{  a_{i}\right\}  _{i=1}^{n}$, such that $\sum_{i=1}^{n}%
a_{i}=1.$
\end{lemma}

\begin{proof}
We have
\[
1=\allowbreak\sum_{i=1}^{n}a_{i}\allowbreak=\sum_{i=1}^{n}a_{i}\sigma_{i}%
\frac{1}{\sigma_{i}}\leq\sqrt{\sum_{i=1}^{n}a_{i}^{2}\sigma_{i}^{2}}\sqrt
{\sum_{i=1}^{n}1/\sigma_{i}^{2}}.
\]
Remembering, that $\operatorname*{var}(\sum_{i=1}^{n}a_{i}X_{i})=\sum
_{i=1}^{n}a_{i}^{2}\sigma_{i}^{2},$ we see that $\operatorname*{var}%
(\sum_{i=1}^{n}a_{i}X_{i})\geq1/\sum_{i=1}^{n}1/\sigma_{i}^{2}$. It is known,
that equality in inequality Schwarz that we applied satisfies, when
$a_{i}\sigma_{i}=\lambda/\sigma_{i};i=1,\ldots,n$ for some $\lambda$. Hence,
taking into account condition $1=\allowbreak\sum_{i=1}^{n}a_{i}\allowbreak$,
when $a_{i}=\eta/\sigma_{i}^{2}$, where $\eta=1/\sum_{i=1}^{n}1/\sigma_{i}%
^{2}$. We have moreover ,
\[
\operatorname*{var}(\sum_{i=1}^{n}X_{i}\eta/\sigma_{i}^{2})\allowbreak
=\allowbreak\eta^{2}\sum_{i=1}^{n}\sigma_{i}^{2}/\sigma_{i}^{4}\allowbreak
=\allowbreak1/\sum_{i=1}^{n}1/\sigma_{i}^{2}.
\]

\end{proof}

Having this lemma let us assume, that our sequence of the random variables
$\{X_{n}\}_{n\geq1}$ is such that
\begin{equation}
\sum_{i\geq1}\frac{1}{\sigma_{i}^{2}}<\infty. \label{rozb_suma}%
\end{equation}
Then of course we would have for any system of weights $\left\{
a_{in}\right\}  _{i\geq1}$ such that $\forall n=1,2,\ldots\;\sum_{i=1}%
^{n}a_{in}=1,$%
\[
\operatorname*{var}(\sum_{i=1}^{n}a_{in}X_{i})\geq1/\sum_{i=1}^{n}1/\sigma
_{i}^{2}\geq1/\sum_{i\geq1}\frac{1}{\sigma_{i}^{2}}>0.
\]
If the sequence $\{X_{n}\}_{n\geq1}$ satisfied\ generalized $SLLN$, then of
course respective sequence of averages (i.e. the sequence $\sum_{i=1}%
^{n}a_{in}X_{i})$ would converge with probability to zero. In particular, the
sequence of variances of its elements would converge\ to zero, which is
impossible in this case. Thus, the sequence of independent random variables
having variances satisfying condition (\ref{rozb_suma}) does not satisfy any
generalized SLLN.

\subsection{For random variables possessing variances}

\subsubsection{For uncorrelated random variables.}

Properties of the sequence of averages $\left\{  \bar{X}_{i}\right\}
_{i\geq1}$ are described in Lemmas \ref{lemosr} and \ref{prostypomocniczy}.
For convenience, we will recall below, these results compiled in one, new lemma.

\begin{lemma}
\label{o_zbieznosci}If the series $\sum_{i\geq1}\mu_{i}^{2}EX_{i+1}^{2}$ is
convergent, then: \newline\emph{i)} series $\sum_{i\geq0}\mu_{i}\bar{X}%
_{i}^{2}$ is convergent almost surely, \newline\emph{ii) }sequence $\left\{
\frac{\sum_{i=1}^{n}\alpha_{i}\bar{X}_{i}^{2}}{\sum_{i=1}^{n}\alpha_{i}%
}\right\}  _{n\geq1}$ converges almost surely to zero,\newline$\allowbreak
$\emph{iii)} \thinspace$\bar{X}_{n}$ $\underset{n\rightarrow\infty
}{\longrightarrow}$ $0$ if and only if, series $\sum_{i\geq0}\mu_{i}\bar
{X}_{i}X_{i+1}$ is convergent almost surely, \newline\emph{iv) }Let $\left\{
n_{k}\right\}  _{k\geq1}$ be a sequence of indices defined by the relationship
$\sum_{i=n_{k}+1}^{n_{k+1}}\mu_{i}\allowbreak=\allowbreak O(1)$, then $\bar
{X}_{n_{k}}\underset{k\rightarrow\infty}{\longrightarrow}0$ almost surely.
\end{lemma}

\begin{proof}
Proofs of \emph{i)}, \emph{ii)}, \emph{iii)}, \emph{iv) }are repetitions of
proofs of assertions \ref{pierwsza}, \ref{zb_sred}, \ref{zb_kwadrat},
\ref{zb_podciagu} of Lemma \ref{prostypomocniczy}.
\end{proof}

\begin{remark}
\label{o_wyzszych_srednich}Let us notice that having assumed convergence of
the series of assertion \emph{i) }Lemma \ref{o_zbieznosci} one can prove
(making use of the iterative form, the properties of the sequence $\left\{
\mu_{i}\right\}  _{i\geq0}$ and of Lemma \ref{podstawowy}) not only
convergence of the sequence of assertion \emph{ii)} of this lemma, but also
the convergence of e.g. sequence
\[
\left\{  \frac{\sum_{i=1}^{n}\alpha_{i}\left(  \sum_{j=0}^{i}\alpha
_{j}\right)  ^{k-1}\bar{X}_{i}^{2}}{\left(  \sum_{j=0}^{n}\alpha_{j}\right)
^{k}}\right\}  _{n\geq1}%
\]
for $k=1,2,\ldots\,.$
\end{remark}

Moreover, we have of course Menchoff ' Theorem \ref{menshov} to examine of the
general case of the sequence $\left\{  X_{n}\right\}  _{n\geq1}$ and Theorem
\ref{zb_mart} (Doob) for examining of the case of the sequence $\left\{
X_{n}\right\}  _{n\geq1}$ consisting of martingale differences, which are used
to examine almost sure convergence of the sequence $\left\{  \bar{X}%
_{n}\right\}  _{n\geq1},$together with Lemma \ref{podstawowy} gives the
following theorem.

\begin{theorem}
\label{prawowl}Let the sequence $\left\{  X_{n}\right\}  _{n\geq1}$ consist of
uncorrelated random variables having finite second moments. \newline If series
$\sum_{i\geq1}\mu_{i}^{2}EX_{i+1}^{2}\left(  \ln i\right)  ^{2}$ is convergent
\newline or if convergent is the series $\sum_{i\geq1}\mu_{i}^{2}EX_{i+1}^{2}$
and additionally sequence $\left\{  X_{n}\right\}  _{n\geq1}$ consist of
martingale differences, then:

\begin{enumerate}
\item $\bar{X}_{n}$ $\underset{n\rightarrow\infty}{\longrightarrow}$ $0$
almost surely

\item series $\sum_{i\geq1}\mu_{i}X_{i+1}$ converges almost surely to some
square integrable random variable $S.$
\end{enumerate}
\end{theorem}

If additionally system $\{X_{n}\}_{n\geq1}$ is PSO (see page.
\pageref{system PSO}), then we have Lemma \ref{oo}, and if additionally we
know, that $\sup_{n\geq1}\mu_{n}\ln n<\infty$, then we get Theorem \ref{o PSO}.

As a corollary, we have a classical theorem concerning the sequence of
independent random variables that satisfy SLLN.

\begin{corollary}
Let $\left\{  X_{i}\right\}  _{i\geq1}$ be a sequence independent random
variables having finite variances such that $\sum_{i\geq1}\operatorname*{var}%
(X_{i})/(i+1)^{2}<\infty$. Then sequence $\left\{  X_{i}\right\}  _{i\geq1}$
satisfies SLLN.
\end{corollary}

\begin{proof}
It follows from previous theorems, since, firstly the sequence $\{X_{i}%
\}_{i\geq1}$ is a sequence of martingale differences, secondly we took
$\mu_{i}=1/(i+1)$.
\end{proof}

\begin{example}
The simplest application of Theorem \ref{prawowl}, is to consider a sequence
$\left\{  X_{n}\right\}  _{n\geq1}$ with zero expectations and increasing
variances e.g. following the scheme : $\operatorname{var}(X_{n})\approx
\frac{n^{\alpha}}{\ln^{\beta}n}$. Then, as it follows from the above mentioned
theorems, the sequence\newline$\left\{  Y_{n}=\frac{\sum_{i=1}^{n}X_{i}}%
{n}\right\}  _{n\geq1}$ converges almost surely to zero, if only: either
$\alpha<1$, or if $\alpha=1$ and $\beta>3$ in the general (uncorrelated) case
and either $\alpha<1$, or $\alpha=1$ and $\beta>1$ of the case of martingale
differences. Good illustration of SLLN in this case is the example and
simulation discussed in example \ref{slabe_pwl}. In order to find out, how the
value of the coefficient $\alpha$ influences the quality of convergence we
will present new simulation. We took the same random variables, as in example
\ref{slabe_pwl} with the small difference that now $\beta=\frac{3}{14}$, that
is $\alpha=\frac{6}{14}<0.5$. Again, as before, we expose every 1 hundredth
observation, i.e. for $J=200$, $400$, $\ldots,4000000$ we present $y_{J}$%

%TCIMACRO{\FRAME{ftbpF}{2.8928in}{1.67in}{0pt}{}{}{ob48hh01.eps}%
%width 2.8928in;  height 1.67in;  depth 0pt;  original-width 6.2379in;
%original-height 3.589in;  cropleft "0";  croptop "1";  cropright "1";
%cropbottom "0";  filename 'OB48HH01.eps';file-properties "XNPEU";}}}%
%BeginExpansion
\begin{figure}[ptb]%
\centering
\includegraphics[
height=1.67in,
width=2.8928in
]%
{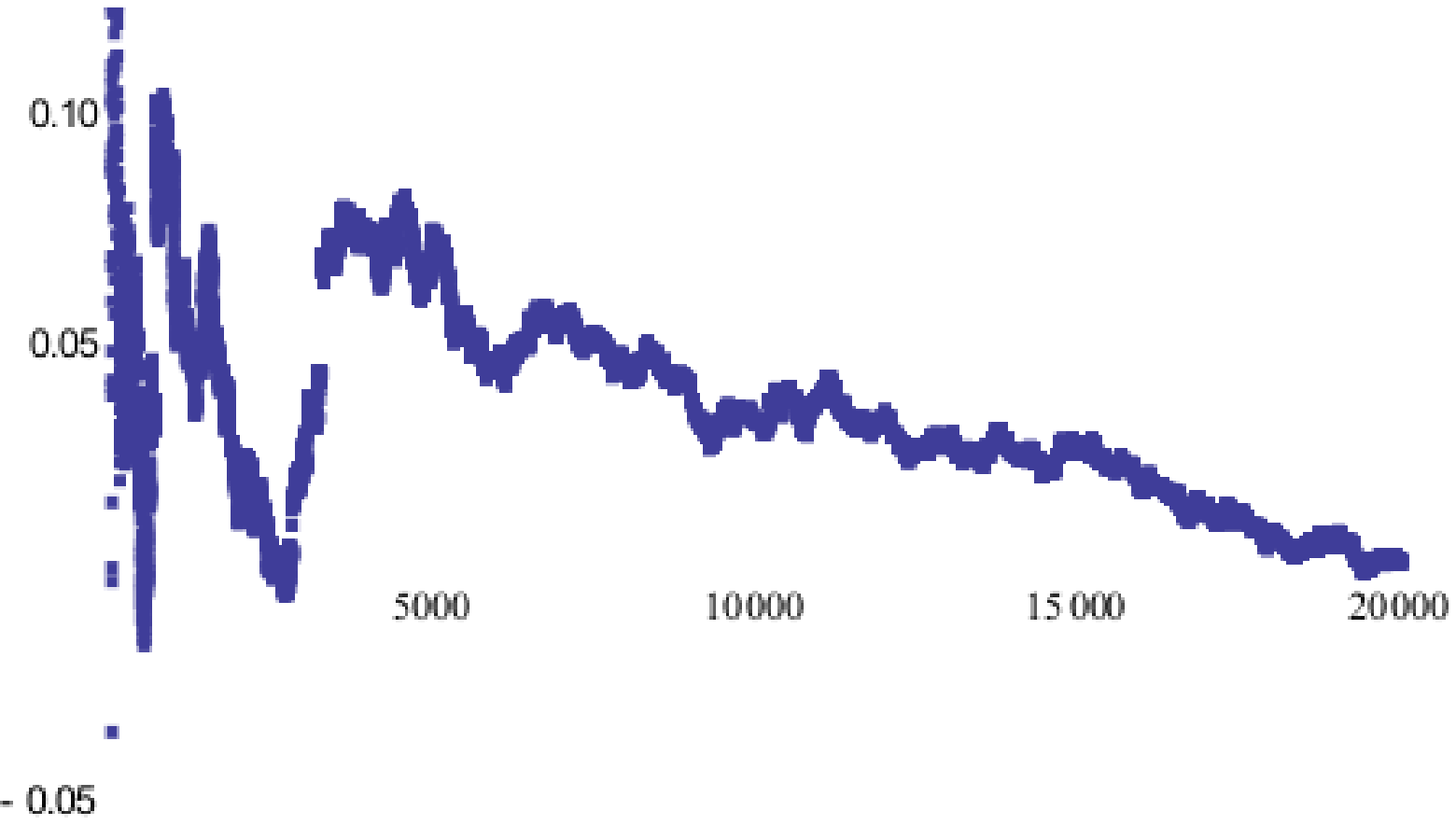}%
\end{figure}
%EndExpansion

\end{example}

Much more interesting is, however the usage of the above mentioned theorem to
examine the of speed convergence in \textquotedblright
ordinary\textquotedblright\ LLN. More precisely, let us assume, that we are
interested in speed of convergence of the sequence $\left\{  \frac{\sum
_{i=1}^{n}X_{i}}{n}\right\}  _{n\geq1}$ to zero. This means that we want to
find such, an increasing number sequence $\left\{  \chi_{n}\right\}  _{n\geq
1}$, that the sequence \newline$\left\{  \chi_{n}\frac{\sum_{i=1}^{n}X_{i}}%
{n}\right\}  _{n\geq1}$ converges almost surely to zero. Such considerations
will be performed in the next example.

\begin{example}
Let us consider, sequence of uncorrelated random variables $\left\{
X_{i}\right\}  _{i\geq1}$ such that $EX_{i}=0$, $EX_{i}^{2}=1;\allowbreak
i\geq1$. Let us set
\[
Y_{n}=\frac{\sum_{i=1}^{n}X_{i}}{\sqrt{n\ln(n+1)\ln\ln^{2}(n+2)}};n\geq1.
\]
Hence $\chi_{n}=\sqrt{\frac{n}{\ln(n+1)\ln\ln^{2}(n+2)}}$. Let us present
$Y_{n}$ in the iterative form.
\[
Y_{n+1}=(1-\nu_{n})Y_{n}+\frac{X_{n+1}}{\sqrt{\left(  n+1\right)  \ln
(n+2)\ln\ln^{2}\left(  n+3\right)  }}%
\]
where
\[
\nu_{n}=1-\sqrt{\frac{n\ln(n+1)\ln\ln^{2}(n+2)}{\left(  n+1\right)  \ln\left(
n+2\right)  \ln\ln^{2}(n+3)}}\approx\frac{1}{2n}+\frac{1}{2n\ln n}+o(\frac
{1}{n\ln n}).
\]
$\allowbreak$ Since, the series $\sum_{n\geq1}\frac{1}{n\ln(n+1)\ln\ln
^{2}(n+2)}$ converges, we deduce that if only additionally we will assume,
that the sequence $\left\{  X_{n}\right\}  _{n\geq1}$ is the sequence of
martingale differences, then sequence $\left\{  Y_{n}\right\}  _{n\geq1}$
converges almost surely to zero.

On the other hand, however, if one does not assume, that the sequence
$\left\{  X_{n}\right\}  _{n\geq1}$ consists of martingale differences, then
knowing, that the series
\[
\sum_{n\geq1}\frac{1}{n\ln(n+1)\ln\ln^{2}(n+2)}\ln^{2}n,
\]
does not converge then on the basis of Menchoff' Theorem \ref{menshov} we
deduce that series
\[
\sum_{n\geq1}X_{n}/\sqrt{n\ln(n+1)\ln\ln^{2}(n+2)}%
\]
may not converge almost surely, although it does converge in mean square. It
depends on further properties of distributions of variables $\left\{
X_{i}\right\}  _{i\geq1}$, not only on the properties of their second moments.
Similarly the sequence $\left\{  Y_{n}\right\}  $ may converge almost surely,
or may not depending on the properties of the sequence $\left\{
X_{i}\right\}  _{i\geq0}$. Following assertion \emph{iv) }of \emph{\ }Lemma
\ref{prostypomocniczy} there exists a subsequence of the sequence $\left\{
Y_{n}\right\}  _{n\geq1}$ that converges almost surely. From the form of the
sequence $\left\{  \nu_{n}\right\}  $ we deduce that the sequence $\left\{
n_{k}\right\}  $ can be chosen to be $\left\{  2^{k}\right\}  $. Hence, on the
basis of assertion \emph{iv)} of the above mentioned lemma, we deduce that
$Y_{2^{k}}\underset{k\rightarrow\infty}{\longrightarrow}0$ almost surely.
Moreover, from assertion \emph{ii)} we infer, for example, that the sequence
$\left\{  \frac{\sum_{i=1}^{n}Y_{i}^{2}/\sqrt{i}}{\sqrt{n}}\right\}  $
converges almost surely to zero. Strict justification of this fact, we leave
to the reader as an exercise.
\end{example}

Let us now suppose, that weights $\left\{  \alpha_{i}\right\}  _{i\geq0}$ are
constant. Hence, $\mu_{i}=\frac{1}{i+1}$, $i\geq0$, $\bar{X}_{n}=\frac{1}%
{n}\sum_{i=1}^{n}X_{i}$. Following remark \ref{o_wyzszych_srednich} under the
assumption, that series $\sum\mu_{i}\bar{X}_{i}^{2}$ is convergent with
probability $1$, it follows that $\frac{\sum_{i=1}^{n}i^{2}\bar{X}_{i}^{2}%
}{n^{3}}\rightarrow0$ with probability $1$ ($k=1$, $\alpha_{i}=i^{2})$. Let
$\beta$ will be any nonnegative number. From assertion $\emph{iii)}$ of Lemma
\ref{Cesaro} we have:
\[
A_{n}^{\beta+1}=O(n^{\beta+1}),\sum_{i=1}^{n}\left(  A_{n-i}^{\beta-1}\right)
^{2}\allowbreak=\allowbreak\sum_{i=1}^{n}\left(  A_{i}^{\beta-1}\right)
^{2}\allowbreak=\allowbreak\sum_{i=1}^{n}O(i^{2\beta-2})\allowbreak
=\allowbreak O(n^{2\beta-1}).
\]
Hence, $\frac{\sqrt{\sum_{i=1}^{n}\left(  A_{n-i}^{\beta-1}\right)  ^{2}%
}n^{3/2}}{A_{n}^{\beta+1}}\allowbreak\cong\allowbreak O(1)$ and consequently
with probability $1$ we have:
\[
\left\vert \frac{\sum_{i=1}^{n}A_{n-i}^{\beta-1}i\bar{X}_{i}}{A_{n}^{\beta+1}%
}\right\vert \leq\frac{\sqrt{\sum_{i=1}^{n}\left(  A_{n-i}^{\beta-1}\right)
^{2}}n^{3/2}}{A_{n}^{\beta+1}}\sqrt{\frac{\sum_{i=1}^{n}i^{2}\bar{X}_{i}^{2}%
}{n^{3}}}\rightarrow0;\,n\rightarrow\infty.
\]
Hence comparing above mentioned expression with assertion \emph{iv)} of Lemma
\ref{Cesaro} we see (taking $\alpha=1$ and $\beta>1)$, that the sequence
$\{X_{n}\}_{n\geq1}$ is $(C,\beta)$ summable for $\beta>1$. We have thus
proved the following Theorem, being a generalization of Theorems $1$ and $2$
of the paper \cite{Moricz83b}.

\begin{theorem}
\label{(C,alfa)}If $\{X_{n}\}_{n\geq1}$ is a sequence of the random variables
such that
\[
\sum_{n\geq1}\frac{EX_{n}^{2}}{(n+1)^{2}}<\infty,\;\sum_{n\geq1}\frac{E\bar
{X}_{n}^{2}}{n+1}<\infty,
\]
where $\bar{X}_{n}=\frac{1}{n}\sum_{i=1}^{n}X_{i}$, then the sequence
$\{X_{n}\}_{n\geq1}$ is $(C,\alpha)$ summable for $\alpha>1$.
\end{theorem}

\begin{remark}
Of course, if we will assume, that the sequence $\{X_{n}\}_{n\geq1}$ consists
of orthogonal random variables, then the first of the conditions of the above
mentioned theorem entails the second one (compare assertion Lemma
\ref{o_zbieznosci}). Hence, in this case we get precisely above mentioned
theorems of the quoted paper.
\end{remark}

To end this part let us recall, that for the uncorrelated random variables
condition $\sum_{n\geq1}\frac{EX_{n}^{2}}{(n+1)^{2}}<\infty$ is not sufficient
for SLLN to occur. It is so since we have the following theorem.

\begin{theorem}
\label{o_rozbieznosci}Let $\left\{  \sigma_{n}^{2}\right\}  _{n\geq1}$ be a
sequence of positive numbers such that series :
\[
\,\sum_{n\geq1}\frac{EX_{n}^{2}\ln^{2}(n+1)}{n^{2}}%
\]
is divergent, and a number sequence $\left\{  \frac{\sigma_{n}^{2}}{n^{2}%
}\right\}  _{n\geq1}$ is non-increasing. Then one can construct the
probability space and define a sequence $\{X_{n}\}_{n\geq1}$ of uncorrelated
random variables with zero expectations such that $EX_{n}^{2}=\sigma_{n}%
^{2},\;n\geq1$ and with probability $1$
\[
\underset{n\rightarrow\infty}{\lim\inf}\,\left\vert \frac{\sum_{i=1}^{n}X_{i}%
}{n}\right\vert =\infty.
\]

\end{theorem}

\begin{proof}
Is long and complicated, one can find it in the paper \cite{Tandori72}.
\end{proof}

\subsubsection{For correlated random variables.}

Lemmas \ref{lemosr} and \ref{prostypomocniczy} supply tools for better
analysis of this of the case. We have the following theorem:

\begin{theorem}
\label{mpwl} If the following series
\begin{equation}
\sum_{n=1}^{\infty}\mu_{n}^{2}\operatorname*{var}(X_{n+1}),\;\;\sum
_{n=1}^{\infty}\mu_{n}\sqrt{\operatorname*{var}(X_{n+1})\operatorname*{var}%
(\overline{X}_{n})}, \label{war_zbieznosci}%
\end{equation}
are convergent, where we denoted as usually $\overline{X}_{n}\allowbreak
=\allowbreak\sum_{i=0}^{n-1}\alpha_{i}X_{i+1}/\sum_{i=0}^{n-1}\alpha
_{i},\;\allowbreak\overline{\left\{  \alpha_{i}\right\}  }\allowbreak
=\allowbreak\left\{  \mu_{i}\right\}  $, then $\{X_{n}\}_{n\geq1}$ satisfies
\emph{generalized} \emph{SLLN} with weights $\left\{  \alpha_{i}\right\}  .$
\end{theorem}

\begin{proof}
Let $\left\{  X_{i}\right\}  _{i\geq1}$ be a sequence of the random variables
possessing variances. Without loss of generality one can assume, that
$EX_{i}=0$. It is easy to notice, that the sequence $\left\{  \overline{X}%
_{n}\right\}  $ satisfies the following recurrent relationship:
\[
\overline{X}_{i+1}=(1-\mu_{i})\overline{X}_{i}+\mu_{i}X_{i+1}.
\]
We multiply side by side this identity by itself, obtaining:
\[
\overline{X}_{i+1}^{2}=(1-\mu_{i}(2-\mu_{i}))\overline{X}_{i}^{2}%
+\allowbreak2(1-\mu_{i})\mu_{i}\overline{X}_{i}X_{i+1}+\allowbreak\mu_{i}%
^{2}X_{i+1}^{2}.
\]
Now we use numerical lemma \ref{podstawowy} of chapter \ref{zbiez} and see
that $\overline{X}_{i}^{2}$ converges to zero almost surely, when the series
$\sum_{i\geq0}\mu_{i}^{2}X_{i+1}^{2}$ and $\sum_{i\geq0}\mu_{i}\overline
{X}_{i}X_{i+1}$ converge almost surely. Now it is enough to apply Schwarz'
inequality to the series $\sum_{i\geq0}\mu_{i}\overline{X}_{i}X_{i+1}$ and use
corollary \ref{zbieznoscbezwzgledna} of Lebesgue' Theorem.
\end{proof}

If we set $\alpha_{i}=\frac{1}{i+1};$ $i=0,1,\ldots$, then from this Theorem
follows evident corollary.

\begin{corollary}
If the following series
\[
\sum_{i\geq1}\operatorname{var}(X_{i})/i^{2},\;\;\sum_{i\geq1}\frac{1}{i}%
\sqrt{\operatorname{var}(X_{i+1})\operatorname{var}\left(  \frac{1}{i}%
\sum_{k=1}^{i}X_{k}\right)  }%
\]
are convergent, then the sequence $\left\{  X_{i}\right\}  _{i\geq1}$
satisfies SLLN.
\end{corollary}

\begin{example}
\emph{\ }Let $\left\{  X_{i}\right\}  _{i\geq1}$ will consist of the following
random variables: \newline\emph{i)}
\[
EX_{i}=0;i\geq1,\left\vert \operatorname*{cov}(X_{i},X_{j})\right\vert
\leq\left\{
\begin{tabular}
[c]{lll}%
$C(\max(i,j))^{\alpha}$ & for & $\left\vert i-j\right\vert \leq4$\\
$0$ & for & $\left\vert i-j\right\vert >4$%
\end{tabular}
\ ,\right.  \alpha<1,
\]

$\emph{ii)}$
\[
EX_{i}=0\,;i\geq1\left|  \operatorname*{cov}(X_{i},X_{j})\right|  \leq\frac
{C}{2^{\left|  i-j\right|  }};i,j\geq1.
\]

It turns out that the sequences of i) and ii) satisfy SLLN. To show this, let
us notice that estimation $E\left\vert \bar{X}_{n}\right\vert ^{2}$, examining
convergence of series $\sum_{i\geq1}\frac{E\left\vert X_{i}\right\vert ^{2}%
}{i+1}$ is essential. In the case \emph{i) }we have:
\begin{align*}
E\left\vert \bar{X}_{n}\right\vert ^{2}  &  =\frac{1}{n^{2}}E(\sum_{i=1}%
^{n}X_{i})^{2}\leq\\
&  \leq\frac{C}{n^{2}}(\sum_{i=1}^{n}i^{\alpha}+2\sum_{i=2}^{n}i^{\alpha
}\allowbreak+2\sum_{i=3}^{n}i^{\alpha}\allowbreak+2\sum_{i=4}^{n}i^{\alpha
}+\allowbreak2\sum_{i=5}^{n}i^{\alpha}))\\
&  \leq\frac{9C}{n^{1-\alpha}}.
\end{align*}
Hence for this case we have
\[
\sum_{n=1}^{\infty}\operatorname*{var}(X_{n+1})/n^{2}\allowbreak\approx
\sum_{n=1}^{\infty}Cn^{-(2-\alpha)},
\]
while for the second
\[
\sum_{n=1}^{\infty}\frac{1}{n}\sqrt{\operatorname*{var}(X_{n+1}%
)\operatorname*{var}(\frac{1}{n}\sum_{i=1}^{n}X_{i})}\allowbreak\approx
\sum_{n=1}^{\infty}\frac{C}{n^{1-\alpha/2+1/2-\alpha/2}}.
\]
The first of these series converges, when $2-\alpha>1$, or $\alpha<1$, while
the second series,when $3/2-\alpha>1$, i.e. when $\alpha<1/2.$

In of the case \emph{ii) }we have:
\begin{align*}
E\left\vert \bar{X}_{n}\right\vert ^{2}  &  =\frac{1}{n^{2}}E(\sum_{i=1}%
^{n}X_{i})^{2}\\
&  \leq\allowbreak\frac{1}{n^{2}}(nC+\allowbreak2C\frac{1}{2}(n-1)+\allowbreak
2C\frac{1}{4}(n-2)+\allowbreak\ldots\allowbreak+2C\frac{1}{2^{n-1}%
}(n-(n-1)))\allowbreak\\
&  \leq\frac{3C}{n}.
\end{align*}
Hence $\sqrt{E\left\vert \bar{X}_{n}\right\vert ^{2}}\approx O(1/\sqrt{n})$.
Thus, the series $\sum_{i\geq1}\frac{1}{n+1}\sqrt{E\left\vert \bar{X}%
_{n}\right\vert ^{2}}\sqrt{E\left\vert X_{n+1}^{2}\right\vert }$ converges.
\end{example}

\begin{remark}
Analyzing the above mentioned example, it is easy to notice, that the strong
law of large numbers is satisfied also by the following sequences $\left\{
X_{i}\right\}  _{i\geq1}$ of dependent random variables:
\[
EX_{i}=0\,;i\geq1,~\left\vert \operatorname*{cov}(X_{i},X_{j})\right\vert
\leq\allowbreak\frac{C}{\eta(\left\vert i-j\right\vert )};i,j\geq1,
\]
where the function $\eta$ satisfies the condition:\emph{\ }series $\sum
_{n\geq1}\frac{1}{n^{3/2}}\sqrt{\sum_{i=1}^{n}\frac{1}{\eta(i)}}$ is
convergent. It entails, that $E\bar{X}_{n}^{2}\leq\frac{C}{n}\sum_{i=1}%
^{n}\frac{1}{\eta(i)}$. It is also not difficult to notice, that functions
$\eta$ satisfying these conditions are e.g. $\eta(x)=\left\vert x\right\vert
^{\beta};\beta>0,\,\eta(x)=\log^{1+\gamma}(1+|x|);\gamma>0$, and so
on.\emph{\ }Below we will generalize this example. Random variables satisfying
the above mentioned conditions will be called \emph{quasi-stationary.}
\end{remark}

Theorem \ref{mpwl} gives possibility of getting the laws of large numbers for
dependent random variables. Conditions for \textquotedblright
dependence\textquotedblright\ can be expressed in terms of the covariances of
the random variables $\{X_{n}\}_{n\geq1}$. Theorem \ref{mpwl} provides quick,
\textquotedblright easy to apply\textquotedblright\ convergence criteria. In
order to present more subtle ones for SLLN to be satisfied by dependent random
variables in concise form, one has to\emph{\ assume} \emph{more }about the
mutual dependence of the\emph{\ }elements of the sequence $\{X_{n}\}_{n\geq1}%
$. Hence, let us consider an important class of the random variables, namely
the so-called \emph{quasi-stationary }random variables.
\index{variables!quasi-stationary}%

\begin{definition}
Random variables $\{X_{n}\}_{n\geq1}$ are \emph{quasi-stationary, } if
\emph{i) }$\forall n\geq1$: $E\left\vert X_{n}\right\vert ^{2}<\infty$,
\emph{ii)} $\forall i,j\geq1:$%
\[
\left\vert \operatorname{cov}(X_{i},X_{j})\right\vert \allowbreak
\leq\allowbreak\rho(\left\vert i-j\right\vert )\allowbreak\sqrt
{\operatorname*{var}(X_{i})\operatorname*{var}(X_{j})},
\]
for some sequence $\left\{  \rho(i)\right\}  _{i\geq0}$ nonnegative numbers
such, that $\rho(0)=1$.
\end{definition}

For such sequences we will give a generalization of Lemma
\ref{fundamental_inequality}, and also Rademacher-Menshoff's Theorem
\ref{menshov}. Let us suppose, that $\{X_{n}\}_{n\geq1}$ is the sequence of
\emph{quasi-stationary }random variables \emph{\ }with zero expectations, such
that $\operatorname*{var}(X_{i})=\sigma_{i}^{2}$, $\left\vert
\operatorname{cov}(X_{i},X_{j})\right\vert \leq\rho(\left\vert i-j\right\vert
)\sigma_{i}\sigma_{j}$. Let us introduce also the following denotations:
\[
\mathbf{R}_{n}=[r_{ij}]_{1\leq i,j\leq n}=\left[  \rho(|i-j|)\right]  _{1\leq
i,j\leq n},\;\mathbf{\xi}_{n}^{T}=[\beta_{1}\sigma_{1},\ldots,\beta_{n}%
\sigma_{n}].
\]

We have the following generalization of Lemma \ref{fundamental_inequality}.

\begin{lemma}
\label{uog_fundamental}Let us denote $S_{i}=\sum_{j=1}^{i}\beta_{j}X_{j}$
:$i=1,2,\ldots,n$. We have
\[
E\underset{1\leq i\leq n}{\max}S_{i}^{2}=M(n)\leq O\left(  \log^{2}n\right)
\left(  \sum_{i=0}^{n}\rho(i)\right)  \sum_{i=1}^{n}\beta_{i}^{2}\sigma
_{i}^{2}.
\]

\end{lemma}

\begin{proof}
Let $\nu(\omega)$ will be the smallest index such that
\[
\underset{1\leq i\leq n}{\max}S_{i}^{2}=S_{\nu}^{2}.
\]
Repeating part of arguments and calculations from the proof of Lemma
\ref{fundamental_inequality} we get:
\[
S_{\nu}^{2}\leq\left(  \sum_{j=1}^{\nu}A_{\nu-j}^{-1/2}A_{j}^{-1/2}\left\vert
S_{j}^{-1/2}\right\vert \right)  ^{2}\leq\sum_{j=1}^{\nu}\left(  A_{\nu
-j}^{-1/2}\right)  ^{2}\sum_{j=1}^{n}\left(  A_{j}^{-1/2}\left\vert
S_{j}^{-1/2}\right\vert \right)  ^{2}.
\]%
\begin{equation}
\sum_{j=1}^{\nu}\left(  A_{\nu-j}^{-1/2}\right)  ^{2}=1+\sum_{j=1}^{\nu
}O(\frac{1}{\nu-j})=O(\ln\nu)\leq O(\ln n). \label{suma_A}%
\end{equation}%
\begin{align*}
ES_{k}^{2}  &  =\left(  \sum_{i=1}^{n}\beta_{i}^{2}\sigma_{i}^{2}+2\sum_{1\leq
i<j\leq n}\beta_{i}\beta_{j}\operatorname{cov}(X_{i},X_{j})\right)  \leq\\
&  \leq\mathbf{\xi}_{n}^{T}\mathbf{R}_{n}\mathbf{\xi}_{n}\allowbreak
\leq\allowbreak\lambda_{n}\mathbf{\xi}_{n}^{T}\mathbf{\xi}_{n},
\end{align*}
where $\lambda_{n}$ denotes the greatest eigenvalue of matrix $\mathbf{R}_{n}$
(its spectral norm). Since, that spectral norm does not exceed any other
matrix norm (compare theorem. 6.1.3 in \cite{Lankaster69}) we have:
\[
ES_{k}^{2}\leq2\left(  \sum_{i=0}^{n-1}\rho(i)\right)  \sum_{i=1}^{n}\beta
_{i}^{2}\sigma_{i}^{2},
\]
since we have
\[
\left\Vert R_{n}\right\Vert =\underset{1\leq i\leq n}{\max}\sum_{j=1}%
^{n}r_{ij}=\underset{1\leq i\leq n}{\max}\left(  2\sum_{j=0}^{i-1}\rho
(j)+\sum_{j=i}^{n-1}\rho(j)\right)  \leq2\sum_{j=0}^{n-1}\rho(j).
\]

Moreover, :
\begin{align*}
E\sum_{j=1}^{n}\left(  A_{j}^{-1/2}\left\vert S_{j}^{-1/2}\right\vert \right)
^{2}  &  =\sum_{j=1}^{n}E\left(  A_{j}^{-1/2}S_{j}^{-1/2}\right)  ^{2}=\\
&  =\sum_{j=1}^{n}E\left(  \sum_{k=1}^{j}A_{j-k}^{-1/2}\beta_{k}X_{k}\right)
^{2}\leq\\
&  \leq2\sum_{j=1}^{n}\sum_{i=0}^{j-1}\rho(i)\sum_{k=1}^{j}\beta_{k}^{2}%
\sigma_{k}^{2}\left(  A_{j-k}^{-1/2}\right)  ^{2}\leq\\
&  \leq2\sum_{k=0}^{n}\rho(k)\sum_{k=1}^{n}\beta_{k}^{2}\sigma_{k}^{2}\left(
O\left(  1\right)  +O\left(  \frac{1}{2}\right)  +\ldots+O\left(  \frac{1}%
{k}\right)  \right)  \leq\\
&  \leq O(\log n)\sum_{k=0}^{n}\rho(k)\sum_{k=1}^{n}\beta_{k}^{2}\sigma
_{k}^{2}.
\end{align*}
Combining (\ref{suma_A}) and the above mentioned result we get assertion of
our lemma.
\end{proof}

Having this lemma it is easy to get Theorem similar to Rademacher-Menchoff's Theorem.

\begin{theorem}
\label{uog_RM}Let $\{X_{n}\}_{n\geq1}$ be a sequence of quasi-orthogonal
random variables such that $\operatorname*{var}(X_{i})=\sigma_{i}^{2}$,
$\left\vert \operatorname{cov}(X_{i},X_{j})\right\vert \leq\rho(\left\vert
i-j\right\vert )\sigma_{i}\sigma_{j}$. If series
\[
\sum_{i\geq1}\beta_{i}^{2}\sigma_{i}^{2}\log^{2}i\sum_{j=0}^{i}\rho(j)
\]
is convergent, then series $\sum_{i\geq1}\beta_{i}X_{i}$ is convergent with
probability $1.$
\end{theorem}

Using this theorem and making use of Lemma \ref{podstawowy} one can easily get
the following result.

\begin{corollary}
Let $\{X_{n}\}_{n\geq1}$ be a sequence of quasi-orthogonal random variables
such that $\operatorname*{var}(X_{i})=\sigma_{i}^{2}$, $\left\vert
\operatorname{cov}(X_{i},X_{j})\right\vert \leq\rho(\left\vert i-j\right\vert
)\sigma_{i}\sigma_{j}$. If series
\[
\sum_{i\geq1}\frac{\alpha_{i-1}^{2}}{\left(  \sum_{j=0}^{i-1}\alpha
_{j}\right)  ^{2}}\sigma_{i}^{2}\log^{2}i\sum_{j=0}^{i}\rho(j)
\]
is convergent, then the sequence $\left\{  \frac{\sum_{i=1}^{n}\alpha
_{i-1}X_{i}}{\sum_{i=1}^{n}\alpha_{i-1}}\right\}  _{n\geq1}$ is convergent
with probability 1 to zero.
\end{corollary}

Also from Theorem \ref{uog_RM} it is easy also to get the following facts.

\begin{proposition}
Let $\{X_{n}\}_{n\geq1}$ be a sequence of quasi-orthogonal random variables
such that $\operatorname*{var}(X_{i})=\sigma_{i}^{2}$, $\left\vert
\operatorname{cov}(X_{i},X_{j})\right\vert \leq\rho(\left\vert i-j\right\vert
)\sigma_{i}\sigma_{j}$. Let $\left\{  \beta_{i}\right\}  _{i\geq1}$ be a
number sequence such that convergent is series:
\[
\sum_{i\geq1}\beta_{i}^{2}\sigma_{i}^{2}.
\]
Then with probability $1$ we have $\sum_{i=1}^{n}\beta_{i}X_{i}=O(\log
n\sqrt{\sum_{i=1}^{n}\rho(i)}).$
\end{proposition}

Further from this Proposition follows the following one:

\begin{proposition}
Let $\{X_{n}\}_{n\geq1}$ be a sequence of quasi-orthogonal random variables
such that $\operatorname*{var}(X_{i})=\sigma_{i}^{2}$, $\left\vert
\operatorname{cov}(X_{i},X_{j})\right\vert \leq\rho(\left\vert i-j\right\vert
)\sigma_{i}\sigma_{j}$. If the series
\begin{equation}
\sum_{i\geq1}\frac{\sigma_{i}^{2}}{i^{2}} \label{war_zb}%
\end{equation}
is convergent, then with probability $1$%
\[
\underset{n\rightarrow\infty}{\lim}\frac{\sum_{i=1}^{n}\left(  X_{i}%
-EX_{i}\right)  }{n\log(n+1)\sqrt{\sum_{i=0}^{n}\rho(i)}}=0.
\]

\end{proposition}

This proposition was proved by other methods in the of paper \cite{Moricz85c}.
In this paper there is also another proof of Lemma \ref{uog_fundamental}, and
what is more construction of such sequence $\{X_{n}\}_{n\geq1}$ orthogonal
random variables, for which condition (\ref{war_zb}) is satisfied, and
\[
\underset{n\rightarrow\infty}{\lim\inf}\,\frac{1}{n\lambda_{n}}\left\vert
\sum_{i=1}^{n}\left(  X_{i}-EX_{i}\right)  _{n\geq1}\right\vert \allowbreak
=\allowbreak\infty
\]
with probability 1 for every sequence $\left\{  \lambda_{n}\right\}  $ such,
that $\underset{n\rightarrow\infty}{\lim}\lambda_{n}/\log n=0.$

SLLN for correlated random variables will be illustrated by the following example:

\begin{example}
Let sequences of the random variables $\left\{  \xi_{i}\right\}  _{i\geq1}$be
the solutions of difference equations of the form
\[
\xi_{i+1}=\sum_{j=0}^{q-1}\gamma_{j}\xi_{i-j}+\zeta_{i+1};i\geq0,
\]
where the sequence $\left\{  \zeta_{i}\right\}  _{i\geq1}$ consists of
uncorrelated random variables having zero means and identical finite variances
and the values $\xi_{-q+1},\allowbreak\ldots\allowbreak,\xi_{0}$ are given.
Sequence of such solutions are called autoregressive time series of order
$q$\emph{\
\index{Sequence!autoregresive}%
} briefly AR($q$)- sequence. If additionally solutions of the algebraic
equation
\[
x^{q}-\gamma_{0}x^{q-1}-\ldots-\gamma_{q-1}=0
\]
lie inside unit circle the complex plane, then respective time series is
called \emph{stationary}.
\end{example}

One considered stationary $\mathbf{AR}(2)-$sequence $\left\{  \xi_{i}\right\}
_{i\geq1}$ with zero mean

$3500$ observations were made. Values of averages $\left\{  S_{i}\right\}
$with selected numbers were the following: $S_{100}\allowbreak=\allowbreak
0,179;$ $S_{500}\allowbreak=\allowbreak0,02;$ $S_{1000}\allowbreak
=\allowbreak0,033;$ $S_{2000}\allowbreak=\allowbreak-0,004;$ $S_{3500}%
\allowbreak=\allowbreak-0,003$

Variability of those averages we will illustrate by the following plots%

%TCIMACRO{\FRAME{itbpF}{2.0678in}{1.2626in}{0in}{}{}{pwl_ar.eps}%
%width 2.0678in;  height 1.2626in;  depth 0in;  original-width 17.311in;
%original-height 10.5421in;  cropleft "0";  croptop "1";  cropright "1";
%cropbottom "0";  filename 'pwl_ar.eps';file-properties "XNPEU";}}}%
%BeginExpansion
{\includegraphics[
height=1.2626in,
width=2.0678in
]%
{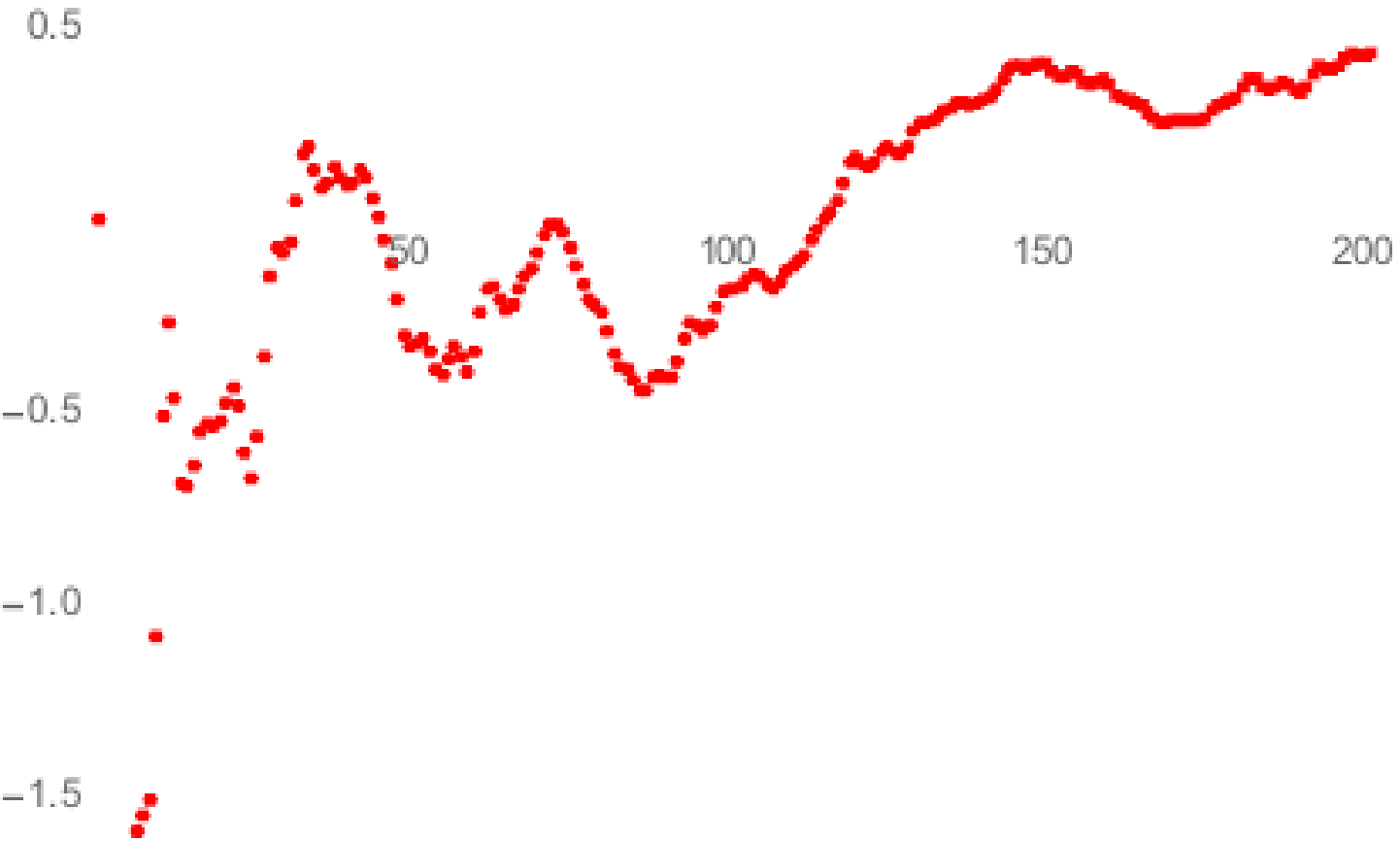}%
}%
%EndExpansion%
%TCIMACRO{\FRAME{itbpF}{1.8127in}{1.2574in}{0in}{}{}{pwl_ar2.eps}%
%{\special{ language "Scientific Word";  type "GRAPHIC";  display "USEDEF";
%valid_file "F";  width 1.8127in;  height 1.2574in;  depth 0in;
%original-width 17.0048in;  original-height 10.5421in;  cropleft "0";
%filename 'pwl_ar2.eps';file-properties "XNPEU";}}}%
%BeginExpansion
{\includegraphics[
height=1.2574in,
width=1.8127in
]%
{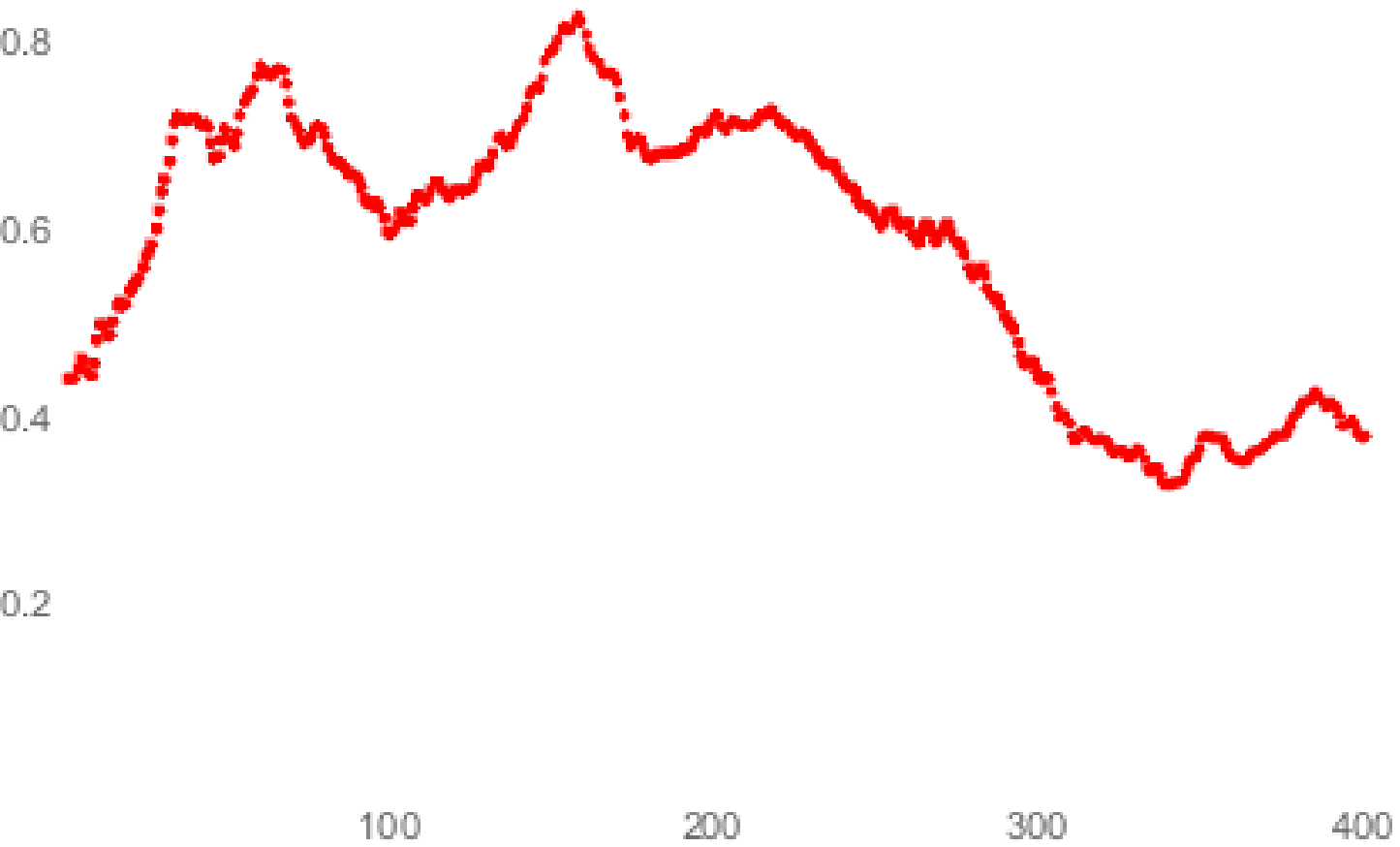}%
}%
%EndExpansion
%

%TCIMACRO{\FRAME{itbpF}{2.0643in}{1.2004in}{0in}{}{}{pwl_ar3.eps}%
%width 2.0643in;  height 1.2004in;  depth 0in;  original-width 18.1178in;
%original-height 10.4988in;  cropleft "0";  croptop "1";  cropright "1";
%cropbottom "0";  filename 'pwl_ar3.eps';file-properties "XNPEU";}}}%
%BeginExpansion
{\includegraphics[
height=1.2004in,
width=2.0643in
]%
{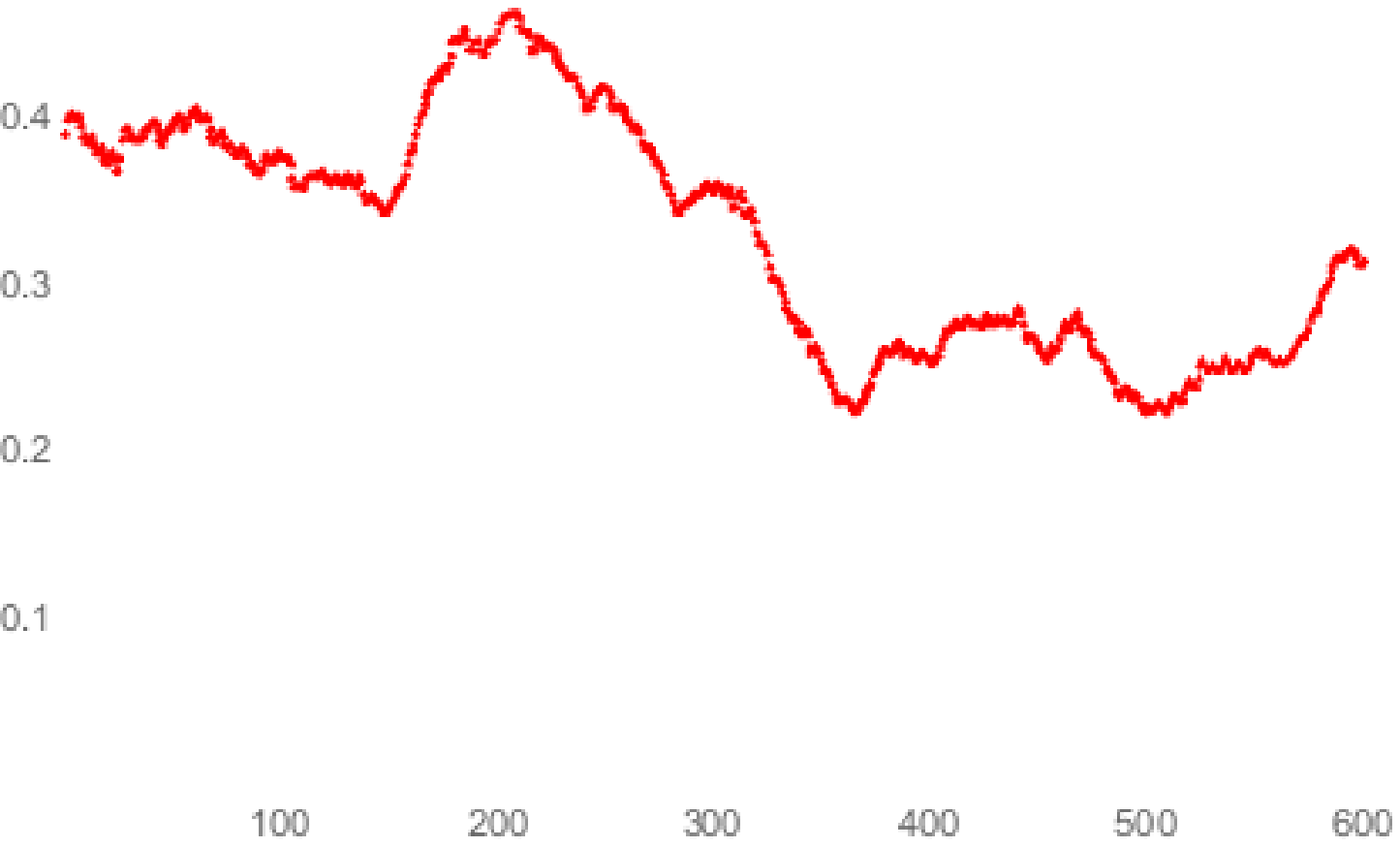}%
}%
%EndExpansion%
%TCIMACRO{\FRAME{itbpF}{1.8092in}{1.0516in}{0in}{}{}{pwl_ar3.eps}%
%width 1.8092in;  height 1.0516in;  depth 0in;  original-width 18.1178in;
%original-height 10.4988in;  cropleft "0";  croptop "1";  cropright "1";
%cropbottom "0";  filename 'pwl_ar3.eps';file-properties "XNPEU";}}}%
%BeginExpansion
{\includegraphics[
height=1.0516in,
width=1.8092in
]%
{pwl_ar3.eps}%
}%
%EndExpansion
%

%TCIMACRO{\FRAME{itbpF}{2.0954in}{1.1882in}{0in}{}{}{pwl_ar4.eps}%
%width 2.0954in;  height 1.1882in;  depth 0in;  original-width 18.6618in;
%original-height 10.5421in;  cropleft "0";  croptop "1";  cropright "1";
%cropbottom "0";  filename 'pwl_ar4.eps';file-properties "XNPEU";}}}%
%BeginExpansion
{\includegraphics[
height=1.1882in,
width=2.0954in
]%
{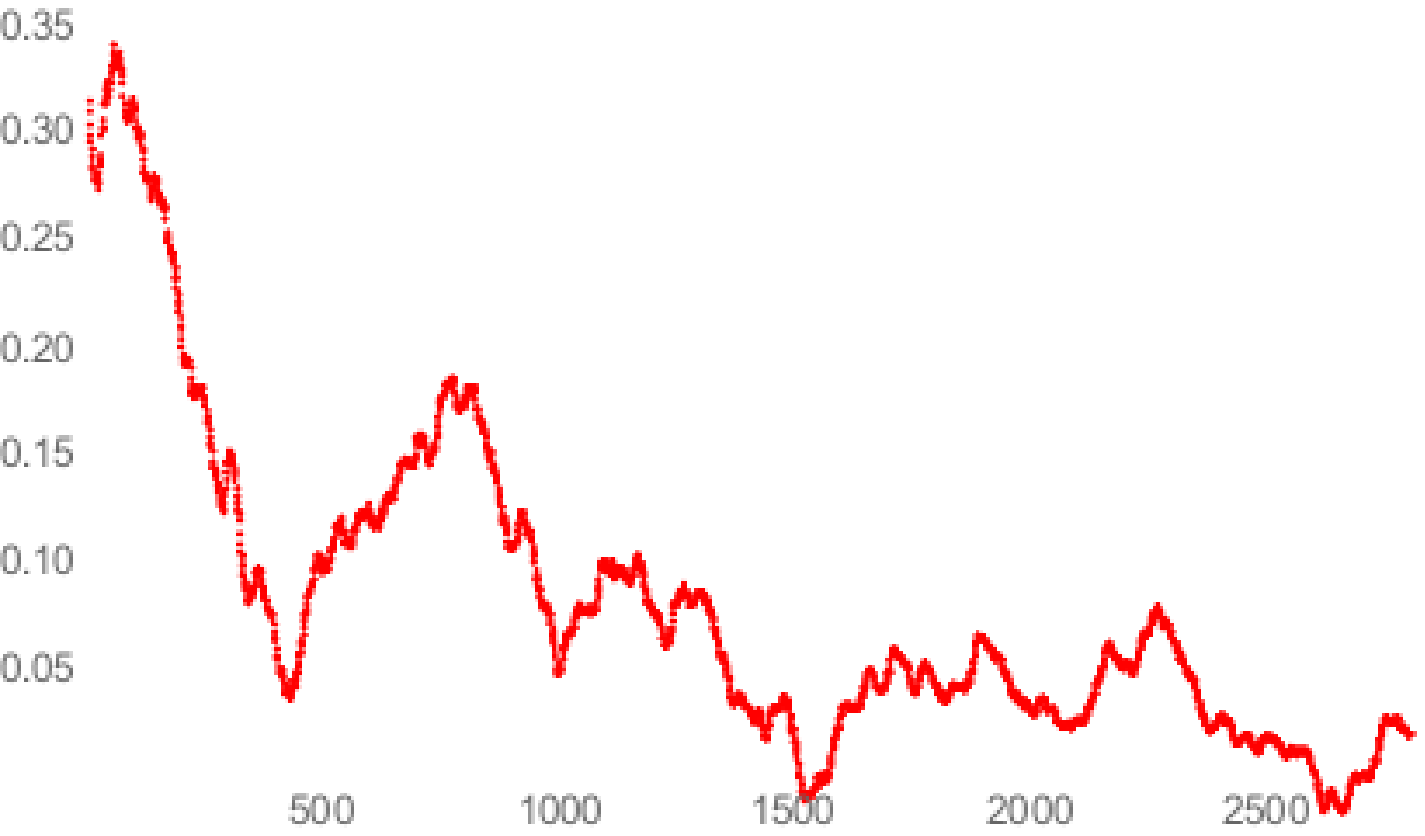}%
}%
%EndExpansion

\subsubsection{Global Central Limit Theorem Almost Surely}

As the second, not a typical example of application of SLLN we will discuss
result presented in the paper \cite{Szablowski972} concerning the so-called
local and global central limit theorems almost surely. This result is a
generalization of results from papers \cite{Brosamler88}, \cite{Csaki93},
\cite{Schatte91}. These papers concern phenomena noticed by Brosamler in the
first of these papers. We mean the so-called local central limit theorem
almost surely%
\index{Theorem!Brosamler}
(LCTGAS). More precisely, let $\left\{  X_{i}\right\}  _{i\geq1}$ be a
sequence independent random variables having identical distributions. Let us
denote
\[
S_{k}=\sum_{i=1}^{k}X_{i},k=1,2,\ldots.
\]
Let us select two sequences of real numbers $\left\{  \alpha_{i}\right\}
_{i\geq1}$ and $\left\{  \beta_{i}\right\}  _{i\geq1}$ such that $\alpha
_{i}\leq\beta_{i}$ ; $i\geq1$. Let $p_{k}=P(\alpha_{k}\leq S_{k}<\beta_{k})$.
Let us set:
\[
\eta_{k}=\left\{
\begin{array}
[c]{lll}%
\frac{I(\alpha_{k}\leq S_{k}<\beta_{k})}{p_{k}}, & gdy & p_{k}\neq0\\
1, & gdy & p_{k}=0
\end{array}
\right.  ,
\]
where $I(A)$ is the characteristic function of the event $A$. It turns out
that selecting proper assumptions concerning random variables $\left\{
X_{i}\right\}  _{i\geq0}$, and also sequences $\left\{  \alpha_{i}\right\}  $
and $\left\{  \beta_{i}\right\}  $, we observe convergence:
\begin{equation}
\frac{1}{\ln n}\sum_{k=1}^{n}\frac{\eta_{k}}{k}\underset{n\rightarrow
\infty}{\longrightarrow}1 \label{local_ctg_as}%
\end{equation}
with probability $1$. Brosamler first noticed\ this phenomenon for sequences
$\left\{  X_{i}\right\}  _{i\geq0}$ having second moments and sequences
$\left\{  \alpha_{i}\right\}  $ and $\left\{  \beta_{i}\right\}  $ of the
following of the form: $\alpha_{k}=-\infty,\beta_{k}=x\sigma\sqrt{k},k\geq1$,
where $x$ is any real number, $\sigma^{2}$ is a variance of variable $X_{1}$.
More precisely, in Brosamler's Theorem the following convergence:
\begin{equation}
\frac{1}{\ln n}\sum_{k=1}^{n}I(S_{k}\leq x\sigma\sqrt{k}%
)\underset{n\rightarrow\infty}{\longrightarrow}\Phi(x), \label{brosamler}%
\end{equation}
was proved with probability 1, here $\Phi(x)$ is distribution Normal $N(0,1)$.
However remembering that on the basis of CLT (see \ref{sek_ctg}) sequence
$\left\{  p_{k}\right\}  $ converges in this case just to $\Phi(x)$ one can
notice clear connection between convergence (\ref{local_ctg_as}) a
(\ref{brosamler}).

The phenomenon shown in (\ref{brosamler}) was called global central limit
theorem almost surely (GCTGAS). During following years one generalized and
improved this result. In particular, one considered conditions, under which we
have convergence (\ref{local_ctg_as}). The result of Cs\`{a}ki, F\"{o}ldes and
R\'{e}v\'{e}sz of the paper \cite{Csaki93} concerns just this type of
convergence. In the paper \cite{Szablowski972} this result has been
generalized, some other of similar type results has been proved estimating
also the speed with which these convergences happen. More precisely, one was
able to find an increasing number sequence $\left\{  \gamma_{k}\right\}  $
such that the sequence
\[
\gamma_{n}\frac{1}{\ln n}\sum_{k=1}^{n}\frac{\eta_{k}-1}{k}%
\]
still converges to zero with probability $1$. In obtaining this result one
used Lemma \ref{prostypomocniczy}. The result of the paper
\cite{Szablowski972} is the following:

\begin{theorem}
Let sequence $\left\{  X_{i}\right\}  _{i\geq0}$ be a sequence independent
random variables having identical distributions such that $EX_{1}=0$,
$EX_{1}^{2}=\sigma^{2}$, $E\left\vert X_{1}\right\vert ^{3}<\infty$. Moreover,
let us suppose, that \newline either \textit{1)} distribution random variable
$X_{1}$ has bounded density and the following conditions are satisfied by
sequences $\left\{  \alpha_{k}\right\}  $, $\left\{  \beta_{k}\right\}  $ and
$\left\{  p_{k}\right\}  :$%
\begin{align*}
\beta_{k}-\alpha_{k}  &  \leq ck,\text{where }c\text{ is some constant,}\\
\sum_{k=1,p_{k}\neq0}^{n}\frac{1}{k^{2}p_{k}}  &  =O(\ln n),
\end{align*}
or: \textit{2)} sequences $\left\{  \alpha_{k}\right\}  $, $\left\{  \beta
_{k}\right\}  $ and $\left\{  p_{k}\right\}  $ satisfy the conditions:
\begin{align*}
\beta_{k}-\alpha_{k}  &  \leq c\sqrt{k},\text{where }c\text{ is some
constant,}\\
\sum_{k=1,p_{k}\neq0}^{n}\frac{\ln k}{k^{3/2}p_{k}}  &  =O(\ln n),
\end{align*}
then:
\end{theorem}

\begin{example}
\begin{theorem}
\textit{a)} series $\sum_{i\geq1}\frac{1}{i\ln_{+}i}(\eta_{i}-1)$ converges
with probability $1$, \newline$(\ln_{+}x\allowbreak\overset{df}{=}%
\allowbreak\max(\log x,1)\allowbreak:x>0)$

\textit{b) }$\frac{1}{A_{n}}\sum_{i=0}^{n}a_{i}\eta_{i+1}%
\underset{n\rightarrow\infty}{\longrightarrow}1$ with probability $1$, where
we denoted:
\[
a_{i}=\frac{\ln_{+}^{\nu}(i+1)}{i+1}:i=0,1,\ldots,
\]
for $\nu>-1$ and $\frac{1}{A_{n}}\sum_{i=0}^{n}a_{i}\ln_{+}\ln_{+}%
(i+1)\eta_{i+1}\underset{n\rightarrow\infty\infty}{\longrightarrow}1$ for
$\nu=-1$, and $A_{n}=\sum_{i=0}^{n-1}a_{i}.$

\textit{c) }$\frac{1}{n}\sum_{i=1}^{n}\frac{\eta_{i}-1}{\ln_{+}i}%
\underset{n\rightarrow\infty}{\longrightarrow}0$ with probability $1$.

If additionally we assume, that the sequence $\left\{  p_{k}\right\}  $ is
such that $\underset{k\,\longrightarrow\infty}{\lim\inf}\,\frac{1}{p_{k}%
}<\infty$, then

\textit{d) }with probability $1$ for $\gamma>\frac{2}{3}:$%
\[
\frac{\ln_{+}^{1/4}n}{\ln_{+}^{\gamma}\ln_{+}n}\frac{1}{\ln_{+}n}\sum
_{k=1}^{n}\frac{\eta_{k}-1}{k}\underset{n\rightarrow\infty}{\longrightarrow
}0,
\]
and Moreover, series
\[
\sum_{i\geq1}\frac{1}{i\ln_{+}^{3/4}i\ln_{+}\ln_{+}^{\gamma}i}(\eta_{i}-1)
\]
converges with probability $1$.
\end{theorem}
\end{example}

\begin{proof}
The proof uses two fundamental facts taken from the paper \cite{Csaki93}.
Namely, by assumptions of this theorem we have the following estimation:
\begin{align}
\operatorname*{var}(\eta_{n})  &  \cong\frac{1}{p_{n}}\label{oszac0}\\
\operatorname{var}(\frac{1}{\ln_{+}n}\sum_{i=1}^{n}\frac{\eta_{i}}{i})  &
\approx O(\frac{1}{\ln_{+}n}),\label{oszac1}\\
\operatorname{var}(\sum_{i=1}^{n}\frac{\eta_{i}-1}{i})  &  \approx O(\ln
_{+}n). \label{oszac2}%
\end{align}
We will prove assertion $d)$ first. Let us denote
\[
Z_{n}=\frac{\ln_{+}^{1/4}n}{\ln_{+}^{\gamma}\ln_{+}n}\frac{1}{\ln_{+}n}%
\sum_{k=1}^{n}\frac{\eta_{k}-1}{k}=\frac{1}{\ln_{+}^{3/4}n\ln_{+}^{\gamma}%
\ln_{+}n}\sum_{k=1}^{n}\frac{\eta_{k}-1}{k}.
\]
It is easy to check, that the sequence $\left\{  Z_{n}\right\}  _{n\geq1}$
satisfies for large $n$ recurrent relationship
\[
Z_{n+1}=(1-\frac{1}{2(n+1)\ln_{+}(n+1)}+v_{n+1})Z_{n}+\frac{\eta_{n+1}%
-1}{\left(  n+1\right)  \ln_{+}^{3/4}n\ln_{+}^{\gamma}\ln_{+}n},
\]
where $v_{n}=o(\frac{1}{n\ln_{+}n})$. On the base of (\ref{oszac2}) we deduce
that
\begin{equation}
\operatorname*{var}(Z_{n})\cong O(\frac{1}{\ln_{+}^{1/2}n\ln_{+}^{2\gamma}%
\ln_{+}n}). \label{oszac3}%
\end{equation}
Hence the series
\[
\sum_{n\geq1}\frac{1}{\left(  n+1\right)  \ln_{+}(n+1)}Z_{n}^{2}%
\]
converges with probability $1$. Moreover, it is easy to notice, that the
series
\[
\sum_{n\geq1}\frac{\left(  \eta_{n+1}-1\right)  ^{2}}{\left(  \left(
n+1\right)  \ln_{+}^{3/4}n\ln_{+}^{\gamma}\ln_{+}n\right)  ^{2}}%
\]
converges with probability $1$. Thus, on the base of Lemma
\ref{prostypomocniczy} we deduce that the sequence $\left\{  Z_{n}\right\}
_{n\geq1}$ converges with probability $1$ to zero if and only if, series
\[
\sum_{n\geq1}\frac{\eta_{n+1}-1}{\left(  n+1\right)  \ln_{+}^{3/4}n\ln
_{+}^{\gamma}\ln_{+}n}Z_{n}%
\]
converges with probability $1$. We have however been using (\ref{oszac0}) and
(\ref{oszac3}):
\begin{align*}
&  \sum_{n\geq1}\frac{\sqrt{\operatorname*{var}\left(  \eta_{n+1}-1\right)  }%
}{\left(  n+1\right)  \ln_{+}^{3/4}n\ln_{+}^{\gamma}\ln_{+}n}\sqrt
{\operatorname*{var}\left(  Z_{n}\right)  }\\
&  \leq\sum_{n\geq1}\frac{\sqrt{1/p_{n+1}}}{\left(  n+1\right)  \ln_{+}%
^{3/4}n\ln_{+}^{\gamma}\ln_{+}n}\sqrt{O(\frac{1}{\ln_{+}^{1/2}n\ln
_{+}^{2\gamma}\ln_{+}n})}<\infty.
\end{align*}
Hence $Z_{n}\underset{n\rightarrow\infty}{\rightarrow}0$ with probability $1$.

In order to get the second part of assertion $d)$ let us notice that the
series
\[
\sum_{n\geq1}\frac{1}{\left(  n+1\right)  \ln_{+}(n+1)}\sqrt{EZ_{n}^{2}}%
\]
converges, since we have (\ref{oszac3}). Hence, it converges with probability
$1$ together with the series
\[
\sum_{n\geq1}\frac{1}{\left(  n+1\right)  \ln_{+}(n+1)}Z_{n}.
\]
Now we apply Lemma \ref{podstawowy} and infer, that the series $\sum_{i\geq
1}\frac{1}{i\ln_{+}^{3/4}i\ln_{+}\ln_{+}^{\gamma}i}(\eta_{i}-1)$ converges
almost surely.

In order to prove assertion $a)$, $b)$ and $c)$ we use the main result of the
paper \cite{Csaki93}. Namely, with assumptions theorems it follows that the
sequence $\left\{  \frac{1}{\ln_{+}n}\sum_{i=1}^{n}\frac{\eta_{i}-1}%
{i}\right\}  _{n\geq1}$ converges to zero with probability $1$. Taking into
account (\ref{oszac1}) we deduce that the series
\begin{equation}
\sum_{n\geq1}\frac{1}{n\ln_{+}n}\frac{1}{\ln_{+}n}\sum_{i=1}^{n}\frac{\eta
_{i}-1}{i} \label{szereg}%
\end{equation}
converges with probability $1$. Let us denote $B_{i}=\sum_{j=0}^{i-1}\frac
{1}{j+1}$. It is known, that $\ln_{+}n-B_{n}\cong0.577$ for large $n$.
Denoting
\[
\alpha_{n}=\allowbreak\frac{1}{n+1},\;\mu_{n}=\allowbreak\frac{1}%
{(n+1)B_{n+1}},\,Y_{n}=\allowbreak\eta_{n}-1,\overline{Y}_{n}=\frac{1}%
{B_{n-1}}\sum_{i=1}^{n}\frac{Y_{i}}{i},
\]
we see that $\overline{Y}_{n}\underset{n\rightarrow\infty}{\rightarrow}0$
a.s., and series $\sum_{n\geq1}\mu_{n}\overline{Y}_{n}$ converges a.s. hence
on the basis of Lemma \ref{podstawowy} we deduce that also the series
$\sum_{n\geq1}\mu_{n}Y_{n+1}$ converges with probability $1$. \newline It
remained to show, that series the $\sum_{n\geq1}\left\vert \mu_{n}-\frac
{1}{\left(  n+1\right)  \ln_{+}(n+1)}\right\vert \left\vert Y_{n+1}\right\vert
$ converges almost surely. It is however very easy, since
\[
\left\vert \mu_{n}-\frac{1}{\left(  n+1\right)  \ln_{+}(n+1)}\right\vert
\cong\frac{const}{\left(  n+1\right)  \ln_{+}^{2}(n+1)},
\]
and $E\left\vert Y_{n}\right\vert \leq1+E\left\vert \eta_{n}\right\vert =2.$

The idea of the proof of assertions $b)$ and $c)$ is very similar and we will
not present those proofs with all the details. Let us denote $A_{i}=\sum
_{j=1}^{i}a_{i-1}\cong\frac{\ln_{+}^{\nu+1}i}{\nu+1}$ for $\nu>-1$ and
$\ln_{+}\ln_{+}i$ for $\nu=-1$. Hence, for $\nu>-1$ $\mu_{i}\allowbreak
\cong\allowbreak\frac{\nu+1}{\left(  i+1\right)  \ln_{+}(i+1)}$ and $\mu
_{i}\allowbreak\cong\allowbreak\frac{1}{\left(  i+1\right)  \ln_{+}%
(i+1)\ln_{+}\ln_{+}(i+1)}$. On the base of already proved assertion $a)$ we
deduce that series
\[
\sum_{i\geq1}\mu_{i}Y_{i+1}%
\]
is convergent almost surely for all $\nu\geq-1$. Hence, on the basis of Lemma
\ref{podstawowy} we deduce, from the convergence to zero of the sequence
$\overline{Y}_{n}$, that in this case takes the form described in assertion
$b)$. Similarly, on the basis of the proven assertion $a)$ changing definition
of the random variables $Y_{i}=\frac{\eta_{i}-1}{\ln_{+}i}$ and elements of
the sequence $a_{i}=1$ and $\mu_{i}=\frac{1}{i+1}$ we deduce convergence to
zero of the sequence $\overline{Y}_{n}=\frac{1}{n}\sum_{j=1}^{i}Y_{i}$ i.e. we
have assertion $c).$
\end{proof}

\section{Monte Carlo methods \label{MonteCarlo}}

\subsection{Monte Carlo methods}

Let us start with an example. Suppose, that we want to estimate values of
some, complicated integrals over the composite area. More precisely, let us
suppose, that we are interested in calculating
\[
\underset{V}{\int}f(\mathbf{x)}d\mathbf{x}\overset{df}{=}I,
\]
where $V$ is some bounded subset of $%
%TCIMACRO{\U{211d} }%
%BeginExpansion
\mathbb{R}
%EndExpansion
^{d}$. Suppose further, that we can to find such number $a>0$, that
$V\subset<0,a>^{d}\overset{df}{=}B$. Let $\mathbf{X}$ will be $d-$dimensional
random variable having uniform distribution on $B$. Let us consider random
variable $Y=f(\mathbf{X})I(\mathbf{X}\in V)$. Notice that $EY=\frac{1}{a^{d}%
}I$. Let us generate sequence of independent observations random vector
$\mathbf{X}$, that is $\left\{  \mathbf{X}_{i}\right\}  _{i\geq1}$. This
sequence in a natural way generates a sequence of the random variables
$\left\{  Y_{i}\right\}  _{i\geq1}$. The assumptions of Theorem \ref{kolmogor}
are satisfied and we can deduce, that $\frac{1}{n}\sum_{i=1}^{n}%
Y_{i}\underset{n\rightarrow\infty}{\longrightarrow}\frac{1}{a^{d}}I$ with
probability 1. Since, also assumptions of proposition \ref{ctg_iid} are
satisfied we can e.g. estimate necessary number of observations to ensure a
given accuracy with probability not less than any given beforehand number.
\newline As a concrete example let us consider $d=1$, $f(x)=\sqrt{1-x^{4}}$,
$V=<0,1>$. In other words, we want to estimate integral $I=\int_{0}^{1}%
\sqrt{1-x^{4}}dx$. To do so, we generate sequence of independent observations
of variables $X_{i}{\sim U(0;1)}$ and consider sequence $\left\{  Y_{i}%
=\sqrt{1-X_{i}^{4}}\right\}  $. Of course, we have $EY_{1}=I$. Now one has to
find the minimal number $n$ for which condition:
\begin{equation}
P(\left\vert \frac{1}{n}\sum_{i=1}^{n}Y_{i}-I\right\vert \leq0,01)\geq0.98,
\label{w-k}%
\end{equation}
is satisfied. Using CTG we get:
\[
P(\left\vert \frac{1}{n}\sum_{i=1}^{n}Y_{i}-I\right\vert \leq0,01)\approx
2\Phi(\frac{0.1\sqrt{n}}{\sqrt{V}}),
\]
where $V$ is here the variance of the random variable $Y$, and $\Phi
(x)=\frac{1}{\sqrt{2\pi}}\int_{0}^{x}\exp(-\frac{t^{2}}{2})dt$ is the
so-called Laplace function. Because of the condition (\ref{w-k}) we have
$\sqrt{n}\geq100\sqrt{V}\Phi^{-1}(0,49)\approx233\sqrt{V}$. Let us estimate
now the quantity $V$. We have $V=EY^{2}-I^{2}=\int_{0}^{1}(1-x^{4})dx-I^{2}$.
Let us notice that for $x\in<0,1>$ we have $\sqrt{1-x^{4}}\geq\sqrt{1-x^{2}}$
and consequently we see that $I\geq\int_{0}^{1}\sqrt{1-x^{2}}dx=\allowbreak
\frac{1}{4}\pi$. Hence, $V\leq\frac{4}{5}-\left(  \frac{1}{4}\pi\right)
^{2}=0,\,18315$. Hence, $n\geq233^{2}\ast0\allowbreak,\,18315=\allowbreak
9943,03$. In other words it is enough perform $n\allowbreak=\allowbreak9944$
observations of the random variables $Y$ (trivial task, if it is to performed
on today's computer), in order to be sure with probability not less than $.98$
that quantity $\frac{1}{n}\sum_{i=1}^{n}Y_{i}$ approximates unknown integral
$\int_{0}^{1}\sqrt{1-x^{4}}dx$ with accuracy not greater than $0,01$.

Let us pay attention to the following features of the above-mentioned example:

\begin{enumerate}
\item versions of SLLN and of CTG used in the above-mentioned example were
very simple

\item potential complications and difficulties were connected with:

\begin{enumerate}
\item generating sequences of independent random variables having identical
distributions uniform on $<0,1>$,

\item generating sequences of independent random variables $\left\{
Y_{i}\right\}  $, whose expectations we would like to estimate.
\end{enumerate}
\end{enumerate}

Mentioned above features characterizes the majority of tasks of Monte Carlo
method, that is estimating values of unknown quantities (most often in the
form of expectations of some random variables) with the help of computer
simulations. Similar features can be found in typical problems of stochastic
optimization. That is to say finding minima of functions of the form
$g(y)=EF(y,X)$.

As it was mentioned before strictly probabilistic part of such tasks is rather
simple and typical. Usually it concerns the application of simple versions of
laws of large numbers and central theorems limit (point 1.). Difficulties in
this type of tasks are connected with the use of good and efficient generator
of pseudo-random number generator (point 2.a.) and possibly by setting the
problem that has translated the usually deterministic problem into the
probabilistic language (point 2.b.).

It is worth to mention, that estimated quantities can have a form of solutions
of the system of deterministic equations (generally nonlinear) or finding
extreme values of some functions or functionals.

Similar features one finds also in problems of stochastic optimization and
parametric estimation. Generally speaking and also simplifying, one has to
find zeros of maxima of the functions $g(y)=EF(y,X)$ in the situation when one
cannot observe values of functions $g$, but only values $F(y,X_{i})$ for any
$y$ and $i$ where $\{X_{n}\}_{n\geq1}$ is the sequence independent random
variables having identical, known distributions. If we give up the recursive
form of such problem, then for fixed $y$ we observe the sequence of values
$\left\{  F(y,X_{i})\right\}  _{i=1}^{n}$. Quantity $\frac{1}{n}\sum_{i=1}%
^{n}F(y,X_{i})$ approximates $g(y)$ with accuracy depending on $n$. There
exist numerical methods, that allow deducing where approximately lies zero of
examined function by knowing its approximate values. Similarly, in the case of
seeking a maximum of some function whose values cannot be observed directly.
It is now enough to apply these methods for observed approximations of values
$g(y_{k})$, $k=1,\ldots,m$.

Difficulties here are connected with the choice of the right numerical method
and not with the probabilistic model of this problem. It is here very simple.
A detailed presentation of this type of applications would lead us too far in
numerical methods.

In the next chapter we will consider similar tasks, but for the more
complicated, not a typical version of the probabilistic model. Namely, we will
consider iterative (or recurrent) versions. Or using other words, we will
assume additionally that for the given point $y$ one can observe only finite
(often equal to one) number of values of the sequence $\left\{  F(y,X_{i}%
)\right\}  $. If additionally, one would depart from requirements of identity
of distributions and independence of elements of the sequence \thinspace
\thinspace$\{X_{n}\}_{n\geq1}$, then we have precisely considered below the
problem of stochastic approximation, or having specially chosen function $F$
the problem of density estimation considered in the next chapter.

\chapter{Stochastic approximation\label{aproksymacja}}

\section{Introduction}

Stochastic approximation concerns the following problems. Let us assume that
there is given a function $f$ (as yet of one variable) $f:%
%TCIMACRO{\U{211d} }%
%BeginExpansion
\mathbb{R}
%EndExpansion
\mathbb{\rightarrow}%
%TCIMACRO{\U{211d} }%
%BeginExpansion
\mathbb{R}
%EndExpansion
$, not necessarily continuous, but such, that:
\[
\exists\theta\in%
%TCIMACRO{\U{211d} }%
%BeginExpansion
\mathbb{R}
%EndExpansion
:\,\left(  \forall x>\theta\,:f(x)>0\,(f(x)<0)\right)  \&\left(  \forall
x<\theta:f(x)<0\,(f(x)>0)\right)  ,
\]
or other words, on the left and on the right from some point $\theta$ the
function has different signs. Let us further suppose, that values of functions
$f$ are not observed straightforwardly. More precisely, every observed value
of these functions is burdened with some random error. In other words, for
every point $x$ we observe the quantity
\[
y_{i}(x)=f(x)+\eta_{i}(x),i\geq1,
\]
where $\eta_{i}(x)$ is a random variable such that $\forall x$ $E\eta
_{i}(x)=0$. \ Notice that its distribution may depend on $x$. The aim of
stochastic approximation procedures is to find point $\theta$, using only the
observed values \newline$\left\{  y_{i};\allowbreak i\geq1\right\}  .$

Stochastic approximation procedures are based on the following idea. Suppose,
that in $n$ -theorem step we have some estimator $x_{n}$ of the point $\theta$
and let us assume, that the function $f$ is positive to the right of $\theta$
and negative to the left. If it happens, that the observed value at this
(estimated so far), point $x_{n}$ is less than zero, then we increase
estimator a bit (more precisely by $\mu_{n}y_{n+1}(x_{n})$, where $\mu_{n}\in%
%TCIMACRO{\U{211d} }%
%BeginExpansion
\mathbb{R}
%EndExpansion
^{+}$ and generally\ $\mu_{n}<1)$, if however the observed value is greater
than zero, then the estimator will be decreased a bit (more precisely by
$-\mu_{n}y_{n+1}(x_{n})$, where $\mu_{n}\in%
%TCIMACRO{\U{211d} }%
%BeginExpansion
\mathbb{R}
%EndExpansion
^{+}$ and generally\ $\mu_{n}<1)$.

In other words, considered algorithm can be presented in the following way:
\begin{equation}
x_{n+1}=x_{n}-\mu_{n}y_{n+1}(x_{n}). \label{aprstoch1}%
\end{equation}
In the present chapter we will examine if and if so then, how quickly this
procedure converges do $\theta.$

There exist stochastic approximation procedures concerned, so to say, with the
problem of minimization of functions in random conditions. We will discuss
such procedures in subsection \ref{rozszerzenia}.

Let us see in some examples, that indeed procedures (\ref{aprstoch1}) are convergent.

\begin{example}
In the first example the function, whose zero is sought, it is the function
$f(x)=(x-3)\exp(-.1(x-3))$\newline%

%TCIMACRO{\FRAME{itbpF}{2.5287in}{1.6847in}{0in}{}{}{ocs4di00.eps}%
%width 2.5287in;  height 1.6847in;  depth 0in;  original-width 2.9853in;
%original-height 1.9804in;  cropleft "0";  croptop "1";  cropright "1";
%cropbottom "0";  filename 'OCS4DI00.eps';file-properties "XNPEU";}}}%
%BeginExpansion
{\includegraphics[
height=1.6847in,
width=2.5287in
]%
{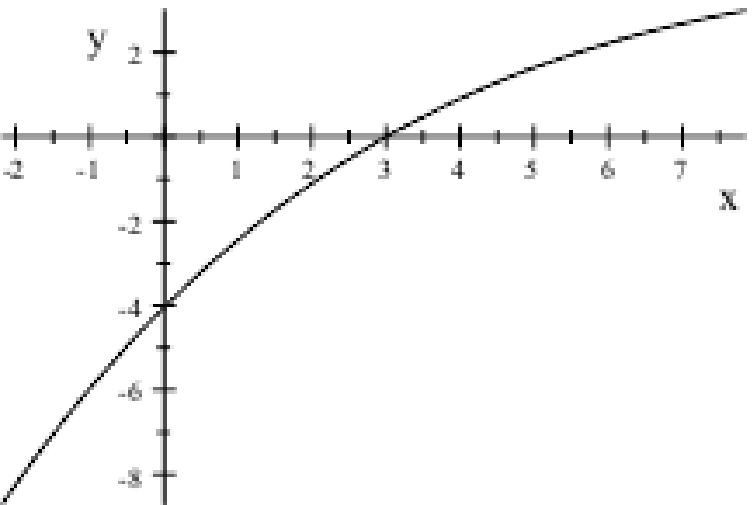}%
}%
%EndExpansion

having the plot presented above. One made $N=5000$ observations $\xi
_{1},\ldots,\xi_{N}$ of the random variables having Normal $\mathcal{N}(0,4)$
distribution and one considered procedure of the form
\begin{equation}
y_{i}=y_{i-1}-\frac{1}{i}\left(  f(y_{i-1})+\xi_{i}\right)  ;\,\,\,\,y_{0}=0.
\label{przykl-proc}%
\end{equation}
As the result of its operation, we got $y_{N}=2,9232$. The course of
iterations had the following plot:%

%TCIMACRO{\FRAME{itbpF}{3.1548in}{1.8957in}{0in}{}{}{aprox1.eps}%
%width 3.1548in;  height 1.8957in;  depth 0in;  original-width 5.5962in;
%original-height 3.3529in;  cropleft "0";  croptop "1";  cropright "1";
%cropbottom "0";  filename 'aprox1.eps';file-properties "XNPEU";}}}%
%BeginExpansion
{\includegraphics[
height=1.8957in,
width=3.1548in
]%
{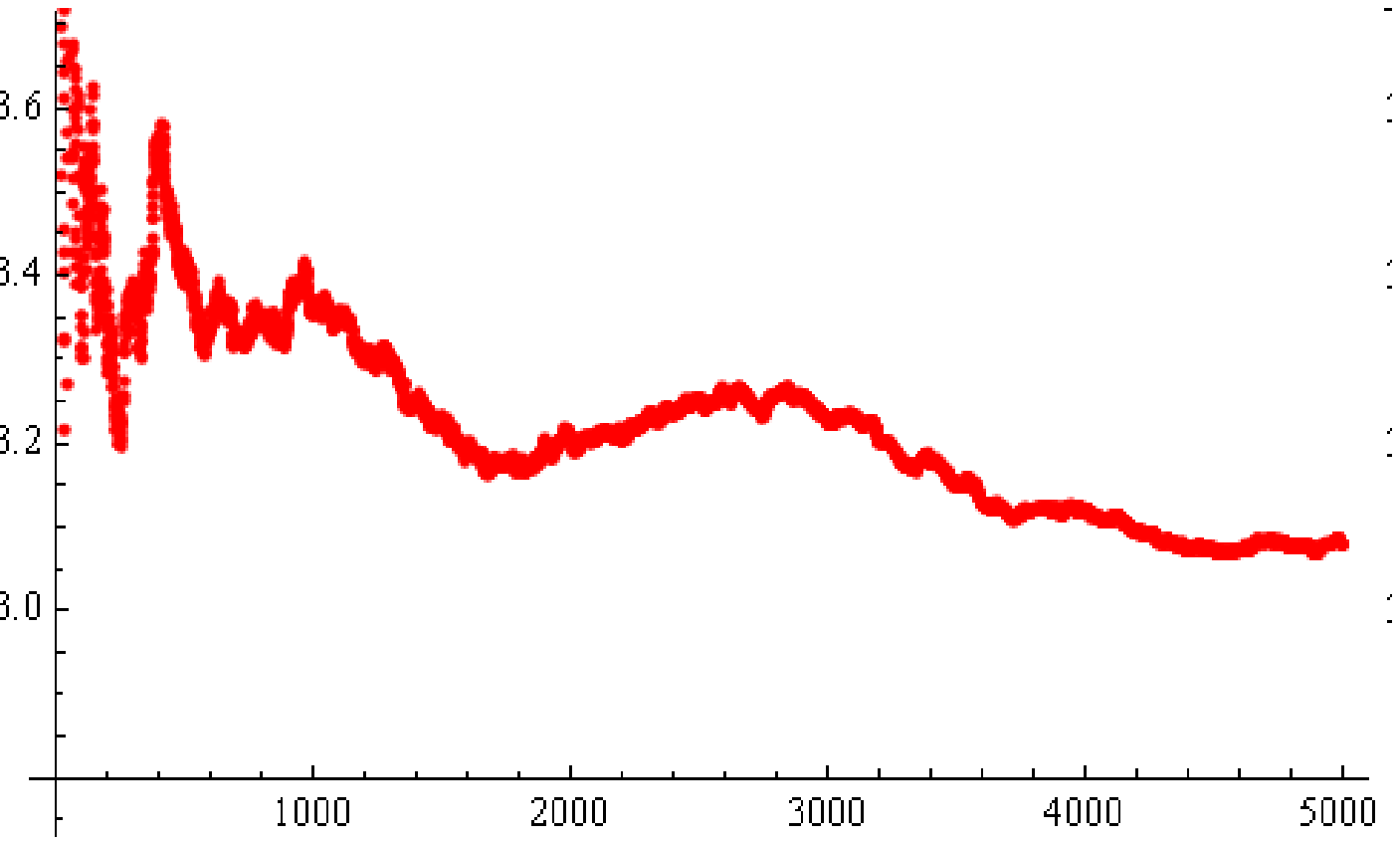}%
}%
%EndExpansion
.
\end{example}

\begin{example}
\label{przyklad_z_mala_funkcja}In the second example, one considered similar
function $f$, namely $f(x)=(x-3)\exp(-(x-3))$\newline%

%TCIMACRO{\FRAME{itbpF}{3.0139in}{2.0038in}{0in}{}{}{ob48hh03.eps}%
%width 3.0139in;  height 2.0038in;  depth 0in;  original-width 2.9698in;
%original-height 1.9649in;  cropleft "0";  croptop "1";  cropright "1";
%cropbottom "0";  filename 'OB48HH03.eps';file-properties "XNPEU";}}}%
%BeginExpansion
{\includegraphics[
height=2.0038in,
width=3.0139in
]%
{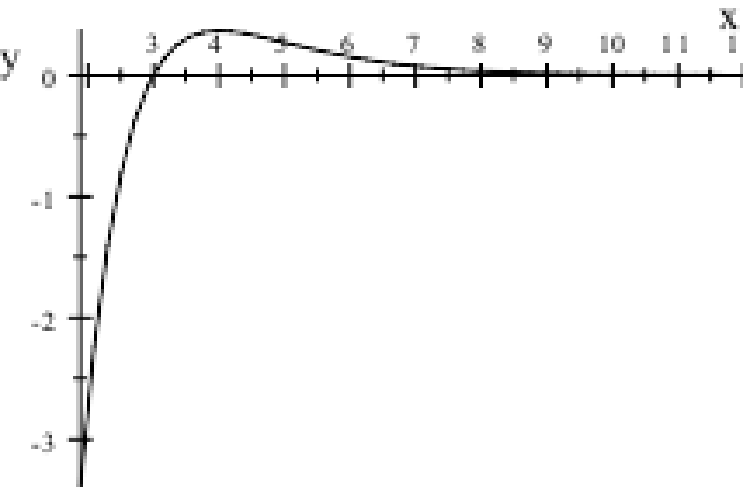}%
}%
%EndExpansion

\end{example}

having a plot as above. Similarly, as before, one took $N=5000$ observations
of the random variables $\xi_{1},\ldots,\xi_{N}$ having Normal $\mathcal{N}%
(0,4)$ distributions and one considered procedure \ref{przykl-proc} with
initial condition $y_{0}=1$. As the result of operating this procedure one got
$x_{N}=9,27$, and the course of iterations was\ the following:%

%TCIMACRO{\FRAME{itbpF}{2.9853in}{1.8784in}{0in}{}{}{aprox2.eps}%
%width 2.9853in;  height 1.8784in;  depth 0in;  original-width 6.2111in;
%original-height 3.8977in;  cropleft "0";  croptop "1";  cropright "1";
%cropbottom "0";  filename 'aprox2.eps';file-properties "XNPEU";}}}%
%BeginExpansion
{\includegraphics[
height=1.8784in,
width=2.9853in
]%
{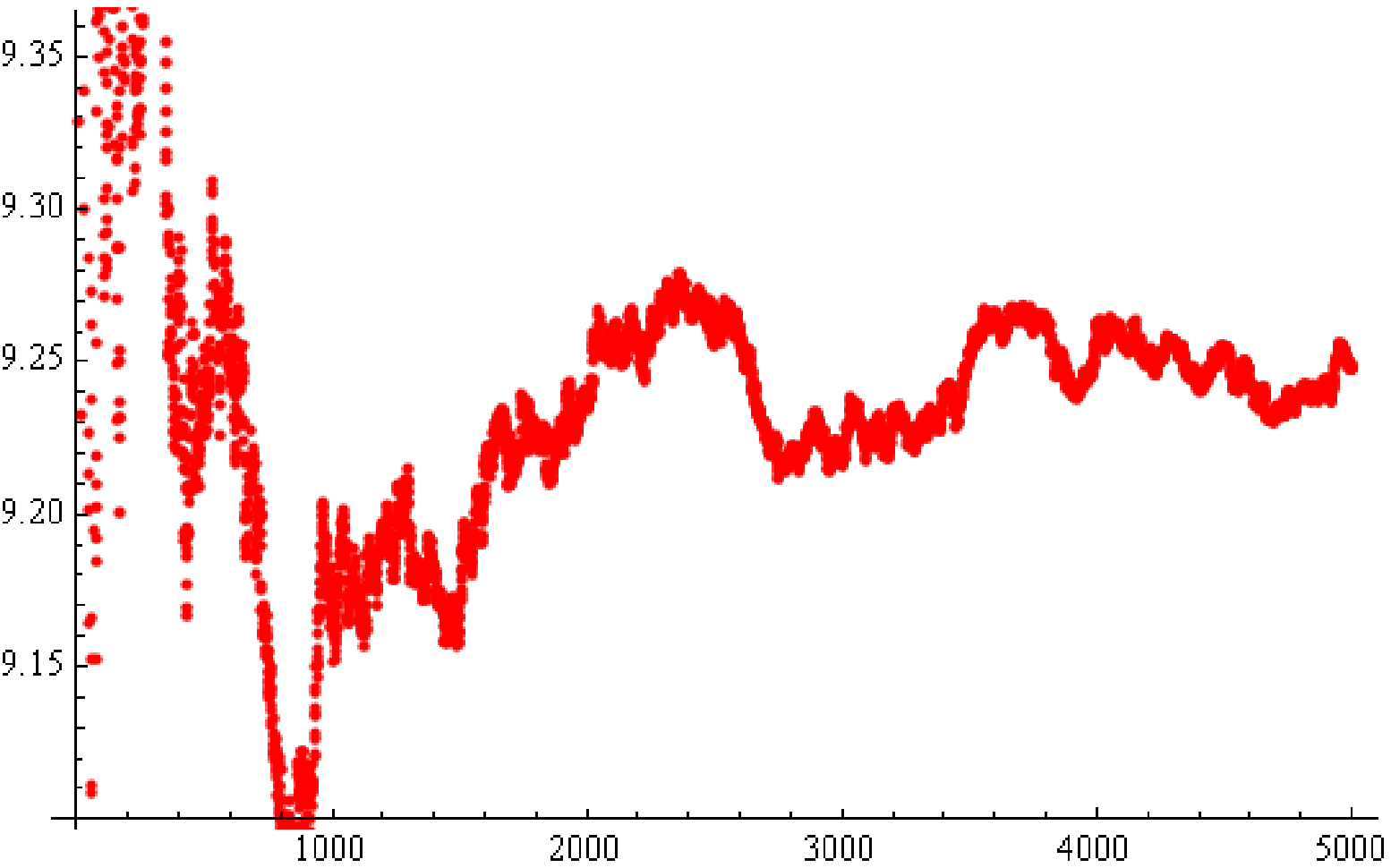}%
}%
%EndExpansion
.

\begin{example}
Why don't we observe convergence here (or in fact we observe very slow
convergence) and what is to be done in order to improve this convergence. It
will follow from the presented in the sequel mathematical analysis of the
stochastic approximation procedures.
\end{example}

\begin{example}
\label{kwantyl}In the next example, we will seek quantiles of distribution on
the basis of observations of the random sample drawn from this distribution.
This example is different from the previous ones in that now random
disturbances of the function values (whose zero, we are looking for) will
depend in this case on the values of the estimator. In particular, situation
considered in this example we will look for the $.85$ quantile of the
distribution $\mathcal{N}(0,2)$. To do so, we fix the number of iterations
$N=5000$, next we generate a sequence $\xi_{1},\ldots,\xi_{N}$ of independent
observations from this distribution. Let us define the following function
\[
v(x,z)=\left\{
\begin{array}
[c]{lll}%
1, & gdy & x\leq z\\
0, & gdy & x>z
\end{array}
\right.  .
\]
Let us notice that $Ev(\xi_{1},z)=F_{\xi}(z)$, where $F_{\xi}$ denotes cdf of
the random variable $\xi_{1}$. Hence, one can write
\[
v(\xi_{i},z)-.85=F_{\xi}(z)-.85+\zeta_{i}(z),
\]
where we denoted $\zeta_{i}(z)=v(\xi_{i},z)-F_{\xi}(z)$. The r\^{o}le of
disturbances play in this case random variables $\left\{  \zeta_{i}%
(z)\right\}  $, and $F_{\xi}(z)-.85$ is a function, whose zero is sought. We
have here $\forall z\in%
%TCIMACRO{\U{211d} }%
%BeginExpansion
\mathbb{R}
%EndExpansion
:E\zeta_{i}(z)=0$. We will consider the following procedure
\[
z_{i}=z_{i-1}-\frac{1}{i}\left(  v(\xi_{i},z_{i-1})-.85\right)  ;z_{0}=0.
\]
Let us notice that the sequence of the random variables $\zeta_{i}(z_{i-1})$
is a sequence of martingale differences, i.e. if we denote $\mathcal{G}%
_{i}=\sigma(z_{0},\ldots,z_{i})$, then $E(\zeta_{i}(z_{i-1})|\mathcal{G}%
_{i})=0$ almost surely. Theoretical value the quantile we are looking for is
equal $2.07029$. After $N=5000$ iterations one obtained $z_{N}=1.6145$. Hence,
convergence was very bad. It follows also from the plot illustrating this
example:
%TCIMACRO{\FRAME{dtbpFU}{2.4189in}{1.4019in}{0pt}{\Qcb{.}}{}{aprstoch3.eps}%
%{\special{ language "Scientific Word";  type "GRAPHIC";  display "USEDEF";
%valid_file "F";  width 2.4189in;  height 1.4019in;  depth 0pt;
%original-width 10.9823in;  original-height 13.1434in;  cropleft "0";
%filename 'aprstoch3.eps';file-properties "XNPEU";}}}%
%BeginExpansion
\begin{center}
\includegraphics[
height=1.4019in,
width=2.4189in
]%
{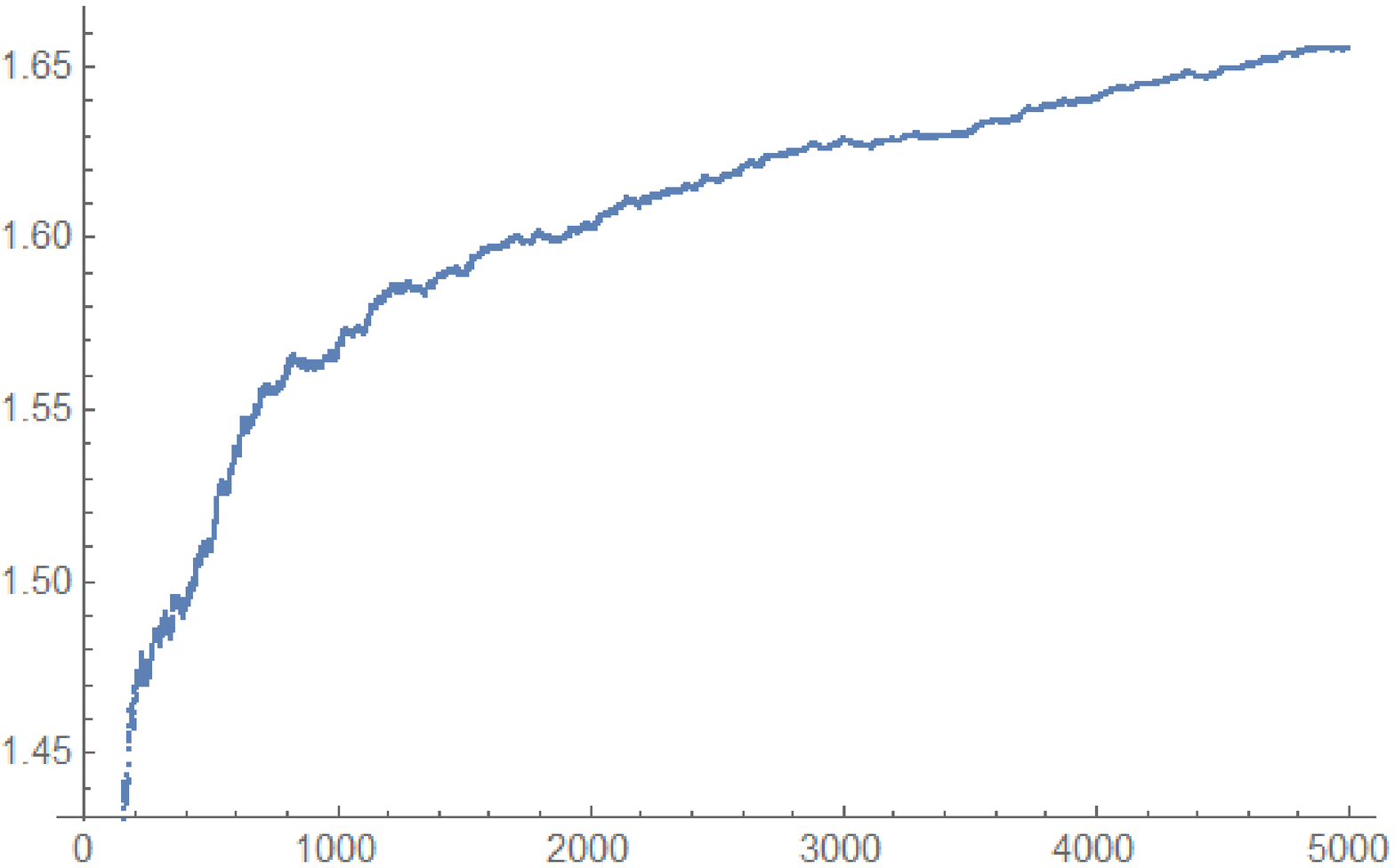}%
\\
.
\end{center}
%EndExpansion
In order to improve the performance of the procedure instead of the $\left\{
\mu_{i}=\frac{1}{i+1}\right\}  $, one took a sequence $\left\{  \mu
_{i}^{^{\prime}}\allowbreak=\allowbreak\frac{1}{\left(  i+1\right)  ^{.75}%
}\right\}  $ and the following procedure was considered:
\[
zx_{i}=zx_{i-1}-\mu_{i}^{^{\prime}}\left(  v(\xi_{i},zx_{i-1})-.85\right)
;zx_{0}=0.
\]
As the plot below shows substantial improvement of the quality of convergence
was observed. In particular, we got $zx_{N}=2.0002$.
%TCIMACRO{\FRAME{dtbpF}{2.8522in}{2.1534in}{0pt}{}{}{aprstoch4.eps}%
%width 2.8522in;  height 2.1534in;  depth 0pt;  original-width 4.3587in;
%original-height 3.2828in;  cropleft "0";  croptop "1";  cropright "1";
%cropbottom "0";  filename 'aprstoch4.eps';file-properties "XNPEU";}}}%
%BeginExpansion
\begin{center}
\includegraphics[
height=2.1534in,
width=2.8522in
]%
{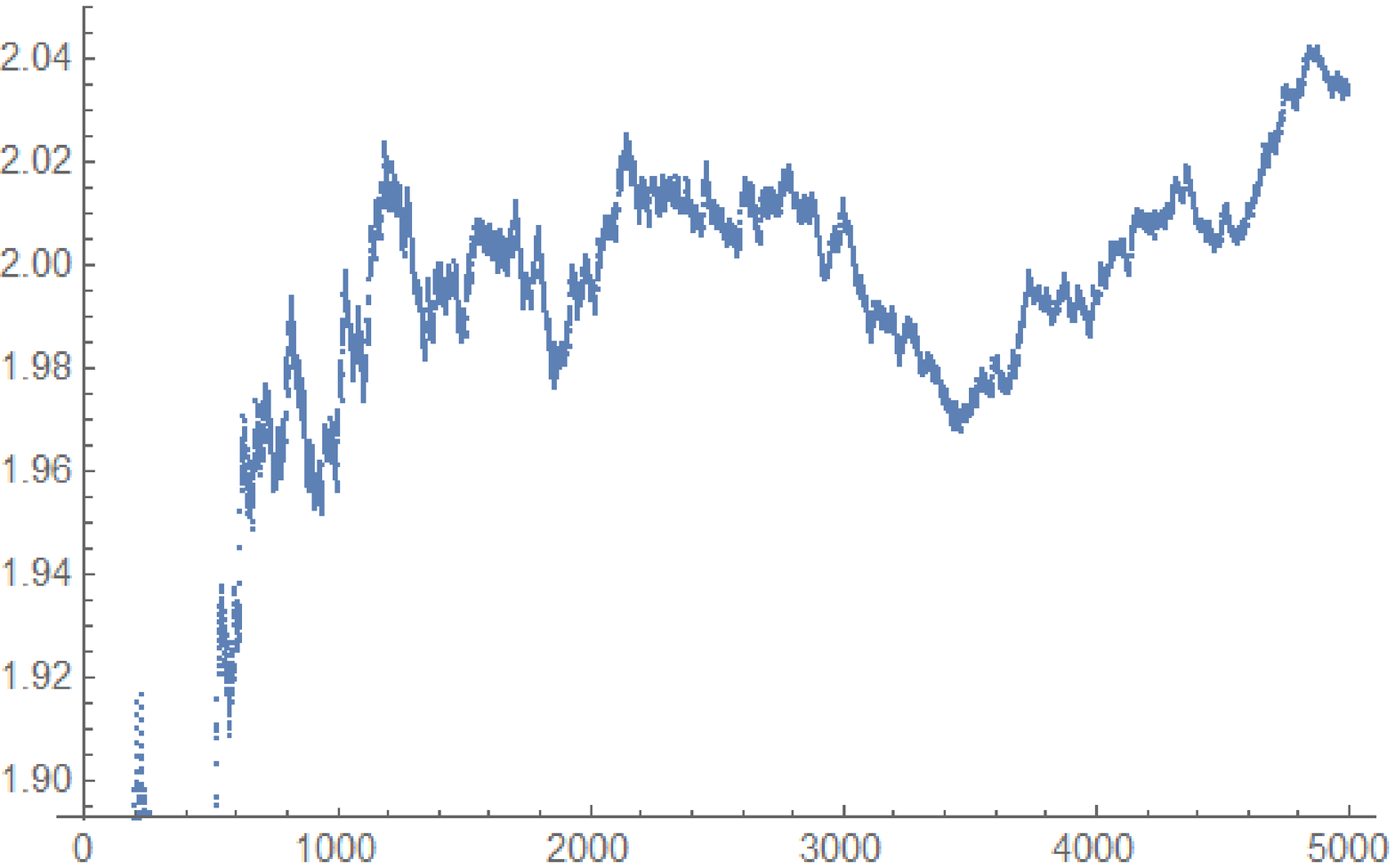}%
\end{center}
%EndExpansion
The stochastic approximation procedure was proposed in 1951 by Robbins and
Monro in the paper \cite{Robbins51} in the simplest version and its
mean-squares convergence was proved. In the next 48 years, the idea personated
in this paper was improved and generalized many times. Moreover, it became an
inspiration and the origin of several branches of applied mathematics. There
exist a few books dedicated to stochastic approximation. One of the eldest is
undoubtedly is the monograph of Nevelson and Chasminskij \cite{Nevelson72}.
There exists also very good monograph of J. Koronacki \cite{Koronacki89} in
Polish dedicated to stochastic approximation and based on it the so-called
stochastic optimization. The approach presented in this monograph differs from
the one followed in this book in the assumptions imposed on the disturbances.
In the monograph of Koronacki most often it is assumed that the disturbances
are martingale differences that (see. definition page. \pageref{martyngaly} in
Appendix \ref{martyngaly}) or are independent.
\end{example}

Indeed it is a very important class of disturbances however the above
mentioned assumption turns out to be unnecessary in many cases. Moreover, it
seems, that approach presented below is more natural (at least for the not
very experienced reader) since it exploits connections of stochastic
approximation with laws of large numbers. In fact, it turns out, that
stochastic approximation is somewhat as a connection of laws of large numbers
with deterministic procedures of finding zeros or minima of functions. These
deterministic problems are discussed in every book on numerical methods like
e.g. \cite{Demidowicz65}\cite{Ralston75}. Problems of finding minima in
different spaces and with different restrictions are discussed e.g. in
monographs \cite{Findeisen77} and \cite{Luenberger73}.

There exists extremely rich literature concerning stochastic approximation and
problems that grew on its ground. In this book, we will present only the main
chain of problems that can be derived from the main idea of Robbins and Monro.
As far as the related problems are concerned, we will refer the reader to the
literature. We hope that after understanding the main ideas the reader will be
able to study all related to stochastic approximation problem without great difficulties.

In the sequel, we will use the following denotations and conventions.

$\mathbf{x}$ denotes vector, usually the column with coordinates $x_{i}$,
$\mathbf{x}^{\prime}$ - its transposition. $\mathbf{x}^{\prime}\mathbf{y}$ is
thus a scalar product of vectors $\mathbf{x}$ and $\mathbf{y}$ i.e. the
quantity $\sum_{i}x_{i}y_{i}$. Let us denote also $\left\vert \mathbf{x}%
\right\vert =\sqrt{\mathbf{x}^{\prime}\mathbf{x}}.$

\section{The simplest version}

Let $\mathbf{f}:%
%TCIMACRO{\U{211d} }%
%BeginExpansion
\mathbb{R}
%EndExpansion
^{m}$ $\rightarrow%
%TCIMACRO{\U{211d} }%
%BeginExpansion
\mathbb{R}
%EndExpansion
^{m}$ will be such function, that:
\begin{align}
{\large \exists\mathbf{\theta}}  &  \in%
%TCIMACRO{\U{211d} }%
%BeginExpansion
\mathbb{R}
%EndExpansion
^{m},\delta>0,\,{\large \forall}\mathbf{x\in}%
%TCIMACRO{\U{211d} }%
%BeginExpansion
\mathbb{R}
%EndExpansion
^{m}:(\mathbf{x-\theta)}^{^{\prime}}\mathbf{f(x)\geq}\delta\left\vert
\mathbf{x-\theta}\right\vert ^{2},\label{war1}\\
{\large \exists\kappa}_{1},\kappa_{2}  &  >0,{\large \forall}\mathbf{x}\in%
%TCIMACRO{\U{211d} }%
%BeginExpansion
\mathbb{R}
%EndExpansion
^{m}:\left\vert \mathbf{f(x)}\right\vert \leq\kappa_{1}\left\vert
\mathbf{x}-\mathbf{\theta}\right\vert +\kappa_{2}. \label{war2}%
\end{align}
Moreover, let $\left\{  \mathbf{\xi}_{i}\right\}  _{i\geq1}$ be a sequence of
random vectors such that the series
\begin{equation}
\sum_{i\geq0}\mu_{i}\mathbf{\xi}_{i+1} \label{sumaszumow}%
\end{equation}
converges almost surely for some normal sequence $\left\{  \mu_{i}\right\}
_{i\geq0}$.

Let us consider procedure of the form:
\begin{equation}
\mathbf{x}_{0}=\mathbf{xo,\,\,x}_{n+1}=\mathbf{x}_{n}-\mu_{n}(\mathbf{f(x}%
_{n})+\mathbf{\xi}_{n+1});n\geq0. \label{apr1}%
\end{equation}

\begin{remark}
Notice that the second of the above-mentioned conditions, i.e. (\ref{war2}),
states, that the function $\mathbf{f}$ 'grows not faster than linearly', or
possibly (if $\kappa_{1}=0)$ does not exceed a constant for large of
$\mathbf{x}$. Whereas the first of these conditions, i.e. (\ref{war1}) states,
that the function $\mathbf{f}$ grows 'at least linearly' for large
$\mathbf{x}$, not excluding the closest neighborhood of the point
$\mathbf{\theta}$. Particularly important in the proof of the below mentioned
theorems will turn out the fact of 'linear estimation' of the function's
$\mathbf{f}$ behavior. If the coordinates of the vector $\mathbf{f}$ are
differentiable, and $J_{\mathbf{f}}$ denotes the Jacobi matrix of mapping
$\mathbf{f}$, then condition (\ref{war1}) implies, that eigenvalues of the
matrix $J_{\mathbf{f}}$ are at a point $\mathbf{\theta}$ all not less than
$\delta.$
\end{remark}

\begin{remark}
Let us recall that the sequence of Riesz's means $\left\{  \bar{X}%
_{n}\right\}  _{n\geq1}$ of the sequence $\{X_{n}\}_{n\geq1}$ with respect to
sequence weights $\left\{  \alpha_{n}\right\}  _{n\geq0}\allowbreak
=\widehat{\allowbreak\left\{  \mu_{n}\right\}  }_{n\geq0}$ can be presented in
the following way:
\[
\bar{X}_{n+1}=(1-\mu_{n})\bar{X}_{n}+\mu_{n}X_{n+1},\;n\geq0\text{ with the
condition }\bar{X}_{0}=0.
\]
It is not difficult to notice, that one can also present the above mentioned
recurrent relationship in another way. Namely,
\begin{equation}
\bar{X}_{n+1}=\bar{X}_{n}-\mu_{n}\left(  \bar{X}_{n}-X_{n+1}\right)
,\;n\geq0\text{ with the condition }\bar{X}_{0}=0. \label{rek_sred}%
\end{equation}
Let us assume that $\forall n\geq1:$ $EX_{n}=\theta$ and let us denote
$\xi_{n}=\theta-X_{n}$, $f(x)=x-\theta$. Recursive equation (\ref{rek_sred})
will now assume the following form:
\[
\bar{X}_{n+1}=\bar{X}_{n}-\mu_{n}\left(  f\left(  \bar{X}_{n}\right)
+\xi_{n+1}\right)  ,\;n\geq0\text{ with the condition }\bar{X}_{0}=0.
\]
Let us recall that phenomenon of convergence of the sequence $\left\{  \bar
{X}_{n}\right\}  _{n\geq1}$ to $\theta$ we called the law of large numbers.
The conditions assuring this are discussed in chapter \ref{simpwl}. This
remark indicates thus strong connections of laws of large numbers with
stochastic approximation. In the light of the above-mentioned observations,
the fact that strong laws of large numbers are satisfied is nothing else but
a.s. convergence of stochastic approximation procedures with linear function
$f$.
\end{remark}

We have the following theorem:

\begin{theorem}
\label{najprostsze}Let us assume that the function $\mathbf{f}$ satisfies
conditions (\ref{war1}), (\ref{war2}) for some point $\mathbf{\theta}\in%
%TCIMACRO{\U{211d} }%
%BeginExpansion
\mathbb{R}
%EndExpansion
^{m}$, while disturbances $\left\{  \mathbf{\xi}_{i}\right\}  _{i\geq1}$
satisfy condition (\ref{sumaszumow}) with some normal sequence $\left\{
\mu_{i}\right\}  _{i\geq0}$. Then procedure (\ref{apr1}) converges almost
surely to the point $\mathbf{\theta.}$
\end{theorem}

\begin{proof}
Let us denote $\mathbf{S}_{n}=\sum_{i\geq n}\mu_{i}\mathbf{\xi}_{i+1}$. It
follows the assumption that the sequence $\left\{  \mathbf{S}_{n}\right\}
_{n\geq1}$ converges almost surely to zero. Moreover, we have $\mathbf{S}%
_{n}=\mathbf{S}_{n+1}+\mu_{n}\mathbf{\xi}_{n+1}$. Let us subtract
$\mathbf{S}_{n+1}+\mathbf{\theta}$ from both sides of procedure (\ref{apr1}).
We get then:
\begin{equation}
\mathbf{x}_{n+1}-\mathbf{\theta-S}_{n+1}=\mathbf{x}_{n}-\mathbf{\theta
}-\mathbf{S}_{n}-\mu_{n}\mathbf{f(x}_{n}). \label{pom1}%
\end{equation}
Let us denote: $d_{n}=\left\vert \mathbf{x}_{n}-\mathbf{\theta-S}%
_{n}\right\vert $. Multiplying both sides of equality (\ref{pom1}) by its
transposition and using definition of $d_{n}$, we get:
\begin{equation}
d_{n+1}^{2}=d_{n}^{2}-2\mu_{n}(\mathbf{x}_{n}-\mathbf{S}_{n}-\mathbf{\theta
)}^{^{\prime}}\mathbf{f(x}_{n})+\mu_{n}^{2}|\mathbf{f(x}_{n})|^{2}.
\label{pom2}%
\end{equation}
Taking advantage assumption (\ref{war1}) we get :
\[
(\mathbf{x}_{n}-\mathbf{S}_{n}-\mathbf{\theta)}^{^{\prime}}\mathbf{f(x}%
_{n})\geq\delta\left\vert \mathbf{x}_{n}-\mathbf{\theta}\right\vert
^{2}-\mathbf{S}_{n}^{^{\prime}}\mathbf{f(x}_{n})=
\]%
\[
=\delta\left\vert \mathbf{x}_{n}-\mathbf{S}_{n}-\mathbf{\theta+S}%
_{n}\right\vert ^{2}-\mathbf{S}_{n}^{^{\prime}}\mathbf{f(x}_{n})=
\]%
\[
=\delta d_{n}^{2}+2\delta(\mathbf{x}_{n}-\mathbf{S}_{n}-\mathbf{\theta
)}^{^{\prime}}\mathbf{S}_{n}+\delta\left\vert \mathbf{S}_{n}\right\vert
^{2}-\mathbf{S}_{n}^{^{\prime}}\mathbf{f(x}_{n})\geq
\]%
\[
\geq\delta d_{n}^{2}-2\delta d_{n}\left\vert \mathbf{S}_{n}\right\vert
+\delta\left\vert \mathbf{S}_{n}\right\vert ^{2}-\mathbf{S}_{n}^{^{\prime}%
}\mathbf{f(x}_{n}).
\]
Now let us notice that $\left\vert \mathbf{x}_{n}-\mathbf{\theta}\right\vert
=\left\vert \mathbf{x}_{n}-\mathbf{S}_{n}-\mathbf{\theta+S}_{n}\right\vert
\leq d_{n}+\left\vert \mathbf{S}_{n}\right\vert $. Hence, we have further
\begin{gather*}
(\mathbf{x}_{n}-\mathbf{S}_{n}-\mathbf{\theta)}^{^{\prime}}\mathbf{f(x}%
_{n})\geq\delta d_{n}^{2}-2\delta d_{n}\eta_{n}+\delta\eta_{n}^{2}-\eta
_{n}(\kappa_{1}d_{n}+\kappa_{1}\eta_{n}+\kappa_{2})=\\
=\delta d_{n}^{2}-d_{n}\eta_{n}(2\delta+\kappa_{1})-\eta_{n}(\kappa_{2}%
+\kappa_{1}\eta_{n}-\delta\eta_{n}),
\end{gather*}
where we denoted $\eta_{n}=\left\vert \mathbf{S}_{n}\right\vert $. Moreover,
we have:
\begin{gather*}
\left\vert \mathbf{f(x}_{n})\right\vert ^{2}\leq(\kappa_{1}\left\vert
\mathbf{x}_{n}-\mathbf{\theta}\right\vert +\kappa_{2})^{2}\leq\\
\leq(\kappa_{1}d_{n}+\kappa_{1}\eta_{n}+\kappa_{2})^{2}\leq3(\kappa_{1}%
^{2}d_{n}^{2}+\kappa_{1}^{2}\eta_{n}^{2}+\kappa_{2}^{2}).
\end{gather*}

Thus, we have recurrent relationship:
\begin{gather*}
d_{n+1}^{2}\leq(1-2\mu_{n}\delta+3\mu_{n}^{2}\kappa_{1}^{2})d_{n}^{2}+2\mu
_{n}d_{n}\eta_{n}\left(  2\delta+\kappa_{1}\right)  +\\
+2\mu_{n}\eta_{n}(\kappa_{2}+\kappa_{1}\eta_{n}-\delta\eta_{n})+3\mu_{n}%
^{2}(\kappa_{1}^{2}\eta_{n}^{2}+\kappa_{2}^{2})\\
\overset{df}{=}(1-2\mu_{n}\delta+3\mu_{n}^{2}\kappa_{1}^{2})d_{n}^{2}+\mu
_{n}d_{n}g_{n}+\mu_{n}h_{n},
\end{gather*}

where
\[
g_{n}=2\eta_{n}(2\delta+\kappa_{1}),h_{n}=2\eta_{n}(\kappa_{2}+\kappa_{1}%
\eta_{n}-\delta\eta_{n})+3\mu_{n}(\kappa_{1}^{2}\eta_{n}^{2}+\kappa_{2}^{2}).
\]
Let us notice that $g_{n}\underset{n\rightarrow\infty}{\longrightarrow}0$ and
$h_{n}\underset{n\rightarrow\infty}{\longrightarrow}0$ almost surely, since
the sequences $\left\{  \eta_{n}\right\}  _{n\geq1}$ and $\left\{  \mu
_{n}\right\}  _{n\geq1}$ converge to zero almost surely. Let us consider the
first $N$ such that the quantity $(1-2\mu_{n}\delta+3\mu_{n}^{2}\kappa_{1}%
^{2})$ is positive. Since, $\mu_{n}\underset{n\rightarrow\infty
}{\longrightarrow}0$, $N$ exists. Let us now examine now, for which $d_{n}$
the following inequality is satisfied:
\begin{gather*}
(1-2\mu_{n}\delta+3\mu_{n}^{2}\kappa_{1}^{2})d_{n}^{2}+\mu_{n}d_{n}g_{n}%
+\mu_{n}h_{n}\leq\\
\leq(1-\mu_{n}\delta+3\mu_{n}^{2}\kappa_{1}^{2})d_{n}^{2}\overset{df}{=}%
\lambda_{n}d_{n}^{2}.
\end{gather*}
Of course it happens, when $\delta d_{n}^{2}-d_{n}g_{n}-h_{n}\geq0$. That is,
when $d_{n}\geq\epsilon_{n}$, where $\epsilon_{n}$ is a positive root of the
equation:
\begin{equation}
\delta x^{2}-xg_{n}-h_{n}=0 \label{rownanie}%
\end{equation}

Since, that $g_{n},\,h_{n}\underset{n\rightarrow\infty}{\longrightarrow}0$
almost surely, $\epsilon_{n}\underset{n\rightarrow\infty}{\longrightarrow}0$
almost surely. Moreover, for $d_{n}\leq\epsilon_{n}$ and for $n\geq N$ we
have:
\begin{gather*}
d_{n+1}^{2}\leq(1-2\mu_{n}\delta+3\mu_{n}^{2}\kappa_{1}^{2})\epsilon_{n}%
^{2}+\mu_{n}\epsilon_{n}g_{n}+\mu_{n}h_{n}=\\
\lambda_{n}\epsilon_{n}^{2}-\mu_{n}\delta\epsilon_{n}^{2}+\mu_{n}\epsilon
_{n}g_{n}+\mu_{n}h_{n}=\lambda_{n}\epsilon_{n}^{2}.
\end{gather*}
Hence, in both cases we have:
\begin{equation}
d_{n+1}^{2}\leq\lambda_{n}\max(d_{n}^{2},\epsilon_{n}^{2}).
\label{oszac_podst}%
\end{equation}
Let us now notice that
\[
\forall k\geq1:\prod_{i=k}^{n}\lambda_{i}\allowbreak\leq\exp(-\sum_{i=k}%
^{n}\mu_{i}(\underset{n\rightarrow\infty}{\delta-2\mu_{i}\kappa_{n}%
^{2})\longrightarrow}0,
\]
since the sequence $\left\{  \mu_{i}\right\}  _{i\geq0}$ is normal. Thus,
assumptions Lemma \ref{zmaxem} are satisfied. We deduce, that $d_{n}%
\underset{n\rightarrow\infty}{\longrightarrow}0$ almost surely, and since
$\mathbf{S}_{n}\underset{n\rightarrow\infty}{\longrightarrow}0$, hence and
$\left\vert \mathbf{x}_{n}-\mathbf{\theta}\right\vert \underset{n\rightarrow
\infty}{\longrightarrow}0$ almost surely.
\end{proof}

\section{Remarks and commentaries}

\begin{remark}
Let us notice that if there are no disturbances, i.e. $\forall n\geq
0:\mathbf{S}_{n}=0$, then we have a procedure:
\begin{equation}
\mathbf{x}_{n+1}-\mathbf{\theta}=\mathbf{x}_{n}-\mathbf{\theta}-\mu
_{n}\mathbf{f(x}_{n}). \label{bezzakl}%
\end{equation}
Estimating similarly as above and using same denotations, we get:
\[
d_{n+1}^{2}\leq(1-2\mu_{n}\delta+3\mu_{n}^{2}\kappa_{1}^{2})d_{n}^{2}+3\mu
_{n}^{2}\kappa_{2}^{2}.
\]
If additionally we assume that the function $\mathbf{f}$ Lipschitz i.e., that
$\kappa_{2}=0$, then we have estimation:
\[
d_{n+1}^{2}\leq(1-2\mu_{n}\delta+3\mu_{n}^{2}\kappa_{1}^{2})d_{n}^{2}.
\]
Iterating this inequality from $i=N$ do $n-1$ we get:
\[
d_{n}^{2}\leq d_{N}^{2}\prod_{i=N}^{n-1}(1-2\mu_{i}\delta+3\mu_{i}^{2}%
\kappa_{1}^{2})\leq d_{N}^{2}\exp(-\sum_{i=N}^{n-1}\mu_{i}(2\delta-3\mu
_{i}\kappa_{1}^{2})).
\]
Hence we can conclude, that in order to ensure convergence of the procedure
(\ref{bezzakl}), the sequence $\left\{  \mu_{n}\right\}  _{i\geq0}$ does not
have to converge to zero. Its best choice is a constant sequence, satisfying
inequality $0<\mu_{i}<\frac{2\delta}{3\kappa_{1}^{2}}$. It confirms known
property of deterministic procedures 'seeking zeros of functions' (comp. e.g.
\cite{Ralston75}).
\end{remark}

\begin{remark}
If, however $\kappa_{2}\neq0$, sequence $\mu_{n}$ must converge to zero, in
order to guarantee convergence of the sequence $d_{n}$ to zero.
\end{remark}

\begin{remark}
Condition (\ref{sumaszumow}) is e.g,. satisfied, when the random variables
$\left\{  \xi_{i}\right\}  _{i\geq1}$ are martingale differences such that
$\underset{n}{\sup}E\xi_{n}^{2}<\infty$, a sequence $\left\{  \mu_{i}\right\}
_{i\geq0}$ is such that $\sum_{i\geq0}\mu_{i}^{2}<\infty$. These are the
typical, appearing in the majority of theorems concerning convergence of
stochastic approximation procedures. Developed in chapter \ref{zbiez} methods
of summing the series of dependent random variables, enable to extend class
sequences $\left\{  \xi_{i}\right\}  _{i\geq1}$and $\left\{  \mu_{i}\right\}
_{i\geq0}.$
\end{remark}

\begin{remark}
\label{mpwl_szumow}Let us notice also, that assumption that the disturbances
$\left\{  \mathbf{\xi}_{i}\right\}  _{i\geq1}$ have to satisfy condition
(\ref{sumaszumow}) can be weakened a little, by subtracting from both sides of
(\ref{apr1}) $\mathbf{\theta}$ and the equation:
\begin{equation}
\mathbf{\zeta}_{n+1}=(1-\mu_{n})\mathbf{\zeta}_{n}+\mu_{n}\mathbf{\xi}_{n+1}.
\label{srednieszumy}%
\end{equation}
It is easy to notice, that $\mathbf{\zeta}_{n}=\frac{\sum_{i=0}^{n-1}%
\alpha_{i}\mathbf{\xi}_{i+1}}{\sum_{i=0}^{n-1}\alpha_{i}};$ $n\geq1$ where
$\overline{\left\{  \alpha_{i}\right\}  }=\left\{  \mu_{i}\right\}  $. Hence,
instead of demanding that instead the series (\ref{sumaszumow}), converges
a.s. we demand that the sequence $\left\{  \mathbf{\zeta}_{n}\right\}
_{n\geq1}$ converges to zero, that we demand that the sequence of disturbances
$\left\{  \mathbf{\xi}_{n}\right\}  _{n\geq1}$ satisfies generalized laws of
large numbers. We will get them after a little algebra:
\[
\mathbf{x}_{n+1}-\mathbf{\theta-\zeta}_{n+1}=\mathbf{x}_{n}-\mathbf{\theta
-\zeta}_{n}-\mu_{n}(\mathbf{f(x}_{n})+\mathbf{\zeta}_{n}),
\]
(compare with the formula (\ref{pom1})) and further we argue as above,
assuming, that the sequence $\left\{  \mathbf{\zeta}_{n}\right\}  _{n\geq0}$
converges almost surely to zero. This would lead to slight complications in
estimation similar to presented above that lead to formula (\ref{pom2}).
\end{remark}

\begin{remark}
Let us notice that in turn that the sequence $\left\{  \mathbf{S}_{n}\right\}
$ converges to zero more quickly than is the convergence of the sequence
$\left\{  \mu_{n}\right\}  _{n\geq0}$ to zero. Hence, the choice of the
sequence " of amplifiers" $\left\{  \mu_{n}\right\}  _{n\geq0}$ has to be a
compromise: the slower this sequence converges to zero, more quickly converges
to zero the sequence $\left\{  \prod_{i=N}^{n}(1-2\mu_{i}\delta+3\mu_{i}%
^{2}\kappa_{1}^{2})\right\}  _{n\geq N}$. On the other hand more quickly
sequence $\left\{  \mu_{n}\right\}  _{n\geq0}$ converges to zero, the quicker
is the convergence of the sequence $\left\{  \mathbf{S}_{n}\right\}  _{n\geq
1}$ to zero.
\end{remark}

\begin{remark}
\label{aspekty}It follows from the above mentioned considerations that one can
distinguish so to say two aspects of the convergence of the procedure
(\ref{apr1}): \newline-\emph{deterministic, }associated with the deterministic
procedure:
\begin{equation}
\mathbf{y}_{n+1}=\mathbf{y}_{n}-\mu_{n}\mathbf{f(y}_{n}), \label{determin}%
\end{equation}
finding zero of functions the $\mathbf{f}$ and \newline-\emph{random,
}associated with 'averaging' of the disturbances $\left\{  \mathbf{\xi}%
_{i}\right\}  _{i\geq1}.$

To ensure quick convergence of the procedure (\ref{determin}) it is suggested
that the sequence $\left\{  \mu_{n}\right\}  _{n\geq0}$ converges to zero as
slow as possible (in the extreme case when $\kappa_{2}=0$ it can be constant).
On the other hand, the random aspect of convergence of the procedure
(\ref{apr1}) requires that the sequence $\left\{  \mu_{i}\right\}  _{i\geq0}$
converged to zero as quickly as possible. Hence, it seems that a reasonable
choice of the sequence $\left\{  \mu_{i}\right\}  _{i\geq1}$ is the following:
\newline first, we keep sequence relatively slowly converging to zero, in
order to reach the area close to the solution as quickly as possible, then we
increase the rate with which the sequence $\left\{  \mu_{i}\right\}  _{i\geq
0}$ decreases to zero in order to start 'averaging' the noises.

One could of course reason more subtly basing on the estimation
(\ref{oszac_podst}). It is not difficult then to notice, that the speed of the
deterministic aspect is connected with the speed of convergence to zero of the
sequence%
\[
\prod_{i=k}^{n}\lambda_{i}\cong\exp(-2\delta\sum_{i=k}^{n}\mu_{i}).
\]
The speed of the random aspect is connected with the speed of convergence to
zero of the sequence $\left\{  \epsilon_{n}^{2}\right\}  _{n\geq0}$, which on
its side is associated with the speed of convergence of expectations of this
sequence that is roughly $\left\{  \mu_{n}\right\}  _{n\geq1}.$

Since, the demand of such choice of the sequence, to make the two sequences
\newline$\left\{  \exp(-2\delta\sum_{i=k}^{n}\mu_{i})\right\}  $ and $\left\{
\mu_{n}\right\}  _{n\geq0}$ possibly quickly simultaneously converge to zero
contains a contradiction, it seems that the only reasonable choice of the
sequence $\left\{  \mu_{n}\right\}  _{n\geq0}$ is to select it to be of the
form $\mu_{n}\allowbreak=\allowbreak a/(n+1)$ for suitable constant $a$. Let
us notice that such choice gives $\exp(-2\delta\sum_{i=k}^{n}\mu
_{i})\allowbreak\cong\allowbreak An^{-2\delta a}$ for some $A$. How to select
coefficient $a?$ Well in such a way as to make $2\delta a\allowbreak
>\allowbreak1$. The choice of coefficient $a$ and estimation of the speed of
convergence of stochastic approximation procedures were subjects of research
of many mathematicians. Precise analysis of the possible choice of $\left\{
\mu_{i}\right\}  _{i\geq0}$ can be found in papers by Fabian \cite{Fabian60},
Kushner and Gavin \cite{KushnerGavin73}, Koronacki \cite{Koronacki80}.
\end{remark}

\section{Extensions and generalizations}

So far we have assumed that the function $\mathbf{f}$ satisfied a condition:
\begin{equation}
{\large \exists}\mathbf{\theta}\in%
%TCIMACRO{\U{211d} }%
%BeginExpansion
\mathbb{R}
%EndExpansion
^{m}{\large ,}\delta\,>0{\large \,\,\,}\forall\mathbf{x\in}%
%TCIMACRO{\U{211d} }%
%BeginExpansion
\mathbb{R}
%EndExpansion
^{m}:\left(  \mathbf{x-\theta}\right)  ^{^{\prime}}\mathbf{f(x)>}%
\delta\left\vert \mathbf{x-\theta}\right\vert ^{2} \label{warzdelta}%
\end{equation}
It turns out that this condition can be weakened relatively substantially:
namely function $\mathbf{f}$ does not have to be almost linear in the
neighborhood of point $\mathbf{\theta}$, as it was assumed in condition
(\ref{warzdelta}), but it is enough, that the condition (\ref{warzdelta}) is
satisfied outside every ringlike neighborhood of the point $\theta$. More
precisely, now we will assume instead condition (\ref{warzdelta}) that the
following condition is satisfied:
\begin{equation}
{\large \exists}\mathbf{\theta}\in%
%TCIMACRO{\U{211d} }%
%BeginExpansion
\mathbb{R}
%EndExpansion
^{m}{\large ,\,\,\,\mathbf{\forall}}1{\large >}\varepsilon
>0:\underset{1/\varepsilon\geq\left\vert \mathbf{x-\theta}\right\vert
\geq\varepsilon}{\inf}\left(  \mathbf{x-\theta}\right)  ^{^{\prime}%
}\mathbf{f(x)>0.} \label{warnasup}%
\end{equation}

\begin{remark}
Condition (\ref{warnasup}) is equivalent to the following one:
\begin{align}
\mathbf{\exists\theta}{\large \,\,}\forall\varepsilon &  \in(0,1)\,\,\,\exists
\delta{\large >0}:1/\varepsilon\geq\left\vert \mathbf{x-\theta}\right\vert
\geq\varepsilon\label{warzdeltaeps}\\
&  \Rightarrow\left(  \mathbf{x-\theta}\right)  ^{^{\prime}}\mathbf{f(x)\geq
}\delta\mathbf{(}\varepsilon\mathbf{)}\left\vert \mathbf{x-\theta}\right\vert
^{2}.\nonumber
\end{align}

\end{remark}

\begin{proof}
The fact that the condition (\ref{warzdeltaeps}) is implied by the condition
(\ref{warnasup}), is obvious. In order to show, that from the condition
(\ref{warnasup}) follows condition (\ref{warzdeltaeps}), let us consider
quantity $\frac{(\mathbf{x}-\mathbf{\theta)}^{^{\prime}}\mathbf{f(x)}%
}{\left\vert \mathbf{x-\theta}\right\vert ^{2}}$. If limes inferior of this
quantity over all $\mathbf{x}$ satisfying $1/\varepsilon\geq\left\vert
\mathbf{x-\theta}\right\vert \geq\varepsilon$ is greater than zero, then
indeed the condition (\ref{warnasup}) implies a condition (\ref{warzdeltaeps}%
). Hence, let us suppose, that limes inferior is equal zero. Then for every
$\chi>0$ there would exist a point $\mathbf{x}_{\chi}$, satisfying inequality
$1/\varepsilon\geq\left\vert \mathbf{x}_{\chi}\mathbf{-\theta}\right\vert
\geq\varepsilon$ and such that $(\mathbf{x}_{\chi}-\mathbf{\theta)}^{\prime
}\mathbf{f(\mathbf{x}_{\chi})\leq\chi}\left\vert \mathbf{x}_{\chi
}\mathbf{-\theta}\right\vert ^{2}\leq\chi/\varepsilon^{2}$. Taking into
account freedom of choice of $\chi$ and the closeness of the set $\left\{
\mathbf{x:\varepsilon\allowbreak\leq}\left\vert \mathbf{x-\theta}\right\vert
\allowbreak\leq\allowbreak1/\varepsilon\right\}  $, $\allowbreak\varepsilon
\in(0,1)$, This inequality, contradicts condition (\ref{warnasup}).
\end{proof}

Moreover, let us suppose, that the following condition concerning the function
$\mathbf{f}$:%

\begin{equation}
{\large \exists\kappa}_{1},\kappa_{2}>0{\large \,\,\forall}\mathbf{x}\in%
%TCIMACRO{\U{211d} }%
%BeginExpansion
\mathbb{R}
%EndExpansion
^{m}:\left\vert \mathbf{f(x)}\right\vert \leq\kappa_{1}\left\vert
\mathbf{x}-\mathbf{\theta}\right\vert +\kappa_{2}, \label{war26}%
\end{equation}
is also satisfied.\ Further, let us assume almost sure convergence of the
series
\begin{equation}
\sum_{i\geq0}\mu_{i}\mathbf{\xi}_{i+1}. \label{sumaszumow6}%
\end{equation}
The sequence $\left\{  \mu_{i}\right\}  _{i\geq0}$ is assumed to be normal.

Let us consider procedure of the form:
\begin{equation}
\mathbf{x}_{0}=\mathbf{xo,\,\,x}_{n+1}=\mathbf{x}_{n}-\mu_{n}(\mathbf{f(x}%
_{n})+\mathbf{\xi}_{n+1});n\geq0. \label{apr16}%
\end{equation}

\begin{theorem}
\label{najprostsze_eps}Suppose that the function $\mathbf{f}$ satisfies
conditions (\ref{warzdeltaeps}), (\ref{war26}) for some $\mathbf{\theta}\in%
%TCIMACRO{\U{211d} }%
%BeginExpansion
\mathbb{R}
%EndExpansion
^{m}$, and noises $\left\{  \mathbf{\xi}_{i}\right\}  _{i\geq1}$ satisfy
condition (\ref{sumaszumow6}) for any normal sequence $\left\{  \mu
_{i}\right\}  _{i\geq0}$. Let us suppose also, that the sequence $\left\{
\mathbf{x}_{n}\right\}  _{n\geq1}$ is bounded with probability $1$, that is,
there exists such random variable $M$, that $P(\underset{n\geq1}{\sup
}\left\vert \mathbf{x}_{n}\right\vert \allowbreak\leq\allowbreak M)=1$. Then,
the procedure (\ref{apr16}) converges almost surely to the point
$\mathbf{\theta.}$
\end{theorem}

\begin{proof}
Let us set, as before, $\mathbf{S}_{n}=\sum_{i\geq n}\mu_{i}\mathbf{\xi}%
_{i+1}$. We know that $\left\{  \mathbf{S}_{n}\right\}  _{n\geq1}$ converges
almost surely to zero, by assumption. Let us subtract $\mathbf{S}%
_{n+1}+\mathbf{\theta}$ from both sides of procedure (\ref{apr16}). We get
then:
\begin{equation}
\mathbf{x}_{n+1}-\mathbf{\theta-S}_{n+1}=\mathbf{x}_{n}-\mathbf{\theta
}-\mathbf{S}_{n}-\mu_{n}\mathbf{f(x}_{n}). \label{pom16}%
\end{equation}
Let us denote: $d_{n}=\left\vert \mathbf{x}_{n}-\mathbf{\theta-S}%
_{n}\right\vert $. Multiplying both sides of (\ref{pom16}) by their
transposition and using definition $d_{n}$ we get:
\begin{equation}
d_{n+1}^{2}=d_{n}^{2}-2\mu_{n}(\mathbf{x}_{n}-\mathbf{S}_{n}-\mathbf{\theta
)}^{^{\prime}}\mathbf{f(x}_{n})+\mu_{n}^{2}|\mathbf{f(x}_{n})|^{2}.
\label{pom26}%
\end{equation}

Taking advantage of assumptions (\ref{warnasup}), (\ref{war26}) we get :
\begin{gather*}
(\mathbf{x}_{n}-\mathbf{S}_{n}-\mathbf{\theta)}^{^{\prime}}\mathbf{f(x}%
_{n})\geq\left(  \mathbf{x}_{n}-\mathbf{\theta}\right)  ^{^{\prime}%
}\mathbf{f(x}_{n})-\left\vert \mathbf{S}_{n}\right\vert \left\vert
\mathbf{f(x}_{n})\right\vert \geq\\
\geq-\eta_{n}\left(  \kappa_{1}\left\vert \mathbf{x}_{n}-\mathbf{\theta
}\right\vert +\kappa_{2}\right)  \geq\\
\geq-\eta_{n}\kappa_{1}d_{n}-\eta_{n}^{2}\kappa_{1}-\eta_{n}\kappa_{2},
\end{gather*}
where we denoted as before $\eta_{n}=\left\vert \mathbf{S}_{n}\right\vert $.
Moreover, we have:
\begin{gather*}
\left\vert \mathbf{f(x}_{n})\right\vert ^{2}\leq(\kappa_{1}\left\vert
\mathbf{x}_{n}-\mathbf{\theta}\right\vert +\kappa_{2})^{2}\leq\\
\leq(\kappa_{1}d_{n}+\kappa_{1}\eta_{n}+\kappa_{2})^{2}\leq\\
\leq3(\kappa_{1}^{2}d_{n}^{2}+\kappa_{1}^{2}\eta_{n}^{2}+\kappa_{2}^{2}).
\end{gather*}

We have thus recurrent relationship:
\begin{gather*}
d_{n+1}^{2}\leq(1+3\mu_{n}^{2}\kappa_{1}^{2})d_{n}^{2}+2\mu_{n}d_{n}\eta
_{n}\kappa_{1}+2\mu_{n}\kappa_{1}\eta_{n}^{2}+2\mu_{n}\eta_{n}\kappa_{2}%
+3\mu_{n}^{2}(\kappa_{1}^{2}\eta_{n}^{2}+\kappa_{2}^{2})\\
\overset{df}{=}(1+3\mu_{n}^{2}\kappa_{1}^{2})d_{n}^{2}+\mu_{n}d_{n}g_{n}%
+\mu_{n}h_{n}.
\end{gather*}

Let us notice that $g_{n}\underset{n\rightarrow\infty}{\longrightarrow}0$ and
$h_{n}\underset{n\rightarrow\infty}{\longrightarrow}0$ a.s.. Let us consider
the function
\[
K_{n}(\upsilon)=\sqrt{(1+3\mu_{n}^{2}\kappa_{1}^{2})\upsilon^{2}+\mu
_{n}\upsilon g_{n}+\mu_{n}h_{n}}.
\]
We have of course: $\forall\upsilon\geq0:K_{n}(\upsilon)\geq\upsilon$ a.s. and
$K_{n}(\upsilon)\underset{n\rightarrow\infty}{\longrightarrow}\upsilon$.
Moreover, let us notice that:
\[
d_{n}\leq\upsilon\Rightarrow d_{n+1}\leq K_{n}(\upsilon).
\]
If however, we will assume, that $1/\varepsilon\geq\left\vert \mathbf{x}%
_{n}-\mathbf{\theta}\right\vert \geq\varepsilon$, then we have:
\begin{align*}
d_{n+1}^{2}  &  =d_{n}^{2}-2\mu_{n}(\mathbf{x}_{n}-\mathbf{S}_{n}%
-\mathbf{\theta)}^{^{\prime}}\mathbf{f(x}_{n})+\mu_{n}^{2}|\mathbf{f(x}%
_{n})|^{2}\\
&  \leq(1-2\mu_{n}\delta(\varepsilon)+3\mu_{n}^{2}\kappa_{1}^{2})d_{n}^{2}%
+\mu_{n}d_{n}g_{n}^{^{\prime}}(\varepsilon)+\mu_{n}h_{n},
\end{align*}
where we denoted $g_{n}^{^{\prime}}(\varepsilon)=\eta_{n}\left(
2\delta(\varepsilon)+\kappa_{1}\right)  $. Let $\epsilon_{n}(\varepsilon)$
will be positive root equation:
\begin{equation}
\delta(\varepsilon)x^{2}-xg_{n}^{^{\prime}}(\varepsilon)-h_{n}=0.
\label{rownanie*}%
\end{equation}

Let us notice that because of properties of the sequences $\left\{
g_{n}^{^{\prime}}(\varepsilon)\right\}  $ and $\left\{  h_{n}\right\}  $ we
see that $\forall\varepsilon>0:\,\epsilon_{n}(\varepsilon
)\underset{n\rightarrow\infty}{\longrightarrow}0$ a.s. Arguing as in the proof
of previous theorem we get for $1/\varepsilon\geq\left\vert \mathbf{x}%
_{n}-\mathbf{\theta}\right\vert \geq\varepsilon:$%
\[
d_{n+1}^{2}\leq(1-\mu_{n}\delta(\varepsilon)+3\mu_{n}^{2}\kappa_{1}^{2}%
)\max\left(  d_{n}^{2},\epsilon_{n}^{2}(\varepsilon)\right)  ,
\]
or equivalently, that
\begin{equation}
d_{n+1}\leq\lambda_{n}(\varepsilon)\max(d_{n},\epsilon_{n}(\varepsilon)),
\label{nierownosc5}%
\end{equation}
where we denoted: $\lambda_{n}(\varepsilon)=\sqrt{1-\mu_{n}\delta
(\varepsilon)+3\mu_{n}^{2}\kappa_{1}^{2}}$. Let us notice also, that
\begin{equation}
\forall\varepsilon>0\prod_{i=0}^{n}\lambda_{i}(\varepsilon
)\underset{n\rightarrow\infty}{\longrightarrow}0\,a.s.,\underset{n>k}{\sup
}\prod_{i=k}^{n}\lambda_{i}(\varepsilon)<\infty\text{ a.s.}.
\label{o_lambdach}%
\end{equation}
Let us take any $1/M(\omega)\geq\varepsilon>0$. Let $\iota$ will be such
random index, that for $i\geq\iota\;\;$ $\eta_{i}\leq\frac{\varepsilon}{4}$.
Let further $\upsilon$ will be such a positive number, and $\iota_{1}$ such
random index, that for $i>\iota_{1}$
\[
\frac{\varepsilon}{4}\leq\upsilon<\frac{3\varepsilon}{4};\,K_{i}(\upsilon
)\leq\frac{3\varepsilon}{4}.
\]
Let finally $\iota_{2}$ be such a random index, that for $i\geq\iota_{2}\,:$%
\[
\,\lambda_{i}(\upsilon-\frac{\varepsilon}{4})<1,\epsilon_{i}(\upsilon
-\frac{\varepsilon}{4})\leq\frac{3\varepsilon}{4}.
\]
From assumptions it follows that indices $\iota$, $\iota_{1}$ and $\iota_{2}$
are finite almost everywhere. Let $\iota^{\ast}$ will be the first after
$\max(\iota,\iota_{1},\iota_{2})$ random moment such that $d_{\iota^{\ast}%
}<\frac{3\varepsilon}{4}$. From our assumptions it follows that $\iota^{\ast
}<\infty$ a.s. Since if it was otherwise, i.e. such $\iota^{\ast}$ would not
exist, then we would have for all $k>\max(\iota,\iota_{1},\iota_{2})$ always
inequalities
\[
2/\varepsilon\allowbreak>\allowbreak1/\varepsilon\allowbreak\geq\allowbreak
M\allowbreak\geq\allowbreak\left\vert \mathbf{x}_{k}-\mathbf{\theta
}\right\vert \allowbreak\geq\allowbreak\frac{3}{4}\varepsilon\allowbreak
-\frac{1}{4}\varepsilon\allowbreak=\allowbreak\frac{\varepsilon}{2},
\]
which is impossible because we have (\ref{nierownosc5}) and the property
(\ref{o_lambdach}). If $\left\vert \mathbf{x}_{\iota^{\ast}}-\mathbf{\theta
}\right\vert \geq\upsilon-\frac{\varepsilon}{4}$, then
\[
d_{\iota^{\ast}+1}\allowbreak\leq\allowbreak\lambda_{\iota^{\ast}}%
(\upsilon-\frac{\varepsilon}{4})\allowbreak\max(d_{\iota^{\ast}}%
,\epsilon_{\iota^{\ast}}(\upsilon-\frac{\varepsilon}{4})\allowbreak
<\allowbreak\frac{3\varepsilon}{4}%
\]
and consequently
\[
\left\vert \mathbf{x}_{\iota^{\ast}+1}-\mathbf{\theta}\right\vert
\allowbreak\leq d_{\iota^{\ast}+1}+\eta_{\iota^{\ast}+1}<\varepsilon,
\]
if $\left\vert \mathbf{x}_{\iota^{\ast}}-\mathbf{\theta}\right\vert
<\upsilon-\frac{\varepsilon}{4}$, then
\[
d_{\iota^{\ast}}\leq\allowbreak\left\vert \mathbf{x}_{\iota^{\ast}%
}-\mathbf{\theta}\right\vert \allowbreak+\allowbreak\theta_{\iota^{\ast}%
}\allowbreak<\allowbreak\upsilon-\frac{\varepsilon}{4}\allowbreak
+\frac{\varepsilon}{4}\allowbreak=\allowbreak\upsilon
\]
hence $d_{\iota^{\ast}+1}<\frac{3\varepsilon}{4}$, that is also $\left\vert
\mathbf{x}_{\iota^{\ast}+1}-\mathbf{\theta}\right\vert <\varepsilon$. Arguing
in the similar way for $d_{\iota^{\ast}+1},\,d_{\iota^{\ast}+2}$ and so on, we
get $\forall k\geq1$ $d_{\iota^{\ast}+k}<\frac{3\varepsilon}{4}$, $\left\vert
\mathbf{x}_{\iota^{\ast}+k}-\mathbf{\theta}\right\vert <\varepsilon$. Hence,
the sequence $\left\{  \mathbf{x}_{n}\right\}  $ converges almost surely.
\end{proof}

In the sequel of this chapter, we will try to understand the behavior of the
procedure from example \ref{przyklad_z_mala_funkcja}. If analyzing closely the
function $f(x)$ \textquotedblright whose zeros are sought \textquotedblright%
\ by this procedure, we notice, that for $x>4$ values of this function
decrease very quickly to zero. For $x\approx12$ we have $f(x)\approx
1.\,1107\times10^{-3}$, that is practically zero. In connection with this one
can state, that practically the set of zeros of the function $f(x)$ contains
the subset $\left\{  x:x\geq12\right\}  $. Consequently , the procedure is
convergent to one of these zeros! In order to analyze more precisely, such and
similar situations first we have to, consider the problem of boundedness in
$L_{2}$ and with probability $1$ of stochastic approximation procedures.

\subsection{Boundedness\label{ograniczonosc}}

The result that we will prove below will concern a bit more general situation
than the one considered in the procedure (\ref{apr1}). Instead, we will have
to impose some restrictions on distributions of $\left\{  \mathbf{\xi}%
_{n}\right\}  _{n\geq0}$. Let us notice that so far the only assumption
imposed on distributions of noises was the requirement of convergence of the
series $\sum_{i\geq1}\mu_{i}\mathbf{\xi}_{i}$, or even more generally, basing
on Remark \ref{mpwl_szumow}, fulfillment of the generalized strong laws of
large numbers by the noises $\left\{  \mathbf{\xi}_{i}\right\}  _{i\geq1}$.
Now we will assume that noises $\left\{  \mathbf{\xi}_{n}\right\}  _{n\geq0}$
can depend on the previously found estimators $\mathbf{x}_{1},\ldots
,\mathbf{x}_{n-1}$ of the point $\mathbf{\theta}$. More precisely, we will
consider procedure of the form:%

\begin{equation}
\mathbf{x}_{n+1}=\mathbf{x}_{n}-\mu_{n}\left(  \mathbf{f}(\mathbf{x}%
_{n})+\mathbf{\xi}_{n+1}(\mathbf{x}_{n})\right)  ;n\geq0,\mathbf{x}%
_{0}=\mathbf{x0,} \label{proc_uog}%
\end{equation}
wherein the sequence of the disturbances we will be assumed to satisfy:
\begin{equation}
E\left(  \mathbf{\xi}_{n+1}(\mathbf{x}_{n})|\mathcal{F}_{n}\right)  =0,\exists
L>0:E\left(  \left\vert \mathbf{\xi}_{n+1}(\mathbf{x}_{n})\right\vert
^{2}|\mathcal{F}_{n}\right)  \leq L(1+\left\vert \mathbf{x}_{n}\right\vert
^{2}), \label{zal_ogr2}%
\end{equation}
where we denoted $\mathcal{F}_{n}=\sigma(\mathbf{x}_{0},\ldots,\mathbf{x}%
_{n})$. We will be concerned with conditions, under which this procedure is
bounded in $L_{2}$. Let us suppose also, that the function $\mathbf{f}$
satisfies condition (\ref{war2}). Let us multiply both sides of the equation
(\ref{proc_uog}) by its transposition and let us calculate the expectation of
both sides. Let us denote by $D_{n}=E\left\vert \mathbf{x}_{n}-\mathbf{\theta
}\right\vert ^{2}$. We get then utilizing first of the assumptions
\ref{zal_ogr2} and the property $\left\vert \mathbf{x}_{n}\right\vert ^{2}%
\leq2\left\vert \mathbf{x}_{n}-\mathbf{\theta}\right\vert ^{2}+2\left\vert
\mathbf{\theta}\right\vert ^{2}$
\begin{equation}
D_{n+1}\leq D_{n}-2\mu_{n}E(\mathbf{x}_{n}-\mathbf{\theta)}^{\prime
}\mathbf{f(x}_{n})+\mu_{n}^{2}(L(1+2D_{n}+2\left\vert \mathbf{\theta
}\right\vert ^{2})+2\kappa_{1}^{2}D_{n}+2\kappa_{2}^{2}). \label{rown_ na_l2}%
\end{equation}

From this inequality follows the following lemma:

\begin{lemma}
\label{ogr_l2} Let disturbances $\left\{  \mathbf{\xi}_{n}\right\}  _{n\geq1}$
satisfy conditions (\ref{zal_ogr2}). Let us suppose also, that the normal
sequence $\left\{  \mu_{n}\right\}  _{n\geq0}$ satisfies additionally a
condition :
\begin{equation}
\sum_{n}\mu_{n}^{2}<\infty, \label{suma_mi2}%
\end{equation}
while the mapping $\mathbf{f}$ satisfies a condition:
\begin{equation}
\forall\mathbf{x\,:(x-\theta)}^{\prime}\mathbf{f(x)\geq0}, \label{zal_ogr}%
\end{equation}
and condition (\ref{war2}). Then sequence $\left\{  \mathbf{x}_{n}\right\}
_{n\geq0}$, of random vectors generated by the procedure (\ref{proc_uog}%
)\textbf{\ }is bounded in $L_{2}$ and with probability $1.$
\end{lemma}

\begin{proof}
On the base of inequality (\ref{rown_ na_l2}) and assumptions (\ref{zal_ogr})
of this lemma we get recurrent relationship:
\[
D_{n+1}\leq D_{n}(1+q\mu_{n}^{2})+C\mu_{n}^{2},
\]
where we denoted%
\[
C=L(1+2\left\vert \mathbf{\theta}\right\vert ^{2})+2\kappa_{2}^{2}%
,q=2(L+\kappa_{1}^{2}).
\]
In order to show boundedness of the sequence $\left\{  D_{n}\right\}  $, let
us denote $P_{n}=\prod_{i=n}^{\infty}(1+q\mu_{i}^{2})$. Elements of the
sequence $\left\{  P_{n}\right\}  $ are finite by our assumptions, and
Moreover, they decrease. Let us denote also $\Delta_{n}=D_{n}P_{n}$. We have:
\[
\Delta_{n+1}\leq P_{n+1}\left(  D_{n}(1+q\mu_{n}^{2})+C\mu_{n}^{2}\right)
\leq\Delta_{n}+P_{0}C\mu_{n}^{2}.
\]
Now it is easy to deduce, that $\Delta_{n}\leq P_{0}C\sum_{i=0}^{n-1}\mu
_{i}^{2}$. Thus, indeed the sequences $\left\{  \Delta_{n}\right\}  _{n\geq1}$
and $\left\{  D_{n}\right\}  _{n\geq1}$ are bounded. In order to show
boundedness with probability $1$ of the sequence $\left\{  \left\vert
\mathbf{x}_{n}-\mathbf{\theta}\right\vert \right\}  $, let us introduce
denotation
\[
\mathbf{Y}_{n}=P_{n}\left\vert \mathbf{x}_{n}-\mathbf{\theta}\right\vert
^{2}+P_{0}C\sum_{i\geq n}\mu_{i}^{2}.
\]
We have
\begin{align*}
E(\mathbf{Y}_{n+1}|\mathcal{F}_{n})  &  \leq P_{n+1}\left(  \left\vert
\mathbf{x}_{n}-\mathbf{\theta}\right\vert ^{2}(1+q\mu_{n}^{2})+C\mu_{n}%
^{2}\right)  +P_{0}C\sum_{i\geq n+1}\mu_{i}^{2}\\
&  \leq P_{n}\left\vert \mathbf{x}_{n}-\mathbf{\theta}\right\vert ^{2}%
+P_{0}C\sum_{i\geq n}\mu_{n}^{2}=\mathbf{Y}_{n}%
\end{align*}
Hence the sequence $\left\{  \mathbf{Y}_{n}\right\}  _{n\geq1}$ is a
nonnegative supermartingale Hence, on the basis of Doob's Theorem
\ref{zb_mart}, converges almost surely to finite limit.
\end{proof}

Thus, we see, that boundedness of the sequence of approximations under rather
loose requirements concerning mappings $\mathbf{f}$ (assumptions
(\ref{zal_ogr}) and (\ref{war2})) already requires some ordered probabilistic
structure of disturbances $\left\{  \mathbf{\xi}_{n}\right\}  $ (assumption of
being martingale differences). One can expect, that in the more complicated
cases of stochastic approximation procedures this assumption to will be also active.

\begin{remark}
Let us notice also, that assumptions (\ref{zal_ogr2}) may not be imposed to
guarantee the almost sure convergence, if it was known, that mapping
$\mathbf{f}$ satisfied condition (\ref{war1}) or (\ref{warzdeltaeps}).
\end{remark}

There exists a way to omit those intensified requirements concerning
disturbances, and aiming to get boundedness of stochastic approximation
procedures. Namely, if we found a bounded set $V$, in which the unknown
parameter would lie for sure, then one could consider the following procedure:

\emph{if in the }$n-$\emph{theorem} \emph{iterative} \emph{step } \emph{the
quantity }
\[
\mathbf{p}_{n}=\mathbf{x}_{n}-\mu_{n}\left(  \mathbf{f}(\mathbf{x}%
_{n})+\mathbf{\xi}_{n+1}(\mathbf{x}_{n})\right)
\]
\emph{lied} \emph{\ inside the set }$V$, \emph{then as} \emph{\ }%
$\mathbf{x}_{n+1}$ \emph{we take }$\mathbf{p}_{n}$, \emph{if we have}
$\mathbf{p}_{n}\notin V$ \emph{then for }$\mathbf{x}_{n+1}$ \emph{we take some
point of } $V.$

It remains to select this point. We have great freedom and it would be good to
select this point properly. It turns out that if the set $V$ is bounded\emph{,
closed and convex} and if for $\mathbf{x}_{n+1}$ we would take orthogonal
projection $\mathbf{\pi}_{n}$ of the point $\mathbf{p}_{n}=\mathbf{x}_{n}%
-\mu_{n}\left(  \mathbf{f}(\mathbf{x}_{n})+\mathbf{\xi}_{n+1}(\mathbf{x}%
_{n})\right)  $ on $V$, then for any point $\mathbf{a\in V}$ $\;\left\vert
\mathbf{p}_{n}-\mathbf{a}\right\vert \geq\left\vert \mathbf{\pi}%
_{n}-\mathbf{a}\right\vert $. In other words, if instead of procedure
(\ref{proc_uog}) we consider the procedure:%

\begin{equation}
\mathbf{x}_{n+1}=\left\{
\begin{array}
[c]{lll}%
\mathbf{p}_{n}=\mathbf{x}_{n}-\mu_{n}\left(  \mathbf{f}(\mathbf{x}%
_{n})+\mathbf{\xi}_{n+1}(\mathbf{x}_{n})\right)  , & \text{when} &
\mathbf{p}_{n}\in V\\
\mathbf{\pi}_{n}=\text{projection }\,\mathbf{p}_{n}\,\text{on}\,V, &
\text{when} & \mathbf{p}_{n}\notin V
\end{array}
\right.  , \label{mod_rzut}%
\end{equation}
then we can do the analysis of the convergence of these procedures and make
use of ordinary, considered earlier, estimation. This fact follows the
following lemmas.

\begin{lemma}
\label{o_zb_wyp}Let $K$ will be sphere, and $V$ a closed, convex subset of $%
%TCIMACRO{\U{211d} }%
%BeginExpansion
\mathbb{R}
%EndExpansion
^{d}$. Let further $\mathbf{p}\in K\backslash V$, and further let
$\mathbf{\pi}$ be a orthogonal projection of $\mathbf{p}$ on $V$. If the
center of the sphere $K$ lies in the set $V$, then $\mathbf{\pi\in}K.$
\end{lemma}

Proof of this lemma is purely geometrical and is based on the following
auxiliary lemma:

\begin{lemma}
\label{o_stozku}Let $K\subset%
%TCIMACRO{\U{211d} }%
%BeginExpansion
\mathbb{R}
%EndExpansion
^{d}$ will be sphere, and $C\subset%
%TCIMACRO{\U{211d} }%
%BeginExpansion
\mathbb{R}
%EndExpansion
^{d}$ closed, convex cone with apex at the center of the sphere. Let further
$\mathbf{p}\in K\backslash C$, and let $\mathbf{\pi}$ be an orthogonal
projection of $\mathbf{p}$ on $C$. Then $\mathbf{\pi\in}K.$
\end{lemma}

\begin{proof}
Proof of this fact is very simple, that is why we will only sketch it. Let
$\mathbf{a}$ denote the center of the sphere $K$. Let us first consider the
two-dimensional situation. Remembering, that orthogonal projection is also the
point of $K$ closest to the projected point $\mathbf{p}$, we show the
truthfulness of the assertion in a two-dimensional situation with ease. Next,
let us consider the general situation and we will argue as follows. First, let
us notice that the three points $\pi$, $\mathbf{p}$ and $\mathbf{a}$ do not
lie on one straight line. Hence, one can draw a three-dimensional plane by
them and reduce the situation to a two-dimensional one.

Let us return to the proof of lemma \ref{o_zb_wyp}. Let $\mathbf{a}$ denote
center of the sphere. Through every point $\mathbf{d}$ of the set $V\backslash
K$ let us draw a ray $\mathbf{da}$. Collection of these rays forms a cone $C$
with apex at $\mathbf{a}$. We will show that $C\cap K\subset V$. Let
$\mathbf{c}$ be any point of the set $C\cap K$ and let $R_{\mathbf{c}}$ will
be the ray passing through $\mathbf{c}$. It follows from the construction of
the cone that there exists point $\mathbf{d\in}V\backslash K$ lying on
$R_{\mathbf{c}}$. From convexity of $V$ it follows that the segment
$\mathbf{da\subset}V$. However, from the definition of the set $V\backslash K$
it follows that $\mathbf{c\in da}$. Hence, indeed $C\cap K\subset V$.
Moreover, we have $V=(V\cap C)\cup(V\backslash C)$. Taking into account the
construction of the cone it is clear, that the set $V\backslash C$ lies inside
the sphere $K$. Let us consider now any point $\mathbf{c\in}V\cap C\subset C$.
It is clear that the distance $\mathbf{cp}$ is smaller than the distance of a
point $\mathbf{p}$ from its projection on $C$ (denoted by \textbf{\thinspace
}$\mathbf{p}^{\prime})$. From Lemma \ref{o_stozku} It follows, that
$\mathbf{p}^{\prime}\in K\cap C\subset V$. Hence, for every point
$\mathbf{d\in}V\cap C$ we can indicate a point $\mathbf{p}^{\prime}\in K\cap
V$, that lies closer to $\mathbf{p}$ than $\mathbf{d}$. Hence, and projection
of $\mathbf{p}$ on $V$ must lie inside $K.$
\end{proof}

The assumption that the unknown $\mathbf{\theta}$ lies inside some known set
and, that because of this one has to look for $\mathbf{\theta}$ in this
particular set is equivalent to the assumption, that there exist restrictions
imposed on the position of the point $\mathbf{\theta}$. This leads us to
stochastic approximation procedures with restrictions. As it turns out this
class of procedures was and is well known and examined. In particular, the
so-called Kiefer -Wolfowitz version of these procedures turned out to be
important and led to the creation of a new chapter of numerical methods called
stochastic optimization. We will return to these problems in sections
\ref{optymalizacja} and \ref{dalej}.

\section{More complex procedures\label{rozszerzenia}}

Let us notice that so far disturbances of observations were coming as if from
another source than the values of estimators, i.e. points $\left\{
\mathbf{x}_{n}\right\}  $. One cannot thus apply existing methods to examine
the convergence discussed above procedure estimating given quantile of unknown
distribution. Let us recall this example (example \ref{kwantyl}) : $\left\{
\xi_{n}\right\}  _{n\geq1}$ was a\ sequence of independent random variables
drawn from the Normal distribution(generally having cdf $F$). Observations at
point $x$ were given by the formula: $Z_{i}(x)=I(\xi_{i}<x)-.85$ (generally
$Z_{i}(x)=I(\xi_{i}<x)-\alpha)$. Let us notice that we have here
$EZ_{i}(x)=F(x)-.85$ (generally $EZ_{i}(x)=F(x)-\alpha)$, or even we have here
somewhat stronger property that $E\left(  Z_{i}(x)|\xi_{i-1},\ldots,\xi
_{1}\right)  =F(x)-.85$. Hence, let us write%

\[
Z_{i}(x)=F(x)-\alpha+\zeta_{i}(x),
\]
where the sequence of the random variables $\left\{  \zeta_{i}(x)\right\}
_{n\geq1}$ has the following property $E\zeta_{i}(x)=0$ or more generally :
\[
E\left(  \zeta_{i}(x)|\xi_{i-1},\ldots,\xi_{1}\right)  =0\,\,a.s.
\]
Let us recall the used above procedure :
\[
x_{i}=x_{i-1}-\frac{1}{i}Z_{i}(x_{i-1});\;x_{0}=xo;\;i\geq1.
\]
We have here
\[
E\left(  Z_{i}(x_{i-1})-F(x_{i-1})+\alpha|x_{i-1},\ldots,x_{1}\right)  =0
\]
and%
\[
E\left(  Z_{i}(x_{i-1})-F(x_{i-1})+\alpha|\xi_{i-1},\ldots,\xi_{1}\right)
=0,
\]
since of course $\sigma\left(  x_{i},\ldots,x_{1}\right)  \allowbreak
\subset\allowbreak\sigma\left(  \xi_{i}\ldots,\xi_{1}\right)  $, $i\geq1$.
This is a property defining the \emph{\ martingale difference }with respect to
the filtration $\left\{  \sigma\left(  \xi_{1},\ldots,\xi_{n}\right)
\right\}  _{n\geq1}$ (compare definition \ref{def_mart} and situation
considered in the previous section in Lemma \ref{ogr_l2})\emph{. }

In the sequel we will assume, that the normal sequence $\left\{  \mu
_{n}\right\}  _{n\geq0}$ satisfies additionally condition:
\begin{equation}
\sum_{n\geq1}\mu_{n}^{2}<\infty. \label{suma_kw_mi}%
\end{equation}

\begin{remark}
Let us notice that instead of one function $\mathbf{f}$, whose zero has been
sought, one can use a sequence of functions $\left\{  \mathbf{f}_{n}\right\}
_{n\geq1}$ such that, we have e.g.
\begin{equation}
\mathbf{\exists\theta\in}%
%TCIMACRO{\U{211d} }%
%BeginExpansion
\mathbb{R}
%EndExpansion
^{m}{\large ,}\exists\left\{  {\large \delta}_{n}\right\}  _{n\geq
1}{\large \,\forall}\mathbf{x\in}%
%TCIMACRO{\U{211d} }%
%BeginExpansion
\mathbb{R}
%EndExpansion
^{m}:\left(  \mathbf{x-\theta}\right)  ^{^{\prime}}\mathbf{f}_{n}%
\mathbf{(x)>\delta}_{n}\left\vert \mathbf{x-\theta}\right\vert ^{2}%
;\underset{n\,\rightarrow\infty}{\lim\inf}\,\delta_{n}>0 \label{warzdeltan}%
\end{equation}
and
\begin{align}
{\large \,\exists}\left\{  \kappa_{1n}\right\}  _{n\geq1},\left\{  \kappa
_{2n}\right\}  _{n\geq1},{\large \forall}\mathbf{x}  &  \in%
%TCIMACRO{\U{211d} }%
%BeginExpansion
\mathbb{R}
%EndExpansion
^{m}:\left\vert \mathbf{f}_{n}\mathbf{(x)}\right\vert \leq\kappa
_{1n}\left\vert \mathbf{x}-\mathbf{\theta}\right\vert +\kappa_{2n}%
,\label{war2n}\\
\underset{n\,\rightarrow\infty}{\lim\inf}\,\left(  \kappa_{1n}+\kappa
_{2n}\right)   &  <\infty.
\end{align}

\end{remark}

Let us consider the procedure:
\begin{equation}
\mathbf{x}_{n+1}=\mathbf{x}_{n}-\mu_{n}\mathbf{F}_{n}(\mathbf{x}_{n},\xi
_{n}),\;n\geq1. \label{nowaproc}%
\end{equation}

Let us denote
\[
\mathbf{G}_{n}(\mathbf{x})=E\left(  \mathbf{F}_{n}(\mathbf{x},\xi_{n}%
)|\xi_{n-1},\ldots,\xi_{0}\right)  .
\]

We have theorem:

\begin{theorem}
\label{rozne_F}Let us assume that functions $\left\{  \mathbf{G}%
_{n}(\mathbf{x})\right\}  _{n\geq1}$ satisfy conditions (\ref{warzdeltan}) and
(\ref{war2n}) at some point $\theta$. Let us suppose also, that noises
\[
\mathbf{\zeta}_{n}(\mathbf{x)=F}_{n}(\mathbf{x,\xi}_{n})-\mathbf{G}%
_{n}(\mathbf{x),}n=1,2,\ldots\;,
\]
satisfy the condition:
\begin{equation}
E\left(  \left\vert \mathbf{\zeta}_{n}(\mathbf{x)}\right\vert ^{2}%
\mathbf{|\xi}_{n-1},\ldots,\xi_{1}\right)  \leq L_{n}(1+\left\vert
\mathbf{x}\right\vert ^{2}),a.s.\,\sup L_{n}<0\;a.s. \label{naszumy1}%
\end{equation}
Then, under the assumption, that the normal sequence $\left\{  \mu
_{i}\right\}  _{i\geq0}$ satisfies assumptions (\ref{suma_kw_mi}), the
procedure (\ref{nowaproc}) converges almost surely do $\theta.$
\end{theorem}

\begin{proof}
Let us notice firstly, that assumptions of Lemma \ref{ogr_l2} would be
satisfied, if only denotations were changed. In such case our procedure is
bounded with probability $1$ and also in $L_{2}$. Following this fact and the
conditions (\ref{war2n}) and (\ref{naszumy1}) we get boundedness with
probability of the sequences $\left\{  E\left\vert \mathbf{\zeta}%
_{n}(\mathbf{x}_{n})\right\vert ^{2}\right\}  _{n\geq1}$ and $\left\{
\left\vert \mathbf{G}_{n}(\mathbf{x}_{n})\right\vert ^{2}\right\}  _{n\geq1}$
and consequently convergence of the following series.
\[
\sum_{n\geq1}\mu_{n}^{2}\left\vert \mathbf{\zeta}_{n}(\mathbf{x}%
_{n})\right\vert ^{2}\text{ and }\sum_{n\geq1}\mu_{n}^{2}\left\vert
\mathbf{G}_{n}(\mathbf{x}_{n})\right\vert ^{2}.
\]
Having proven this fact, we see that the sequence the $\left\{  \sum_{i=0}%
^{n}\mu_{i}\mathbf{\zeta}_{i+1}(\mathbf{x}_{i})\right\}  _{n\geq1}$ is a
martingale convergent in $L_{2}$, hence also almost surely. Thus, the sequence
of random vectors $\left\{  \mathbf{S}_{n}=\sum_{i=n}\mu_{i}\mathbf{\zeta
}_{i+1}(\mathbf{x}_{i})\right\}  _{n\geq1}$ converges almost surely to zero.
Subtracting $\mathbf{S}_{n+1}$ from both sides of the procedure
(\ref{nowaproc}) we get recurrent relationship combining only functions
$\mathbf{G}_{n}$. Further, we proceed as in the proof of theorem
\ref{najprostsze}.
\end{proof}

Procedures of this type, i.e. with functions $\mathbf{f}_{n}$ depending on
iteration number, and also possibly with disturbances depending on the
iteration number and so far obtained estimator, are used in the so-called
identification of discrete stochastic processes. The problem of identification
will be discussed in chapter \ref{identifikacja}. In this section, we will
present, however a theorem on the convergence of the procedure that is a
generalization of Theorem \ref{rozne_F} and useful just for identification
purposes, and also in problems of the so-called stochastic optimization
discussed briefly below.

In order to do it swiftly, we will prove a few useful numerical lemmas.

\begin{lemma}
\label{prosty-rek}Let us assume that number sequence $\left\{  d_{n}\right\}
_{n\geq1}$satisfies recurrent relationship:
\begin{equation}
d_{n+1}^{2}\leq\left[  1-2\delta_{n}\mu_{n}+\mu_{n}^{2}\gamma_{n}\right]
^{+}d_{n}^{2}+\mu_{n}g_{n}d_{n}+\mu_{n}h_{n}, \label{rek-dn1}%
\end{equation}
where $\left\{  g_{n}\right\}  _{n\geq1}$, $\left\{  h_{n}\right\}  _{n\geq1}%
$, $\left\{  \delta_{n}\right\}  _{n\geq1}$ and $\left\{  \gamma_{n}\right\}
_{n\geq1}$ are some number sequences, satisfying the following assumptions:
\[
\underset{n\,\rightarrow\infty}{\lim\inf}\,\delta_{n}%
>0;\underset{n\,\rightarrow\infty}{\lim}\mu_{n}\gamma_{n}=0.
\]
Then the following recursive relationship is satisfied:
\[
d_{n+1}\leq\lambda_{n}\max(d_{n},\epsilon_{n}),
\]
where $\lambda_{n}=\sqrt{1-\mu_{n}\delta_{n}+\mu_{n}^{2}\gamma_{n}}$, and the
sequence $\left\{  \epsilon_{n}\right\}  _{n\geq1}$ consists of positive roots
of the equations:
\[
\epsilon_{n}^{2}\delta_{n}-\epsilon_{n}g_{n}-h_{n}=0;n\geq1;
\]
In particular, if $g_{n}\underset{n\rightarrow\infty}{\longrightarrow}0$,
$h_{n}\underset{n\rightarrow\infty}{\longrightarrow}0$, then $d_{n}%
\underset{n\rightarrow\infty}{\longrightarrow}0.$
\end{lemma}

\begin{proof}
The proof was already a few times presented (without this particular
statement) when we presented proofs of theorems on convergence stochastic
approximation procedures.
\end{proof}

\begin{lemma}
\label{nieprostyrek}Let us assume that number sequence $\left\{
d_{n}\right\}  _{n\geq1}$ satisfies recurrent relationship:
\begin{equation}
d_{n+1}^{2}\leq\left[  1-2\delta_{n}\mu_{n}+\mu_{n}^{2}\gamma_{n}\right]
^{+}d_{n}^{2}+\mu_{n}g_{n}d_{n}+\mu_{n}h_{n}, \label{rek-dn}%
\end{equation}
where $\left\{  g_{n}\right\}  _{n\geq1}$, $\left\{  h_{n}\right\}  _{n\geq1}%
$, $\left\{  \delta_{n}\right\}  _{n\geq1}$ and $\left\{  \gamma_{n}\right\}
_{n\geq1}$ are some number sequences, satisfying the following assumptions:
\begin{align}
\delta_{n}  &  =\delta_{n}^{^{\prime}}+\delta_{n}^{^{\prime\prime}%
};\underset{n\,\rightarrow\infty}{\lim\inf}\,\delta_{n}^{^{\prime}%
}>0;\underset{n\,\rightarrow\infty}{\lim}\mu_{n}\delta_{n}^{^{\prime}%
}=0,\label{rozkl-delt}\\
\infty &  >\underset{m>n}{\sup}\left\vert \sum_{i=n}^{m}\mu_{i}\delta
_{i}^{^{\prime\prime}}\right\vert ;\underset{n\,\rightarrow\infty}{\lim}%
\mu_{n}\delta_{n}^{^{\prime\prime}}=0;\label{sum-deltabis}\\
0  &  =\underset{n\,\rightarrow\infty}{\lim}\mu_{n}\gamma_{n}.
\label{limdelta}%
\end{align}
Then, the following recursive relationship is satisfied by $\left\{
d_{n}\right\}  _{n\geq1}$:
\[
d_{n+1}\leq\lambda_{n}\max(d_{n},\epsilon_{n}),
\]
where $\lambda_{n}=\sqrt{\frac{M_{n}}{M_{n+1}}\left(  1-\mu_{n}\delta
_{n}^{^{\prime}}\exp(2\mu_{n}\delta_{n}^{^{\prime\prime}})+\mu_{n}^{2}%
\gamma_{n}\exp(2\mu_{n}\delta_{n}^{^{\prime\prime}})\right)  }$, $\epsilon
_{n}=\tau_{n}/M_{n}$, $\tau_{n}$ is a positive root of the equation:
\[
\tau_{n}^{2}\delta_{n}^{^{\prime}}\exp(2\mu_{n}\delta_{n}^{^{\prime\prime}%
})-\tau_{n}g_{n}\exp(\mu_{n}\delta_{n}^{^{\prime\prime}})\sqrt{M_{n+1}}%
-h_{n}M_{n+1}=0,
\]
and
\[
M_{n}=\underset{m\rightarrow\infty}{\lim\inf}\,\exp\left(  -2\sum_{i=n}^{m}%
\mu_{i}\delta_{i}^{^{\prime\prime}}\right)  .
\]

\end{lemma}

\begin{proof}
Let us denote: $M_{n}^{m}=\exp(-2\sum_{i=n}^{m}\mu_{i}\delta_{i}%
^{^{\prime\prime}})$. Let $N_{1}$ be the first natural number such that
$\left\vert \mu_{n}\delta_{n}^{^{\prime\prime}}\right\vert <1/2$. From
assumptions it follows that $N_{1}$ is finite.\ Further, let $N$ will be such
index, that for $n\geq N:$ $1-2\delta_{n}\mu_{n}+\mu_{n}^{2}\gamma_{n}\geq0$.
Again from assumptions it follows that $N$ is a finite number. Taking
advantage of Proposition \ref{o_exp}, let us notice that for $n\geq
\max(N,N_{1})$ we have:
\begin{equation}
d_{n+1}^{2}\leq\exp(-2\mu_{n}\delta_{n}^{^{\prime\prime}})d_{n}^{2}+(-2\mu
_{n}\delta_{n}^{^{\prime}}+\mu_{n}^{2}\gamma_{n})d_{n}^{2}+\mu_{n}g_{n}%
d_{n}+\mu_{n}h_{n}. \label{nowa-nier}%
\end{equation}
Let us set: $d_{n}^{\ast}=d_{n}\sqrt{M_{n}}$. Let us multiply both sides of
(\ref{nowa-nier}) by $M_{n+1}^{m}$, $m>n+1$. We get then
\begin{gather}
M_{n+1}^{m}(d_{n+1})^{2}\leq M_{n+1}^{m}d_{n}^{2}\exp(-2\mu_{n}\delta
_{n}^{^{\prime\prime}})+(-2\mu_{n}\delta_{n}^{^{\prime}}+\mu_{n}^{2}\gamma
_{n})d_{n}^{2}M_{n+1}^{m}+\nonumber\\
+\mu_{n}g_{n}d_{n}M_{n+1}^{m}+\mu_{n}h_{n}M_{n+1}^{m}=\nonumber\\
=M_{n}^{m}d_{n}^{2}(1-2\mu_{n}\delta_{n}^{^{\prime}}\exp(2\mu_{n}\delta
_{n}^{^{\prime\prime}})+\mu_{n}^{2}\gamma_{n}\exp(2\mu_{n}\delta_{n}%
^{^{\prime\prime}}))+\label{nier_pom1}\\
+\mu_{n}g_{n}\exp(\mu_{n}\delta_{n}^{^{\prime\prime}})\sqrt{M_{n+1}^{m}}%
\sqrt{M_{n}^{m}d_{n}^{2}}+\mu_{n}h_{n}M_{n+1}^{m}.\nonumber
\end{gather}
Let us denote
\begin{align*}
g_{n}^{\ast}  &  =g_{n}\underset{m>n}{\sup}\exp(\mu_{n}\delta_{n}%
^{^{\prime\prime}})\sqrt{M_{n+1}^{m}},\;h_{n}^{\ast}=h_{n}\underset{m>n}{\sup
}M_{n+1}^{m},\\
\delta_{n}^{\ast}  &  =\delta_{n}^{^{\prime}}\exp(2\mu_{n}\delta_{n}%
^{^{\prime\prime}}),\;\gamma_{n}^{\ast}=\gamma_{n}\exp(2\mu_{n}\delta
_{n}^{^{\prime\prime}})
\end{align*}
and let us pass with $m$ to infinity in (\ref{nier_pom1}). We get then:
\[
\left(  d_{n+1}^{\ast}\right)  ^{2}\leq\left(  1-2\mu_{n}\delta_{n}^{\ast}%
+\mu_{n}^{2}\gamma_{n}^{\ast}\right)  (d_{n}^{\ast})^{2}+\mu_{n}g_{n}^{\ast
}d_{n}^{\ast}+\mu_{n}h_{n}^{\ast}.
\]
Now we apply Lemma \ref{prosty-rek} and see, that $d_{n+1}^{\ast}\leq
\lambda_{n}^{\ast}\max(d_{n}^{\ast},\epsilon_{n}^{\ast})$, where
\newline$\lambda_{n}^{\ast}=\sqrt{1-\mu_{n}\delta_{n}^{\ast}+\mu_{n}^{2}%
\gamma_{n}^{\ast}}$ and $\epsilon_{n}^{\ast}$ is a positive root of the
equation:
\[
x^{2}\delta_{n}^{\ast}-xg_{n}^{\ast}-h_{n}^{\ast}=0.
\]
Returning to 'without star' variables\ we get assertion lemma.
\end{proof}

\begin{theorem}
\label{split_delta_proc}Let us assume that functions $\left\{  \mathbf{G}%
_{n}(\mathbf{x})\right\}  _{n\geq1}$ satisfy conditions (\ref{rozkl-delt}) and
(\ref{war2n}) at some point $\mathbf{\theta}$. Let us suppose also, that
noises $\mathbf{\zeta}_{n}(\mathbf{x)=F}_{n}(\mathbf{x,\xi}_{n})-\mathbf{G}%
_{n}(\mathbf{x)}$ satisfy condition (\ref{naszumy1}). Then, under the
assumption, that normal sequence $\left\{  \mu_{i}\right\}  _{i\geq0}$
satisfies assumptions (\ref{suma_kw_mi}), the procedure (\ref{nowaproc})
converges almost surely do $\theta.$
\end{theorem}

\begin{proof}
Is similar to the proof of Theorem \ref{rozne_F}, with the proviso that it
exploits Lemma \ref{nieprostyrek}.
\end{proof}

\begin{remark}
Let us notice that in likewise way one can prove other, similar theorems
concerning convergence, combining different assumptions dealing with the form
of functions $\left\{  \mathbf{F}_{n}\right\}  _{n\geq1}$ and disturbances. In
particular, instead of Theorem \ref{split_delta_proc}, one can consider
theorems similar to Theorems \ref{najprostsze} and \ref{najprostsze_eps}.
\end{remark}

\section{Complements}

\subsection{Introduction to stochastic optimization\label{optymalizacja}}

The procedures for seeking zeros of the system of functions discussed so far
are called \emph{Robbins-Monro }procedures. Procedures of searching for
extremes of the systems of functions in the random environment are called
\emph{Kiefer-Wolfowitz procedures }since this type of procedures appeared for
the first time in the paper of Kiefer and Wolfowitz \cite{Kiefer52}\emph{. }We
will be concerned in a moment with one-dimensional versions of procedures of
this type.

Let us assume that there is given a function $\psi:%
%TCIMACRO{\U{211d} }%
%BeginExpansion
\mathbb{R}
%EndExpansion
\mathbb{\rightarrow}%
%TCIMACRO{\U{211d} }%
%BeginExpansion
\mathbb{R}
%EndExpansion
$, whose minimum is at point $\theta$. We would like to find this point, but
we cannot observe values of functions $\psi$. Instead, these values can be
measured disturbed i.e. with certain random error. Let us take into account
convergent to zero sequence $\left\{  c_{n}\right\}  _{n\geq1}$ of positive
numbers and let us assume, that the function $\psi$ is a differentiable
function\emph{, }having derivative satisfying so-called global \emph{Lipschitz
condition}. Let us notice that values of functions
\[
\hat{\psi}_{n}(x)=\frac{\psi(x+c_{n})-\psi(x-c_{n})}{2c_{n}}%
\]
converge to $\psi^{\prime}(x)$ at every point $x\in%
%TCIMACRO{\U{211d} }%
%BeginExpansion
\mathbb{R}
%EndExpansion
$. As stated above values of functions $\hat{\psi}$ can not be observed
straightforwardly, but only one can observe values of functions
\[
\Psi_{n}(x)=\psi(x)+\xi_{n},
\]
where $\left\{  \xi_{n}\right\}  $ is a sequence of the random variables with
zero mean and finite variances. Point $\theta$ will be estimated with the help
of the sequence:
\begin{equation}
x_{n+1}=x_{n}-\mu_{n}\left(  \frac{\Psi_{2n+1}(x_{n}+c_{n})-\Psi_{2n}%
(x_{n}-c_{n})}{2c_{n}}\right)  . \label{kiefer}%
\end{equation}
Let us assume that the series
\begin{equation}
\sum_{n=1}^{\infty}\frac{\mu_{n}(\xi_{2n+1}-\xi_{2n})}{c_{n}} \label{szumy}%
\end{equation}
converges with probability $1$. If, e.g., in the simplest situation random
variables $\left\{  \xi_{n}\right\}  $ are \emph{martingale differences}
having jointly bounded\emph{\ }variances, then it is enough to assume that
e.g.
\begin{equation}
\sum_{n=1}^{\infty}\left(  \frac{\mu_{n}}{c_{n}}\right)  ^{2}<\infty.
\label{kwmi/cn}%
\end{equation}
to get convergence.

Let us expand function $\psi(x+c_{n})$ at some point $x$. We have
\[
\psi(x+c_{n})=\psi(x)+c_{n}\psi^{^{\prime}}(x)+r_{n}(x),
\]
where $r_{n}(x)$ is a residue satisfying condition $\left\vert r_{n}%
(x)\right\vert \leq c_{n}^{2}L$, where $L$ is a global Lipschitz constant,
whose existence we postulated. Hence,
\[
\hat{\psi}_{n}(x)=\frac{\psi(x+c_{n})-\psi(x-c_{n})}{2c_{n}}=\psi^{^{\prime}%
}(x)+R_{n}(x),
\]
where $\left\vert R_{n}(x)\right\vert \leq c_{n}L.$

Let us notice that from the previously discussed theorems it follows that to
make the procedure (\ref{kiefer}) convergent one needs, that:%

\begin{equation}
\forall\varepsilon>0:1/\varepsilon\geq\left|  x-\theta\right|  \geq
\varepsilon\Rightarrow\exists\delta(\varepsilon):(x-\theta)\hat{\psi}%
_{n}(x)\geq\delta_{n}(\varepsilon)(x-\theta)^{2}, \label{pierwszy}%
\end{equation}

\begin{equation}
{\large \exists\kappa}_{1},\kappa_{2},{\large \forall}n\geq1{\large ,}\left|
\hat{\psi}_{n}(x)\right|  \leq\kappa_{1n}\left|  x-\theta\right|  +\kappa
_{2n}. \label{drugi}%
\end{equation}

Let us notice however, that in the first case we have for $1/\varepsilon
\geq\left\vert x-\theta\right\vert \geq\varepsilon$:
\begin{align*}
(x-\theta)\hat{\psi}_{n}(x)  &  =(x-\theta)\psi^{^{\prime}}(x)+(x-\theta
)R_{n}(x)\geq\\
&  (x-\theta)\psi^{^{\prime}}(x)-\left\vert x-\theta\right\vert ^{2}%
c_{n}L/\varepsilon.
\end{align*}
Hence, if the gradient of the function $\psi$ satisfies a condition:
\[
\forall\varepsilon>0:1/\varepsilon\geq\left\vert x-\theta\right\vert
>\varepsilon\Rightarrow\exists\delta^{^{\prime}}(\varepsilon):(x-\theta
)\psi^{^{\prime}}(x)\geq\delta^{^{\prime}}(\varepsilon)(x-\theta)^{2},
\]
then condition \ref{pierwszy} is satisfied with constant $\delta
_{n}(\varepsilon)=\delta^{^{\prime}}(\varepsilon)+\delta_{n}^{^{\prime\prime}%
}(\varepsilon)$, where $\delta_{n}^{^{\prime\prime}}(\varepsilon
)=-c_{n}L/\varepsilon$. Similarly postulate of the existence of global
Lipschitz constant implies satisfaction of the condition \ref{drugi}. Thus,
following the standard way, as in the proofs of the previous two theorems on
the convergence of stochastic approximation procedures, we reach the true
estimation for $1/\varepsilon\geq\left\vert x_{n}-\theta\right\vert
\geq\varepsilon$:
\begin{equation}
d_{n+1}^{2}\leq\lbrack1-2\mu_{n}\delta_{n}(\varepsilon)+\mu_{n}^{2}\kappa
_{1n}]^{+}d_{n}^{2}+\mu_{n}d_{n}g_{n}(\varepsilon)+\mu_{n}h_{n}(\varepsilon),
\label{kiefer-oszac}%
\end{equation}
where
\[
d_{n}=\left\vert x_{n}-\theta-G_{n}\right\vert ,G_{n}=\sum_{i\geq n}\mu
_{i}\frac{\xi_{2i+1}-\xi_{2i}}{c_{i}},
\]
and sequences $\left\{  g_{n}(\varepsilon)\right\}  $ and $\left\{
h_{n}(\varepsilon)\right\}  $ depend on the sequences $\left\{  \mu
_{n}\right\}  $, $\left\{  \delta_{n}(\varepsilon)\right\}  $, $\left\{
\kappa_{1n}\right\}  $, $\left\{  \kappa_{2n}\right\}  $ and have property:
\[
\forall\varepsilon>0:g_{n}(\varepsilon)\underset{n\rightarrow\infty
}{\longrightarrow}0;h_{n}(\varepsilon)\underset{n\rightarrow\infty
}{\longrightarrow}0
\]
with probability $1$. Thus, one can argue as in the proof of Theorem
\ref{split_delta_proc} (with sequence $\delta_{n}$ depending on $\varepsilon
)$, using Lemma \ref{nieprostyrek}, and get convergence of the sequence
$\left\{  d_{n}\right\}  $ to zero with probability $1$. In order to apply
this lemma, one has to assume, that the condition (\ref{sum-deltabis}) (i.e.
$\underset{n}{\sup}\left\vert \sum_{i=0}^{n-1}\mu_{i}\delta_{i+1}%
^{^{\prime\prime}}\right\vert <\infty$ $a.s$. and $\mu_{n}\delta
_{n+1}\rightarrow0,\;n\rightarrow\infty)$ is satisfied. Remembering the form
of the sequence $\left\{  \delta_{n}^{^{\prime\prime}}(\varepsilon)\right\}
_{n\geq1}$ it is easy to notice, that this condition will be satisfied when:
\begin{equation}
\sum_{n\geq1}\mu_{n}c_{n}<\infty. \label{summicn}%
\end{equation}

Thus, we have sketched the proof the following theorem:

\begin{theorem}
\label{kiefwolf}. Let number sequences $\left\{  \mu_{n}\right\}  _{n\geq0}$
and $\left\{  c_{n}\right\}  _{n\geq0}$ be chosen in such a way that they
satisfy the conditions (\ref{suma_kw_mi}, \ref{kwmi/cn}, \ref{summicn}). Let
us suppose also, that disturbances $\left\{  \xi_{n}\right\}  $ are such that
the condition (\ref{kwmi/cn}) guarantees convergence of the series
(\ref{szumy}). Let us assume further that the function $\psi:%
%TCIMACRO{\U{211d} }%
%BeginExpansion
\mathbb{R}
%EndExpansion
\mathbb{\rightarrow}%
%TCIMACRO{\U{211d} }%
%BeginExpansion
\mathbb{R}
%EndExpansion
$ is differentiable at every point and that its derivative satisfies the
global Lipschitz condition, and also conditions (\ref{pierwszy}, \ref{drugi})
with the selected sequence $\left\{  c_{n}\right\}  $. Then the procedure
(\ref{kiefer}) converges with probability $1$ do $\theta.$
\end{theorem}

\begin{remark}
Conditions (\ref{kwmi/cn}) and (\ref{summicn}) that are to be satisfied by the
sequences of coefficients $\left\{  \mu_{i}\right\}  _{i\geq0}$ and $\left\{
c_{i}\right\}  _{i\geq0}$ are known and appear already in the above mentioned
paper of Kiefer and Wolfowitz. Proof of Theorem \ref{kiefwolf} is of course
different. It is interesting, however, this classical theorem was proved as a
particular case of application of Theorem \ref{split_delta_proc}.
\end{remark}

\begin{remark}
We would like to remind in this place, that the procedure (\ref{kiefer}) was
an inspiration for many other authors to find the extension, generalization,
and improvements of the classical procedure. As a result, this procedure was a
germ, around which arose new branch of numerical methods namely stochastic
optimization. There exists huge literature dedicated to it. It will be
partially discussed in section \ref{dalej}.
\end{remark}

\subsection{Speed of convergence}

Let us notice that the presented so far theorems enable to examine the speed
of convergence in LLN.\emph{\ }Let us consider a recursive form of LLN, i.e.%

\begin{equation}
\overline{X}_{n+1}=(1-\mu_{n})\overline{X}_{n}+\mu_{n}X_{n+1}, \label{pwl1}%
\end{equation}
where $\left\{  X_{i}\right\}  _{i\geq1}$ is the sequence of the random
variables with zero expectation. It is known, that if the sequence $\left\{
\mu_{i}\right\}  _{i\geq0}$ satisfies the conditions:
\begin{equation}
\mu_{0}=0,\,\,\mu_{n}\in(0,1),\,\,\,\,\sum_{i\geq0}\mu_{i}=\infty, \label{mi}%
\end{equation}
then sequence $\left\{  \overline{X}_{n}\right\}  _{n\geq1}$ can be expressed
by the formula:
\begin{equation}
\overline{X}_{n}=\frac{\sum_{i=0}^{n-1}\alpha_{i}X_{i+1}}{\sum_{i=0}%
^{n-1}\alpha_{i}}, \label{sr1}%
\end{equation}
where $\alpha_{0}=1$, $\alpha_{n}=\mu_{n}/\prod_{i=1}^{n}(1-\mu_{i})$,
$n\geq1$. Now let $\left\{  \beta_{n}\right\}  _{n\geq1}$ will be some
strictly increasing number sequence. Let us notice that if we denote:
$Z_{n}=\beta_{n}\overline{X_{n}}$, then we get recurrent relationship:
\begin{equation}
Z_{n+1}=(1+\gamma_{n})(1-\mu_{n})Z_{n}+\beta_{n}\mu_{n}X_{n+1}, \label{sz_pwl}%
\end{equation}
where we denoted for symmetry of formulae $\gamma_{n}=(\beta_{n+1}-\beta
_{n})/\beta_{n}$, whose convergence can be examined by the known methods.

Using similar technic one can estimate the speed of convergence of stochastic
approximation procedures.\emph{\ }

More on this topic of speed of convergence of stochastic approximation
procedures one can find e.g. in the papers: \cite{Fabian67} and
\cite{Ruppert82}. In particular, one can find conditions to be imposed on
functions $\mathbf{f}$, under which for $\gamma<1/2$ the sequence $\left\{
n^{\gamma}(\mathbf{x}_{n}-\mathbf{\theta)}\right\}  _{n\geq1}$ converges to
zero almost surely$.$

There exist also papers dedicated to the problem of stopping stochastic
approximation procedures, that is to the following problem. Let us consider
the procedure (\ref{nowaproc}). One has to find such stopping moment (see
Appendix \ref{czas}) $\tau$, such that with probability, not less than
$\delta>0$ the following condition was satisfied:
\[
\left\vert \mathbf{x}_{\tau}-\theta\right\vert <\varepsilon,
\]
where $\varepsilon$ and $\delta$ are given beforehand numbers. Unfortunately,
satisfactory stopping rule $\tau$ was not found so far. Many attempts to find
such a stopping rule were undertaken. Their description can be found in e.g.
\cite{Farrell62}, \cite{Sielkien73} or \cite{Yin90}.

As far as connections stochastic approximation procedures with the Central
Limit Theorem is concerned\emph{, }we have the following particular result.

Let us consider one-dimensional stochastic approximation procedure:
\begin{equation}
X_{n+1}=X_{n}-\frac{a}{n+1}\left(  f(X_{n})+\xi_{n+1}\right)  , \label{as1}%
\end{equation}
with function $f$ satisfying the following condition:
\begin{align}
{\large \forall\epsilon}  &  >0:\underset{|x-\theta|>\epsilon}{\sup
}f(x)(x-\theta)>0,\label{1}\\
{\large \exists\kappa}_{1},\kappa_{2}  &  :|f(x)|\leq\kappa_{1}|x-\theta
|+\kappa_{2},\label{2}\\
\exists B  &  >0:f(x)=B(x-\theta)+\delta(x),\label{lin-zero}\\
\delta(x)  &  =o(x-\theta),\text{if }x\rightarrow\theta.\nonumber
\end{align}

\begin{theorem}
\label{ctg}Let will be given a stochastic approximation procedure (\ref{as1})
with function $f$ satisfying conditions (\ref{1}, \ref{2}, \ref{lin-zero}).
Let us suppose additionally, that $\left\{  \xi_{i}\right\}  _{i\geq0}$ is a
sequence of independent random variables with zero expectations and variances
equal to $\sigma^{2}$. \newline If
\[
aB>\frac{1}{2},
\]
then sequence random variables
\begin{equation}
\sqrt{n}(X_{n}-\theta) \label{unorm}%
\end{equation}
has asymptotically Normal distribution:
\[
N(0,\frac{a^{2}\sigma^{2}}{2aB-1}).
\]

\end{theorem}

\begin{proof}
One can find in \cite{Nevelson72}. It is not very simple and elementary that
is why we do not present it here.
\end{proof}

\subsection{Trends in developments\label{dalej}}

Observation expressed in Remark \ref{aspekty} is the base of the division of
the set of problems and methods associated with a stochastic approximation.
Namely, assuming the simple stochastic structure of noises (most often
independence, more seldom the fact, martingale difference assumption and some
assumption concerning the existence of moments) one considers more and
complicated cases of functions $\mathbf{f}$ (in Robbins-Monro version of
stochastic approximation procedures ) and also of functions $\Psi$ (in
Kiefer-Wolfowitz version). As far as Robbins-Monro version is concerned the
generalizations and extensions went mainly into the direction of considering
functions that have many zeros. Procedures look for any of them. Even more,
procedures approach the set of zeros of such functions and then in the limit
we are able to give the probability distribution on this set. There exists a
series of papers dedicated to these problems. To mention only a few more
interesting. These are first of all papers of H. Kushner \cite{Kushner72},
\cite{KushnerGavin73}, \cite{Kusher72b}, \cite{Kushner84}.

As far as methods of minimization in random conditions are concerned, i.e.
extensions of stochastic approximation procedures in Kiefer-Wolfowitz version
again there exists many papers dedicated to this problem. As it was stated
before the set of these extensions created new branch of numerical mathematics
-stochastic optimization. There exists very rich literature dedicated to this
discipline. Application of various methods of deterministic optimization turns
out very fruitful. Probably one can risk a statement, that every known and
proven optimization method (see e.g. monograph \cite{Findeisen77}) has already
its stochastic counterpart. One can mention here papers of Ruszczy\'{n}ski and
others \cite{NorRuszcz98}, \cite{PflRuszcz98}, \cite{ErmoRuszc97},
\cite{Ruszcz97}, \cite{ErmRuszc96}, \cite{RuszcSys86}, \cite{RuszcSys86a}.

Problems of boundedness of stochastic approximation procedures and also
stochastic approximation procedures utilizing information on the position of
the looked for point, that was briefly discussed in subsection
\ref{ograniczonosc}, were also generalized and extended and consist an
important part of the stochastic optimization. The problem of utilizing
existing information to search for zeros, or minima of functions are very
tempting, and moreover , obtained there results very important. It is worth to
mention that in the case of deterministic methods of minimization one
distinguishes the bounds of equality and inequality type. One distinguishes
also the fact if inequality restrictions form a convex set or not. Finally,
there exist typical methods of solving optimization problems with restrictions
such as the method of Lagrange multiplier, a method of penalty functions, a
method of admissible directions. The point is that the methods of stochastic
optimization are classified in the same way as the deterministic ones. One has
developed stochastic counterparts of the above mentioned method, of solving
these problems. Thus, the detailed discussion of stochastic optimization
extends far beyond this book. We refer the interested reader to e.g.
\cite{Ruszc84}, \cite{KushnerClark78}, \cite{Kushner84},
\cite{KushnerSanvincente74}, \cite{KushnerGavin74}, \cite{Ruszc80}.

Finally, let us mention also, that there exists also other direction of
generalization and extension stochastic approximation procedures. Namely, we
mean stochastic approximation procedures in infinite dimensional spaces. Such
procedures are most often constructed and used to find functions having
defined properties: finding zero of a mapping in some functional space into
itself (in the case of generalized Robbins-Monro procedure) or minimizing
functional\ (in the case of generalization of Kiefer-Wolfowitz procedure).
Historically, it was already Dvoretzki in 1956 in the paper \cite{Dvoretzky56}
considered a similar situation.\ Further, one should mention papers of
Schmetterer \cite{Schmetterer58}, Venter \cite{Venter67} or Yin \cite{Yin92}.
Problems of convergence of such procedures are difficult and will not be
presented here. There exists, however, one exception. Namely, we mean
iterative procedures of density and regression estimation and also some
iterative procedures of identification. We will dedicate to these problems in
the next two chapters. It turns out that although procedures of density and
regression estimation constitute a separate branch of nonparametric
estimation, we can view their iterative versions as stochastic approximation
procedures in Robbins-Monro versions acting in infinite dimensional spaces.

\chapter{Density and Regression estimation\label{metody_jadrowe}}

In this chapter, the so-called kernel methods of density and regression
estimation are discussed.

\section{Basic ideas}

Any density function will be called \emph{kernel }%
\index{Kernel}%
. Let hence $K(x)$ be any kernel. Function $F_{y,h}(x)=\frac{1}{h}K(\frac
{x-y}{h})$ has the following property:
\begin{equation}
{\large \forall}y\in%
%TCIMACRO{\U{211d} }%
%BeginExpansion
\mathbb{R}
%EndExpansion
,h>0,\,\,\int_{%
%TCIMACRO{\U{211d} }%
%BeginExpansion
\mathbb{R}
%EndExpansion
}F_{y,h}(x)dx=1, \label{calka}%
\end{equation}
hence is also a kernel. If $h<1$, then the plot of the function $F_{y,h}$ is,
if compared with the plot of the function $K$, shifted by $y$ and
\textquotedblright restricted to values in the neighborhood of the point
$y$\textquotedblright\ i.e. e.g., in the case when the support of the density
$K$ is bounded, then the support of the function $F_{y,h}$ is a subset of
support of the function $K$. e.g. if $H(x)=\left\{
\begin{array}
[c]{ccc}%
0 & if & x<0\\
1 & if & x\geq0
\end{array}
\right.  $
\begin{equation}
K(x)=\left\{
\begin{array}
[c]{ccc}%
1-\left\vert x\right\vert  & \;for & \left\vert x\right\vert \leq1\\
0 & \;for & \left\vert x\right\vert >1
\end{array}
\right.  \label{jadro}%
\end{equation}%
%TCIMACRO{\FRAME{dtbpF}{2.4872in}{1.5437in}{0pt}{}{}{k3.eps}%
%width 2.4872in;  height 1.5437in;  depth 0pt;  original-width 3.6115in;
%original-height 2.2312in;  cropleft "0";  croptop "1";  cropright "1";
%cropbottom "0";  filename 'K3.eps';file-properties "XNPEU";}}}%
%BeginExpansion
\begin{center}
\includegraphics[
height=1.5437in,
width=2.4872in
]%
{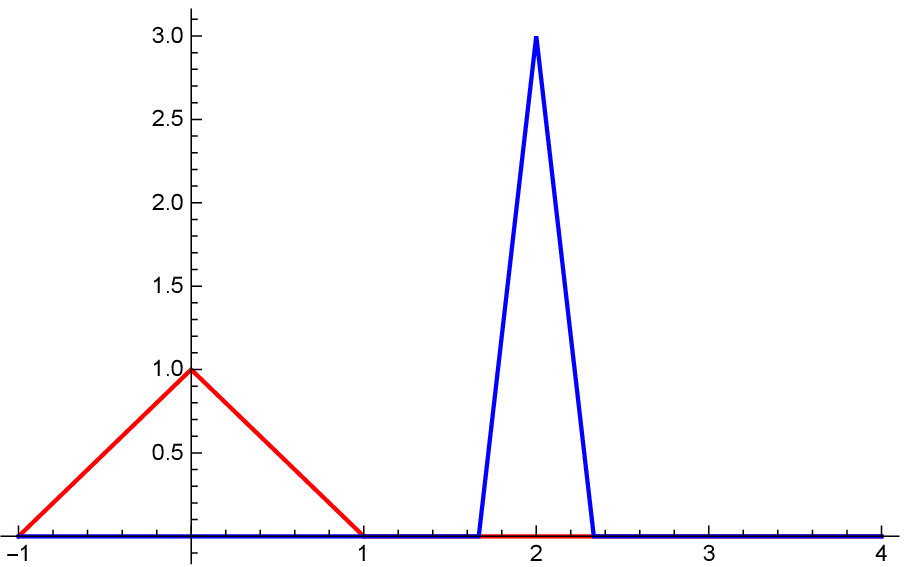}%
\end{center}
%EndExpansion

i.e. it has the plot as on the above figure drawn in red, while the function
$3K(3(x-2))$ has the plot presnted above in blue.

\begin{example}
In the sequel, the following general theorem will be of use.
\end{example}

\begin{theorem}
\label{wlasnosci_jadra} Let $f$ and $g$ be two Lebesgue integrable functions. Then

$i)\qquad\int_{%
%TCIMACRO{\U{211d} }%
%BeginExpansion
\mathbb{R}
%EndExpansion
}\left\vert f(x)g(y-x)\right\vert dx\leq\int_{%
%TCIMACRO{\U{211d} }%
%BeginExpansion
\mathbb{R}
%EndExpansion
}\left\vert f(x)\right\vert dx\int_{%
%TCIMACRO{\U{211d} }%
%BeginExpansion
\mathbb{R}
%EndExpansion
}\left\vert g(x)\right\vert dx$ (Young inequality )%
\index{Inequality!Young}%
,

and moreover, if additionally we assume, that $\int_{%
%TCIMACRO{\U{211d} }%
%BeginExpansion
\mathbb{R}
%EndExpansion
}g(x)dx=1$, then

$ii)\qquad\underset{h\downarrow0}{\lim}\int_{%
%TCIMACRO{\U{211d} }%
%BeginExpansion
\mathbb{R}
%EndExpansion
}\left|  \int_{%
%TCIMACRO{\U{211d} }%
%BeginExpansion
\mathbb{R}
%EndExpansion
}\frac{1}{h}f(x)g(\frac{y-x}{h})dx-f(y)\right|  dy=0,$

If additionally the function $\hat{g}(x)=\underset{\left\vert y\right\vert
\geq\left\vert x\right\vert }{\sup}\left\vert g(y)\right\vert $ is integrable, then

$iii)\qquad\underset{h\downarrow0}{\lim}\int_{%
%TCIMACRO{\U{211d} }%
%BeginExpansion
\mathbb{R}
%EndExpansion
}\frac{1}{h}f(y)g(\frac{y-x}{h})dy=f(x)$ for almost all $x\in%
%TCIMACRO{\U{211d} }%
%BeginExpansion
\mathbb{R}
%EndExpansion
$.
\end{theorem}

\begin{proof}
Can be found in the book\cite{devroy} (see theorem. 1 on page 16). The proof
is not probabilistic and is based on (particularly on the assertion
\textit{iii)) }of the\textit{\ }Lebesgue Theorem%
\index{Theorem!Lebesgue on points of continuity}
on density points (see, e.g. textbook of \L ojasiewicz \cite{lojasiewicz}) and
that is why we will not give it here.
\end{proof}

We have also the following generally theorem, being in fact version Lemma
Scheff\'{e}'s (see Appendix \ref{scheffe})

\begin{theorem}%
\index{Theorem!Glick}%
(Glick) Let $\{f_{n}\}$ be a sequence of density estimators of the density
$f$. If $f_{n}\rightarrow f$ \ in probability (with probability $1$) for
almost all $x$ as $n\rightarrow\infty$, then $\int_{%
%TCIMACRO{\U{211d} }%
%BeginExpansion
\mathbb{R}
%EndExpansion
}\left\vert f_{n}(x)-f(x)\right\vert dx\rightarrow0$ in probability (with
probability $1$) as $n\rightarrow\infty.$
\end{theorem}

\begin{proof}
On can to find in the paper \cite{Glick74}. It is very similar to the proof of
Scheff\'{e}'s Lemma presented in the Appendix \ref{scheffe}.
\end{proof}

Let us fix some kernel $K(x)$. Let us suppose now, that we make $N$
observations of some random variable $X$ having a density $f(x)$, obtaining a
number sequence $x_{1},\ldots,x_{N}$. Let $h=h(N)$ be some sequence of
positive numbers, decreasing to zero together with $N$. Let us consider
function
\begin{equation}
\tilde{f}_{N}(y)=\frac{1}{N}\sum_{i=1}^{N}\frac{1}{h}K\left(  \frac{y-x_{i}%
}{h}\right)  . \label{est-gest}%
\end{equation}
It is a density, since we have%

\[
{\large \forall}N\in%
%TCIMACRO{\U{2115} }%
%BeginExpansion
\mathbb{N}
%EndExpansion
,\,\,y\in%
%TCIMACRO{\U{211d} }%
%BeginExpansion
\mathbb{R}
%EndExpansion
:\tilde{f}_{N}(y)\geq0,\;\,\,\int_{%
%TCIMACRO{\U{211d} }%
%BeginExpansion
\mathbb{R}
%EndExpansion
}\tilde{f}_{N}(y)dy=1.
\]
In order to analyze relationship of this function with the density function
$f$, let us consider the problem from the probabilistic point of view. Let be
given sequence $X_{1},\ldots,X_{N}$ of \emph{independent random variables
having the same distribution} with the density $f$. Let us consider random
variables:
\begin{equation}
f_{N}(y)=\frac{1}{N}\sum_{i=1}^{N}\frac{1}{h(N)}K\left(  \frac{y-X_{i}}%
{h(N)}\right)  . \label{los-est}%
\end{equation}
It is clear, that for $\forall n\in%
%TCIMACRO{\U{2115} }%
%BeginExpansion
\mathbb{N}
%EndExpansion
$ and $\forall y\in%
%TCIMACRO{\U{211d} }%
%BeginExpansion
\mathbb{R}
%EndExpansion
$ : $f_{N}(y)\geq0$ with probability $1$. Moreover, $\int_{%
%TCIMACRO{\U{211d} }%
%BeginExpansion
\mathbb{R}
%EndExpansion
}f_{N}(y)dy=1$ with probability $1$. $f_{N}(y)$ is a random variable, whose
one of the realizations is $\tilde{f}_{N}(y).$

Let us calculate sequence $\left\{  \phi_{N}\right\}  $ Fourier
transformations of functions $f_{N}(y)$. We have:
\[
\phi_{N}(t)=\int_{%
%TCIMACRO{\U{211d} }%
%BeginExpansion
\mathbb{R}
%EndExpansion
}f_{N}(y)\exp(ity)dy=\frac{1}{N}\sum_{i=1}^{N}\int_{%
%TCIMACRO{\U{211d} }%
%BeginExpansion
\mathbb{R}
%EndExpansion
}\frac{1}{h(N)}K\left(  \frac{y-X_{i}}{h(N)}\right)  \exp(ity)dy,
\]
but
\[
\int_{%
%TCIMACRO{\U{211d} }%
%BeginExpansion
\mathbb{R}
%EndExpansion
}\frac{1}{h(N)}K\left(  \frac{y-X_{i}}{h(N)}\right)  \exp(ity)dy\allowbreak
\allowbreak=\allowbreak\int_{%
%TCIMACRO{\U{211d} }%
%BeginExpansion
\mathbb{R}
%EndExpansion
}K(z)\exp\left(  itX_{i}+itzh(N)\right)  dz.
\]
Denoting $\varphi(t)=\int_{%
%TCIMACRO{\U{211d} }%
%BeginExpansion
\mathbb{R}
%EndExpansion
}K(z)\exp(itz)dz$, we get:
\[
\phi_{N}(t)=\varphi(th(N))\frac{1}{N}\sum_{i=1}^{N}\exp(itX_{i}).
\]
Let us notice that $\forall t\in%
%TCIMACRO{\U{211d} }%
%BeginExpansion
\mathbb{R}
%EndExpansion
:\varphi(th(N))\underset{N\rightarrow\infty}{\longrightarrow}\varphi
(0)\allowbreak=\allowbreak\int K(z)dz\allowbreak=\allowbreak1$ (since $K$ is
density and $h(N)\rightarrow0$, when $N\rightarrow\infty$). Moreover, taking
into account, that random variables $\left\{  X_{i}\right\}  _{i\geq1}$ are
independent, they satisfy LLN in version of Kolmogorov's (see theorem.
\ref{kolmogor}) and so we see that
\[
\forall t\in%
%TCIMACRO{\U{211d} }%
%BeginExpansion
\mathbb{R}
%EndExpansion
:\frac{1}{N}\sum_{i=1}^{N}\exp(itX_{i})\underset{N\rightarrow\infty
}{\longrightarrow}\varphi_{X}(t)
\]
almost surely, where by $\varphi_{X}(t)$ we denoted characteristic function of
the random variable $X_{1}$. Thus, the sequence of the random
variables$\left\{  f_{n}(y)\right\}  _{n\geq1}$ converges for almost every
$\omega$ in the distributive sense (as a function of $y$) to the distribution
of the random variable $X_{1}$(see Appendix \ref{dystrybucje} particularly
Theorem \ref{zbiez_dystrybucji}). It means, e.g., that
\begin{equation}
{\large \forall}x{\large \in}%
%TCIMACRO{\U{211d} }%
%BeginExpansion
\mathbb{R}
%EndExpansion
:\int_{-\infty}^{x}f_{n}(y)dy\underset{n\rightarrow\infty}{\longrightarrow
}\int_{-\infty}^{x}f(y)dy\text{ almost surely.} \label{zbiez_dystr}%
\end{equation}
It turns out that there exists a rich literature concerning density estimation
and one can give deeper and more detailed theorem on convergence.

\begin{remark}
Let us notice that in order to show weak convergence (i.e. in fact,
convergence of characteristic functions) of the sequence of densities to
limiting density, one does not have to assume independence of observations. As
it turned out from the above calculations, it was enough, that the law of
large numbers was satisfied for random variables $\left\{  Y_{i}=\exp
(itX_{i})\right\}  _{i\geq1}$ for every $t\in%
%TCIMACRO{\U{211d} }%
%BeginExpansion
\mathbb{R}
%EndExpansion
$. Further, to get this law of large numbers, satisfied it is enough (as it
follows e.g. from Theorem \ref{mpwl}), that covariances $\operatorname*{cov}%
(Y_{i},Y_{j})$ decreased with $\left\vert i-j\right\vert $ sufficiently
quickly to zero. How to check this, depends on the particular form of the
sequence of the random variables $\left\{  X_{i}\right\}  _{i\geq1}$. In any
case it is enough only of two-dimensional distributions of this sequence.
\end{remark}

\begin{remark}%
\index{Histogram}%
The other way of density estimation, mentioned in section \ref{sek_ctg}, is
the estimation with the help of histograms. The histogram can be obtained in
the following way. Let us assume, that we are interested in estimating the
density of the random variable $X$. To this end

a) we observe $N$ independent realizations of this random variable getting
values $x_{1}$, $x_{2}$, $\ldots$, $x_{N}.$

b) we divide the interval of variability of the random variable $X$ on
$k\geq2$ disjoint subintervals with the help of points $y_{1},y_{2},\ldots$,
$y_{k-1}$. Next we count how many points among $x_{1},\ldots,x_{N}$ fell into
every of the subintervals $\Delta_{j}\overset{df}{=}<y_{j-1},y_{j})$,
$j=1,\ldots,k$, where we assumed for simplicity $y_{0}=-\infty$ and
$y_{k}=\infty$. In other words, let us calculate: numbers $n_{j}=\#\left\{
x_{i}:x_{i}\in\Delta_{j}\right\}  $. Histogram it is a step that on
$\Delta_{j}$ assumes value $\frac{n_{j}}{N}$. In other words
$Histogram(y)\allowbreak=\allowbreak\sum_{j=1}^{k}\frac{n_{j}}{N}I(\Delta
_{j})(y).$

It is not difficult to notice, that the better histogram approximates the
density of a random variable the greater must be the number of observations
$N$, and the number of intervals $k$. However the ratio $N/k$ should also be
great. The point is that every one of the intervals $\Delta_{j}$ should
contain sufficiently many observations (there should be satisfied modified law
of large numbers).

A drawback of the density estimation with the help histograms is, of course,
the fact, that the histogram is a step function, hence discontinuous. Much
better results we get using kernel methods described in this chapter.
\end{remark}

Let us start by analyzing a few examples. In each of them, the estimation was
based on $n=5000$ simulations.

\begin{example}
We assume in this example $n=5000$, $h(n)=n^{-.5}$. The estimated density was
the density of the uniform distribution on the segment $<0,1>$, density
distribution $U(0,1)$. Two estimation were done : the first with the kernel
given by the formula (\ref{jadro}) (plotted in red), and the second with the
so-called Cauchy kernel, that is the function $\frac{1}{\pi(1+x^{2})}$ was
taken to be a kernel (was plotted in blue).
%TCIMACRO{\FRAME{itbpF}{3.0649in}{1.977in}{0in}{}{}{kernel_uniform1.eps}%
%width 3.0649in;  height 1.977in;  depth 0in;  original-width 8.8323in;
%original-height 5.6818in;  cropleft "0";  croptop "1";  cropright "1";
%cropbottom "0";  filename 'kernel_uniform1.eps';file-properties "XNPEU";}}}%
%BeginExpansion
{\includegraphics[
height=1.977in,
width=3.0649in
]%
{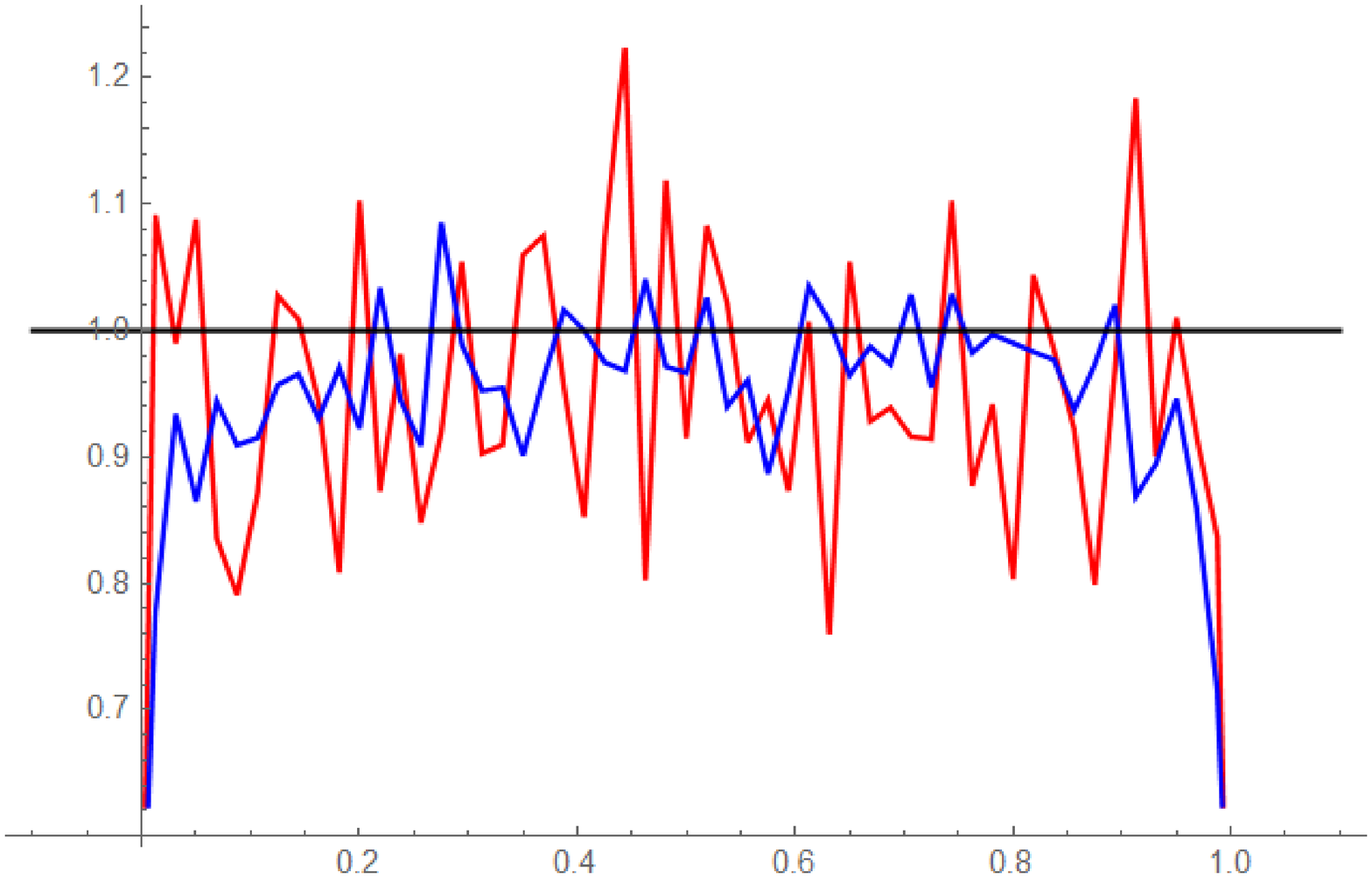}%
}%
%EndExpansion
.

As one can see an improvement in the quality of estimators one could be observed.
\end{example}

\begin{example}
In this example density function of the exponential distribution $Exp(1)$,
that is function $\exp(-x)$ for $x\geq0$ was estimated. Parameter $n$ and
kernels were the same, as in the previous example. One set $h(n)=n^{-.4}$. As
in the previous example, estimator obtained with the help of triangular kernel
was plotted in red, while in blue the one obtained with the help Cauchy kernel.%

%TCIMACRO{\FRAME{itbpF}{2.6368in}{1.6146in}{0in}{}{}{kernel_exp1.eps}%
%width 2.6368in;  height 1.6146in;  depth 0in;  original-width 8.6559in;
%original-height 5.2831in;  cropleft "0";  croptop "1";  cropright "1";
%cropbottom "0";  filename 'kernel_exp1.eps';file-properties "XNPEU";}}}%
%BeginExpansion
{\includegraphics[
height=1.6146in,
width=2.6368in
]%
{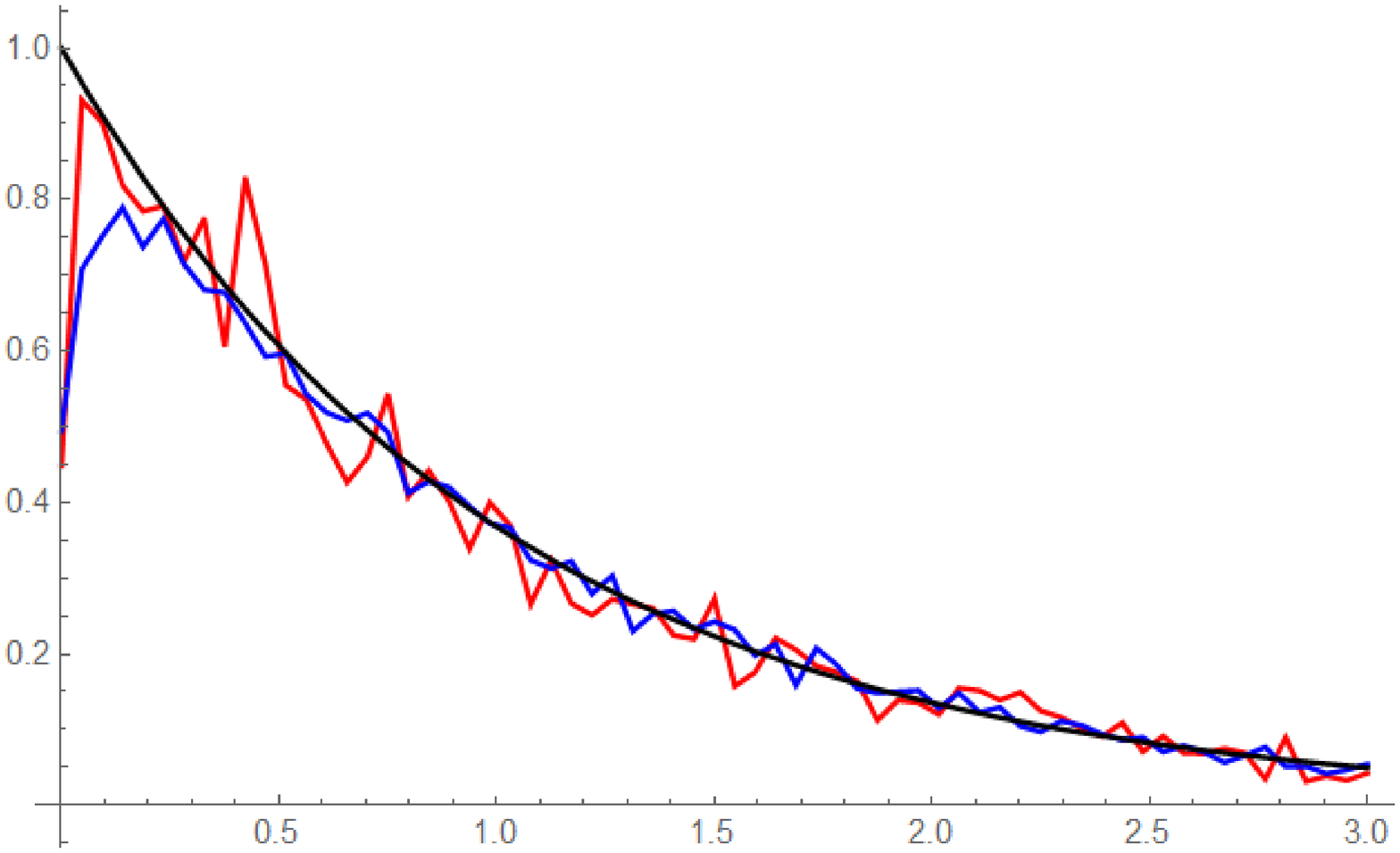}%
}%
%EndExpansion

\end{example}

\begin{example}
In this example density function of the arc sinus distribution that is the
function $\frac{1}{\pi\sqrt{1-x^{2}}}$ for $\left\vert x\right\vert <1$ was
estimated. Parameter $n$ and kernels were the same, as in the previous
example. One took $h(n)=n^{-.3}$. As in the previous example in red was
plotted estimator obtained with the help of triangular kernel, while in blue
with the help Cauchy kernel.%

%TCIMACRO{\FRAME{itbpF}{2.5097in}{1.6233in}{0in}{}{}{kernel_arcsin1.eps}%
%width 2.5097in;  height 1.6233in;  depth 0in;  original-width 7.5005in;
%original-height 4.8369in;  cropleft "0";  croptop "1";  cropright "1";
%cropbottom "0";  filename 'kernel_arcsin1.eps';file-properties "XNPEU";}}}%
%BeginExpansion
{\includegraphics[
height=1.6233in,
width=2.5097in
]%
{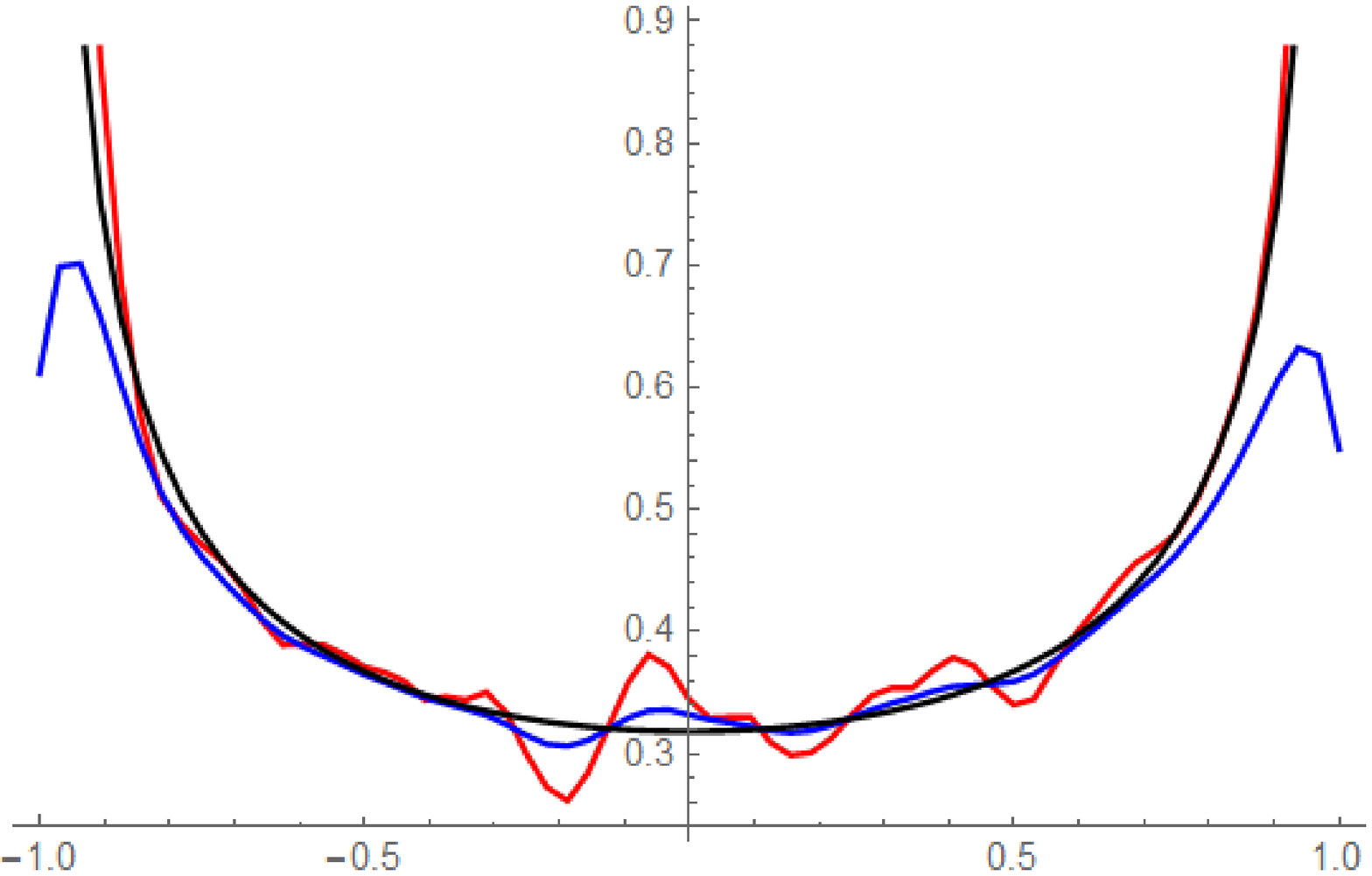}%
}%
%EndExpansion

\end{example}

\begin{remark}
Let us notice that calculation and reasoning used to justify meaningfulness of
the kernel estimator (\ref{los-est}) is universal in this sense, that it
refers also to random variables not having densities. This argumentation can
be the base for considerations of kernel estimators of cumulative distribution
functions. Namely, denoting by $F_{K}$ the cumulative distribution function
kernel $K$ i.e. $F_{K}(x)=\int_{-\infty}^{x}K(z)dz$, we get from formula
(\ref{zbiez_dystr})
\[
\frac{1}{n}\sum_{i=1}^{n}F_{K}\left(  \frac{x-X_{i}}{h_{n}}\right)
\underset{n\rightarrow\infty}{\rightarrow}F_{X}(x),
\]
with probability $1$, where $F_{X}$ denotes cumulative distribution function
random variable $X_{1}$. As simulations show, this method is good, efficient
and, as it was mentioned, universal. As it seems the first one, who
noticed\ the possibilities embedded in this method of cdf estimation
was\ Azzalini (1981) (see \cite{Azzalini81}). It seems also, that this method
is rather a little known and requires research. \newline We will illustrate it
by the following example. One took $N=6000$ observations discrete random variable%

\[
X=\left\{
\begin{array}
[c]{ccc}%
-1 & \text{with probability} & 1/8\\
0 & \text{with probability} & 4/8\\
2 & \text{with probability} & 2/8\\
3 & \text{with probability} & 1/8
\end{array}
\right.  .
\]
We took either $F_{K}(x)\allowbreak=\allowbreak1/2+\arctan(x)/\pi$, or
\[
F_{K}=\left\{
\begin{array}
[c]{ccc}%
0 & for & x<-\sqrt{5}\\
\frac{3\sqrt{5}}{20}\left(  x-\frac{1}{15}x^{3}+\frac{2\sqrt{5}}{3}\right)  &
for & -\sqrt{5}\leq x<\sqrt{5}\\
1 & for & x\geq\sqrt{5}%
\end{array}
\right.  ,
\]
i.e. so-called Epanechnikov's kernel. The results were presented in the figure
below. Here estimator with Cauchy kernel is plotted in red, estimator with
Epanechnikov's kernel was plotted in blue, and cdf of the random variable $X$
was plotted in black. One took $h(N)=N^{-.4}$%

%TCIMACRO{\FRAME{itbpF}{2.783in}{1.7487in}{0in}{}{}{kernel_dystr3.eps}%
%width 2.783in;  height 1.7487in;  depth 0in;  original-width 7.5005in;
%original-height 4.6994in;  cropleft "0";  croptop "1";  cropright "1";
%cropbottom "0";  filename 'kernel_dystr3.eps';file-properties "XNPEU";}}}%
%BeginExpansion
{\includegraphics[
height=1.7487in,
width=2.783in
]%
{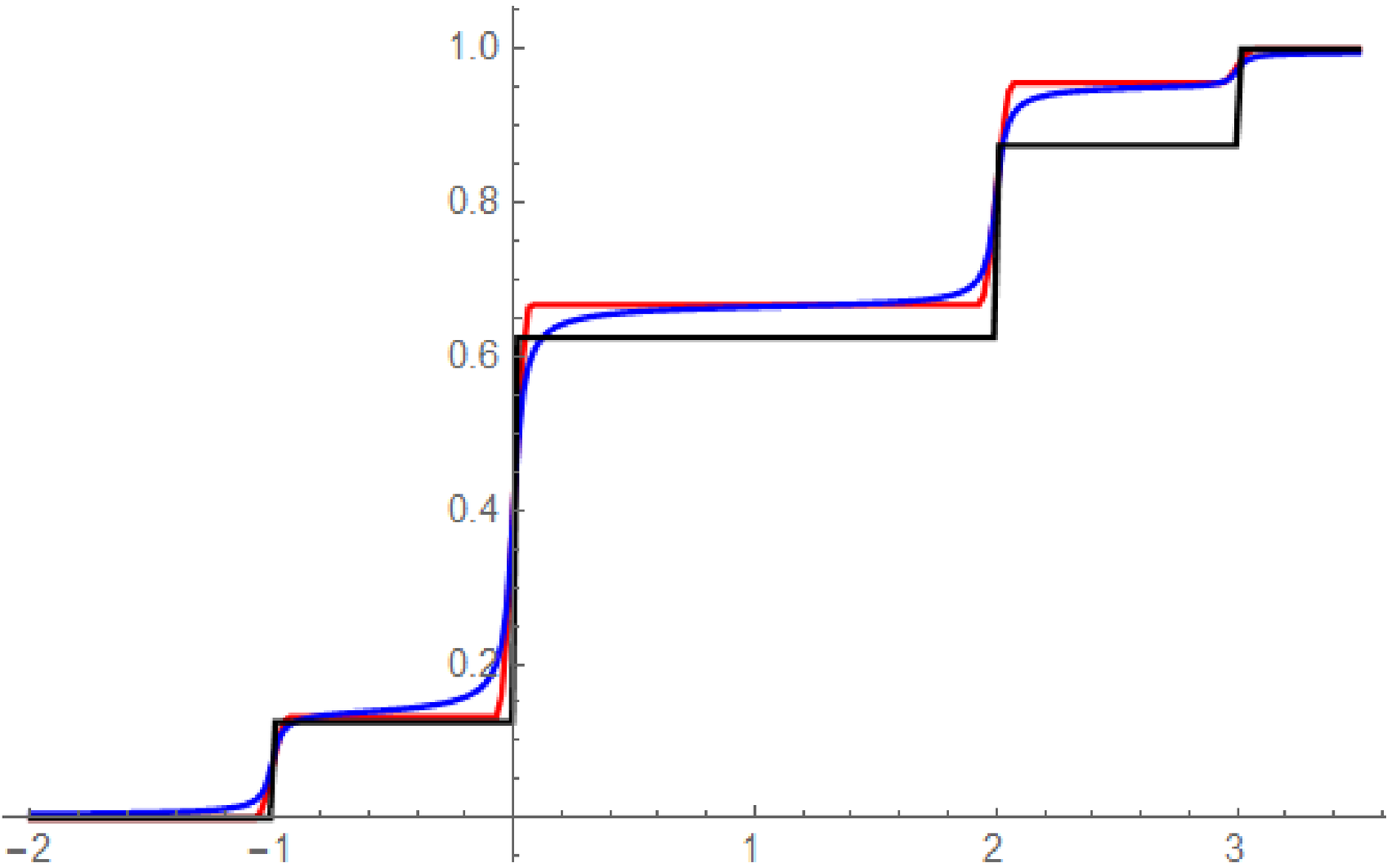}%
}%
%EndExpansion
.
\end{remark}

\subsection{Properties of the basic estimator}

Since the main aim of this chapter is not the exhausting presentation of
methods of density estimation, but only to indicate that there exists strong
connections of these issues with problems in the laws of large numbers and
relatively precise analysis of this variant of the density estimation method,
that can be presented in the iterative form. Consequently , we will very
briefly present main results of more than 30 years of research dedicated to
density estimation. A series of books as well as long review articles
dedicated to this problem was written. Density estimation methods, as it turns
out, one split into two big sets. Basing on \emph{mean square error}
\[
MISE(h,n)=E\int\left(  f_{n}(y)-f(y)\right)  ^{2}dy,
\]
and basing on the so-called \emph{\ }$L_{1}-$\emph{error}%
\[
MI(L_{1})E=E\int\left\vert f_{n}(y)-f(y)\right\vert dy.
\]
Of course, one can consider other metrics in functional spaces and there exist
papers considering them, but the two metrics defined by the above mentioned
formulae are the most important and about 99\% of the literature is dedicated
to them. Among the works on density estimation basing on $MISE$ let us mention
monograph of Silverman \cite{Silverman86} and a few papers among them, the
eldest as well as a few the newest since they contain references to the
earlier appears: \cite{Rosenblatt56},$\allowbreak$\cite{Parzen62}%
,$\allowbreak$\cite{Koronacki94}, \cite{SheaHettDo94}, \cite{Terrell92},
\cite{Terrell80}, \cite{Mammen97a}, \cite{Park98}, \cite{Wand95}.

As far as the second measure of errors one has to mention nomography of
\cite{devroy}.

We will now present a few general properties of this estimation method, and
then briefly some results concerning the optimal choice of the sequence of
coefficients $\left\{  h(N)\right\}  _{N\geq1}$ i.e. the so-called
\textquotedblright window width\textquotedblright\ (or \textquotedblright band
width\textquotedblright) and of the form of the kernel.

Let $X_{1},\ldots,X_{n}$ be a sequence $n$ i.i.d. random variables, having
density $f$. Previously we considered estimator%
\index{Estimator!basic of the density}
\begin{equation}
f_{n}(y)=\frac{1}{n}\sum_{i=1}^{n}\frac{1}{h_{n}}K\left(  \frac{y-X_{i}}%
{h_{n}}\right)  , \label{podstaw}%
\end{equation}
where $\left\{  h_{n}\right\}  $ is non-increasing number sequence, convergent
to zero. In $d$-dimensional version this estimator has the following form:
\begin{equation}
f_{n}(\mathbf{y})=\frac{1}{n}\sum_{i=1}^{N}\frac{1}{h_{n}^{d}}K\left(
\frac{\mathbf{y}-\mathbf{X}_{i}}{h_{n}^{d}}\right)  , \label{podstaw1}%
\end{equation}

where $\mathbf{y\in}%
%TCIMACRO{\U{211d} }%
%BeginExpansion
\mathbb{R}
%EndExpansion
^{d}$ a $\left\{  \mathbf{X}_{1},\ldots,\mathbf{X}_{n}\right\}  $ is a simple
random sample of $d-$dimensional vectors with the density $f(\mathbf{x).}$

The basic properties of this estimator are presented by the following lemmas
and Theorems.

\begin{lemma}
\label{wlasnosci_estjadr}Estimator (\ref{podstaw1}) has the following
properties:\newline\emph{i)}
\begin{align*}
Ef_{n}(\mathbf{y})  &  =\int\frac{1}{h_{n}^{d}}K\left(  \frac{\mathbf{y}%
-\mathbf{x}}{h_{n}}\right)  f(\mathbf{x})d\mathbf{x}=\int K(\mathbf{x}%
)f(\mathbf{y}-h_{n}^{d}\mathbf{x})d\mathbf{x},\\
b(\mathbf{y})\overset{df}{=}Ef_{n}(\mathbf{y})-f(\mathbf{y})  &  =\int
K(\mathbf{x})\left(  f(\mathbf{y}-h_{n}^{d}\mathbf{x})-f(\mathbf{y})\right)
d\mathbf{x}%
\end{align*}
\emph{ii)}
\begin{align*}
\operatorname*{var}(f_{n}(\mathbf{y}))  &  =\frac{1}{n}\left[  \int\frac
{1}{h_{n}^{2d}}K^{2}\left(  \frac{\mathbf{y}-\mathbf{x}}{h_{n}^{d}}\right)
f(\mathbf{x})d\mathbf{x}-\left(  Ef_{n}(\mathbf{y})\right)  ^{2}\right] \\
&  =\frac{1}{n}\left[  \int\frac{1}{h_{n}^{d}}K^{2}(\mathbf{x})f(\mathbf{y}%
-\mathbf{x}h_{n}^{d})d\mathbf{x}-\left(  \int K(\mathbf{x})f(\mathbf{y}%
-h_{n}^{d}\mathbf{x})d\mathbf{x}\right)  ^{2}\right]  .
\end{align*}
\newline\emph{iii) }
\[
MISE(h,n)=E\int\left(  f_{n}(\mathbf{y})-f(\mathbf{y})\right)  ^{2}%
d\mathbf{y}+\int\left(  \operatorname*{var}(f_{n}(\mathbf{y}))+b^{2}%
(\mathbf{y})\right)  d\mathbf{y}%
\]

\end{lemma}

\begin{proof}
First equalities of both assertions one gets on the basis of assumptions on
the sameness of distributions of variables $X_{1},\ldots,X_{n}$.\ Further,
equalities follow from an elementary change of variables in respective
integrals and from the fact, that the mean square error is equal to the sum of
the variance and the square of bias.
\end{proof}

\begin{theorem}
[Devroye]%
\index{Theorem!Devroy}%
\cite{devroy}Let $K$ be any kernel on $%
%TCIMACRO{\U{211d} }%
%BeginExpansion
\mathbb{R}
%EndExpansion
^{d}$. Let us denote $J_{n}=\int_{%
%TCIMACRO{\U{211d} }%
%BeginExpansion
\mathbb{R}
%EndExpansion
}\left\vert f_{n}(y)-f(y)\right\vert dy$. Then the following statements are
equivalent: \newline\emph{i) }$J_{n}\underset{n\rightarrow\infty
}{\longrightarrow}0$ mod $P$, for some density $f$,\newline\emph{ii) }%
$J_{n}\underset{n\rightarrow\infty}{\longrightarrow}0$ mod $P$, for any
density $f,$\newline\emph{iii) }$J_{n}\underset{n\rightarrow\infty
}{\longrightarrow}0$ almost surely, for any $f$,\newline\emph{iv) }%
$J_{n}\underset{n\rightarrow\infty}{\longrightarrow}0$ exponentially (i.e. for
any $\varepsilon>0$ there exist such $r$ and $n_{0}$, that $P(J_{n}%
>\varepsilon)\leq\exp(-rn)$, for any $n\geq n_{0})$ for any $f$,\newline%
\emph{v) }$\underset{n\,\rightarrow\infty}{\lim}h_{n}=0$,
$\underset{n\,\rightarrow\infty}{\lim}nh_{n}^{d}=\infty.$
\end{theorem}

\begin{proof}
Proof of this theorem is somewhat long and not very instructive It can be
found in the book\emph{L. Devroye, L. Gy\"{o}rfi: }\cite{devroy}.
\end{proof}

Further considerations, we will lead for the sake of clarity for the
one-dimensional case. We will select the best bandwidth and best kernel among
all kernels satisfying the following conditions:%

\[
\int tK(t)dt=0,\;\;\int t^{2}K(t)dt\overset{df}{=}\kappa^{2}<\infty,\;\;\int
K^{2}(t)dt<\infty
\]
i.e. kernels having zero mean, possessing variances and \textquotedblright
square integrable\textquotedblright.

It turns out that in order to be able to talk about optimal bandwidth one has
to assume, that estimated density is smooth, more precisely has square
integrable, continuous second derivative.

Hence let us suppose, that the estimated density $f$ has continuous second
derivative, i.e. one can expand $f$ at any point in Taylor series in the
following way:
\[
f(y-xh_{n})=f(y)-xh_{n}f^{\prime}(y)+\frac{1}{2}x^{2}h_{n}^{2}f^{\prime\prime
}(y)+o(x^{2}h_{n}^{2}).
\]
Using this expansion and basing on Lemma \ref{wlasnosci_estjadr} we get:
\begin{align*}
b(y)  &  =-h_{n}f^{\prime}(y)\int tK(t)dt+\frac{1}{2}h_{n}^{2}f^{\prime\prime
}(y)\kappa^{2}+o(h_{n}^{2})=\\
&  =\frac{1}{2}h_{n}^{2}f^{\prime\prime}(y)\kappa^{2}+o(h_{n}^{2}),\\
\operatorname*{var}(f_{n}(y))  &  =\frac{1}{n}\left[  \frac{1}{h_{n}}f(y)\int
K^{2}(t)dt+f^{\prime}(y)\int tK^{2}(t)dt+o(1)-\left(  f(y)+O(h_{n}%
^{2})\right)  ^{2}\right]  =\\
&  =\frac{1}{nh_{n}}f(y)\int K^{2}(t)dt+O(\frac{1}{n}).
\end{align*}
Summarizing, we have the following statement.

\begin{proposition}
\label{optymalne}Asymptotically (i.e. for large $n$) the best (in the sense of
minimum of mean square error) bandwidth is
\begin{equation}
h_{\min}=\left[  \frac{\int K^{2}(t)dt}{n\kappa^{4}\int\left(  f^{\prime
\prime}(y)\right)  ^{2}dy}\right]  ^{1/5}. \label{min_szer}%
\end{equation}
The best (in the same sense) kernel, is the Epanechnikov's kernel $K_{E}:$%
\begin{equation}
K_{E}(t)=\left\{
\begin{array}
[c]{ccc}%
\frac{3}{4\sqrt{5}}(1-\frac{1}{5}t^{2}) & ,if & \left\vert t\right\vert
\leq\sqrt{5}\\
0 & ,if & \left\vert t\right\vert >\sqrt{5}%
\end{array}
\right.  . \label{epanecznikov}%
\end{equation}

\end{proposition}

\begin{proof}
[Sketch of the proof]Basing on \ the above-mentioned calculations, we have
approximately:
\begin{align}
\operatorname*{var}(f_{n}(y))+b^{2}(y)  &  \cong\frac{1}{nh_{n}}f(y)\int
K^{2}(t)dt+\frac{1}{4}h_{n}^{4}\left(  f^{\prime\prime}(y)\right)  ^{2}%
\kappa^{4},\nonumber\\
MISE(h,n)  &  \cong\frac{1}{nh_{n}}\int K^{2}(t)dt+\frac{1}{4}h_{n}^{4}%
\kappa^{4}\int\left(  f^{\prime\prime}(y)\right)  ^{2}dy. \label{mise}%
\end{align}
As it is easy to check by differentiating, the band with $h$ minimizing the
above mentioned expression is (\ref{min_szer}). Let us put now this quantity
to (\ref{mise}). We get then
\[
\frac{5}{4}n^{-4/5}\left(  \int\left(  f^{\prime\prime}(y)\right)
^{2}dy\right)  ^{1/5}\kappa^{4/5}\left(  \int K^{2}(t)dt\right)  ^{4/5}.
\]
As it can be seen, the quantity
\[
\kappa^{4/5}\left(  \int K^{2}(t)dt\right)  ^{4/5}=\left(  \kappa\int
K^{2}(t)dt\right)  ^{4/5}.
\]
depends on of the form of the kernel. Hence, the best kernel would minimize
the quantity%
\begin{equation}
\kappa\int K^{2}(t)dt. \label{miarajadra}%
\end{equation}
Let us recall now, that when random variable $X$ has a density $f_{X}(x)$,
expectation $m$ and variance $\sigma^{2}$, then the random variable
$Y=\frac{X-m}{\sigma}$ has the density $f_{Y}(y)\allowbreak=\allowbreak\sigma
f_{X}(m+\sigma y)$, expectation zero and variance equal to $1$. Let us denote
$\tilde{K}(y)=\kappa K(\kappa y)$. We have%

\[
\int\left[  \tilde{K}(y)\right]  ^{2}dy=\int\kappa^{2}K^{2}(\kappa
y)dy=\kappa\int K^{2}(x)dx.
\]
Hence quantity (\ref{miarajadra}) does not depend on the variance of the
kernel and hence on the choice of the optimal kernel, we can select kernels
having variance equal to $1$ and of course satisfying remaining conditions,
that were imposed on the considered kernels. Hodges and Lehman in the paper
\cite{Hodges56} solved the problem of choosing the density minimizing the
quantity $\int K^{2}(x)dx$ under conditions $\int K(x)dx\allowbreak
=\allowbreak1$ and $\int x^{2}K(x)dx\allowbreak=\allowbreak1$. It turned out,
that this density is the so-called Epanechnikov's density (\ref{epanecznikov}).
\end{proof}

Epanechnikov's density has the following plot:%

%TCIMACRO{\FRAME{dtbpF}{2.6455in}{1.6812in}{0pt}{}{}{kernel_ep.eps}%
%width 2.6455in;  height 1.6812in;  depth 0pt;  original-width 7.5005in;
%original-height 4.7547in;  cropleft "0";  croptop "1";  cropright "1";
%cropbottom "0";  filename 'kernel_ep.eps';file-properties "XNPEU";}}}%
%BeginExpansion
\begin{center}
\includegraphics[
height=1.6812in,
width=2.6455in
]%
{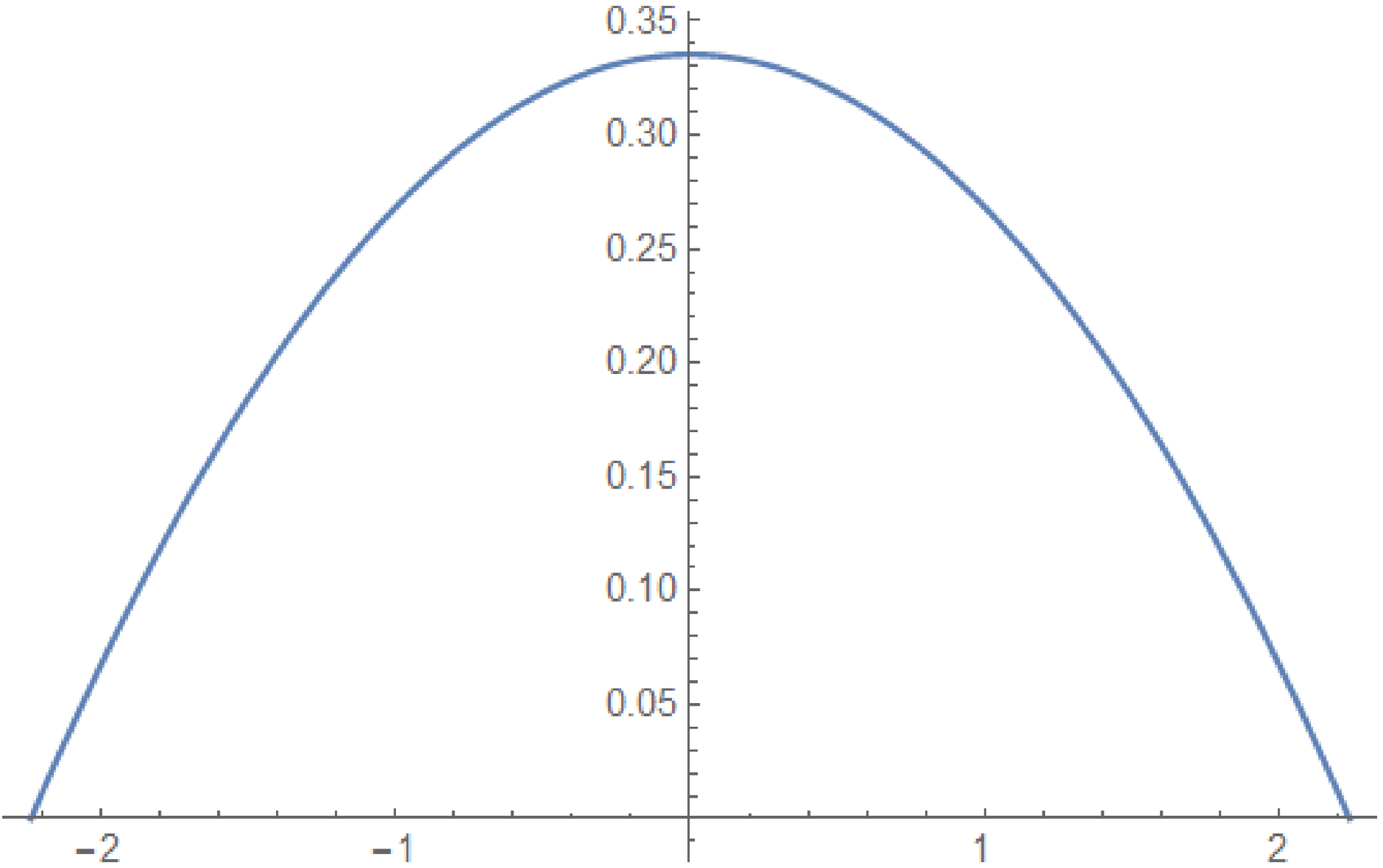}%
\end{center}
%EndExpansion
. Let us consider a functional\ $C(K)=\left(  \int t^{2}K(t)dt\right)
^{1/2}\int K^{2}(t)dt$ defined for symmetric kernels. As we remember for
Epanechnikov's kernel this functional assumes its smallest value equal
\[
\int_{-\sqrt{5}}^{\sqrt{5}}\left(  K_{E}(t)\right)  ^{2}d\allowbreak=\frac
{3}{5\sqrt{5}}.
\]
In the above mentioned monograph of Silverman \cite{Silverman86} one defines
the following quantity
\[
eff(K)=\frac{C(K_{E})}{C(K)}=\frac{3}{5\sqrt{5}}\frac{1}{\left(  \int
t^{2}K(t)dt\right)  ^{1/2}\int K^{2}(t)dt},
\]
called kernel's effectiveness or effectiveness of the kernel $K$. This
parameter was calculated for several substantially different symmetric
kernels. It turned out that, from this point of view the difference e.g.
between Epanechnikov's kernel and the rectangular one equal to $1/2$ for
$\left\vert t\right\vert \leq1$ was very small. Summarizing effectiveness of
many popular kernels is close to $1$ and in any case greater than $0.9$.

Proposition \ref{optymalne} has unfortunately only theoretical meaning, since
it is not known how much is $\int\left(  f^{\prime\prime}(x)\right)  ^{2}dx$.
Jones, Marron, and Sheather in a review paper \cite{Jones96}, discuss
different estimation methods of this parameter on the basis the same
measurements, that are used for the density estimation. During the last $20$
years, many of such estimators we constructed. We will not discuss this
problem since density estimation is not the main subject of this book. We
refer the reader to the literature.

\subsection{Modifications\label{gest_regresyjna}}

Since, those considered density estimators were in of the iterative form,
further analysis will be dedicated to the following modifications of the basic
estimator, whose iterative form nicely fits the assumptions of this book.
Namely, let us consider instead of the estimator (\ref{podstaw}) the following one:%

\index{Estimator!Iterative of density}
\begin{equation}
\hat{f}_{n}(y)=\frac{1}{n}\sum_{i=1}^{n}\frac{1}{h_{i}}K\left(  \frac{y-X_{i}%
}{h_{i}}\right)  , \label{est_iteracyjny}%
\end{equation}
where sequence $\left\{  h_{i}\right\}  $ is some, convergent to zero sequence
of positive numbers, while $\{X_{n}\}_{n\geq1}$ is a sequence of independent
random variables having same one-dimensional distributions, possessing
density. This estimator has the following recursive form:
\begin{equation}
\hat{f}_{n+1}(y)=\left(  1-\frac{1}{n+1}\right)  \hat{f}_{n}(y)+\frac
{1}{\left(  n+1\right)  h_{n+1}}K\left(  \frac{y-X_{n+1}}{h_{n+1}}\right)  .
\label{iter}%
\end{equation}

Let us introduce also the following denotation for the sake of brevity:
\[
\bar{f}_{n}(y)=E\hat{f}_{n}(y).
\]
Taking the expectation of both sides of tha equality (\ref{iter}) we get
\begin{align}
\bar{f}_{n+1}(y)  &  =\left(  1-\frac{1}{n+1}\right)  \bar{f}_{n}(y)+\frac
{1}{\left(  n+1\right)  h_{n+1}}EK\left(  \frac{y-X_{n+1}}{h_{n+1}}\right)
\label{postac_Efn}\\
&  =\left(  1-\frac{1}{n+1}\right)  \bar{f}_{n}(y)+\frac{1}{\left(
n+1\right)  }\int_{%
%TCIMACRO{\U{211d} }%
%BeginExpansion
\mathbb{R}
%EndExpansion
}K(z)f(y-zh_{n+1})dz.\nonumber
\end{align}

Let us denote $\mathcal{G}_{n}=\sigma(X_{1},\ldots,X_{n})$. In further
analysis of the estimator $\hat{f}_{n}$ the following lemma will be of use%
%TCIMACRO{\TEXTsymbol{>}}%
%BeginExpansion
$>$%
%EndExpansion

\begin{lemma}
\label{szacunki_EK_2}If $\underset{n}{\sup}\int_{%
%TCIMACRO{\U{211d} }%
%BeginExpansion
\mathbb{R}
%EndExpansion
}K^{2}(z)f(y-zh_{n})dz<\infty$ and function $\underset{\left\vert y\right\vert
\geq\left\vert x\right\vert }{\sup}K(y)$ is integrable and random variables
$\{X_{n}\}_{n\geq1}$ are independent, then\newline i)
\begin{equation}
EK^{2}\left(  \frac{y-X_{n}}{h_{n}}\right)  =\allowbreak h_{n}\int_{%
%TCIMACRO{\U{211d} }%
%BeginExpansion
\mathbb{R}
%EndExpansion
}K^{2}(z)f(y-zh_{n})dz, \label{EK_kwadrat}%
\end{equation}
\newline ii)
\begin{gather}
E\left\vert \int_{%
%TCIMACRO{\U{211d} }%
%BeginExpansion
\mathbb{R}
%EndExpansion
}W_{n-1}(y)\left[  K\left(  \frac{y-X_{n}}{h_{n}}\right)  -EK\left(
\frac{y-X_{n}}{h_{n}}\right)  \right]  dy\right\vert ^{2}\leq\label{E_W_K_2}\\
\leq h_{n}E\left(  \int_{%
%TCIMACRO{\U{211d} }%
%BeginExpansion
\mathbb{R}
%EndExpansion
}W_{n-1}^{2}(y)dy\right)  \int_{%
%TCIMACRO{\U{211d} }%
%BeginExpansion
\mathbb{R}
%EndExpansion
}K^{2}(z)dz,\nonumber\\
E\left\{  \int_{%
%TCIMACRO{\U{211d} }%
%BeginExpansion
\mathbb{R}
%EndExpansion
}W_{n-1}(y)\left[  K\left(  \frac{y-X_{n}}{h_{n}}\right)  -EK\left(
\frac{y-X_{n}}{h_{n}}\right)  \right]  dy\right\}  =0, \label{EW_K}%
\end{gather}
where $W_{n-1}(y)$ is measurable with respect to $\mathcal{G}_{n-1}$ random
variable such that
\[
E\int_{%
%TCIMACRO{\U{211d} }%
%BeginExpansion
\mathbb{R}
%EndExpansion
}W_{n-1}^{2}(y)dy<\infty,
\]
\textit{iii)}
\[
\frac{1}{h_{n}}EK\left(  \frac{x-X_{n}}{h_{n}}\right)  \underset{n\rightarrow
\infty}{\longrightarrow}f(x),
\]
for almost all $x\in%
%TCIMACRO{\U{211d} }%
%BeginExpansion
\mathbb{R}
%EndExpansion
.$\newline
\end{lemma}

\begin{proof}
\textit{i)} We have:
\begin{align*}
EK^{2}\left(  \frac{y-X_{n}}{h_{n}}\right)   &  =\int_{%
%TCIMACRO{\U{211d} }%
%BeginExpansion
\mathbb{R}
%EndExpansion
}K^{2}\left(  \frac{y-x}{h_{n}}\right)  f(x)dx\allowbreak=\\
&  =h_{n}\int_{%
%TCIMACRO{\U{211d} }%
%BeginExpansion
\mathbb{R}
%EndExpansion
}K^{2}(z)f(y-zh_{n})dz,
\end{align*}
after change of variables $z=(y-x)/h_{n}.$

\textit{ii)} Moreover, we have:
\begin{gather*}
E\left\vert \int_{%
%TCIMACRO{\U{211d} }%
%BeginExpansion
\mathbb{R}
%EndExpansion
}W_{n-1}(y)\left(  K\left(  \frac{y-X_{n}}{h_{n}}\right)  -EK\left(
\frac{y-X_{n}}{h_{n}}\right)  \right)  dy\right\vert ^{2}\leq\\
\leq E\left\vert \sqrt{\int_{%
%TCIMACRO{\U{211d} }%
%BeginExpansion
\mathbb{R}
%EndExpansion
}W_{n-1}^{2}(y)dy\int_{%
%TCIMACRO{\U{211d} }%
%BeginExpansion
\mathbb{R}
%EndExpansion
}\left(  K\left(  \frac{y-X_{n}}{h_{n}}\right)  -EK\left(  \frac{y-X_{n}%
}{h_{n}}\right)  \right)  ^{2}dy}\right\vert ^{2}=\\
=E\left(  \int_{%
%TCIMACRO{\U{211d} }%
%BeginExpansion
\mathbb{R}
%EndExpansion
}W_{n-1}^{2}(y)dy\int_{%
%TCIMACRO{\U{211d} }%
%BeginExpansion
\mathbb{R}
%EndExpansion
}\left(  K\left(  \frac{y-X_{n}}{h_{n}}\right)  -EK\left(  \frac{y-X_{n}%
}{h_{n}}\right)  \right)  ^{2}dy\right)  =\\
=E\left(  \int_{%
%TCIMACRO{\U{211d} }%
%BeginExpansion
\mathbb{R}
%EndExpansion
}W_{n-1}^{2}(y)dyE\left(  \int_{%
%TCIMACRO{\U{211d} }%
%BeginExpansion
\mathbb{R}
%EndExpansion
}\left(  K\left(  \frac{y-X_{n}}{h_{n}}\right)  -EK\left(  \frac{y-X_{n}%
}{h_{n}}\right)  \right)  ^{2}dy|\mathcal{G}_{n-1}\right)  \right)  =\\
=E\left(  \int_{%
%TCIMACRO{\U{211d} }%
%BeginExpansion
\mathbb{R}
%EndExpansion
}W_{n-1}^{2}(y)dy\int_{%
%TCIMACRO{\U{211d} }%
%BeginExpansion
\mathbb{R}
%EndExpansion
}E\left(  \left(  K\left(  \frac{y-X_{n}}{h_{n}}\right)  -EK\left(
\frac{y-X_{n}}{h_{n}}\right)  \right)  ^{2}|\mathcal{G}_{n-1}\right)
dy\right)
\end{gather*}
Taking advantage of the fact, that $\operatorname*{var}(X)\leq EX^{2}$ and the
property $i)$ we get:
\begin{gather*}
E\left(  \int_{%
%TCIMACRO{\U{211d} }%
%BeginExpansion
\mathbb{R}
%EndExpansion
}W_{n-1}^{2}(y)dy\int_{%
%TCIMACRO{\U{211d} }%
%BeginExpansion
\mathbb{R}
%EndExpansion
}E\left(  \left(  K\left(  \frac{y-X_{n}}{h_{n}}\right)  -EK\left(
\frac{y-X_{n}}{h_{n}}\right)  \right)  ^{2}|\mathcal{G}_{n-1}\right)
dy\right)  \leq\\
\leq h_{n}E\left(  \int_{%
%TCIMACRO{\U{211d} }%
%BeginExpansion
\mathbb{R}
%EndExpansion
}W_{n-1}^{2}(y)dy\int_{%
%TCIMACRO{\U{211d} }%
%BeginExpansion
\mathbb{R}
%EndExpansion
}\int_{%
%TCIMACRO{\U{211d} }%
%BeginExpansion
\mathbb{R}
%EndExpansion
}K^{2}(z)f(y-h_{n}z)dzdy\right)  =\\
=h_{n}E\int_{%
%TCIMACRO{\U{211d} }%
%BeginExpansion
\mathbb{R}
%EndExpansion
}W_{n-1}^{2}(y)dy\int_{%
%TCIMACRO{\U{211d} }%
%BeginExpansion
\mathbb{R}
%EndExpansion
}K^{2}(z)dz.
\end{gather*}
In above-mentioned calculations, we used Schwarz inequality and the properties
conditional expectation. Knowing, that one can interchange integration with
respect to $y$ and calculating expectation on the basis of (\ref{E_W_K_2}), we
change this order in (\ref{EW_K}) and we get:
\begin{align*}
&  E\int_{%
%TCIMACRO{\U{211d} }%
%BeginExpansion
\mathbb{R}
%EndExpansion
}W_{n}(y)\left(  K\left(  \frac{y-X_{n+1}}{h_{n+1}}\right)  -EK\left(
\frac{y-X_{n+1}}{h_{n+1}}\right)  \right)  dy\\
&  =\int_{%
%TCIMACRO{\U{211d} }%
%BeginExpansion
\mathbb{R}
%EndExpansion
}E\left(  W_{n}(y)E\left(  K\left(  \frac{y-X_{n+1}}{h_{n+1}}\right)
-EK\left(  \frac{y-X_{n+1}}{h_{n+1}}\right)  |\mathcal{G}_{n}\right)  \right)
dy=0.
\end{align*}

\textit{iii)} is a simple consequence of the formula
\[
\frac{1}{h_{n}}EK(\frac{x-X_{n}}{h_{n}})=\int_{%
%TCIMACRO{\U{211d} }%
%BeginExpansion
\mathbb{R}
%EndExpansion
}K(z)f(x-zh_{n})dz,
\]
and assumed convergence $h_{n}\underset{n\rightarrow\infty}{\longrightarrow}0$
and assertion \textit{iii)} of Theorem \ref{wlasnosci_jadra}.
\end{proof}

The following theorem we get immediately.

\begin{theorem}
\label{o_Ef_n}If sequence $\{X_{n}\}_{n\geq1}$ consists of i.i.d. random
variables, number sequence $\left\{  h_{n}\right\}  _{n\geq1}\,$is such that
\[
\sum_{n\geq1}\frac{1}{n^{2}h_{n}}<\infty\text{ and }\underset{n}{\sup}\int_{%
%TCIMACRO{\U{211d} }%
%BeginExpansion
\mathbb{R}
%EndExpansion
}K^{2}(z)f(y-h_{n}z)dz<\infty,
\]
for almost all $y$, then:
\begin{equation}
\hat{f}_{n}(y)-E\hat{f}_{n}(y)\underset{n\rightarrow\infty}{\longrightarrow
}0\,\,\,\,a.s.\text{for almost all }y\in%
%TCIMACRO{\U{211d} }%
%BeginExpansion
\mathbb{R}
%EndExpansion
, \label{1zb}%
\end{equation}%
\begin{equation}
\int_{%
%TCIMACRO{\U{211d} }%
%BeginExpansion
\mathbb{R}
%EndExpansion
}\left(  \hat{f}_{n}(y)-E\hat{f}_{n}(y)\right)  ^{2}dy\underset{n\rightarrow
\infty}{\longrightarrow}0\,\,\,a.s., \label{2zb}%
\end{equation}

\begin{equation}
\int_{%
%TCIMACRO{\U{211d} }%
%BeginExpansion
\mathbb{R}
%EndExpansion
}\left\vert \hat{f}_{n}(y)-f(y)\right\vert dy\underset{n\rightarrow
\infty}{\longrightarrow}0\;\;a.s.. \label{4zb}%
\end{equation}

\begin{equation}
\hat{f}_{n}(y)\underset{n\rightarrow\infty}{\rightarrow}%
f(y)\,\;\;a.s.\text{for almost all }y\in%
%TCIMACRO{\U{211d} }%
%BeginExpansion
\mathbb{R}
%EndExpansion
, \label{3zb}%
\end{equation}

\end{theorem}

\begin{proof}
Let us denote $T_{n}(y)=\hat{f}_{n}(y)-E\hat{f}_{n}(y)$. Taking the
expectation of both sides of (\ref{iter}) and subtracting from both sides of
this equation we get:
\begin{gather*}
T_{n+1}(y)=\left(  1-\frac{1}{n+1}\right)  T_{n}(y)+\\
+\frac{1}{\left(  n+1\right)  h_{n+1}}\left(  K\left(  \frac{y-X_{n+1}%
}{h_{n+1}}\right)  -EK\left(  \frac{y-X_{n+1}}{h_{n+1}}\right)  \right)  .
\end{gather*}
Let us consider sequence $\left\{  \sum_{n=0}^{N}\frac{1}{\left(  n+1\right)
h_{n+1}}\left(  K\left(  \frac{y-X_{n+1}}{h_{n+1}}\right)  -EK\left(
\frac{y-X_{n+1}}{h_{n+1}}\right)  \right)  \right\}  _{N\geq1}$. It is a
martingale with respect to filtration $\left\{  \mathcal{G}_{N}\right\}
_{N\geq1}$. This martingale\ is convergent for example, when the series:
\[
\sum_{n\geq1}\frac{1}{n^{2}h_{n}^{2}}EK^{2}\left(  \frac{y-X_{n}}{h_{n}%
}\right)
\]
convergent is. Taking advantage of our assumption and assertion \textit{i)
}of\textit{\ }Lemma \ref{szacunki_EK_2}, we see that this series is
convergent, if only series $\sum_{n\geq1}\frac{1}{n^{2}h_{n}}$ is convergent.
It is so since we assumed this convergence. Hence, we have (\ref{1zb}).

To show (\ref{2zb}) let us denote additionally $M_{n}=\int T_{n}^{2}(y)dy$.
For the sequence of the random variables $\left\{  M_{n}\right\}  $ we get the
following recurrent relationship:
\begin{gather}
M_{n+1}=\left(  1-\frac{1}{n+1}\right)  ^{2}M_{n}+\label{iteracje_M_n}\\
+\frac{2(1-\frac{1}{n+1})}{(n+1)h_{n+1}}\int T_{n}(y)\left(  K\left(
\frac{y-X_{n+1}}{h_{n+1}}\right)  -EK\left(  \frac{y-X_{n+1}}{h_{n+1}}\right)
\right)  dy+\nonumber\\
+\frac{1}{(n+1)^{2}h_{n+1}^{2}}\int\left(  K\left(  \frac{y-X_{n+1}}{h_{n+1}%
}\right)  -EK\left(  \frac{y-X_{n+1}}{h_{n+1}}\right)  \right)  ^{2}%
dy.\nonumber
\end{gather}
Let us apply property (\ref{EW_K}) of Lemma \ref{szacunki_EK_2} for
$W_{n}=T_{n}$ and assertion \ref{E_W_K_2} of this lemma. We get then the
following recurrent relationship
\begin{gather*}
EM_{n+1}=\left(  1-\frac{1}{n+1}\right)  ^{2}EM_{n}+\frac{1}{(n+1)^{2}%
h_{n+1}^{2}}\times\\
\times\int E\left(  K\left(  \frac{y-X_{n+1}}{h_{n+1}}\right)  -EK\left(
\frac{y-X_{n+1}}{h_{n+1}}\right)  \right)  ^{2}dy\leq\\
\leq\left(  1-\frac{1}{n+1}\right)  ^{2}EM_{n}+\frac{1}{(n+1)^{2}h_{n+1}^{2}%
}\times\\
\times h_{n+1}\int\int K^{2}(z)f(y-zh_{n+1})dzdy,
\end{gather*}
Taking advantage of assumptions and the convergence of the series $\sum
_{n\geq1}\frac{1}{n^{2}h_{n}}$ on the basis of Lemma \ref{podstawowy} we see
that the sequence $\left\{  EM_{n}\right\}  _{n\geq0}$ converges to
zero.\ Further, let us consider sequence random variables:
\begin{equation}
\left\{  \sum_{n=1}^{N}\frac{1}{nh_{n}}\int T_{n-1}(y)\left(  K\left(
\frac{y-X_{n}}{h_{n}}\right)  -EK\left(  \frac{y-X_{n}}{h_{n}}\right)
\right)  dy\right\}  _{N\geq1}. \label{martngal_pomocniczy}%
\end{equation}
Taking advantage of assertion (\ref{EW_K}) of Lemma \ref{szacunki_EK_2} we see
it is a martingale with respect to filtration $\left\{  \mathcal{G}%
_{N}\right\}  _{N\geq1}$. This martingale converges almost surely, if for
example
\[
\sum_{n=1}^{\infty}\frac{1}{n^{2}h_{n}^{2}}E\left(  \int T_{n-1}(y)\left(
K\left(  \frac{y-X_{n}}{h_{n}}\right)  -EK\left(  \frac{y-X_{n}}{h_{n}%
}\right)  \right)  dy\right)  ^{2}<\infty.
\]
On the base of assertion (\ref{E_W_K_2}) of Lemma \ref{szacunki_EK_2} we see,
that this condition is satisfied, when
\[
\sum_{n=1}^{\infty}\frac{1}{n^{2}h_{n}}EM_{n-1}\int K^{2}(z)dz<\infty.
\]
This condition is satisfied, since the sequence $\left\{  EM_{n}\right\}  $
converges to zero. Hence, returning to the relationship (\ref{iteracje_M_n}),
on the basis of Lemma \ref{podstawowy} we deduce that the sequence of the
random variables $\left\{  M_{n}\right\}  $ converges to zero almost surely,
when converge the following series
\[
\sum_{n=1}^{\infty}\frac{1}{n^{2}h_{n}^{2}}\int\left(  K\left(  \frac{y-X_{n}%
}{h_{n}}\right)  -EK\left(  \frac{y-X_{n}}{h_{n}}\right)  \right)  ^{2}dy
\]
and
\[
\sum_{n=1}^{\infty}\frac{2(1-\frac{1}{n})}{nh_{n}}\int T_{n-1}(y)\left(
K\left(  \frac{y-X_{n}}{h_{n}}\right)  -EK\left(  \frac{y-X_{n}}{h_{n}%
}\right)  \right)  dy.
\]
Almost everywhere convergence of the second one was already above proved.
Convergence almost everywhere of the first one follows from the observation,
that its elements are positive, the inequality $\operatorname*{var}(X)\leq
EX^{2}$ and the equality (\ref{EK_kwadrat}).\newline We will prove (\ref{3zb})
by showing, that $Ef_{n}(y)\underset{n\rightarrow\infty}{\longrightarrow}f(y)$
for almost all $y\in%
%TCIMACRO{\U{211d} }%
%BeginExpansion
\mathbb{R}
%EndExpansion
\mathbb{\,}$. To show this, let us notice that from formula
(\ref{est_iteracyjny}) follows relationship
\[
E\hat{f}_{n}(y)=\frac{1}{n}\sum_{i=1}^{n}E\left(  \frac{1}{h_{i}}K\left(
\frac{y-X_{i}}{h_{i}}\right)  \right)  .
\]
From assertion $iii)$ Lemma \ref{szacunki_EK_2} it follows that $E\left(
\frac{1}{h_{i}}K(\frac{y-X_{i}}{h_{i}})\right)  \underset{i\rightarrow
\infty}{\longrightarrow}f(y)$ for almost all $y$, hence basing on Lemma
\ref{lemosr} we get assertion.

To get (\ref{4zb}), it is enough to recall assertion (\ref{3zb}) and
Scheff\'{e}'s Lemma.
\end{proof}

\begin{remark}
The above-mentioned theorem, slightly differently formulated and with
different proofs appear in papers of different authors (Deheuvels 1974
\cite{Deheuvels74}, Wolverton \& Wagner 1969\cite{Wolverton69}, Yamato 1971
\cite{Yamato71}), Davies 1973 \cite{Davies73}, Carrol 1976\cite{Carrol76},
Ahmad \&Lin 1976\cite{Ahmad76}, Devroye 1979 \cite{Devroye79}, Wegman \&
Davies 1979\cite{Wengman79} and so on). The above-mentioned formulation and
proof seems to be simple and Moreover, use only the means developed in this book.
\end{remark}

\begin{remark}
Let us notice, that in order to show convergence with probability 1 of the
sequence of estimators, we assumed independence of the sequence of
observations $\left\{  X_{i}\right\}  _{i\geq1}$. It was necessary to show the
convergence some of the series (of the series whose sequence of partial sums
is the sequence (\ref{martngal_pomocniczy})). Convergence of this series not
necessarily one has to examine by martingale methods. Possibly such
convergence could have been proved without supposing independence of
observations, using other methods (for example described in chapter
\ref{zbiez}). One has to estimate covariance $\operatorname*{cov}\left(
K\left(  \frac{x-X_{i}}{h_{i}}\right)  ,K\left(  \frac{x-X_{j}}{h_{j}}\right)
\right)  $. But this requires knowledge (partial) of two-dimensional
distributions of the sequence $\{X_{n}\}_{n\geq1}.$
\end{remark}

\begin{example}
\label{przyk_gest_iter}As an example, let us consider sequence random
variables having bimodal density, being a mixture of two Normal distributions
$N(0,1)$ and $N(4,.5)$, with weights respectively $\frac{1}{4}$ and $\frac
{3}{4}$. One performed $N\allowbreak=\allowbreak3000$ iterations
Sequence$\left\{  h_{i}\right\}  _{i\geq0}$ was chosen to be: $h_{i}%
=i^{-.35};i\geq1$. On the figure below one shows plot of the density and its
estimator based on $N$ observations obtained by the iterative method with
Cauchy kernel:%

%TCIMACRO{\FRAME{itbpF}{2.9698in}{1.9285in}{0in}{}{}{kernel_2modal.eps}%
%width 2.9698in;  height 1.9285in;  depth 0in;  original-width 6.2111in;
%original-height 4.0222in;  cropleft "0";  croptop "1";  cropright "1";
%cropbottom "0";  filename 'kernel_2modal.eps';file-properties "XNPEU";}}}%
%BeginExpansion
{\includegraphics[
height=1.9285in,
width=2.9698in
]%
{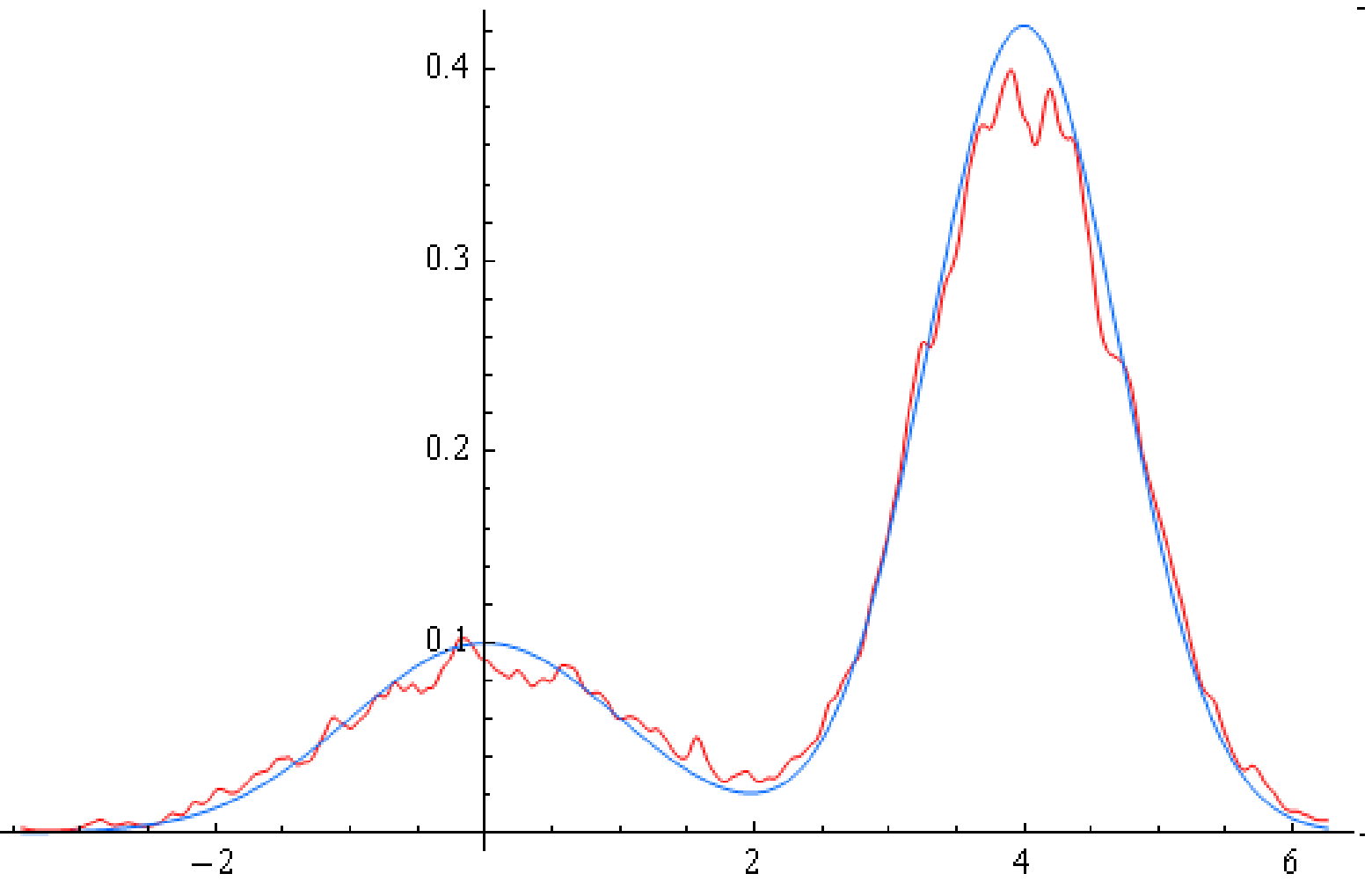}%
}%
%EndExpansion

\end{example}

\section{Introduction to regression estimation}

Let us consider a sequence of independent realizations of the two-dimensional
random variable $(X,Y)$, i.e. the sequence \newline$\left\{  (X_{1}%
,Y_{1}),\ldots,(X_{n},Y_{n})\right\}  $. Let us assume that the random
variable $X$ has density $f$. Regression function $r(x)$ of the random
variable $Y$ on $X$ that is $r(x)=E(Y|X=x)$ will be estimated with the help of
the following estimators:%
\index{Estimator!Regression}
\begin{equation}
\hat{r}_{n}(x)=\frac{\sum_{i=1}^{n}Y_{i}K\left(  \frac{x-X_{i}}{h_{n}}\right)
}{\sum_{i=1}^{n}K\left(  \frac{x-X_{i}}{h_{n}}\right)  }, \label{regresja}%
\end{equation}
where, as before, $K$ is some kernel, a $h_{n}$ is number sequence convergent
to zero. Let us see how it works.

\begin{example}
In this example, as a kernel we take Epanechnikov's one. Sequence of
observations will be simulated with the help of the sequence \newline$\left\{
(X_{1},Y_{1}),\ldots,(X_{n},Y_{n})\right\}  $, where
\[
Y_{i}=r(X_{i})+0.5\ast\xi_{i},
\]
while sequences of the random variables $\left\{  X_{i}\right\}  _{i\geq1}$
and $\left\{  \xi_{i}\right\}  _{i\geq1}$ are independent. Assume that the
function $r(x)$ is defined by $r(x)=x^{2}$, while in the second example
\newline$r1(x)=\left\{
\begin{array}
[c]{lll}%
-1 & for & x<-1\\
x & for & \left\vert x\right\vert \leq1\\
1 & for & x>1
\end{array}
\right.  $. Sequences $\left\{  X_{i}\right\}  _{i\geq1}$ and $\left\{
\xi_{i}\right\}  _{i\geq1}$ are the sequences of independent variables having
distributions $N(0,2)$. Number of observations $n$ is equal to $1000$,
$h(n)=n^{-.4}$. One obtained then, for the functions%

%TCIMACRO{\FRAME{itbpF}{2.7216in}{1.759in}{0in}{}{}{kernel_reg1.eps}%
%width 2.7216in;  height 1.759in;  depth 0in;  original-width 6.2379in;
%original-height 4.0222in;  cropleft "0";  croptop "1";  cropright "1";
%cropbottom "0";  filename 'kernel_reg1.eps';file-properties "XNPEU";}}}%
%BeginExpansion
{\includegraphics[
height=1.759in,
width=2.7216in
]%
{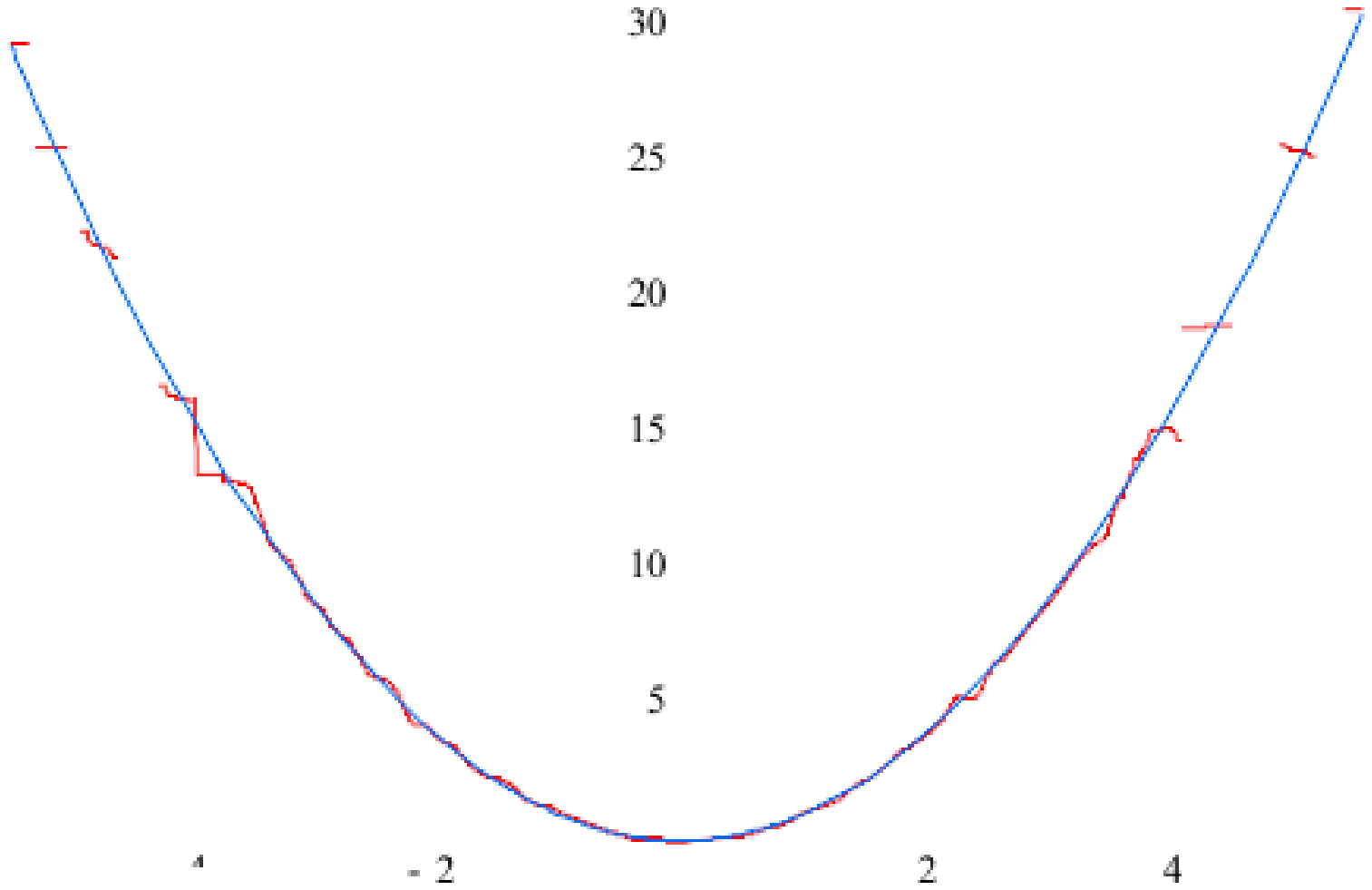}%
}%
%EndExpansion
$r(x)=x^{2}$ and for $r1$%

%TCIMACRO{\FRAME{itbpF}{2.7319in}{1.6985in}{0in}{}{}{kernel_reg2.eps}%
%width 2.7319in;  height 1.6985in;  depth 0in;  original-width 6.2379in;
%original-height 3.87in;  cropleft "0";  croptop "1";  cropright "1";
%cropbottom "0";  filename 'kernel_reg2.eps';file-properties "XNPEU";}}}%
%BeginExpansion
{\includegraphics[
height=1.6985in,
width=2.7319in
]%
{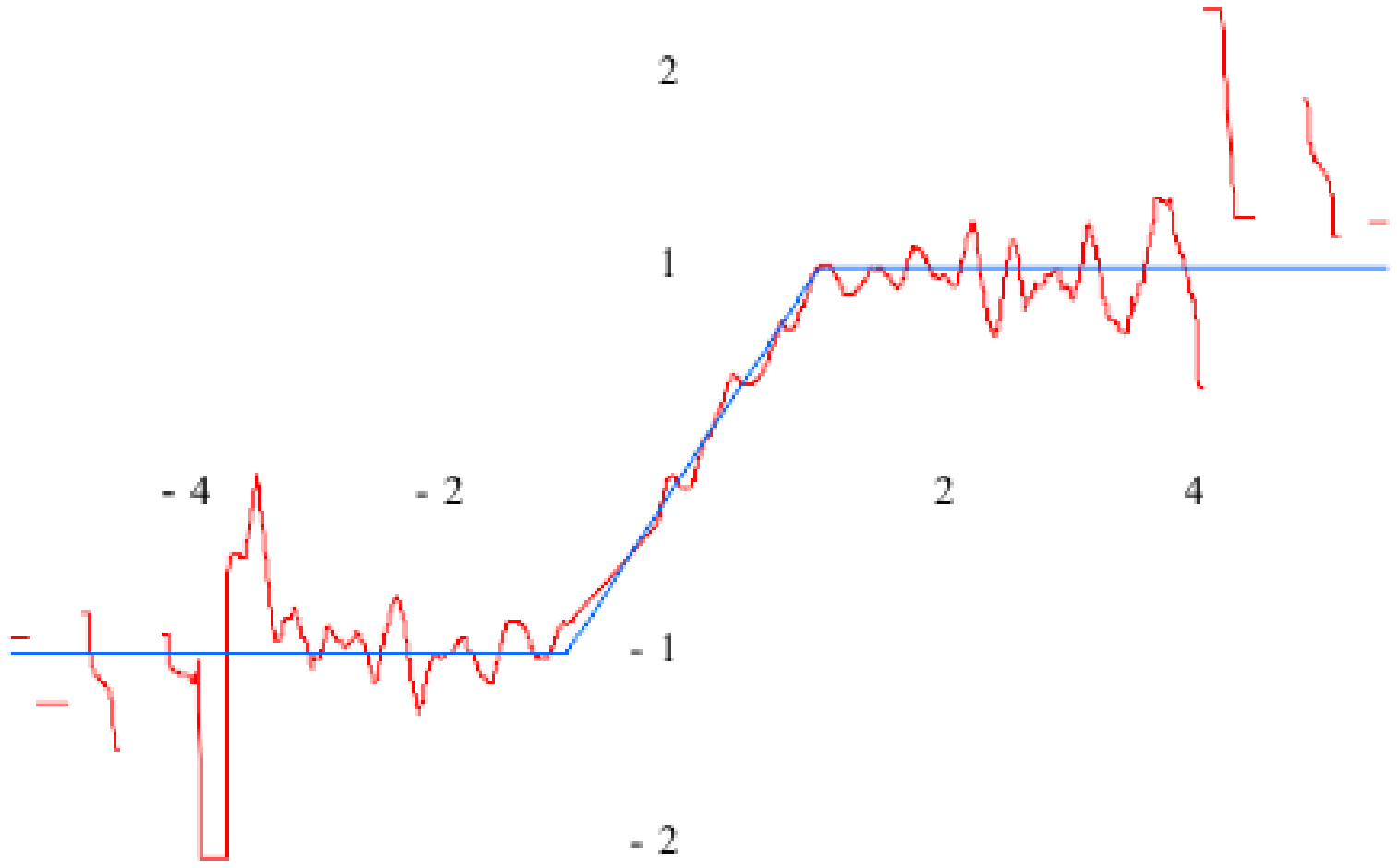}%
}%
%EndExpansion
.
\end{example}

\subsection{Simple regression estimator}

We will be concerned with the case of one-dimensional random variables $X$ and
$Y$ and firstly we will analyze estimator (\ref{regresja}), as $n$ (number of
observations) diverges to $\infty$, $K(x)$ is a fixed kernel, and the number
sequence $\left\{  h_{n}\right\}  _{n\geq1}$ converges to zero, in such a way
that $nh_{n}\underset{n\rightarrow\infty}{\longrightarrow}\infty$. Moreover,
let us denote $\phi_{n}(x)=\frac{1}{n}\sum_{i=1}^{n}\frac{1}{h_{n}}K\left(
\frac{x-X_{i}}{h_{n}}\right)  $. On the base of material of the previous
section, we know, that the sequence $\phi_{n}$ converges pointwise, and also
in the distributive sense to the density $\phi$ of the random variable $X_{1}%
$. We have the following theorem:

\begin{theorem}%
\begin{equation}
\forall A\in\mathcal{B}\,\,\,:\int_{A}\hat{r}_{n}(x)\phi_{n}%
(x)dx\underset{n\rightarrow\infty}{\longrightarrow}\int_{A}r(x)\phi(x)dx.
\label{zb-dystr}%
\end{equation}

\end{theorem}

\begin{proof}
Let us calculate the Fourier transform of the function (see denotations in
Appendix \ref{fourier})
\[
H_{n}(x)=\allowbreak\hat{r}_{n}(x)\phi_{n}(x)=\allowbreak\frac{1}{nh_{n}}%
\sum_{i=1}^{n}Y_{i}K\left(  \frac{x-X_{i}}{h_{n}}\right)  .
\]
We have:
\begin{align*}
\widehat{H}_{n}(t)  &  =\int_{%
%TCIMACRO{\U{211d} }%
%BeginExpansion
\mathbb{R}
%EndExpansion
}\hat{r}_{n}(x)\phi_{n}(x)\exp(itx)dx\\
&  =\frac{1}{nh_{n}}\sum_{i=1}^{n}\int_{%
%TCIMACRO{\U{211d} }%
%BeginExpansion
\mathbb{R}
%EndExpansion
}Y_{i}K\left(  \frac{x-X_{i}}{h_{n}}\right)  \exp(itx)dx,
\end{align*}
but $\int_{%
%TCIMACRO{\U{211d} }%
%BeginExpansion
\mathbb{R}
%EndExpansion
}\frac{1}{h_{n}}K\left(  \frac{x-X_{i}}{h_{n}}\right)  \exp(itx)dx=\int_{%
%TCIMACRO{\U{211d} }%
%BeginExpansion
\mathbb{R}
%EndExpansion
}K(z)\exp(itX_{i}+itzh_{n})dz$. Let us denote $\varphi(t)=\int_{%
%TCIMACRO{\U{211d} }%
%BeginExpansion
\mathbb{R}
%EndExpansion
}K(z)\exp(itz)dz$. Hence:
\[
\widehat{H}_{n}(t)=\varphi(th_{n})\frac{1}{n}\sum_{i=1}^{n}Y_{i}\exp
(itX_{i}).
\]
Let us notice that $\forall t\in%
%TCIMACRO{\U{211d} }%
%BeginExpansion
\mathbb{R}
%EndExpansion
:\varphi(th_{n})\underset{n\rightarrow\infty}{\longrightarrow}\varphi(0)=1$
(since $K$ is a density). Moreover, because random variables $\left\{  \left(
X_{i},Y_{i}\right)  \right\}  _{i\geq1}$ are independent, they satisfy LLN and
we see that
\begin{align*}
\forall t  &  \in%
%TCIMACRO{\U{211d} }%
%BeginExpansion
\mathbb{R}
%EndExpansion
:\allowbreak\frac{1}{n}\sum_{i=1}^{n}Y_{i}\exp(itX_{i})\allowbreak
\underset{n\rightarrow\infty}{\longrightarrow}\allowbreak EY_{1}\exp
(itX_{1})=\\
&  =Er(X_{1})\exp(itX_{1})\allowbreak=\int_{%
%TCIMACRO{\U{211d} }%
%BeginExpansion
\mathbb{R}
%EndExpansion
}r(x)\phi(x)\exp(itx)dx.
\end{align*}
Hence sequence of the random variables $\left\{  \hat{r}_{n}(x)\phi
_{n}(x)\right\}  _{n\geq1}$ converges for almost every elementary event
$\omega$ in the distributive sense to $r(x)\phi(x)$. It means, e.g., that we
have formula (\ref{zb-dystr}) (see Theorem \ref{zbiez_dystrybucji} on page
\pageref{zbiez_dystrybucji}).
\end{proof}

\subsection{Recursive regression estimator}

Similarly as in the case of density estimation, one can consider
\emph{iterative} \emph{forms} of regression estimators, i.e. estimators of the
form:
\begin{equation}
\widehat{R}_{n}(x)=\frac{\sum_{i=1}^{n}\frac{1}{h_{i}}Y_{i}K\left(
\frac{x-X_{i}}{h_{i}}\right)  }{\sum_{i=1}^{n}\frac{1}{h_{i}}K\left(
\frac{x-X_{i}}{h_{i}}\right)  }. \label{est-reg-iter}%
\end{equation}%
\index{Estimator!Iterative of regression}%
The fact, that the sequence $\left\{  \frac{1}{n}\sum_{i=1}^{n}\frac{1}{h_{i}%
}K\left(  \frac{x-X_{i}}{h_{i}}\right)  \right\}  _{n\geq1}$converges under
suitable assumptions to the density of the random variable $X_{1}$, was shown
before. Let us now denote elements of this sequence as before
\[
\phi_{n}(x)=\frac{1}{n}\sum_{i=1}^{n}\frac{1}{h_{i}}K\left(  \frac{x-X_{i}%
}{h_{i}}\right)  .
\]
Let us denote also
\[
Q_{n}(x)=\frac{1}{n}\sum_{i=1}^{n}\frac{1}{h_{i}}Y_{i}K\left(  \frac{x-X_{i}%
}{h_{i}}\right)  .
\]
Hence $\widehat{R}_{n}(x)=\frac{Q_{n}(x)}{\phi_{n}(x)}$. We will show that,
under similar assumptions as in the section \ref{gest_regresyjna}, the
sequence $\left\{  \frac{1}{n}\sum_{i=1}^{n}\frac{1}{h_{i}}Y_{i}K\left(
\frac{x-X_{i}}{h_{i}}\right)  \right\}  _{n\geq1}$ almost surely converges
pointwise to $r(x)\phi(x)$. First, let us notice that this sequence can be
written in the following recursive form:
\begin{equation}
Q_{n+1}(x)=(1-\mu_{n})Q_{n}(x)+\mu_{n}T_{n+1}(x),
\label{iteracje_dla_regresji}%
\end{equation}
where we denoted, $\mu_{n}=\frac{1}{n+1}$, $T_{n+1}=\frac{1}{h_{n+1}}%
Y_{n+1}K\left(  \frac{x-X_{n+1}}{h_{n+1}}\right)  $. Let us denote also
$v(x)=E(Y^{2}|X=x)$, $s^{2}=EY^{2}$. \newline As before, let $\mathcal{G}%
_{n}=\sigma(X_{1},Y_{1},\ldots,X_{n},Y_{n})$ and $\bar{Q}_{n}(x)=EQ_{n}(x)$.
\newline The sequence$\left\{  \bar{Q}_{n}(x)\right\}  _{n\geq1}$ satisfies
the following recursive relationship:
\begin{align}
\bar{Q}_{n+1}(x)  &  =(1-\mu_{n})\bar{Q}_{n}+\mu_{n}\frac{1}{h_{n+1}%
}Er(X_{n+1})K\left(  \frac{x-X_{n+1}}{h_{n+1}}\right)  =\label{postac_EQ_n}\\
&  =(1-\mu_{n})\bar{Q}_{n}+\mu_{n}\int_{%
%TCIMACRO{\U{211d} }%
%BeginExpansion
\mathbb{R}
%EndExpansion
}K(z)r(x-zh_{n+1})f(x-zh_{n+1})dz.\nonumber
\end{align}
The lemma below and following it Theorem are very similar to respectively
Lemma \ref{szacunki_EK_2} and Theorem \ref{o_Ef_n}.

\begin{lemma}
\label{pomocniczy} If \emph{\ }$\underset{n}{\sup}\int_{%
%TCIMACRO{\U{211d} }%
%BeginExpansion
\mathbb{R}
%EndExpansion
}K^{2}(z)v(x-zh_{n})f(x-zh_{n})dz<\infty$ and function $\underset{\left\vert
y\right\vert \geq\left\vert x\right\vert }{\sup}K(y)$ is integrable, then

i)
\begin{equation}
EY_{n}^{2}K^{2}(\frac{x-X_{n}}{h_{n}})=h_{n}\int_{%
%TCIMACRO{\U{211d} }%
%BeginExpansion
\mathbb{R}
%EndExpansion
}K^{2}(z)v(x-zh_{n})f(x-zh_{n})dz, \label{1_pom_reg}%
\end{equation}

ii)
\begin{gather}
E\left\vert \int_{%
%TCIMACRO{\U{211d} }%
%BeginExpansion
\mathbb{R}
%EndExpansion
}W_{n-1}(x)\left[  Y_{n}K\left(  \frac{x-X_{n}}{h_{n}}\right)  -E\left(
r(X_{n})K\left(  \frac{x-X_{n}}{h_{n}}\right)  \right)  \right]  dx\right\vert
^{2}\label{2_pom_reg}\\
\leq h_{n}s^{2}E\left(  \int_{%
%TCIMACRO{\U{211d} }%
%BeginExpansion
\mathbb{R}
%EndExpansion
}W_{n-1}^{2}(x)dx\int_{%
%TCIMACRO{\U{211d} }%
%BeginExpansion
\mathbb{R}
%EndExpansion
}K^{2}(z)dz\right) \nonumber
\end{gather}
and
\begin{equation}
E\left[  \int_{%
%TCIMACRO{\U{211d} }%
%BeginExpansion
\mathbb{R}
%EndExpansion
}W_{n-1}(x)\left(  Y_{n}K(\frac{x-X_{n}}{h_{n}})-E\left(  r(X_{n}%
)K(\frac{x-X_{n}}{h_{n}})\right)  \right)  dx\right]  =0, \label{3_pom_reg}%
\end{equation}
where $W_{n-1}(x)$ is measurable with respect to $\mathcal{G}_{n-1}$ random
variable and such that
\[
E\left(  \int_{%
%TCIMACRO{\U{211d} }%
%BeginExpansion
\mathbb{R}
%EndExpansion
}W_{n-1}^{2}(x)dx\right)  <\infty.
\]

iii)
\[
\frac{1}{h_{n}}E\left(  Y_{n}K(\frac{x-X_{n}}{h_{n}})\right)
\underset{n\rightarrow\infty}{\longrightarrow}r(x)f(x)
\]
for almost all $x\in%
%TCIMACRO{\U{211d} }%
%BeginExpansion
\mathbb{R}
%EndExpansion
.$\newline
\end{lemma}

\begin{proof}
\textit{i)} We have:
\[
EY_{n}^{2}K^{2}\left(  \frac{y-X_{n}}{h_{n}}\right)  =\int_{%
%TCIMACRO{\U{211d} }%
%BeginExpansion
\mathbb{R}
%EndExpansion
}K^{2}\left(  \frac{y-x}{h_{n}}\right)  v(x)f(x)dx=h_{n}\int_{%
%TCIMACRO{\U{211d} }%
%BeginExpansion
\mathbb{R}
%EndExpansion
}K^{2}(z)v(y-zh_{n})f(y-zh_{n})dz,
\]
after change of variables $z=(y-x)/h_{n}.$

\textit{ii)} We have:
\begin{gather*}
E\left\vert \int_{%
%TCIMACRO{\U{211d} }%
%BeginExpansion
\mathbb{R}
%EndExpansion
}W_{n-1}(y)\left(  Y_{n}K\left(  \frac{y-X_{n}}{h_{n}}\right)  -Er(X_{n}%
)K\left(  \frac{y-X_{n}}{h_{n}}\right)  \right)  dy\right\vert ^{2}\leq\\
\leq E\left\vert \sqrt{\int_{%
%TCIMACRO{\U{211d} }%
%BeginExpansion
\mathbb{R}
%EndExpansion
}W_{n-1}^{2}(y)dy\int_{%
%TCIMACRO{\U{211d} }%
%BeginExpansion
\mathbb{R}
%EndExpansion
}\left(  Y_{n}K\left(  \frac{y-X_{n}}{h_{n}}\right)  -Er(X_{n})K\left(
\frac{y-X_{n}}{h_{n}}\right)  \right)  ^{2}dy}\right\vert ^{2}=\\
=E\left(  \int_{%
%TCIMACRO{\U{211d} }%
%BeginExpansion
\mathbb{R}
%EndExpansion
}W_{n-1}^{2}(y)dy\int_{%
%TCIMACRO{\U{211d} }%
%BeginExpansion
\mathbb{R}
%EndExpansion
}\left(  Y_{n}K\left(  \frac{y-X_{n}}{h_{n}}\right)  -Er(X_{n})K\left(
\frac{y-X_{n}}{h_{n}}\right)  \right)  ^{2}dy\right) \\
=E\left(  \int_{%
%TCIMACRO{\U{211d} }%
%BeginExpansion
\mathbb{R}
%EndExpansion
}W_{n-1}^{2}(y)dyE\left(  \int_{%
%TCIMACRO{\U{211d} }%
%BeginExpansion
\mathbb{R}
%EndExpansion
}\left(  Y_{n}K\left(  \frac{y-X_{n}}{h_{n}}\right)  -Er(X_{n})K\left(
\frac{y-X_{n}}{h_{n}}\right)  \right)  ^{2}dy|\mathcal{G}_{n-1}\right)
\right) \\
=E\left(  \int_{%
%TCIMACRO{\U{211d} }%
%BeginExpansion
\mathbb{R}
%EndExpansion
}W_{n-1}^{2}(y)dy\int_{%
%TCIMACRO{\U{211d} }%
%BeginExpansion
\mathbb{R}
%EndExpansion
}E\left(  \left(  Y_{n}K\left(  \frac{y-X_{n}}{h_{n}}\right)  -Er(X_{n}%
)K\left(  \frac{y-X_{n}}{h_{n}}\right)  \right)  ^{2}|\mathcal{G}%
_{n-1}\right)  dy\right) \\
\leq E\left(  \int_{%
%TCIMACRO{\U{211d} }%
%BeginExpansion
\mathbb{R}
%EndExpansion
}W_{n-1}^{2}(y)dy\int_{%
%TCIMACRO{\U{211d} }%
%BeginExpansion
\mathbb{R}
%EndExpansion
}\nu(X_{n})K^{2}\left(  \frac{y-X_{n}}{h_{n}}\right)  dy\right)  =\\
=h_{n}E\left(  \int_{%
%TCIMACRO{\U{211d} }%
%BeginExpansion
\mathbb{R}
%EndExpansion
}W_{n-1}^{2}(y)dy\int_{%
%TCIMACRO{\U{211d} }%
%BeginExpansion
\mathbb{R}
%EndExpansion
}\int_{%
%TCIMACRO{\U{211d} }%
%BeginExpansion
\mathbb{R}
%EndExpansion
}K^{2}(z)v(y-h_{n}z)f(y-h_{n}z)dzdy\right) \\
=h_{n}s^{2}E\left(  \int_{%
%TCIMACRO{\U{211d} }%
%BeginExpansion
\mathbb{R}
%EndExpansion
}W_{n-1}^{2}(y)dy\int_{%
%TCIMACRO{\U{211d} }%
%BeginExpansion
\mathbb{R}
%EndExpansion
}K^{2}(z)dz\right)
\end{gather*}
In the above-mentioned calculations, we used Schwarz inequality and the
properties of the conditional expectation and inequality $\operatorname*{var}%
(Z)\leq EZ^{2}$ true for any random variable $Z$. Knowing that one can
exchange integration with respect to $y$ and and calculating expectation, on
the basis of (\ref{2_pom_reg}) we change the order of integration in
(\ref{3_pom_reg}) and get:
\begin{align*}
&  E\int_{%
%TCIMACRO{\U{211d} }%
%BeginExpansion
\mathbb{R}
%EndExpansion
}W_{n}(y)\left(  Y_{n}K\left(  \frac{y-X_{n+1}}{h_{n+1}}\right)
-Er(X_{n})K\left(  \frac{y-X_{n+1}}{h_{n+1}}\right)  \right)  dy\\
&  =\int_{%
%TCIMACRO{\U{211d} }%
%BeginExpansion
\mathbb{R}
%EndExpansion
}E\left(  W_{n}(y)E\left(  Y_{n}K\left(  \frac{y-X_{n+1}}{h_{n+1}}\right)
-Er(X_{n})K\left(  \frac{y-X_{n+1}}{h_{n+1}}\right)  |\mathcal{G}_{n}\right)
\right)  =0.
\end{align*}

\textit{iii)} is simple consequence of the formula $\frac{1}{h_{n}}%
EY_{n}K(\frac{x-X_{n}}{h_{n}})=\int_{%
%TCIMACRO{\U{211d} }%
%BeginExpansion
\mathbb{R}
%EndExpansion
}K(z)r(x-h_{n})f(x-zh_{n})dz$, assumed convergence $h_{n}%
\underset{n\rightarrow\infty}{\longrightarrow}0$ and assertion \textit{iii)}
of Lemma \ref{wlasnosci_jadra}.
\end{proof}

Immediately we have the following theorem.

\begin{theorem}
If sequence $\left\{  h_{n}\right\}  \,$is such that $\sum_{n\geq1}\frac
{1}{n^{2}h_{n}}<\infty$ and
\[
\underset{n}{\sup}\int_{%
%TCIMACRO{\U{211d} }%
%BeginExpansion
\mathbb{R}
%EndExpansion
}K^{2}(z)v(y-zh_{n})f(y-h_{n}z)dz<\infty
\]
for almost all $y$, then
\begin{equation}
Q_{n}(y)-EQ_{n}(y)\underset{n\rightarrow\infty}{\longrightarrow}%
0\,\,\,\,a.s.\text{for almost all }y\in%
%TCIMACRO{\U{211d} }%
%BeginExpansion
\mathbb{R}
%EndExpansion
\label{1zb_iter}%
\end{equation}%
\begin{equation}
\int_{%
%TCIMACRO{\U{211d} }%
%BeginExpansion
\mathbb{R}
%EndExpansion
}\left(  Q_{n}(y)-EQ_{n}(y)\right)  ^{2}dy\underset{n\rightarrow
\infty}{\longrightarrow}0\,\,\,a.s. \label{2zb_iter}%
\end{equation}%
\begin{equation}
\widehat{R}_{n}(y)\underset{n\rightarrow\infty}{\rightarrow}%
r(y)\,\;\;a.s.\text{\ for almost all }y\in%
%TCIMACRO{\U{211d} }%
%BeginExpansion
\mathbb{R}
%EndExpansion
\label{3zb_iter}%
\end{equation}

\end{theorem}

\begin{proof}
Firstly, let us notice that our assumptions guarantee satisfaction of
assumptions of Theorem \ref{o_Ef_n}. Hence, we have $\hat{f}_{n}%
(y)\underset{n\rightarrow\infty}{\rightarrow}f(y)\,\;\;a.s$. for almost all
$y\in%
%TCIMACRO{\U{211d} }%
%BeginExpansion
\mathbb{R}
%EndExpansion
$ and $\int_{%
%TCIMACRO{\U{211d} }%
%BeginExpansion
\mathbb{R}
%EndExpansion
}\left\vert \hat{f}_{n}(y)-f(y)\right\vert dy\underset{n\rightarrow
\infty}{\longrightarrow}0\;\;a.s$. Let us denote $U_{n}(y)=Q_{n}(y)-EQ_{n}%
(y)$. Taking expectation of both sides (\ref{iteracje_dla_regresji}) and
subtracting those integrals from both sides of this equality we get:
\begin{gather*}
U_{n+1}(y)=(1-\frac{1}{n+1})U_{n}(y)+\frac{1}{\left(  n+1\right)  h_{n+1}%
}\times\\
\times\left(  Y_{n+1}K\left(  \frac{y-X_{n+1}}{h_{n+1}}\right)  -EY_{n+1}%
K\left(  \frac{y-X_{n+1}}{h_{n+1}}\right)  \right)  .
\end{gather*}
Let us consider sequence
\[
\left\{  \sum_{n=0}^{N}\frac{1}{\left(  n+1\right)  h_{n+1}}\left(
Y_{n+1}K\left(  \frac{y-X_{n+1}}{h_{n+1}}\right)  -EY_{n+1}K\left(
\frac{y-X_{n+1}}{h_{n+1}}\right)  \right)  \right\}  _{N\geq1}.
\]
It is a martingale with respect to filtration $\left\{  \mathcal{G}%
_{N}\right\}  _{N\geq1}$. It is convergent for example, when the series
$\sum_{n\geq1}\frac{1}{n^{2}h_{n}^{2}}EY_{n}^{2}K^{2}\left(  \frac{y-X_{n}%
}{h_{n}}\right)  $ is convergent. Taking advantage of assumptions and
assertion \textit{i) }of\textit{\ }Lemma \ref{pomocniczy} we see that this
series is convergent, if only series $\sum_{n\geq1}\frac{1}{n^{2}h_{n}}$ is
convergent. It so because of assumptions concerning $\left\{  h_{n}\right\}
$. Hence, we have (\ref{1zb_iter}).\newline In order to show (\ref{2zb_iter})
let us denote additionally $W_{n}=\int U_{n}^{2}(y)dy$. For the sequence of
the random variables $\left\{  W_{n}\right\}  $ we get the following recurrent
relationship:
\begin{gather}
W_{n+1}=(1-\frac{1}{n+1})^{2}W_{n}+\frac{2(1-\frac{1}{n+1})}{(n+1)h_{n+1}%
}\times\label{iteracje_W_n}\\
\times\int_{%
%TCIMACRO{\U{211d} }%
%BeginExpansion
\mathbb{R}
%EndExpansion
}U_{n}(y)\left(  Y_{n+1}K\left(  \frac{y-X_{n+1}}{h_{n+1}}\right)
-EY_{n+1}K\left(  \frac{y-X_{n+1}}{h_{n+1}}\right)  \right)  dy+\nonumber\\
+\frac{1}{(n+1)^{2}h_{n+1}^{2}}\int_{%
%TCIMACRO{\U{211d} }%
%BeginExpansion
\mathbb{R}
%EndExpansion
}\left(  Y_{n+1}K\left(  \frac{y-X_{n+1}}{h_{n+1}}\right)  -EY_{n+1}K\left(
\frac{y-X_{n+1}}{h_{n+1}}\right)  \right)  ^{2}dy.\nonumber
\end{gather}
Since we have an assertion (\ref{3_pom_reg}) of Lemma \ref{pomocniczy} applied
to $W_{n}\allowbreak=\allowbreak U_{n}$ we have the following recurrent
relationship
\begin{gather*}
EW_{n+1}=(1-\frac{1}{n+1})^{2}EW_{n}+\frac{1}{(n+1)^{2}h_{n+1}^{2}}\times\\
\times\int E\left(  Y_{n+1}K\left(  \frac{y-X_{n+1}}{h_{n+1}}\right)
-EY_{n+1}K\left(  \frac{y-X_{n+1}}{h_{n+1}}\right)  \right)  ^{2}dy\\
\leq(1-\frac{1}{n+1})^{2}EW_{n}+\frac{1}{(n+1)^{2}h_{n+1}^{2}}\times\\
\times h_{n+1}\int\int K^{2}(z)v(y-zh_{n+1})f(y-zh_{n+1})dzdy.
\end{gather*}
Hence on the basis of Lemma \ref{podstawowy} we see that the sequence
$\left\{  EW_{n}\right\}  _{n\geq0}$ converges to zero.\ Further, let us
consider sequence random variables:
\[
\left\{  \sum_{n\geq1}^{N}\frac{1}{nh_{n}}\int U_{n-1}(y)\left(  K\left(
\frac{y-X_{n}}{h_{n}}\right)  -EK\left(  \frac{y-X_{n}}{h_{n}}\right)
\right)  dy\right\}  _{N\geq1}.
\]
Taking advantage of the property (\ref{3_pom_reg}) of Lemma \ref{pomocniczy}
we see that it is a martingale with respect to filtration $\left\{
\mathcal{G}_{n}\right\}  $. This martingale\ this converges almost surely, if
only for example
\[
\sum_{n=1}^{\infty}\frac{1}{n^{2}h_{n}^{2}}E\left(  \int U_{n-1}(y)\left(
Y_{n}K\left(  \frac{y-X_{n}}{h_{n}}\right)  -EY_{n}K\left(  \frac{y-X_{n}%
}{h_{n}}\right)  \right)  dy\right)  ^{2}<\infty.
\]
On the base of assertion \ref{2_pom_reg} of Lemma \ref{pomocniczy} one can
see, that this condition is satisfied when
\[
\sum_{n=1}^{\infty}\frac{1}{n^{2}h_{n}}EW_{n-1}\int K^{2}(z)dz<\infty.
\]
This condition is satisfied, since the sequence $\left\{  EW_{n}\right\}  $
converges to zero and the series $\sum_{n\geq1}\frac{1}{n^{2}h_{n}}$ is
convergent. Hence, returning to the relationship (\ref{iteracje_W_n}) on the
basis of Lemma \ref{podstawowy} we deduce that the sequence random variables
$\left\{  W_{n}\right\}  $ converges to zero almost surely, when converge the
following series
\[
\sum_{n=1}^{\infty}\frac{1}{n^{2}h_{n}^{2}}\int\left(  K\left(  \frac{y-X_{n}%
}{h_{n}}\right)  -EK\left(  \frac{y-X_{n}}{h_{n}}\right)  \right)  ^{2}dy
\]
and
\[
\sum_{n=1}^{\infty}\frac{2(1-\frac{1}{n})}{nh_{n}}\int T_{n-1}(y)\left(
K\left(  \frac{y-X_{n}}{h_{n}}\right)  -EK\left(  \frac{y-X_{n}}{h_{n}%
}\right)  \right)  dy.
\]
Convergence almost everywhere of the second one was already proven above. The
convergence almost everywhere of the first series, follows observation, that
its elements are positive, inequality $\operatorname*{var}(Z)\leq EZ^{2}$ true
for any random variable and from the equality (\ref{1_pom_reg}).

(\ref{3zb_iter}) will be proved by showing, that $EQ_{n}%
(y)\underset{n\rightarrow\infty}{\longrightarrow}r(y)f(y)$ for almost all
$y\in%
%TCIMACRO{\U{211d} }%
%BeginExpansion
\mathbb{R}
%EndExpansion
\mathbb{\,}$. In order to show this, let us notice that from formula
(\ref{postac_EQ_n}) it follows on the basis of Lemma \ref{lemosr}, that
\[
EQ_{n}(y)=\frac{1}{n}\sum_{i=1}^{n}E\left(  \frac{1}{h_{i}}Y_{i}K\left(
\frac{y-X_{i}}{h_{i}}\right)  \right)  .
\]
Moreover, from Lemma \ref{pomocniczy} it follows that $E\left(  \frac{1}%
{h_{i}}Y_{i}K(\frac{y-X_{i}}{h_{i}})\right)  \underset{i\rightarrow
\infty}{\longrightarrow}r(y)f(y)$ for almost all $y$, hence basing on lemma
\ref{lemosr}, we get assertion.
\end{proof}

\begin{example}
Sequence of two-dimensional observations of the random variables was taken as
before, $N\allowbreak=\allowbreak5000$. Observations of the random variables
$\left\{  X_{i}\right\}  _{i\geq1}$ having distribution being a mixture of
Normal distributions (as in example \ref{przyk_gest_iter}). As the second
coordinate of our two-dimensional vector we took the transformed first
coordinate, i.e., more precisely, one took $Y_{i}=f(X_{i})+\xi_{i}$ ;
$i=1,\ldots,N$. Function $f$ was equal to
\[
f(x)=\left\{
\begin{array}
[c]{ccc}%
-1+.2\ast(x+\frac{\pi}{2}) & for & x<-\pi/2\\
\sin x & for & -\pi/2\leq x\leq\pi/2\\
1+.2\ast(x-\frac{\pi}{2}) & for & x>\pi/2
\end{array}
\right.  ,
\]
while sequence $\left\{  \xi_{i}\right\}  _{i\geq1}$ consisted of i.i.d.
having Normal distribution $N(0,1)$. Hence, as it can be seen $E(Y|X)=f(X)$
almost surely. The sequence$\left\{  h_{i}\right\}  $ was\ such as win the
example \ref{przyk_gest_iter}. After $N$ iterations one obtained:\newline%
%TCIMACRO{\FRAME{dtbpF}{2.5547in}{1.5895in}{0pt}{}{}{ob48hh05.eps}%
%width 2.5547in;  height 1.5895in;  depth 0pt;  original-width 6.2379in;
%original-height 3.87in;  cropleft "0";  croptop "1";  cropright "1";
%cropbottom "0";  filename 'OB48HH05.eps';file-properties "XNPEU";}}}%
%BeginExpansion
\begin{center}
\includegraphics[
height=1.5895in,
width=2.5547in
]%
{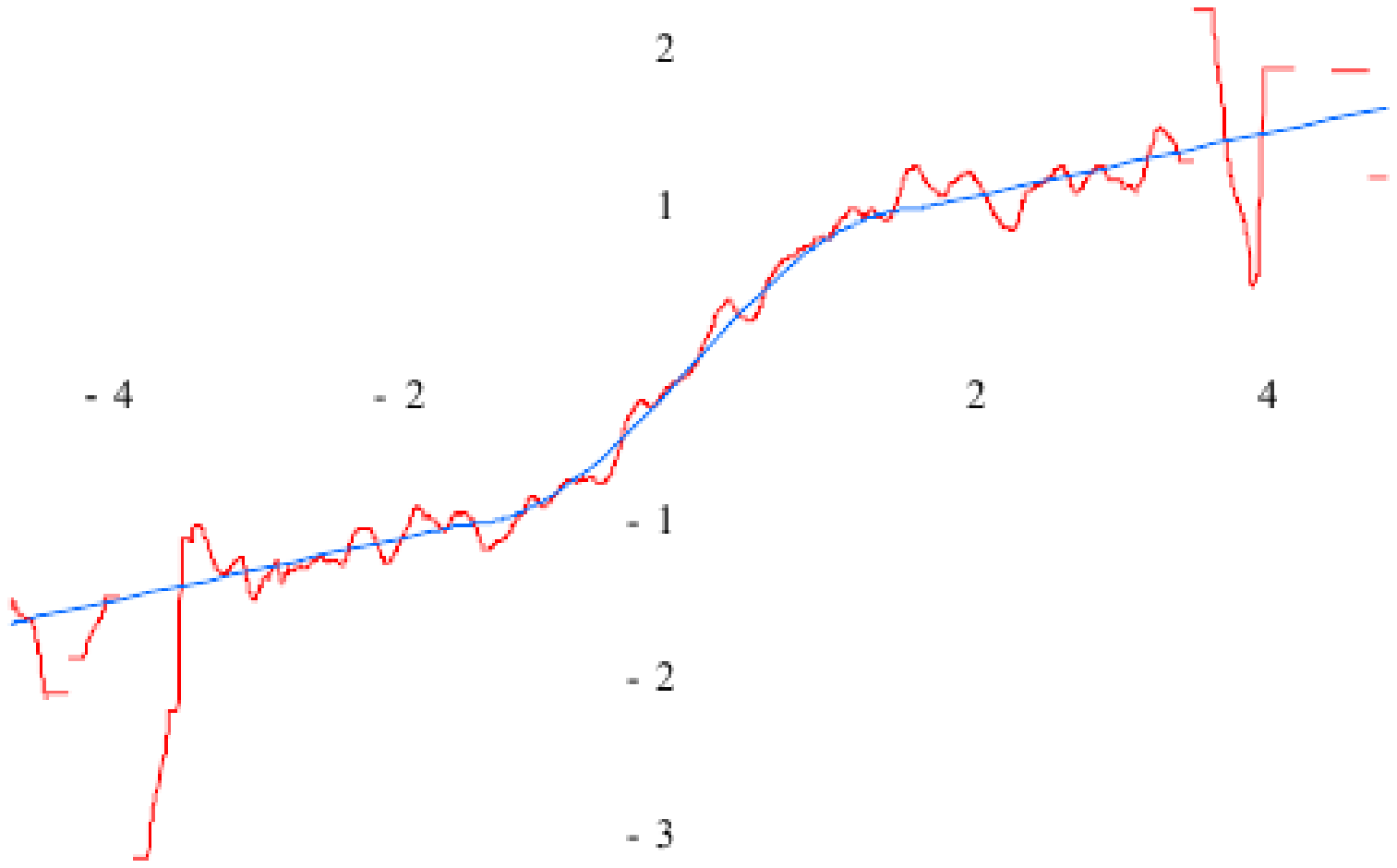}%
\end{center}
%EndExpansion
In the figure above both regression function and its estimator were plotted.
\end{example}

\begin{remark}
To finish this section let us notice that regression estimator
(\ref{est-reg-iter}) can be viewed as a weighted mean (Riesz's mean) of points
$\left\{  Y_{i}\right\}  $ with weights being random functions $\left\{
\frac{1}{h_{i}}K(\frac{x-X_{i}}{h_{i}})\right\}  $. Hence, one can present
regression estimator in the iterative form
\begin{equation}
\hat{R}_{n+1}(x)=(1-M_{n}(x))\hat{R}_{n}(x)+M_{n}(x)Y_{n+1}, \label{uog_iter}%
\end{equation}
where we denoted :
\[
M_{n}(x)=\frac{\frac{1}{h_{n+1}}K\left(  \frac{x-X_{n+1}}{h_{n+1}}\right)
}{\sum_{i=1}^{n+1}\frac{1}{h_{i}}K\left(  \frac{x-X_{i}}{h_{i}}\right)  }.
\]
Formula (\ref{uog_iter}) differs from the so far considered ones in that : a)
weights are now as it was mentioned, random variables correlated with
\textquotedblright averaged variables\textquotedblright\ $\left\{
Y_{i}\right\}  _{i\geq1}$, b) these weights are also functions of some random
variables, that is we deal with indexed families of \textquotedblright random
weights\textquotedblright. The theory of such averages (far reaching
generalizations of Riesz's means) is waiting for development!
\end{remark}

\chapter{Iterative methods of identification\label{identyf}}

\label{identifikacja}The issue of identification concerns the following
problems. Suppose, that we observe some stochastic processes $\mathcal{Y=}%
\left\{  \mathbf{y}_{i}\right\}  _{i\in I}$, where $I$ is some set of indices
( e.g. time instances). $\mathcal{Y}$ is interesting for us for some reasons.
If e.g. it is a sequence of prices of some shares on the stock exchange that
we are interested in, then it is obvious, that we are interested in this
process, and we even would like to know about it that much so as to be able to
predict its future values. It is not difficult to give other, less egocentric
reasons for which some processes can be interesting for us. E.g. vector
$\mathbf{y}_{i}$ can contain information on the levels of water in different
points of some river basin at the moment $i$. It is clear, that it is
extremely important for the whole community residing in a given territory is
to predict the values of this vector in the future moments of time. Prediction
is the next stage. First one has to define a reasonably reliable model of this
process that is to dutifully it.

The problem of identification is very broad and complex. It appears in
different branches of system engineering and control theory. A broad
discussion of issues of identification would require a separate book. Besides,
one would have to introduce reader in issues of control theory, particularly
in the stability theory of differential equations. In this chapter, we want to
indicate applications of ideas developed previously in chapters
\ref{aproksymacja} and \ref{metody_jadrowe}. We do not even pretend to bring a
comprehensive identification of the problem. We will only indicate partial
problems associated with it, that can be attacked by methods presented in this
book. In total in both parts of the present chapter will mainly present
examples of identification indicating, by the way, theoretical problems. For
the reader not interested much in identifications, such presentation is
enough. Readers more interested in identification are referred to literature.
It is very vast and it is impossible to mention all positions dedicated to
those issues. As we already mentioned, in order to understand well these
problems, one has to get knowledge of notions and results associated with
control theory. That is why we advice interested in identification readers to
get familiar with this theory.

We will distinguish \emph{parametric and nonparametric identification}.

\section{Parametric identification}

\subsection{Estimating functions}%

\index{Estimating Function}%

The problem of parametric identifications can be set in the following way. Let
be given observations of the process $\mathcal{Y}$. We assume, that this
process is generated with the help of the following procedure:%

\begin{equation}
\mathbf{y}_{n+1}=\mathbf{H}_{n}\mathbf{(y}_{n},\mathbf{p,\xi}_{n}%
\mathbf{);}n\in I, \label{gen_procesu}%
\end{equation}
where $\left\{  \mathbf{\xi}_{n}\right\}  _{n\in I}$ is some sequence of
random vectors, and functions $\left\{  \mathbf{H}_{n}\right\}  _{n\in I}$ are
known. Suppose that the values of $\mathcal{Y}$ are assumed to be in $%
%TCIMACRO{\U{211d} }%
%BeginExpansion
\mathbb{R}
%EndExpansion
^{d}$. Equation (\ref{gen_procesu}) represents the so-called parametric model
of the identified process. Unknown is only the value of the parameter vector
$\mathbf{p}$, assumed values in some subset $\mathbf{P\subseteq}%
%TCIMACRO{\U{211d} }%
%BeginExpansion
\mathbb{R}
%EndExpansion
^{q}$. One would like to find $\mathbf{p}$. Of course, one can imagine more
complex models of the processes that we are interested in. One has to remember
that we have only a finite number of measurements of the values of the process
at our disposal. If the model contains too many parameters, then having
limited a number of observations, confidence intervals of parameters will be
very large. May it be better to construct a model with a smaller number of
parameters and determine them more precisely? Such questions always accompany
those who deal with identification. When choosing model it is good to the
member, that often excellent results are obtained considering only simple
models of type ARMA of order $2$ or $3$ for the process itself, or for
differences of order at most $3$. It is shown emphatically by examples from
the book of Box and Jenkins \cite{BJ70}.

On the other hand, sometimes we have a sufficiently long sequence of
observations at our disposal, or even limitless in this sense, that
observations come at every time instant (the so-called observations
\emph{on-line}), then, of course, it is tempting to consider the more precise
model. Generally, we will assume, that information about the process' model
are contained in the form of the sequence of the so-called \emph{\ estimating
functions }$\left\{  \mathbf{F}_{i}(\mathbf{Y}_{i},\mathbf{p)}\right\}  _{i\in
I}$, where $\mathbf{Y}_{i}=\{\mathbf{y}_{i},\mathbf{y}_{i-1},\ldots
,\mathbf{y}_{0}\}$. The notion of estimating equations and functions has a
long history, that will not be discussed here in detail. We will mention only,
that these notions were discussed are in the papers: \cite{Durbin62},
\cite{[SZA2]}, \cite{[SZA312]}, \cite{[SZA320]}, \cite{[SZA22]},
\cite{Godambe87}, \cite{Godambe89}, \cite{Godambe91}, \cite{Heyde97}. For the
purpose of this book we will define estimating function in the following way:
a mapping $\mathbf{F}_{i}\allowbreak\mathbf{:}\allowbreak\underset{i+1\text{
times}}{\underbrace{%
%TCIMACRO{\U{211d} }%
%BeginExpansion
\mathbb{R}
%EndExpansion
^{d}\times\ldots\times%
%TCIMACRO{\U{211d} }%
%BeginExpansion
\mathbb{R}
%EndExpansion
^{d}}}\allowbreak\times\allowbreak%
%TCIMACRO{\U{211d} }%
%BeginExpansion
\mathbb{R}
%EndExpansion
^{q}\allowbreak\rightarrow\allowbreak%
%TCIMACRO{\U{211d} }%
%BeginExpansion
\mathbb{R}
%EndExpansion
^{q}$ that is differentiable with respect to the last argument and satisfies
the following conditions:%

\begin{gather}
\forall\mathbf{\hat{p}\in P:}E_{\mathbf{\hat{p}}}\mathbf{F}_{i}\mathbf{(Y}%
_{i},\mathbf{\hat{p})}=0\label{FE_1}\\
\forall\mathbf{\hat{p}\in P}:\det E_{\mathbf{\hat{p}}}\frac{\partial
\mathbf{F}_{i}\mathbf{(Y}_{i},\mathbf{\hat{p})}}{\partial\mathbf{\hat{p}}^{T}%
}\neq0\label{FE_2}\\
\forall\mathbf{\hat{p}\in P}:\det E_{\mathbf{\hat{p}}}\mathbf{F}%
_{i}\mathbf{(Y}_{i},\mathbf{\hat{p})F}_{i}^{T}\mathbf{(Y}_{i},\mathbf{\hat
{p})\neq0} \label{FE_3}%
\end{gather}
is called an estimating function\emph{\ }based on $i+1$ first observations. In
the above mentioned formula $E_{\mathbf{\hat{p}}}(.)$ denotes expectation
under the assumption, that \textquotedblright the true parameter
\textquotedblright\ is $\mathbf{\hat{p}}$, that is with respect to
distribution in which we set $\mathbf{\hat{p}}$ instead $\mathbf{p}$\textbf{.}
The equation:
\[
\mathbf{F}_{i}(\mathbf{Y}_{i},\mathbf{p)=0}%
\]
is called an estimating equation.

What is the connection of the model (\ref{gen_procesu}) with the estimating
function? Generally, one can state that one model can lead to many different
of estimating functions. For example, we can take:
\[
\mathbf{F}_{i}^{(1)}(\mathbf{Y}_{i},\mathbf{\hat{p})=y}_{i}-E_{\mathbf{\hat
{p}}}\mathbf{H}_{i-1}(\mathbf{y}_{i-1},\mathbf{\hat{p},\xi}_{i-1}),\;i\geq2
\]
if $d=q$, or
\begin{align*}
\mathbf{F}_{i}^{(2)}(\mathbf{Y}_{i},\mathbf{\hat{p})}  &  =\mathbf{w}%
_{i}(\mathbf{\hat{p})}\left(  \mathbf{y}_{i}-E_{\mathbf{\hat{p}}}%
\mathbf{H}_{i-1}(\mathbf{y}_{i-1},\mathbf{\hat{p},\xi}_{i-1})\right)
,\;i\geq2\\
\mathbf{F}_{i}^{(3)}(\mathbf{Y}_{i},\mathbf{\hat{p})}  &  =\sum_{j=2}%
^{i}\mathbf{w}_{j}(\mathbf{\hat{p})}\left(  \mathbf{y}_{j}-E_{\mathbf{\hat{p}%
}}\mathbf{H}_{j-1}(\mathbf{y}_{j-1},\mathbf{\hat{p},\xi}_{j-1})\right)
,\;i\geq2,
\end{align*}
where $\left\{  \mathbf{w}_{i}(\mathbf{\hat{p})}\right\}  $ are some $q\times
d$ -matrices with coordinates depending on $\mathbf{\hat{p}}$\textbf{.} Of
course, it may happen, that for functions $\mathbf{F}_{i}^{(2)}$ will not
satisfy condition (\ref{FE_2}) then one has to give up this estimating
function and select the other, modified one. Generally, model of the process
is something given, unalterable. As far as the choice of estimating functions
we have some freedom. We will not discuss these issues.

In the present chapter, we will assume that the sequence of estimating
functions has been already somehow\ chosen.

Let us notice that if $\mathbf{F}_{i}$ is estimating function, then so is also
$\mathbf{\Lambda(\hat{p})F}_{i}$ for nonsingular. matrix $\mathbf{\Lambda}$ of
order $q$. In order to avoid such trivial situations, and also because we are
going to compare different estimating functions, it would be reasonable to
normalize them somehow. To this end we will further assume that
\begin{equation}
\left.  E_{\mathbf{p}}\frac{\partial\mathbf{F}_{i}\mathbf{(Y}_{i}%
,\mathbf{\hat{p})}}{\partial\mathbf{\hat{p}}^{T}}\right\vert _{\mathbf{\hat
{p}=p}}=\mathbf{I.} \label{unormowanie}%
\end{equation}

We will introduce an order inside the set of estimating functions based on $i
$ observations in the following way. Let be given two estimating functions
$\mathbf{F}_{i}$ and $\mathbf{G}_{i}$, basing on the same number of
observations and satisfying condition (\ref{unormowanie}). Then function
estimating $\mathbf{F}_{i}$ is called \emph{not worse} (\emph{better) }than
estimating function $\mathbf{G}_{i}$, when:\emph{\ }
\begin{equation}
\forall\mathbf{\hat{p}\in}%
%TCIMACRO{\U{211d} }%
%BeginExpansion
\mathbb{R}
%EndExpansion
^{q}:tr\left\{  E_{\mathbf{\hat{p}}}\mathbf{F}_{i}\mathbf{(Y}_{i}%
,\mathbf{\hat{p})F}_{i}^{T}\mathbf{(Y}_{i},\mathbf{\hat{p})}\right\}
\mathbf{\leq(<)}tr\left\{  E_{\mathbf{\hat{p}}}\mathbf{G}_{i}\mathbf{(Y}%
_{i},\mathbf{\hat{p})G}_{i}^{T}\mathbf{(Y}_{i},\mathbf{\hat{p})}\right\}
\mathbf{,} \label{liniowy}%
\end{equation}
where $tr(A)$ denotes trace of matrix $A.$

One introduces also partial order in the set of estimating functions in a
similar way. Namely, $\mathbf{F}_{i}$ is not worse estimating function than
$\mathbf{G}_{i}$, if:
\begin{equation}
\forall\mathbf{\hat{p}\in}%
%TCIMACRO{\U{211d} }%
%BeginExpansion
\mathbb{R}
%EndExpansion
^{q}:\text{matrix }E_{\mathbf{\hat{p}}}\mathbf{F}_{i}\mathbf{(Y}%
_{i},\mathbf{\hat{p})F}_{i}^{T}\mathbf{(Y}_{i},\mathbf{\hat{p})-}%
E\mathbf{_{\mathbf{\hat{p}}}G}_{i}\mathbf{(Y}_{i},\mathbf{\hat{p})G}_{i}%
^{T}\mathbf{(Y}_{i},\mathbf{\hat{p})} \label{czesciowy}%
\end{equation}
is negatively semidefinite. It is obvious that if $\mathbf{F}_{i}$ is not
worse than $\mathbf{G}_{i}$ in the sense of partial order given by
(\ref{czesciowy}), then it is also not worse in the sense of linear order
given by (\ref{liniowy}). It turns out that in the sense of partial order
there exists, by some regularity conditions, a maximal element. Namely, we
have the following theorem:

\begin{theorem}
Let $\Phi(\mathbf{Y}_{i},\mathbf{\hat{p})}$ will be the density of the
probability distribution of observations $\mathbf{Y}_{i}$. Let us assume that
$\Phi$ is differentiable with respect to $\mathbf{\hat{p}}$ and that matrix
$\mathbf{V(\hat{p})=}\left\{  E\left[  \frac{\partial\ln\Phi(\mathbf{Y}%
_{i},\mathbf{\hat{p})}}{\partial\mathbf{\hat{p}}^{T}}\frac{\partial\ln
\Phi(\mathbf{Y}_{i},\mathbf{\hat{p})}}{\partial\mathbf{\hat{p}}}\right]
\right\}  ^{-1}$ exists. Then the estimating function
\[
\mathbf{M(Y}_{i},\mathbf{\hat{p})=}\frac{\partial\ln\Phi(\mathbf{Y}%
_{i},\mathbf{\hat{p})}}{\partial\mathbf{\hat{p}}^{T}}%
\]
is the maximal in the sense of partial order introduced by (\ref{czesciowy}).
In other words, for every estimating function $\mathbf{F}_{i}\mathbf{(Y}%
_{i},\mathbf{\hat{p})}$ satisfying condition (\ref{unormowanie}) the matrix:
\[
E_{\mathbf{\hat{p}}}\mathbf{F}_{i}\mathbf{(Y}_{i},\mathbf{\hat{p})F}_{i}%
^{T}\mathbf{(Y}_{i},\mathbf{\hat{p})-V(\hat{p})}%
\]
is positively semidefinite. Moreover, this matrix is a zero matrix if and only
if:
\[
\mathbf{F}_{i}\mathbf{(Y}_{i},\mathbf{\hat{p})=\Lambda(\hat{p})M(Y}%
_{i},\mathbf{\hat{p})}%
\]
for some matrix $\Lambda$ having elements depending only on $\mathbf{\hat{p}%
.}$
\end{theorem}

\begin{proof}
Sketch of the proof. We have $0=E_{\mathbf{\hat{p}}}\mathbf{F}_{i}%
\mathbf{(Y}_{i},\mathbf{\hat{p})\allowbreak=}\int\mathbf{F}_{i}\mathbf{(Y}%
_{i},\mathbf{\hat{p})}\Phi(\mathbf{Y}_{i},\mathbf{\hat{p})}d\mathbf{Y}_{i}$.
Differentiating with respect to $\mathbf{\hat{p}}$ and assuming the
possibility of changing the order of integration and differentiation we get:%

\[
0=\int\left(  \frac{\partial}{\partial\mathbf{\hat{p}}^{T}}\mathbf{F}%
_{i}(\mathbf{Y}_{i},\mathbf{\hat{p})}\right)  \Phi(\mathbf{Y}_{i}%
,\mathbf{\hat{p})}d\mathbf{Y}_{i}+\allowbreak\int\mathbf{F}_{i}\mathbf{(Y}%
_{i},\mathbf{\hat{p})}\left(  \frac{\partial}{\partial\mathbf{\hat{p}}^{T}%
}\Phi(\mathbf{Y}_{i},\mathbf{\hat{p})}\right)  d\mathbf{Y}_{i}.
\]
But
\[
\int\frac{\partial}{\partial\mathbf{\hat{p}}^{T}}\mathbf{F}_{i}(\mathbf{Y}%
_{i},\mathbf{\hat{p})}\Phi(\mathbf{Y}_{i},\mathbf{\hat{p})}d\mathbf{Y}%
_{i}=\allowbreak E_{\mathbf{\hat{p}}}\frac{\partial}{\partial\mathbf{\hat{p}%
}^{T}}\mathbf{F}_{i}(\mathbf{Y}_{i},\mathbf{\hat{p})\allowbreak=I,}%
\]
on the basis of assumptions. Hence%

\[
\mathbf{I=\allowbreak}\int(-\mathbf{F}_{i}\mathbf{(Y}_{i},\mathbf{\hat{p}%
))}\Phi(\mathbf{Y}_{i},\mathbf{\hat{p})}\frac{\partial}{\partial
\mathbf{\hat{p}}^{T}}\ln\Phi(\mathbf{Y}_{i},\mathbf{\hat{p})}d\mathbf{Y}_{i}.
\]
Now we apply generalized by Cramer \cite{Cramer46}, Schwarz inequality to
vectors: $-\mathbf{F}_{i}\mathbf{(Y}_{i},\mathbf{\hat{p})}\Phi^{1/2}%
(\mathbf{Y}_{i},\mathbf{\hat{p})}$ and $\Phi^{1/2}(\mathbf{Y}_{i}%
,\mathbf{\hat{p})M(Y}_{i},\mathbf{\hat{p})}$. As a corollary we get a
statement, that the matrix:
\[
\mathbf{I-}\allowbreak\int\mathbf{F}_{i}\mathbf{(Y}_{i},\mathbf{\hat{p})F}%
_{i}^{T}\mathbf{(Y}_{i},\mathbf{\hat{p})}\Phi(\mathbf{Y}_{i},\mathbf{\hat{p}%
)}d\mathbf{Y}_{i}\mathbf{\times}\allowbreak\int\Phi(\mathbf{Y}_{i}%
,\mathbf{\hat{p})M(Y}_{i},\mathbf{\hat{p})M}^{T}\mathbf{(Y}_{i},\mathbf{\hat
{p})}d\mathbf{Y}_{i}%
\]
is negatively semidefinite. It is easy to get an assertion utilizing this fact
and remembering that $\int\Phi(\mathbf{Y}_{i},\mathbf{\hat{p})M(Y}%
_{i},\mathbf{\hat{p})M}^{T}\mathbf{(Y}_{i},\mathbf{\hat{p})}d\mathbf{Y}%
_{i}\allowbreak=\allowbreak\mathbf{V}^{-1}(\mathbf{\hat{p})}$.
\end{proof}

\begin{remark}
The above-mentioned theorem indicates, that the best estimating equation is
the so-called maximum likelihood equation. Hence, we get different, other than
the traditional justification of the use of the maximum likelihood method.
\end{remark}

Let us assume that there is given a sequence of observed values of estimating
functions $\left\{  \mathbf{F}_{i}(\mathbf{Y}_{i},\mathbf{\hat{p})}\right\}
_{i\geq1}$. Let us pay attention, that on the basis of our assumptions,
equations
\[
E_{\mathbf{p}}\mathbf{F}_{i}(\mathbf{Y}_{i},\mathbf{\hat{p})=0,}%
\]
as functions of $\mathbf{\hat{p}}$ have one common zero equal to $\mathbf{p}$.
Hence, one can use stochastic approximation methods, in order to estimate this
zero. In chapter \ref{aproksymacja} many different versions of stochastic
approximation were discussed. Some of them are really well fitted to be
applied in this case. Particularly useful seem to be a procedure
(\ref{nowaproc}) and the stating its convergence Theorem \ref{rozne_F}. We
will illustrate its use for parametric identification by the following examples.

\begin{example}
\label{poprzedni}Suppose, that process $\left\{  y_{i}\right\}  _{i\geq1}$ is
generated by the system defined by the relationship:
\begin{equation}
y_{i+1}=f(y_{i};p)+\zeta_{i};i\geq0, \label{przyk_rek}%
\end{equation}
where $f(x;p)=\left\{
\begin{array}
[c]{ccc}%
px & ,gdy & x<p\\
p^{2}/2+px/2 & ,gdy & x\geq p
\end{array}
\right.  $, that is function $f$ has e.g. for $p=.9$ the following
plot.\newline%
%TCIMACRO{\FRAME{itbpF}{2.3938in}{1.5912in}{0in}{}{}{obdjym00.eps}%
%width 2.3938in;  height 1.5912in;  depth 0in;  original-width 2.9698in;
%original-height 1.9649in;  cropleft "0";  croptop "1";  cropright "1";
%cropbottom "0";  filename 'OBDJYM00.eps';file-properties "XNPEU";}}}%
%BeginExpansion
{\includegraphics[
height=1.5912in,
width=2.3938in
]%
{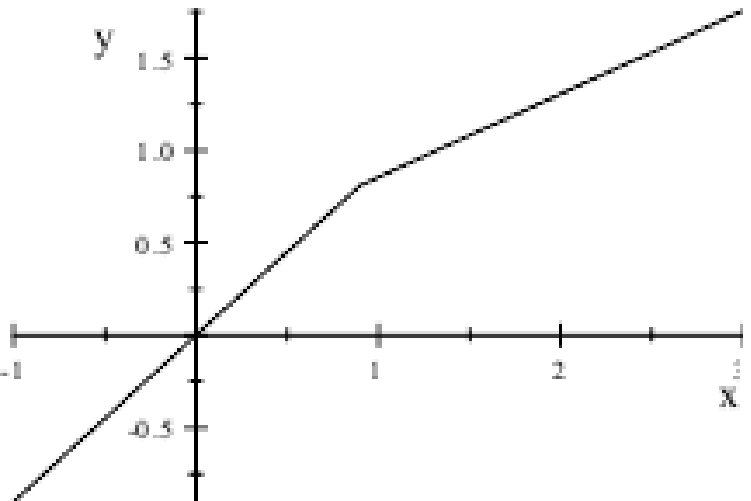}%
}%
%EndExpansion
. In order to be able to analyze further this procedure of identification of
this system, we will assume, that $p>0$. Moreover, it is not difficult to
notice, that to make the system stable with probability $1$, one has to assume
that, $p<1$. It follows from the fact, that for $p>1$ it would happen, that
for some $i$ we would have $y_{i}<0$ and $\xi_{i}$ not very large. Then
$y_{i+1}$ would be also negative and with positive probability $y_{i+1}%
<<y_{i},$. and so on. the subsequent values of the sequence $\{y_{i}\}$ could
decrease to $-\infty$. Let us notice that we have then%

\begin{align*}
Ey_{i+1}  &  =pEy_{i}I(y_{i}<p)+\frac{p^{2}}{2}+\frac{p}{2}Ey_{i}I(y_{i}\geq
p)=\\
&  =pEy_{i}+\frac{p}{2}(p-Ey_{i}I(y_{i}\geq p)).
\end{align*}
$\left\{  \zeta_{i}\right\}  _{i\geq0}$ is the sequence of the random
variables having zero mean and finite variance not unnecessarily independent.
Other assumptions concerning sequence $\left\{  \zeta_{i}\right\}  $ will be
given in the sequel. They will be concerned with ensuring convergence of
respective stochastic approximation procedures. Generally, these assumptions
will impose that the sequence $\left\{  \zeta_{i}\right\}  $ will satisfy
strong laws of large numbers. How to translate this requirement on assumptions
concerning of covariance functions $K(n,k)=\operatorname*{cov}(\zeta_{n}%
,\zeta_{k})$, given is e.g. in theorem \ref{mpwl} or its generalization. The
problem is to find a point $p$, defining function $f$, on the basis of
sequence of observations $\left\{  y_{i}\right\}  _{i\geq0}$. To appreciate
this ability of getting information about the distribution from the sequence
random variables by stochastic approximation, we will plot sequence $\left\{
y_{i}\right\}  _{i\geq0}$ simulated for $p=.9$ and the sequence $\left\{
\zeta_{i}\right\}  _{i\geq0}$, consisting of the time series of type
ARMA(2,2). Parameter $p$ was estimated with the help of a simple stochastic
approximation procedure:
\[
q_{i+1}=q_{i}-\frac{1}{i+1}(y_{i+1}-f(y_{i};q_{i})).
\]
For $p=.9$ one obtained the following plot of iterations:%
%TCIMACRO{\FRAME{dtbpF}{2.3696in}{1.4875in}{0pt}{}{}{ident_z_f_nielin3.eps}%
%width 2.3696in;  height 1.4875in;  depth 0pt;  original-width 5.7795in;
%original-height 3.6192in;  cropleft "0";  croptop "1";  cropright "1";
%cropbottom "0";  filename 'ident_z_f_nielin3.eps';file-properties "XNPEU";}}}%
%BeginExpansion
\begin{center}
\includegraphics[
height=1.4875in,
width=2.3696in
]%
{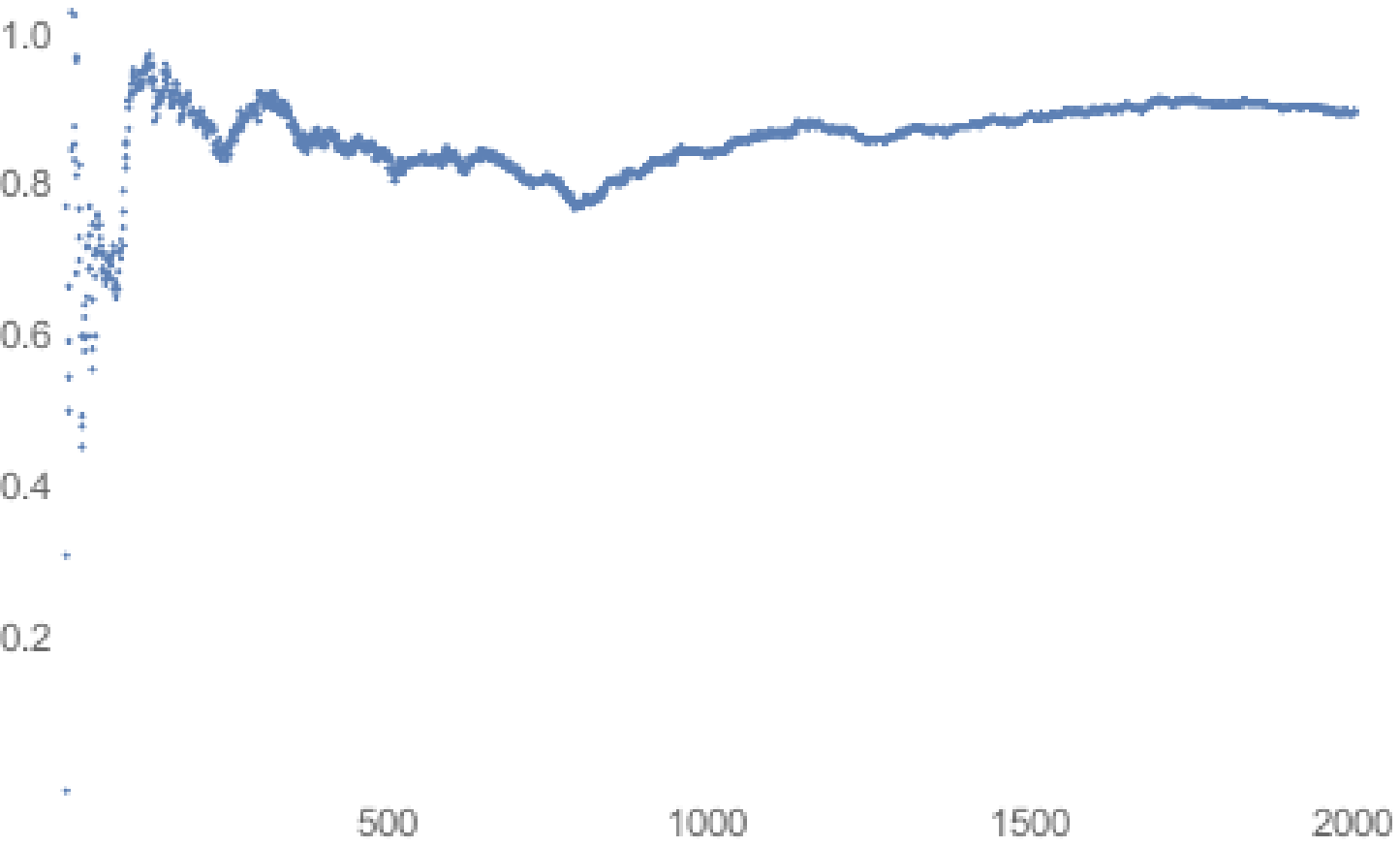}%
\end{center}
%EndExpansion

Convergence of this procedure followed from the modified Theorem
\ref{split_delta_proc}. This modification is set in this that we split
functions $\mathbf{F}_{n}$ on $\mathbf{G}_{n}=E(\mathbf{F}_{n}|\mathcal{F}%
_{n-1})$ and $\mathbf{F}_{n}-G_{n}$, since additive noises do not depend on
the present value of the estimator. This modification is thus in the spirit of
Theorem \ref{najprostsze}. Besides, one has to utilize Theorem \ref{mpwl}. On
the base of Theorem \ref{mpwl} it is somewhat easy to show that series
$\sum_{i\geq1}\frac{1}{i+1}\zeta_{i}$ is convergent with probability $1$,
remembering that the covariance function of the ARMA process is dominated by
the exponential function. Finally, we will mention only, that when applying
Theorem \ref{rozne_F} one used equality:
\[
(q-p)(f(y_{i};q)-f(y_{i};p))\allowbreak=\allowbreak\left\vert p-q\right\vert
^{2}\chi(y_{i};p,q),
\]
where%

\[
\chi(x;p,q)\allowbreak=\allowbreak\left\{
\begin{array}
[c]{ccc}%
x & ,when & x<\min(p,q)\\
\frac{x}{2}+\frac{p+q}{2}+\frac{\max(p,q)(x-\max(p,q)}{2(\max(p,q)-\min(p,q))}
& ,when & \min(p,q)\leq x<\max(p,q)\\
\frac{x}{2}+\frac{p+q}{2} & ,when & x\geq\max(p,q)
\end{array}
\right.  .
\]
It is easy to notice, that
\begin{equation}
\chi(x,p,q)\geq\left\{
\begin{array}
[c]{ccc}%
x & ,gdy & x<\min(p,q)\\
\frac{x+\min(p,q)}{2} & ,gdy & \min(p,q)\leq x<\max(p,q)\\
\frac{x}{2}+\frac{p+q}{2} & ,gdy & x\geq\max(p,q)
\end{array}
\overset{df}{=}\eta(x,p,q).\right.  \label{chi_}%
\end{equation}
Quantity $\chi(y_{i},p,q)$ or its lower bound $\eta(y_{i},p,q)$ given in the
formula (\ref{chi_}) we treat as $\delta$ appearing in Theorem
\ref{split_delta_proc} and we decompose $\delta$ in the following way
$E\eta(y_{i},p,q)+\eta(y_{i},p,q)-E\eta(y_{i},p,q)$. In order to be able to
make use of this theorem one has to show that i) $lim_{i\rightarrow\infty
}E\eta(y_{i},p,q)>0$ and ii) series $\sum_{i\geq0}\frac{1}{i+1}(\eta
(y_{i},p,q)-E\eta(y_{i},p,q))$ is bounded almost surely. Property ii) is
intuitively obvious. In order to make the discussed series\ convergent, one
has to show, that $\operatorname*{var}(\eta(y_{i},p,q))$ is bounded, which in
the face of assumed stationarity of the process $\left\{  y_{i}\right\}  $ is
obvious and also for example to show, that the covariance function of the
process $\left\{  \eta\left(  y_{i},p,q\right)  \right\}  $ decreases
exponentially, which again confronted with the fact that the process $\left\{
y_{i}\right\}  $ is Markov should be satisfied. To show condition i) is may be
a bit more complex. Let us leave it to the interested reader. Let us notice
only that fact that \underline{$lim$}$_{i\rightarrow\infty}E\eta
(y_{i},p,q)\neq0$ is connected with the stationarity of the process $\left\{
y_{i}\right\}  $ and also with the fact that \underline{$lim$}$_{i\rightarrow
\infty}\operatorname*{var}(\zeta_{i})>0.$
\end{example}

\begin{example}
Let the process $\left\{  y_{i}\right\}  _{i\geq0}$ will be generated with the
help of the recursive equation:
\begin{equation}
y_{i+3}=1.6y_{i+2}-1.475y_{i+1}+.7605y_{i}+\zeta_{i+3}, \label{symualacja1}%
\end{equation}
$y_{0},y_{1},y_{2}$ are given numbers, and $\left\{  \zeta_{i}\right\}
_{i\geq0}$ is sequence independent random variables having $N(0,1)$
distributions. Information that is at our disposal, consists of a sequence of
observations $\mathcal{Y}=\left\{  y_{i}\right\}  _{i\geq0}$ of our process
(\ref{symualacja1}). This data one can e.g. put in the following plot in the
so-called phase coordinates $(y_{i},y_{i+1})$.%
\[%
%TCIMACRO{\FRAME{itbpF}{2.9326in}{1.8204in}{0in}{}{}{symulacja2.eps}%
%width 2.9326in;  height 1.8204in;  depth 0in;  original-width 5.7795in;
%original-height 3.576in;  cropleft "0";  croptop "1";  cropright "1";
%cropbottom "0";  filename 'symulacja2.eps';file-properties "XNPEU";}}}%
%BeginExpansion
{\includegraphics[
height=1.8204in,
width=2.9326in
]%
{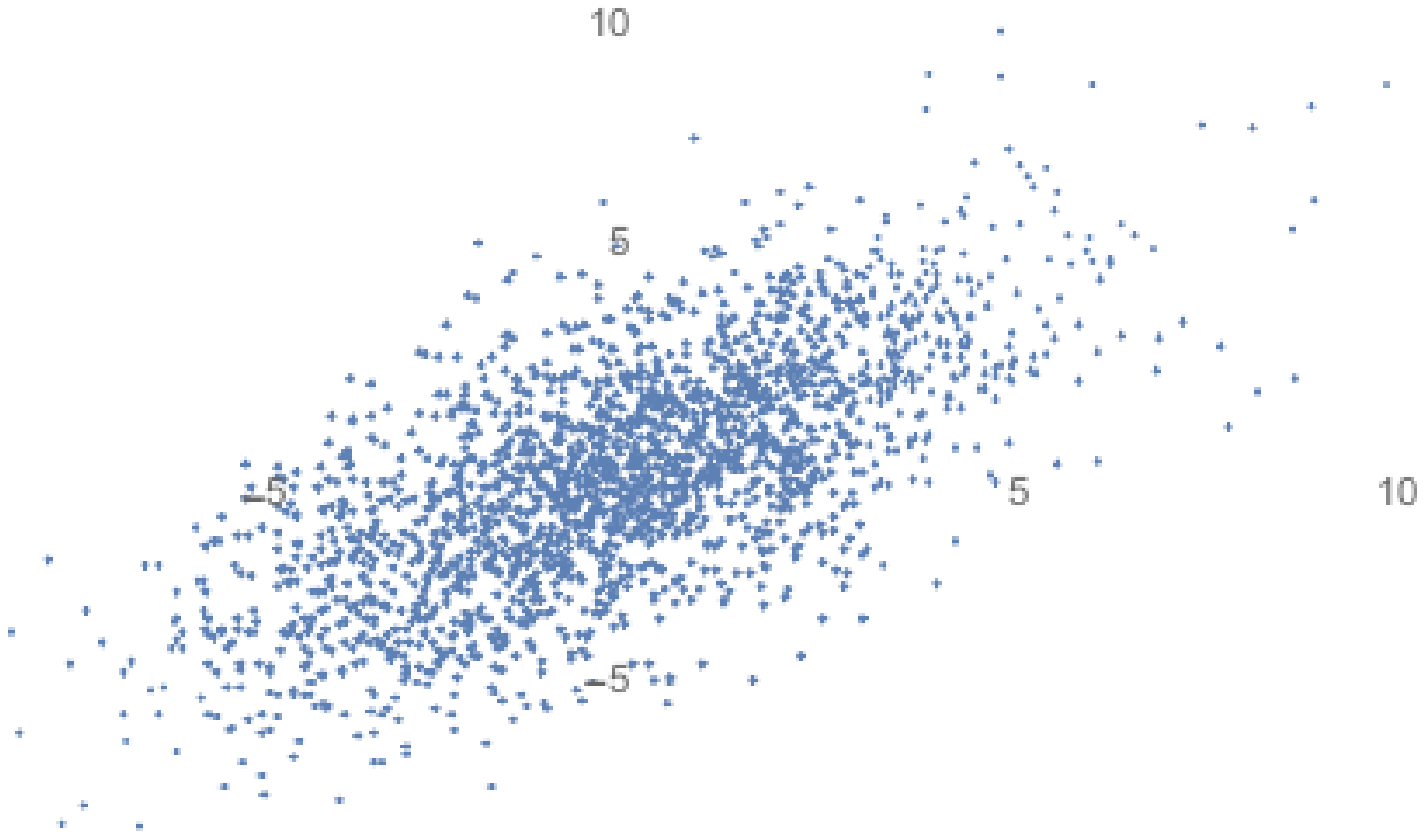}%
}%
%EndExpansion
.
\]
As one can deduce from these plots it would be rather difficult to deduce that
the values of parameters characterizing process $\mathcal{Y}$ are
$(1.6,\allowbreak-1.475,\allowbreak.7605)$. The issue of identification lies
just in finding these parameters on the base of the sequence of observations
$\mathcal{Y}$. Vector of parameters $\mathbf{a}^{T}=(1.6,-1.475,.7605)$ will
be recreated with the help of one of the following procedures:
\begin{equation}
\mathbf{b}_{i+1}=\mathbf{b}_{i}+\frac{1}{i+1}\mathbf{v}_{i}(y_{i+3}%
-\mathbf{b}_{i}^{T}\mathbf{v}_{i});\mathbf{b}_{0}=(0,0,0)^{T};i\geq0,
\label{procedura1}%
\end{equation}
or
\begin{equation}
\mathbf{a}_{i+1}=\mathbf{a}_{i}+\frac{2}{i+1}(\frac{1}{i+1}\sum_{k=0}%
^{i}\mathbf{v}_{k}\mathbf{v}_{k}^{T})^{-1}\mathbf{v}_{i}(y_{i+3}%
-\mathbf{a}_{i}^{T}\mathbf{v}_{i});\mathbf{a}_{0}=(0,0,0)^{T};i\geq0,
\label{procedura2}%
\end{equation}
where we denoted $\mathbf{v}_{i}=(y_{i+2},y_{i+1},y_{i})^{T}$. The results
were the following:

a) for the procedure (\ref{procedura1})%
%TCIMACRO{\FRAME{dtbpF}{2.5555in}{1.5601in}{0pt}{}{}{procedura1.eps}%
%width 2.5555in;  height 1.5601in;  depth 0pt;  original-width 6.4411in;
%original-height 3.9202in;  cropleft "0";  croptop "1";  cropright "1";
%cropbottom "0";  filename 'procedura1.eps';file-properties "XNPEU";}}}%
%BeginExpansion
\begin{center}
\includegraphics[
height=1.5601in,
width=2.5555in
]%
{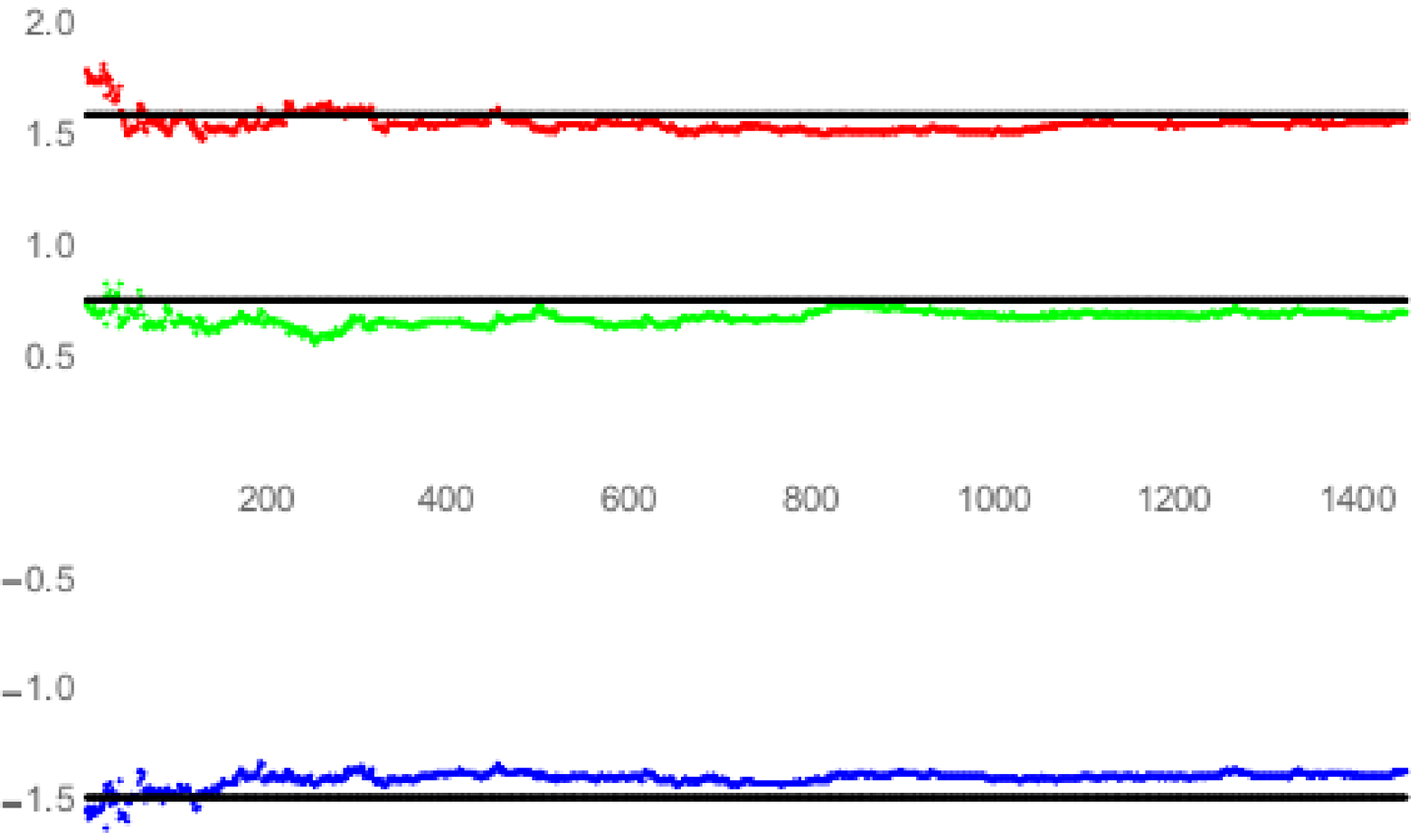}%
\end{center}
%EndExpansion

b) for the procedure (\ref{procedura2}):%
%TCIMACRO{\FRAME{dtbpF}{2.5105in}{1.4909in}{0pt}{}{}{procedura2.eps}%
%width 2.5105in;  height 1.4909in;  depth 0pt;  original-width 5.7795in;
%original-height 3.4238in;  cropleft "0";  croptop "1";  cropright "1";
%cropbottom "0";  filename 'procedura2.eps';file-properties "XNPEU";}}}%
%BeginExpansion
\begin{center}
\includegraphics[
height=1.4909in,
width=2.5105in
]%
{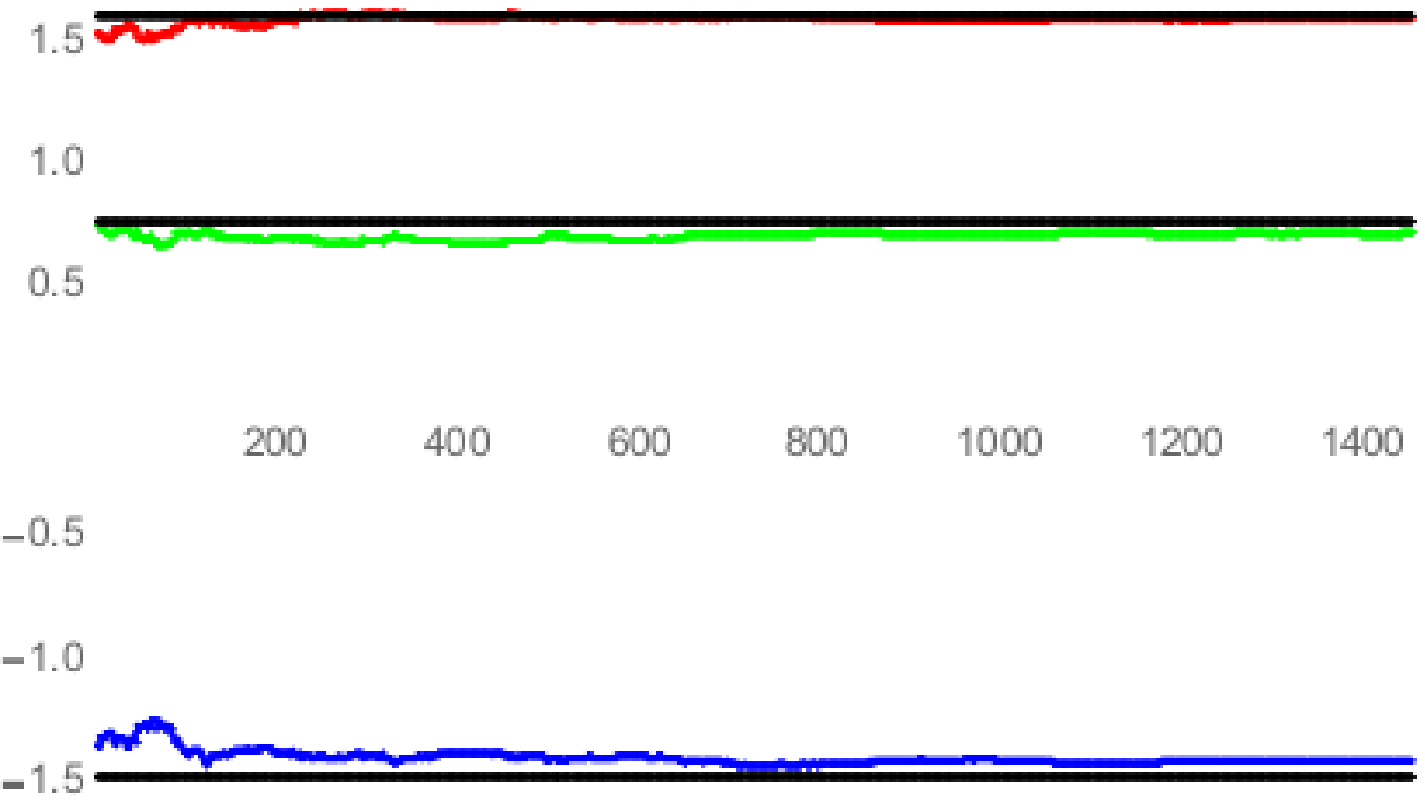}%
\end{center}
%EndExpansion
. As one can observe the convergence was relatively quick. The justification
of convergence is supplied by Lemma \ref{nieprostyrek} and modification based
on it. We used the procedure defined in Theorem \ref{rozne_F}. We will not
provide details. In chapter \ref{aproksymacja} there were presented many
different stochastic approximation procedures, whose convergence have similar
proofs and that one can easily modify and extend. The use of Lemma
\ref{nieprostyrek} seems to be crucial in this case, since a characteristic
feature of both discussed procedures is the fact, that matrix $\frac{\partial
}{\partial\mathbf{a}^{T}}\left(  \mathbf{v}_{i}(y_{i+3}-\mathbf{v}_{i}%
^{T}\mathbf{a)}\right)  =\mathbf{v}_{i}\mathbf{v}_{i}^{T}$ is of order $1$,
hence condition:
\[
\left(  \mathbf{a-\alpha}\right)  ^{T}\mathbf{v}_{i}\mathbf{v}_{i}%
^{T}(\mathbf{a-\alpha)\geq}\delta_{i}\left\vert \mathbf{a-\alpha}\right\vert
^{2};\underset{i\rightarrow\infty}{\lim}\delta_{i}>0
\]
is not satisfied. Instead, one can notice, that matrix $E\mathbf{v}%
_{i}\mathbf{v}_{i}^{T}$ is nonsingular. Hence, one can decompose sequence
$\left\{  \delta_{i}\right\}  $ in the following way: $\delta_{i}=\delta
_{i}^{^{\prime}}+\delta_{i}^{^{\prime\prime}}$, $\underset{i\rightarrow
\infty}{\lim}\delta_{i}^{^{\prime}}>0$, series $\sum_{i\geq0}\mu_{i}\delta
_{i}^{^{\prime\prime}}$ is convergent a.s. and $\underset{i\rightarrow
\infty}{\lim}\mu_{i}\delta_{i}^{^{\prime^{\prime}}}=0$ a.s. so that one can
apply Lemma \ref{nieprostyrek}.
\end{example}

To finish this part dedicated to parametric identification let us mention the
following problem. It concerns the construction of \textquotedblright
optimal\textquotedblright\ identification procedures. Namely, let us treat a
given sequence of estimating functions $\left\{  F_{i}\right\}  _{i\geq1}$ as
a sequence of \textquotedblright elementary estimating
functions\textquotedblright. Suppose, that we will use these functions to
recursive estimation utilizing procedures of stochastic approximation, as it
was done in the above-mentioned examples. Can one, and if so, then how to
improve or modify data coming from estimating functions, in order to
accelerate the convergence of the respective stochastic approximation
procedure. The problem is important and non-trivial. Partial result in this
direction was presented is in the paper \cite{szab5}. It was assumed there,
that together with every estimating function, we have at our disposal some
auxiliary information in the following of the form. Namely, let us assume,
that
\[
\forall i\geq1\exists k\leq i:\left.  E(\mathbf{F(Y}_{i},\mathbf{\hat{p}%
})|\mathbf{Y}_{k})\right\vert _{\mathbf{\hat{p}=p}}=0
\]
The problem, that was aimed to be solved in the discussed paper was : how to
find sequence of \textquotedblright weights\textquotedblright\ - random
variables that depend on the measurements up to moment $k(i)$ and
$\mathbf{\hat{p}}$ so that respective identification procedure with new
estimating functions of the form
\[
\mathbf{\tilde{F}}_{i}\mathbf{(Y}_{i},\mathbf{\hat{p})=w}_{i}(\mathbf{Y}%
_{k(i)}\mathbf{\hat{p})F(Y}_{i},\mathbf{p)}%
\]
converge quicker (the quickest?!). We will not go into details here. Let us
mention only that such \textquotedblright weights \textquotedblright\ were
found. It turns out that they have a relatively simple form when $k(i)=i-1.$

\section{Nonparametric identification}

We want to indicate in this part, the possibilities of using methods of
regression estimation for identification. The general idea behind this method
of identification is the following. Suppose, that some stochastic processes
$\left\{  x_{i}\right\}  _{i\geq1}$ is generated, by the following recursive
equation:
\[
x_{i+1}=f(x_{i})+\xi_{i},
\]
where the sequence $\left\{  \xi_{i}\right\}  _{i\geq1}$ consists of
independent random variables (more precisely, it is enough to assume, that
this sequence this is a sequence of martingale differences) having zero
expectations. Then, of course, we have:
\[
E(x_{i+1}|x_{i})=f(x_{i}).
\]
In one word function $f$ is a regression of \textquotedblright the
next\textquotedblright\ on \textquotedblright the previous\textquotedblright%
\ observation of the process. Sequence of observations $\left\{
x_{i}\right\}  $ of this process contains information about its way of
generation. Wherein we do not have to parametrize function $f$ and seek
\textquotedblright the true values of parameters\textquotedblright\ as we were
doing in the previous section. The only constraint is independence (and
integrability of course) of the sequence of disturbances $\left\{  \xi
_{i}\right\}  $. It is worth to notice, that distributions of these variables
do not have to be identical! The whole procedure is, however sensitive on the
assumption of independence (more precisely on "being a martingale
difference"). That is if this sequence consists of dependent random variables,
then a function $f$ cannot be obtained by the density estimation method. The
parametric method described above should be used instead. The examples below
show that is is so indeed.

\begin{example}
\label{poprzedni1}Let us return to example \ref{poprzedni} and we will assume,
that process $\left\{  y_{i}\right\}  _{i\geq1}$ generated is with the help
equation (\ref{przyk_rek}) with sequence of disturbances $\left\{  \xi
_{i}\right\}  _{i\geq1}$ consisting of independent random variables having
Normal distribution $\xi_{i}\sim N(0,1+\sin^{2}i).$%
\[%
%TCIMACRO{\FRAME{itbpF}{2.9456in}{2.1127in}{0in}{}{}{ide_{n}iepar1.eps}%
%width 2.9456in;  height 2.1127in;  depth 0in;  original-width 5.7917in;
%original-height 4.1476in;  cropleft "0";  croptop "1";  cropright "1";
%cropbottom "0";  filename 'ide_niepar1.eps';file-properties "XNPEU";}}}%
%BeginExpansion
{\includegraphics[
height=2.1127in,
width=2.9456in
]%
{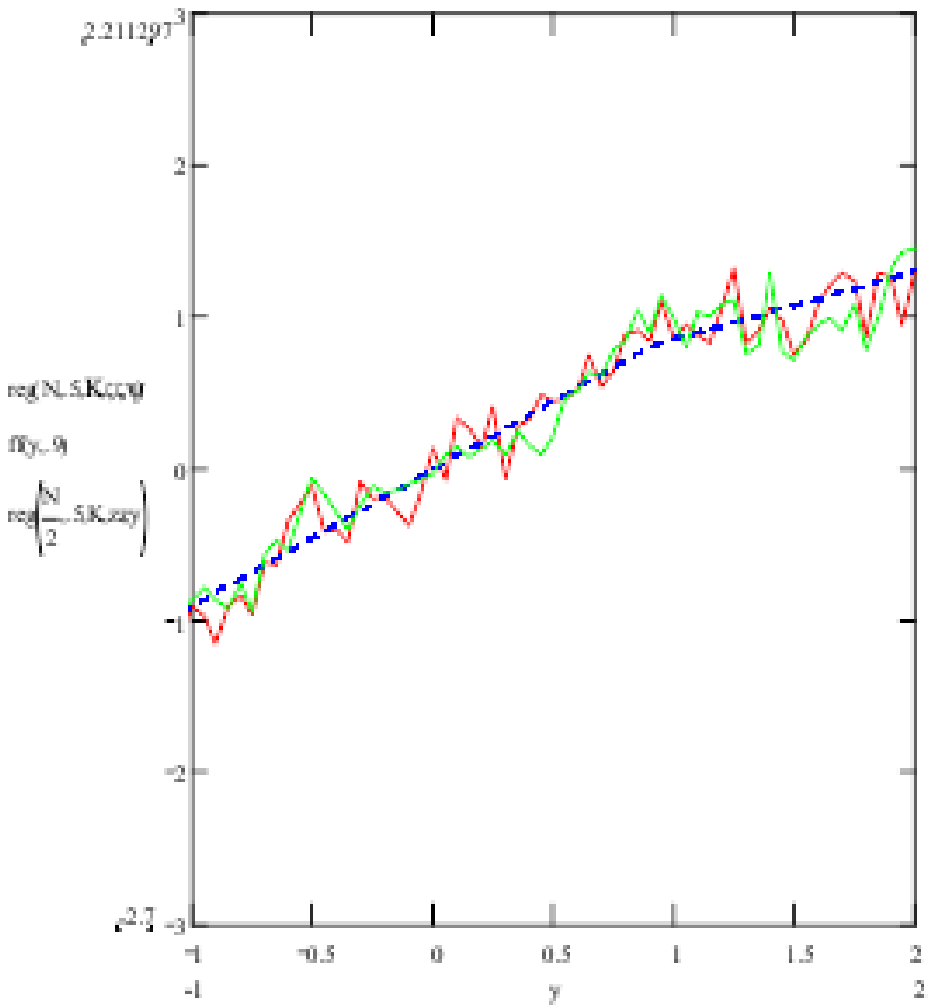}%
}%
%EndExpansion
\]

The nonparametric estimator obtained after $N\allowbreak=\allowbreak6000$
iterations are presented in red. The same estimator obtained after
$N\allowbreak=\allowbreak3000$ iterations are presented in green, while the
estimated function was plotted in blue. We chose density Cauchy distributions
as the kernel, and coefficient $\alpha\allowbreak=\allowbreak.5.$
\end{example}

\begin{example}
In the second example regression function is substantially more complex. It
would require many parameters in order to parametrize . That is if one wanted
to use stochastic approximation one should use its multidimensional version
(converging slower of course). Namely, as a regression function, we took
function $h(x)\allowbreak=\allowbreak\left\{
\begin{array}
[c]{ccc}%
.8x & ,when & x<-2\\
-.4+-.8(x+2) & ,when & -2\leq x<0\\
1 & ,when & 0\leq x<.5\\
1-.9x & ,when & x>.5
\end{array}
\right.  $
%TCIMACRO{\FRAME{dtbpF}{2.4734in}{1.9285in}{0pt}{}{}{ide_niepar2.eps}%
%width 2.4734in;  height 1.9285in;  depth 0pt;  original-width 5.1102in;
%original-height 3.9799in;  cropleft "0";  croptop "1";  cropright "1";
%cropbottom "0";  filename 'ide_niepar2.eps';file-properties "XNPEU";}}}%
%BeginExpansion
\begin{center}
\includegraphics[
height=1.9285in,
width=2.4734in
]%
{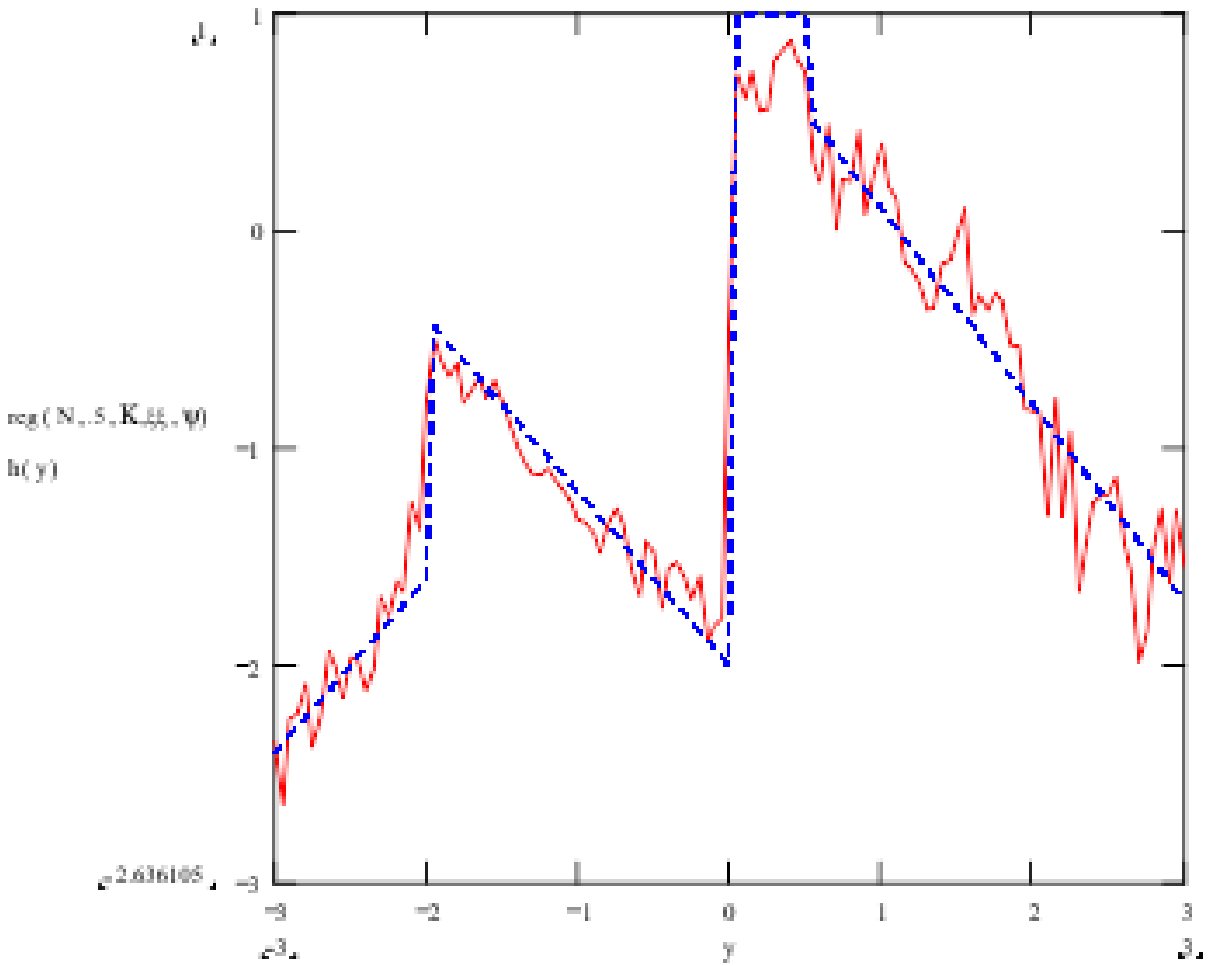}%
\end{center}
%EndExpansion

After $N\allowbreak=\allowbreak6000$ iterations we get the following estimator
of the function $h$ (plotted in blue). The process was disturbed by the noise
as in the previous example. The value of the parameter $\alpha$ and the kernel
were also identical as before.
\end{example}

\begin{example}
In the above-mentioned example, one can notice the superiority of the kernel
method over parametric methods. These are not, as it turns out very sensitive
on the assumptions of independence of disturbances appearing in the processes
equation. In the next example, we will consider the process that was analyzed
in example \ref{poprzedni1} with the proviso that we will assume this time,
that disturbing noises are slightly correlated.
\end{example}

\begin{example}
Namely, we will assume, that the noise $\left\{  \xi_{n}\right\}  $ is
generated by the moving average process of order 3, i.e. $\xi_{n}=\zeta
_{n}+.3\zeta_{n-1}-.2\xi_{n-2}$, where the sequence $\left\{  \zeta
_{n}\right\}  $ is a sequence of independent random variables having Gaussian
distribution $N(0,1)$. The one obtained after $N=6000$ iterations is the
following result:%
%TCIMACRO{\FRAME{dtbpFU}{2.162in}{1.6829in}{0pt}{\Qcb{}}{}{flg2b409.eps}%
%width 2.162in;  height 1.6829in;  depth 0pt;  original-width 5.444in;
%original-height 4.2307in;  cropleft "0";  croptop "1";  cropright "1";
%cropbottom "0";  filename 'FLG2B409.eps';file-properties "XNPEU";}}}%
%BeginExpansion
\begin{center}
\includegraphics[
height=1.6829in,
width=2.162in
]%
{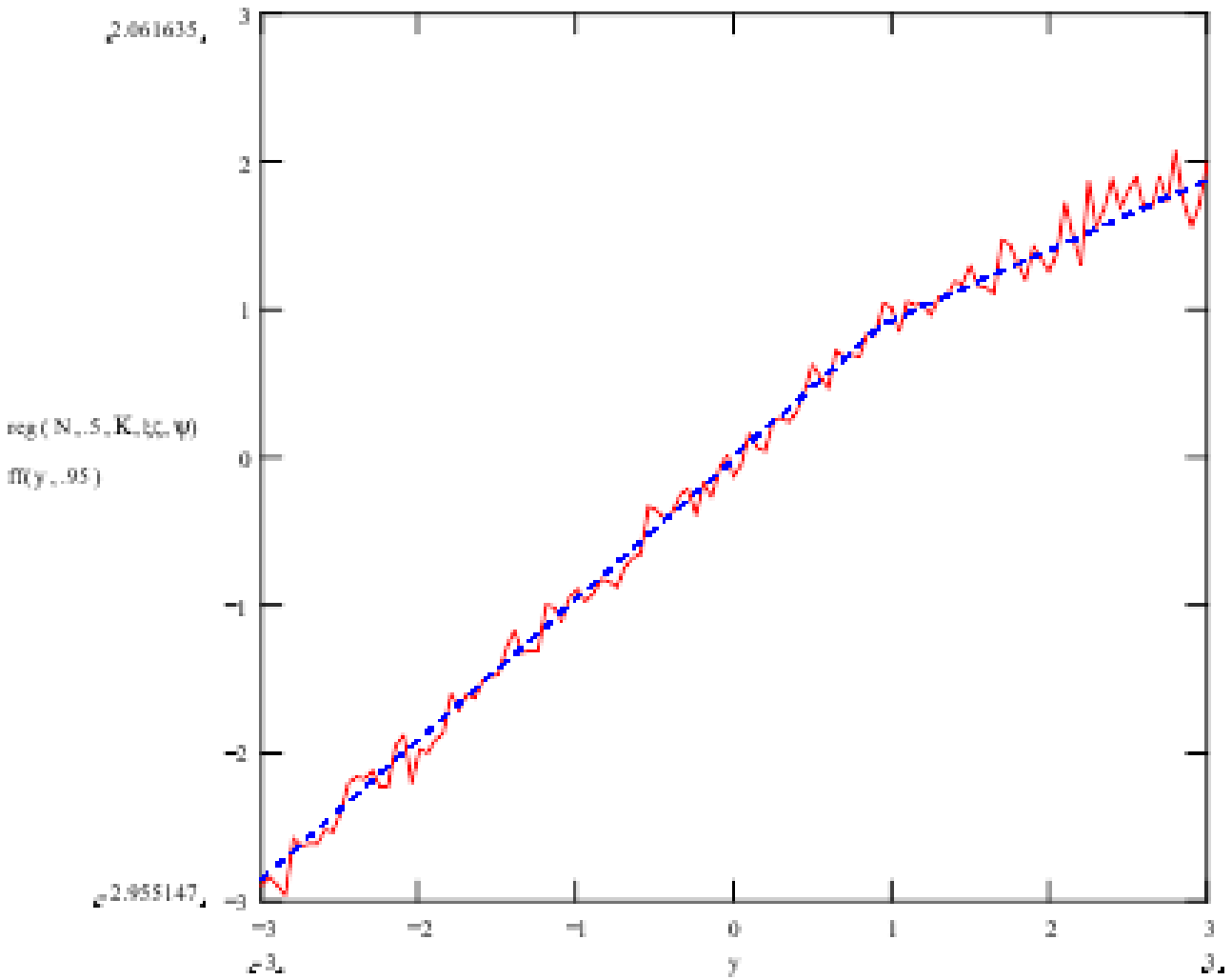}%
\end{center}
%EndExpansion
For the sake for clarity the estimated function was again plotted in blue.
\end{example}

\begin{remark}
Examples considered in the previous and this section clearly show that
parametric methods are substantially quicker. Practically after 200 -300
iterations we got already reasonable approximations of estimated parameters
Applying of nonparametric methods of estimation requires 1000 or more
iterations, to make the estimator visibly approximating estimated function. It
is not very surprising. One should expect this. The nonparametric estimation
has to \textquotedblright examine the shape of the estimated
function\textquotedblright\ and \textquotedblright to find approximated values
of parameters\textquotedblright. In the case of parametric estimation, the
first of these problems are already solved.
\end{remark}

\appendix

\chapter{Calculus of probability}

\section{Probability continuity%
\index{Continuity!Probability}
\label{cptwo}}

Let $(\Omega,\mathcal{F},P)$ be a probability space. Let us consider sequence
of events $\left\{  A_{i}\right\}  _{i\geq1}\subset\mathcal{F}$. We have the
following statement:

\begin{proposition}
i) If the sequence of events $\left\{  A_{i}\right\}  _{i\geq1}$ is
non-decreasing i.e. $\forall i\geq1$: $A_{i}\subseteq A_{i+1}$, then
$P(\bigcup_{i=1}^{\infty}A_{i})\allowbreak=\allowbreak\underset{i\rightarrow
\infty}{\lim}P(A_{i}),$\newline ii) If the sequence of events $\left\{
A_{i}\right\}  _{i\geq1}$ is non-increasing i.e. $\forall i\geq1$:
$A_{i}\supseteq A_{i+1}$, then $P(\bigcap_{i=1}^{\infty}A_{i})\allowbreak
=\allowbreak\underset{i\rightarrow\infty}{\lim}P(A_{i}).$
\end{proposition}

\begin{proof}
It is easy to notice, that assertion $i)$ and $ii)$ are equivalent due de
Morgan's laws. Hence, we will prove assertion $i)$. Let us denote
$C_{1}\allowbreak=\allowbreak A_{1}$, $C_{2}\allowbreak=\allowbreak
A_{2}-A_{1}$, $\ldots.,C_{n+1}\allowbreak=\allowbreak A_{n+1}-A_{n}$, $\ldots$
Events $\left\{  C_{i}\right\}  _{i\geq1}$ are disjoint, and Moreover, we have:%

\[
\bigcup_{i=1}^{\infty}A_{i}=\bigcup_{i=1}^{\infty}C_{i}.
\]
Hence, by the countable additivity of probability, we get:
\[
P(\bigcup_{i=1}^{\infty}A_{i})=\sum_{i=1}^{\infty}P(C_{i}).
\]
Moreover, let us notice, that event $C_{i+1}$ and $A_{i}$ are also disjoint
and $A_{i+1}=C_{i+1}\cup A_{i}$. Hence, $P(C_{i+1})\allowbreak=\allowbreak
P(A_{i+1})\allowbreak-\allowbreak P(A_{i})$, $i=1,2,\ldots$. Thus, $\sum
_{i=1}^{n}P(C_{i})\allowbreak=\allowbreak P(A_{n})$. In other words
$\sum_{i=1}^{\infty}P(C_{i})\allowbreak=\allowbreak\underset{i\rightarrow
\infty}{\lim}P(A_{i}).$
\end{proof}

\begin{remark}
One can easily show, that the property of probability continuity is equivalent
to the properties of countable additivity. One part of this equivalence was
already shown. It remained to show, that from probability continuity follows
countable additivity. This easy task we leave to the reader.
\end{remark}

\section{Chebyshev inequality
\index{Inequality!Chebishev}%
\label{czebyszew}}

Let $X$ be a random variable with expectation $EX$ and variance
$\operatorname*{var}(X)$. Chebyshev inequality states:
\begin{equation}
\forall\varepsilon>0:P\left(  \left\vert X-EX\right\vert \geq\varepsilon
\right)  \leq\frac{\operatorname*{var}(X)}{\varepsilon^{2}}.
\label{nier_Czeb1}%
\end{equation}
This inequality appears often in the following equivalent form:
\begin{equation}
\forall k>0:P\left(  \left\vert X-EX\right\vert <k\sqrt{\operatorname*{var}%
(X)}\right)  >1-\frac{1}{k^{2}}. \label{nier_Czeb2}%
\end{equation}

The proof is based on the so-called inequality Markov%
\index{Inequality!Markov}
\begin{equation}
\forall\epsilon>0:P\left(  Y\geq\epsilon\right)  \leq\frac{EY}{\epsilon}
\label{nier_Markowa}%
\end{equation}
that is true for nonnegative, integrable random variables. Now we set
$Y=E\left(  X-EX\right)  ^{2}$, $\epsilon=\varepsilon^{2}$ in%
\index{Inequality!Markov}%
. Markov's inequality one obtains by taking the expectation of both sides of
the inequality:
\[
I(Y\geq\epsilon)\leq\frac{Y}{\epsilon},
\]
true for all values $Y\geq0$ (make a plot!).

\section{Borel-Cantelli Lemma\label{Borel-Cantelli}}

Let $\{A_{i}\}_{i\geq1}$ will be a family of events. Let us denote
\[
\underset{i\rightarrow\infty}{\lim\inf}\,A_{i}=\bigcap_{i=1}^{\infty}%
\bigcup_{j=i}^{\infty}A_{j}.
\]

Sometimes the event $\underset{n\rightarrow\infty}{\lim\sup\,}A_{n}$ will be
denoted $\left\{  A_{n}:i.o.\right\}  $ coming from \textquotedblright
infinitely often\textquotedblright. Complementary event to $\left\{
A_{n}:i.o.\right\}  $ is an event $\bigcup_{i=1}^{\infty}\bigcap_{j=i}%
^{\infty}A_{j}^{c}$. Events of such form are called lower union of events
$\left\{  A_{i}^{c}\right\}  $. Sometimes it is denoted as
$\underset{n\rightarrow\infty}{\lim\inf}\,A_{n}^{c}$, or $\left\{  A_{n}%
^{c}:f.o.\right\}  $ coming from the words \textquotedblright finitely
often\textquotedblright.

\begin{lemma}
[Borel-Cantelli]%
\index{Lemma!Borel-Cantelli}%
\emph{i) }If $\sum_{i=1}^{\infty}P(A_{i})<\infty$ then, \newline%
$P(\underset{i\rightarrow\infty}{\lim\inf}\,A_{i})=0.$

\emph{ii) }If events $\{A_{i}\}_{i\geq1}$ are independent and $\sum
_{i=1}^{\infty}P(A_{i})=\infty$, \newline then $P(\underset{i\rightarrow
\infty}{\lim\inf}\,A_{i})\allowbreak=\allowbreak1.$

$\emph{iii)}$ If $P(\underset{i\,\longrightarrow\infty}{\lim\inf}%
\,A_{i})\allowbreak=\allowbreak1$ and event $A_{i}$ are independent, \newline
then $\sum_{i=1}^{\infty}P(A_{i})<\infty.$
\end{lemma}

\begin{proof}
\emph{i). }Let us denote $C_{i}\allowbreak=\allowbreak\bigcup_{j=i}^{\infty
}A_{j}$. We have $C_{i+1}\subseteq C_{i}$. Hence, \newline$P(\lim\sup
A_{i})\allowbreak=\allowbreak P\left(  \bigcap_{i=1}^{\infty}C_{i}\right)
\allowbreak=\allowbreak\underset{i\rightarrow\infty}{\lim}P(C_{i})$. Moreover,
$P(C_{i})\leq\sum_{j=i}^{\infty}P(A_{j})\rightarrow0$, as $i\rightarrow\infty
$, since $\sum_{i=1}^{\infty}P(A_{i})<\infty$. \emph{ii) }We have
$\overline{\underset{i\rightarrow\infty}{\lim\inf}\,A_{i}}=\bigcup
_{i=1}^{\infty}\bigcap_{j=i}^{\infty}A_{j}^{c}$. Let us denote $D_{i}%
\allowbreak=\allowbreak\bigcap_{j=i}^{\infty}A_{j}^{c}$. From the property of
probability continuity we have: $P($ $\overline{\underset{i\rightarrow
\infty}{\lim\sup}A_{i}})=\allowbreak\underset{i\rightarrow\infty\,}{\lim
}P(D_{i})$, since $D_{i}\subseteq D_{i+1}$. Moreover, since the events
$\left\{  A_{i}\right\}  $ are independent, we have $P(D_{i})=\allowbreak
\prod_{j=i}^{\infty}P(A_{j}^{c})\allowbreak=\prod_{j=i}^{\infty}%
(1-P(A_{j}))\allowbreak\leq\exp(-\sum_{j=i}^{\infty}P(A_{j}))\allowbreak=0$.
\emph{iii) }Basing on considerations from point \emph{ii) }%
$P(\underset{n\,\longrightarrow\infty}{\lim\inf}\,A_{i})\allowbreak
=\allowbreak1$ implies, that $\underset{i\rightarrow\infty}{\lim}%
P(D_{i})\allowbreak=\allowbreak1$. But then we have: $P(D_{i})=\prod
_{j=i}^{\infty}(1-P(A_{j}))$. If $\sum_{i=1}^{\infty}P(A_{i})\allowbreak
=\allowbreak\infty$, then as it follows from point $\emph{ii)}$ we would have
$\underset{i\rightarrow\infty}{\lim}P(D_{i})\allowbreak=\allowbreak0$, hence
we must have $\sum_{i\geq1}P(A_{i})<\infty.$
\end{proof}

\begin{remark}
Let us notice that the event $\left\{  A_{i}:f.o.\right\}  $ is equivalent to
the event $\left\{  \sum_{i\geq1}I(A_{i})<\infty\right\}  $. Similarly the
event $\left\{  A_{i}:i.o.\right\}  $ is equivalent to the event
\newline$\left\{  \sum_{i\geq1}I(A_{i})=\infty\right\}  .$
\end{remark}

\section{Types of convergence of sequences of the random
variables\label{rzbiez}}

\label{zbieznosc}Let $(\Omega,\mathcal{F},P)$ be a probability space. Let the
sequence $\{X_{n}\}_{n\geq0}$ of the random variables be defined on it. We say
that this sequence :

\begin{enumerate}
\item -Converges\emph{\ }with probability $1$%
\index{Convergence!with probability 1}
to a random variable $X$, when\newline$P\{{\omega:}X_{n}{(\omega)\allowbreak
}\underset{n\rightarrow\infty}{\longrightarrow}X(\omega)\}\allowbreak
\allowbreak=\allowbreak1$. (We write, then $\allowbreak X_{n}%
\underset{n\rightarrow\infty}{\longrightarrow}X$ with probability $1$ or
a.(lmost) s.(\textit{urely} ).

\item -Converges in probability%
\index{Convergence! in probability}
to a random variable $X$, when $\forall\epsilon>0:P\{\omega:|X_{n}%
(\omega)-X(\omega)|>\epsilon\}\longrightarrow0$ as $n\rightarrow\infty$. (We
write, then $X_{n}\underset{n\rightarrow\infty}{\longrightarrow}X$ in
probability or mod P).

\item -Converges in $r$ -th mean
\index{Convergence!in r-th mean}
to $X$ (also with $r$-th mean, or simply in $L_{r}$) ($r>0$), if
$E|X_{n}-X|^{r}\rightarrow0$ for $n\rightarrow\infty$. (We write, then
$X_{n}\overset{(r)}{\rightarrow}X$ or $X_{n}\overset{L_{r}}{\rightarrow}X$, as
$n\rightarrow\infty$). \newline\textbf{Remark}: In of the case $r\allowbreak
=\allowbreak2$ we talk about mean-squares convergence!

\item -Converges weakly%
\index{Convergence!weak}%
\emph{\ }to\textit{\ }$X$ ( according to cumulative distribution function or
in distribution), when the sequence $F_{n}$ (of cdf's of $X_{n}$ ) converges
to $F_{X}$ (cumulative distribution function of $X$) at every continuity point
of the cumulative distribution function $F$. (We write, then $X_{n}%
\overset{\ast}{\rightarrow}X$, where in place of (*) we can put (\textit{w})
or (\textit{D}).
\end{enumerate}

\begin{remark}
Let us denote by $A_{k}$ the following event:
\[
\bigcap_{m=1}^{\infty}\bigcup_{n=m}^{\infty}\{{\omega:}\allowbreak{|}%
X_{n}{(\omega)-}X{(\omega)|>}\frac{{1}}{k}{\}.}%
\]
Using the definition of the limit we see, that convergence with probability
$1$ of the sequence $\{X_{n}\}$ to $X$ is equivalent to the fact that event
$\bigcup_{k=1}^{\infty}A_{k}$ has zero probability. Since we have
\[
\bigcup_{k=1}^{\infty}A_{k}\supset A_{k},k=1,2,\ldots
\]
then
\[
P(\bigcup_{k=1}^{\infty}A_{k})=0\allowbreak\Longrightarrow\allowbreak
P(A_{k})=0,\text{ }k=1,2,\ldots.
\]
Let us notice now, that denoting $B_{m}\allowbreak=\allowbreak\bigcup
_{n=m}^{\infty}\{\omega:\allowbreak{|}X_{n}{(\omega)-}X{(\omega)|>}\frac{{1}%
}{k}\}$ we have $B_{m+1}\subset B_{m}$, $m=1,2,\ldots$ and $A_{k}%
=\bigcap_{m=1}^{\infty}B_{m}$. Hence, once again from continuity of
probability we get: $\underset{m\rightarrow\infty}{\lim}P(B_{m})=0$. It
remains to note, keeping in mind definition of events $B_{m}$, that
\[
B_{m}\supset\allowbreak\{\omega:\left\vert X_{m}(\omega)\allowbreak-X\left(
\omega\right)  \right\vert \allowbreak>\frac{1}{k}\}.
\]
Hence summarizing, if the sequence $\left\{  X_{n}\right\}  _{n\geq1}$
converges with probability $1$ do $X$, then
\[
\forall k\in%
%TCIMACRO{\U{2115} }%
%BeginExpansion
\mathbb{N}
%EndExpansion
\mathbb{\;}P(\left\vert X_{m}-X\right\vert >\frac{1}{k})\underset{m\rightarrow
\infty}{\longrightarrow}0,
\]
that is a convergence with probability follows the convergence with
probability $1$.
\end{remark}

\begin{remark}
\label{p1}The following example shows, that from the convergence in
probability of a sequence one cannot deduce its convergence with probability
$1$. Let us consider the following probability space : $([0,1],\mathcal{B}%
([0,1]),|.|)$ and let us define on it, the following sequence of the random
variables:
\begin{align*}
X_{0}(\omega)  &  =1;\\
X_{1}(\omega)  &  =\left\{
\begin{array}
[c]{lll}%
1 & for & \omega\in\lbrack0,1/2)\\
0 & for & \omega\in\lbrack1/2,1]
\end{array}
\right.  ;\\
X_{2}(\omega)  &  =\left\{
\begin{array}
[c]{lll}%
0 & for & \omega\in\lbrack0,1/2)\\
1 & for & \omega\in\lbrack1/2,1]
\end{array}
\right.  ;\\
X_{3}(\omega)  &  =\left\{
\begin{array}
[c]{lll}%
1 & for & \omega\in\lbrack0,1/3)\\
0 & for & \omega\in\lbrack1/3,1]
\end{array}
\right.  ;\\
X_{4}(\omega)  &  =\left\{
\begin{array}
[c]{lll}%
1 & for & \omega\in\lbrack1/3,2/3)\\
0 & for & \omega\in\lbrack0,1/3)\cup\lbrack2/3,1]
\end{array}
\right.  ;\\
X_{5}(\omega)  &  =\left\{
\begin{array}
[c]{lll}%
1 & for & \omega\in\lbrack2/3,1]\\
0 & for & \omega\in\lbrack0,2/3)
\end{array}
\right.  ;\\
X_{6}(\omega)  &  =\left\{
\begin{array}
[c]{lll}%
1 & for & \omega\in\lbrack0,1/4)\\
0 & for & \omega\in\lbrack1/4,1]
\end{array}
\right.  ;\\
X_{7}(\omega)  &  =\left\{
\begin{array}
[c]{lll}%
1 & for & \omega\in\lbrack1/4,2/4)\\
0 & for & \omega\in\lbrack0,1/4)\cup\lbrack2/4,1]
\end{array}
\right.  ;\text{ and so on}%
\end{align*}
This sequence converges to zero in probability, since for $1>\varepsilon>0$ we
have: $P(\left\vert X_{n}\right\vert >\varepsilon)\allowbreak=P(X_{n}=1)$, and
$\left\{  P(X_{n}=1)\right\}  =\{1,\allowbreak\frac{1}{2},\allowbreak\frac
{1}{2},\allowbreak\frac{1}{3},\allowbreak\frac{1}{3},\allowbreak\frac{1}%
{3},\allowbreak\frac{1}{4}.\allowbreak\frac{1}{4},\allowbreak\frac{1}%
{4},\allowbreak\frac{1}{4},\allowbreak\frac{1}{5},\allowbreak\ldots\}$. This
sequence is however divergent at almost every point, since for fixed
$\omega\neq0,1$ the sequence $\left\{  X_{n}(\omega)\right\}  _{n=1}^{\infty}$
will contain infinitely many $1$. On can notice it from e.g. the figure below,
where, for the clarity, values of functions $X$ were multiplied by decreasing
coefficients.\newline%
%TCIMACRO{\FRAME{itbpF}{2.3151in}{1.5316in}{0in}{}{}{ob48hh07.eps}%
%width 2.3151in;  height 1.5316in;  depth 0in;  original-width 2.9853in;
%original-height 1.9649in;  cropleft "0";  croptop "1";  cropright "1";
%cropbottom "0";  filename 'OB48HH07.eps';file-properties "XNPEU";}}}%
%BeginExpansion
{\includegraphics[
height=1.5316in,
width=2.3151in
]%
{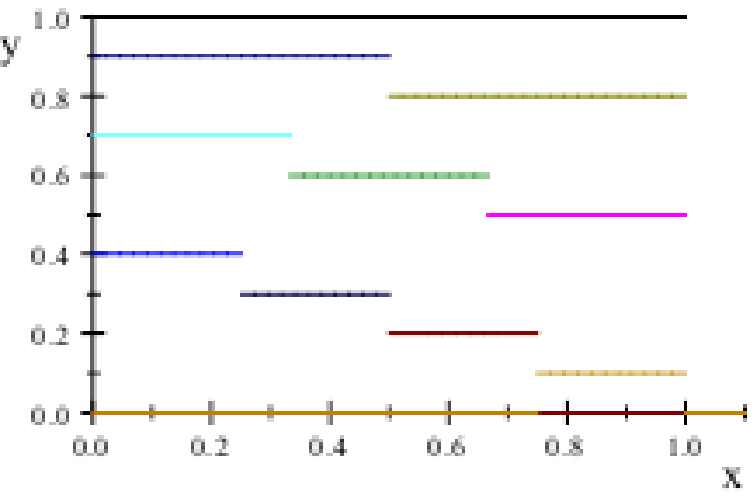}%
}%
%EndExpansion

\end{remark}

\begin{remark}
Analyzing the above mentioned example, one can notice, that from the sequence
$\left\{  X\right\}  _{n}$ (convergent in probability), one was able to select
a subsequence convergent almost surely. Such sequence is e.g. the sequence
$\left\{  X_{0},X_{1},X_{3},X_{6},\ldots\right\}  $. This observation can be
generalized and we have the following theorem.
\end{remark}

\begin{theorem}
[Riesz]%
\index{Theorem!Riesz}%
Every convergent in probability sequence of the random variables contains a
subsequence convergent almost surely. Conversely, if any sequence of the
random variables has the following property: each of its subsequences contains
a convergent almost surely subsequence, the sequence is convergent in probability!
\end{theorem}

Proof of this theorem one can find in e.g. \cite{lojasiewicz}.

Relationships between different types of convergence of sequences random
variables:
\[
\left.
\begin{array}
[c]{c}%
\begin{tabular}
[c]{|l|}\hline
$\text{convergence}$\\
$\text{almost surely}$\\\hline
\end{tabular}
\\%
\begin{tabular}
[c]{|l|}\hline
$\text{convergence }$\\
with$\text{ }r\text{-th mean}$\\\hline
\end{tabular}
\end{array}
\right.  \Rightarrow%
\begin{tabular}
[c]{|l|}\hline
$\text{convergence }$\\
in probability\\\hline
\end{tabular}
\ \Rightarrow%
\begin{tabular}
[c]{|l|}\hline
$\text{convergence }$\\
in$\text{ distribution}$\\\hline
\end{tabular}
\]

Moreover, if th sequence converges in $L_{r}$ and $r\geq s$, then also
converges in $L_{s}.$

Counterexamples:

\begin{enumerate}
\item Let $X$ have Cauchy distribution and $X_{n}=X/n$, $n\geq1$. Then of
course $X_{n}\rightarrow0$, $n\rightarrow\infty$ a.s. but $E\left\vert
X_{n}\right\vert ^{r}=\infty$ for $n,r\geq1.$

\item Sequence considered in remark \ref{p1} converges with any mean to zero,
but of course, does not converge with probability $1$.
\end{enumerate}

\begin{remark}
The fact, that convergence with $r-$th mean implies convergence in probability
follows directly from Chebyshev inequality:
\[
P(\left\vert X_{n}-X\right\vert >\epsilon)\leq E\left\vert X_{n}-X\right\vert
^{r}/\epsilon^{r}.
\]

\end{remark}

\begin{remark}
The fact, that convergence in probability implies weak convergence is given
without the proof. Convergence in probability to a constant imply convergence
in probability to a constant.
\end{remark}

To give other examples of the sequences that are convergent with probability
$1,$ we will use Borel-Cantelli Lemma (see Appendix \ref{Borel-Cantelli}).

Using this lemma, we will give an example of convergent and divergent with
probability $1$ sequence of the random variables.

\begin{example}
Let $\left\{  X_{n}\right\}  _{n\geq1}$ be the following sequence random
variables:
\[
X_{n}=\left\{
\begin{array}
[c]{ll}%
1 & \text{with probability }1/n^{2}\\
0 & \text{with probability }1-1/n^{2}%
\end{array}
\right.  .
\]
Let us notice that we do not specify on which probability space this sequence
is defined and whether or not its elements are independent. From the
Borel-Cantelli Lemma, it follows that since: $\sum_{n\geq1}P(X_{n}%
\neq0)<\infty$, hence with probability $1$ the event $\left\{  X_{n}%
\neq0\right\}  _{n\geq1}$ will happen only a finite number of times. In other
words, starting from some $N$ $\,\,$(may be random), for all $n\geq N$ we will
have, $X_{n}=0$, that is, a sequence $\left\{  X_{n}\right\}  _{n\geq1}$ will
converge \ with probability $1$ to zero.
\end{example}

\begin{example}
Let us consider now a sequence of independent random variables $\left\{
X_{n}\right\}  _{n\geq1}$ having the following distributions:
\[
X_{n}=\allowbreak\left\{
\begin{array}
[c]{ll}%
1 & withprobability1/n\\
0 & withprobability1-1/n
\end{array}
\right.  .
\]
Now we have $\sum_{n\geq1}P(X_{n}=1)=\infty$, and Moreover, events $\left\{
X_{n}=1\right\}  $ are independent, hence according to Borel-Cantelli Lemma
with probability $1$ these events will happen infinite number of times. In
other words, for almost every $\omega$ the sequence $\left\{  X_{n}%
(\omega)\right\}  $ will have an infinite number of $1^{\prime}s$, that is the
sequence cannot converge to zero.
\end{example}

\section{Conditional expectation\label{wwocz}}

\subsection{Definition and basic examples}

Let $(\Omega,\mathcal{F},P)$ be a probability space, $X$- random variable
defined on it and $\mathcal{A}\subset\mathcal{F}$ some $\sigma$-field of
subsets of $\Omega$. If $\sigma(X)\footnote{$\sigma(X)$ - a $\sigma-field$
generated by the random variable $X.$ i.e. $\sigma(X)\allowbreak
\allowbreak=\allowbreak\sigma\left(  X^{-1}(B):B\in\mathcal{B}_{1}\right)  ,$
$\mathcal{B}_{1}$ denotes here Borel $\sigma-$field of subsets on $\mathbb{R}%
$}\subset\mathcal{A}$, then we say that the random variable $X$ is
$\mathcal{A}$-measurable. Let us assume that $E\left\vert X\right\vert
<\infty$ .

\begin{remark}
\label{o_mierzalnych} If $\mathcal{A}\allowbreak=\allowbreak\sigma
(\mathbf{Y})$ for some random vector $\mathbf{Y}$, then random variable
$\mathbf{X}$ is $\mathcal{A}$-measurable if and only if, there exists Borel
function $\mathbf{g}$ such, that $\mathbf{X}=g(\mathbf{Y})$ a.s.
\end{remark}

\begin{definition}
Conditional expectation of a random variable $X$ with respect to $\sigma
$-field $\mathcal{A}$ we call such $\mathcal{A}$-measurable random variable
(denoted $E(X|\mathcal{A})$), that satisfies the following condition:
\begin{equation}
E(YE(X|\mathcal{A}))=E(YX) \label{definicja_warunkowej}%
\end{equation}
for every $\mathcal{A}$-measurable random variable $Y$ such, that
$E|XY|<\infty$.
\end{definition}

\begin{remark}
If $\sigma$- field $\mathcal{A}\allowbreak=\allowbreak\sigma(Y)$, for some
random vector $\mathbf{Y}$, then we write $E(X|\mathbf{Y})$ instead
$E(X|\sigma(\mathbf{Y}))$. (From the remark \ref{o_mierzalnych} it follows
that there exists then a Borel function $h$ such that $E(X|\mathbf{Y}%
)=h(\mathbf{Y})$ a.s.) .
\end{remark}

\begin{example}
Let $\mathcal{A}=\{\emptyset,\Omega\}$. Then every random variable measurable
with respect to $\mathcal{A}$ is almost surely equal to a constant.
$E(X|\mathcal{A)}$ is thus also equal to a constant. Which one? From the
condition (\ref{definicja_warunkowej}) it follows immediately, that this
constant (let us call it $e$ ) has to satisfy the condition :
\[
\forall y\in%
%TCIMACRO{\U{211d} }%
%BeginExpansion
\mathbb{R}
%EndExpansion
:Eye=EyX.
\]
Thus, we immediately deduce, that $E(X|\mathcal{A})=EX$ a.s.
\end{example}

\begin{example}
Let $(X,Y)$ have joint density $f(x,y)$. let us find $E(X|Y)$. We have here
$\mathcal{A}=$ $\sigma(Y$), hence every random variable $\mathcal{A}%
-$measurable is of the form $g(Y)$ for some Borel functions $g$. $E(X|Y$) is
also of this form e.g. for the function $h$. One has to find this function.
From the condition defining conditional expectation we have for every Borel
function $g$:
\[
\iint\nolimits_{%
%TCIMACRO{\U{211d} }%
%BeginExpansion
\mathbb{R}
%EndExpansion
^{2}}xg(y)f(x,y)dxdy=\int\nolimits_{%
%TCIMACRO{\U{211d} }%
%BeginExpansion
\mathbb{R}
%EndExpansion
}g(y)h(y)f_{Y}(y)dy.
\]
Hence
\[
h(y)=\int\nolimits_{%
%TCIMACRO{\U{211d} }%
%BeginExpansion
\mathbb{R}
%EndExpansion
}x\frac{f(x,y)}{f_{Y}(y)}dx.
\]
Interpretation: Let us consider the above mentioned formula and let us reshape
its right hand side slightly in the following way:
\begin{equation}
\int\nolimits_{%
%TCIMACRO{\U{211d} }%
%BeginExpansion
\mathbb{R}
%EndExpansion
}x\frac{f(x,y)}{f_{Y}(y)}dx=\frac{\int_{%
%TCIMACRO{\U{211d} }%
%BeginExpansion
\mathbb{R}
%EndExpansion
}xf(x,y)dxdy}{\int_{%
%TCIMACRO{\U{211d} }%
%BeginExpansion
\mathbb{R}
%EndExpansion
}f(x,y)dxdy}. \label{wyrazenie wrunkowe}%
\end{equation}
Next on the plane $(x,y)$ let us consider a horizontal strip of width $dy$
containing line $y=y0$:%

%TCIMACRO{\FRAME{itbpF}{3.141in}{1.8075in}{0in}{}{}{w-kowa_w_ocz.eps}%
%width 3.141in;  height 1.8075in;  depth 0in;  original-width 25.6642in;
%original-height 14.7174in;  cropleft "0";  croptop "1";  cropright "1";
%cropbottom "0";  filename 'w-kowa_w_ocz.eps';file-properties "XNPEU";}}}%
%BeginExpansion
{\includegraphics[
height=1.8075in,
width=3.141in
]%
{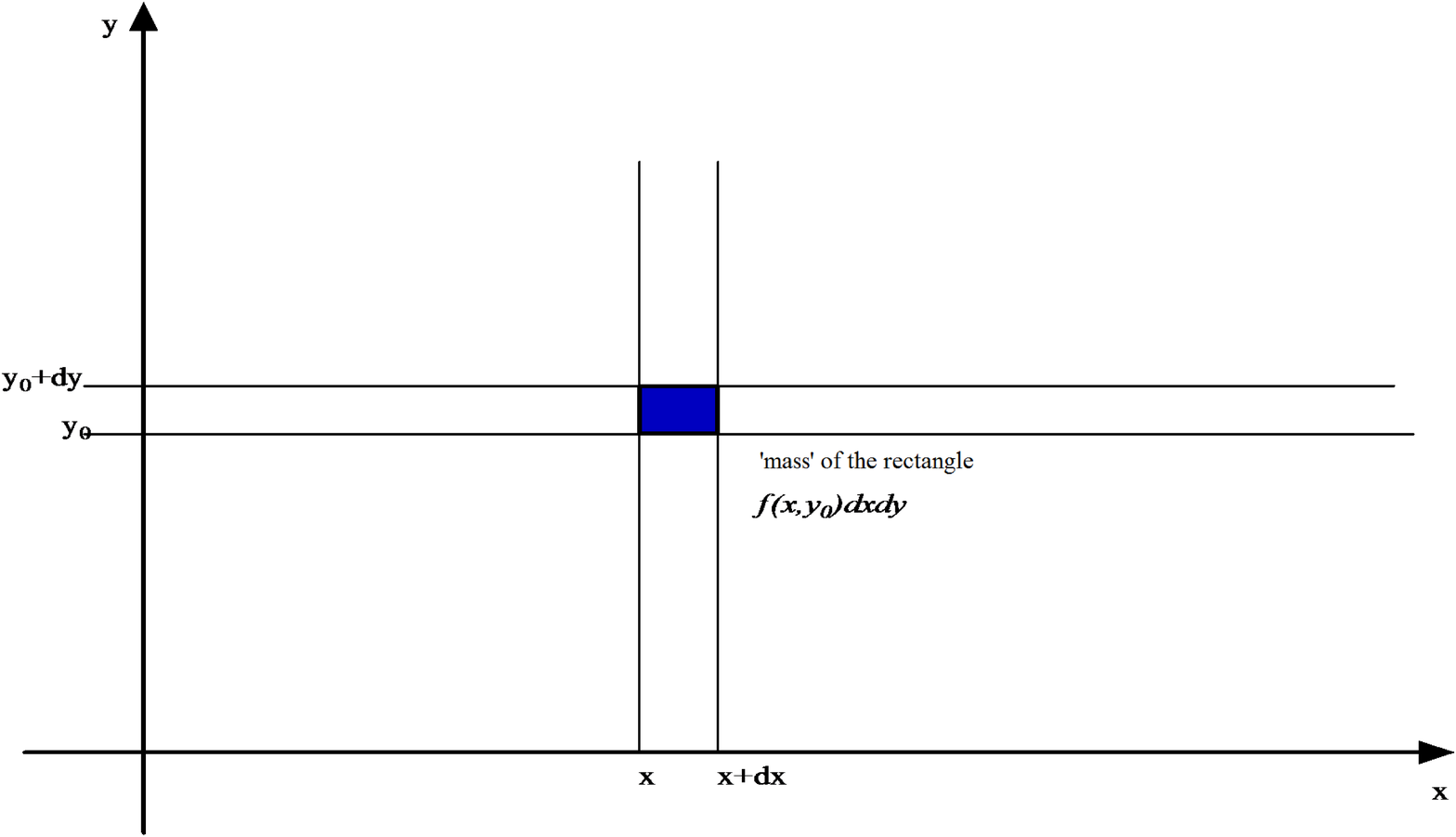}%
}%
%EndExpansion
. The denominator of the expression (\ref{wyrazenie wrunkowe}) is equal to
\textquotedblright sum of moments of mass f(x,y0)dxdy\ with respect to the
axis oy (multiplied by x)\textquotedblright\ that is equal to the moment
(mechanical) of the strip with respect to axis oy. The denominator is equal to
the 'mass' of this strip (sum of masses f(x,y)dxdy along the line ox). The
value of the conditional expectation at point $y0$ that is $h(y0)$ is thus
equal to \textquotedblright center of the mass of the strip with width $dy$
containing straight line $y=y0"$.
\end{example}

Conditional expectation is just the curve joining 'centers of the masses' of
parallel horizontal strips and the conditional variance is a curve joining
moments of inertia of such strips!!!

\begin{remark}
The values of the functions $h(.)$ at point $y$ are traditionally denoted
$E(X|Y=y)$. Analogously, we show, that:
\[
E(w(X)|Y=y)=\int\nolimits_{%
%TCIMACRO{\U{211d} }%
%BeginExpansion
\mathbb{R}
%EndExpansion
}w(x)\frac{f(x,y)}{f_{Y}(y)}dx.
\]

\end{remark}

\subsection{Properties}

\begin{enumerate}
\item \label{istnienie_war_ocz}If $E|X|<\infty$ then for any $\sigma$-field
$\mathcal{A}$ $E(X|\mathcal{A})$ exists and is defined uniquely in the
following sense: if $Z_{1}$ and $Z_{2}$ are two $\mathcal{A}$- measurable
random variables satisfying condition (\ref{definicja_warunkowej}), then
$P(Z_{1}=Z_{2})=1$ or other in words $Z_{1}=Z_{2}$ a.s. .

\begin{remark}
Because of expressed in the above mentioned property properties of the
ambiguity of the $\mathcal{A}$- measurable random variables satisfying
condition (\ref{definicja_warunkowej}) they are called versions of the
conditional expectation $E(X|\mathcal{A}).$
\end{remark}

\item For every $\alpha,\beta${}$\in%
%TCIMACRO{\U{211d} }%
%BeginExpansion
\mathbb{R}
%EndExpansion
$ and random variables $X$, $Y$ such that $E|X|<\infty$, $E|Y|<\infty$ the
following equality{} is satisfied:
\[
E(\alpha X+\beta Y|\mathcal{A})=\alpha E(X|\mathcal{A})+\beta{}E(Y|\mathcal{A}%
).
\]

\begin{proof}
For any $\mathcal{A-}$measurable random variable $T$ we have:
\[
ET(\alpha X+\beta Y)=ETE(\alpha X+\beta Y|\mathcal{A))}.
\]
Let us denote $Z_{1}=E(\alpha X+\beta Y|\mathcal{A)}$. On the other hand by
linearity of expectation we have:
\[
ET(\alpha X+\beta Y)=\alpha EXT+\beta EYT.
\]
Using(\ref{definicja_warunkowej}) we get:
\begin{align*}
\alpha EXT+\beta EYT  &  =\alpha ETE(X|\mathcal{A})+\beta ETE(Y|\mathcal{A}%
)=\\
&  =ET(\alpha E(X|\mathcal{A})+\beta E(Y|\mathcal{A})).
\end{align*}
Random variable $Z_{2}\allowbreak=\allowbreak\alpha E(X|\mathcal{A}%
)\allowbreak+\allowbreak\beta E(Y|\mathcal{A})$ is of course $\mathcal{A}%
$-measurable (as the sum of the random variables is a random variable). Hence,
we have%

\[
ET(\alpha X+\beta Y)=ETZ_{1}%
\]
and
\[
ET(\alpha X+\beta Y)=ETZ_{2}%
\]
for any random variable $\mathcal{A}$ measurable $T$. Hence, by
\ref{istnienie_war_ocz} it follows that $Z_{1}\allowbreak=\allowbreak Z_{2}$
a.s. .
\end{proof}

\item If $X\geq0$ a.s. and $E|X|<\infty$, then $E(X|\mathcal{A})\geq0$ a.s.
\newline for every $\sigma$-field $\mathcal{A}$ $\subset\mathcal{F}$.

\item $E(E(X|\mathcal{A}))=EX$ a.s.

\item[5.] If $E|X|$ $<$ $\infty$ and $\mathcal{A}$ and $\mathcal{B}$ are two
$\sigma-$fields such that $\mathcal{A}\subset\mathcal{B}\subset\mathcal{F}$,
then
\[
E(E(X|\mathcal{B})|\mathcal{A)}=E(X|\mathcal{A)}.
\]
\newline In particular, if $\mathbf{Y}_{1}$ and $\mathbf{Y}_{2}$ are some
random vectors, then $E(E(X|(\mathbf{Y}_{1},\mathbf{Y}_{2}))|\mathbf{Y}%
_{1})=E(X|\mathbf{Y}_{1})$ a.s.

\item[6.] If $E|X|$ $<$ $\infty$ and $X$ is $\mathcal{A}-$measurable random
variable, then
\[
E(X|\mathcal{A)}=X\,\,\,a.s.
\]

\item[7.] If $E|X|$ $<$ $\infty$ and $Y$ is $\mathcal{A}-$measurable random
variable, then
\[
E(YX|\mathcal{A)}=YE(X|\mathcal{A)\,\,}a.s.
\]

\item[8.] \label{niezalezne}If $E|X|$ $<$ $\infty$ and $X$ is a random
variable independent on (i.e. event $\{\omega:$ $X(\omega)$ $<$ $x\}$ and $A$
are independent for every $x\in%
%TCIMACRO{\U{211d} }%
%BeginExpansion
\mathbb{R}
%EndExpansion
$ and $A\in A)$ and $E|X|$ $<$ $\infty$, then
\[
E(X|\mathcal{A})=EX\,\text{a.s.}%
\]

\item[9.] If $E|X|^{2}$ $<$ $\infty$, then
\begin{equation}
\underset{Y}{\min}E(X-Y)^{2}=E(X-E(X|\mathcal{A}))^{2}, \label{minimalizacja}%
\end{equation}
where the minimum is taken with respect to all $\mathcal{A}$ -measurable
random variables.
\end{enumerate}

\begin{remark}
This is the most important (from the point of view of the application)
property of the conditional expectation. Its importance will be better seen if
we, considered a particular case, namely if we assume, that
$\mathcal{A\allowbreak=\allowbreak}\sigma(\mathbf{Z})$ for some random vector.
Then, as we remember all $\mathcal{A}$- measurable random variables are of the
form $g(\mathbf{Z)}$. The property \ref{minimalizacja} will now take the
following form:
\begin{equation}
\underset{g-Borel~function}{\min}E(X-g(\mathbf{Z))}^{2}=E(X-E(X|\mathbf{Z}%
))^{2}. \label{minim_Borel}%
\end{equation}
In other words, conditional expectation is the best (in the mean-squares
sense) approximation of the random variable $X$ with the help of Borel
functions of the vector $\mathbf{Z.}$
\end{remark}

\begin{enumerate}
\item[10.] If $E|X|$ $<$ $\infty$, then for every $\mathcal{A}-$ measurable
random variable $Y$ and such that, $E|XY|$ $<$ $\infty$, random variables $Y $
and $X-E(X|\mathcal{A)}$ are uncorrelated.\newline\textbf{Proof.} We have
$EY(X-E(X|\mathcal{A}))=E\left(  YE\left(  X-E(X|\mathcal{A})|\mathcal{A}%
\right)  \right)  =0.$
\end{enumerate}

\begin{remark}
Conditional expectation (e.g. $E(X|Y=y)$) is a.s. equal to a Borel functions
minimizing expression: $\min E(X-h(Y))^{2}=E(X-E(X|Y))^{2}$. Hereof its name
(sometimes met) nonlinear regression!.
\end{remark}

\section{Uniform integrability\label{UnInt}}

\begin{definition}%
\index{Uniform integrability}%
Class $\mathcal{C}$ $\;$of the random variables is called uniformly\emph{\ }%
integrable, if for every $\varepsilon>0$ exists a constant $K$ such, that
$\forall X\in\mathcal{C}:E(\left\vert X\right\vert I(\left\vert X\right\vert
>K))<\varepsilon.$
\end{definition}

We have two important examples of such classes of the random variables.

\begin{example}
\label{alpha}Let the family of the random variables $\mathcal{C}$ be
defined$\mathcal{\allowbreak}$ by conditions: $\exists p>1;\allowbreak
A\allowbreak\in(0,\infty)\allowbreak\forall X\allowbreak\in\allowbreak
\mathcal{C}:\allowbreak E|X|^{p}\allowbreak<\allowbreak A$, then
$\mathcal{C\allowbreak}$ is uniformly integrable. It follows then from by the
following argument: for $v\geq K>0$ we have $v\leq K^{1-p}v^{p}$. Hence, for
$\forall X\allowbreak\in\allowbreak\mathcal{C}$ we have
\[
E\left(  |X|I\left(  |X|>K\right)  \right)  \leq K^{1-p}E\left(
|X|^{p}I\left(  |X|>K\right)  \right)  \leq K^{1-p}A.
\]

\end{example}

\begin{example}
\label{zdominowane}Let the family of the random variables
$\mathcal{C\allowbreak}$ be defined$\mathcal{\allowbreak}$ by the conditions:
$\exists Y\allowbreak\geq\allowbreak0\allowbreak,E\left(  Y\right)  <\infty
:$\allowbreak$|X|\leq Y$, then $\mathcal{C}$ is uniformly integrable.

Because we have for $K>0$ and $X\in\mathcal{C}$
\[
E\left(  |X|I\left(  |X|>K\right)  \right)  \leq E(YI\left(  Y>K\right)  .
\]
We have also $EY\allowbreak=\allowbreak E(YI\left(  Y>K\right)  \allowbreak
+E(YI\left(  Y\leq K\right)  )$. But $YI\left(  Y\leq K\right)  $ is not
decreasing as a function of $K$ and $\underset{K->\infty}{\lim}$ $YI\left(
Y\leq K\right)  \allowbreak=\allowbreak Y$, Hence, by the so-called Lesbesgue
Theorem, we see, that $E(YI\left(  Y>K\right)  ->$ $0$ when $K->\infty$. In
other words, one can find $K$ so that $E\left(  |X|I\left(  |X|>K\right)
\right)  $ was sufficiently small.
\end{example}

Another important example of the uniformly integrable family of the random
variables is supplied by the following statement:

\begin{proposition}
Let $X$ will be integrable random variable defined on $(\Omega,\mathcal{F}%
,P)$. Then the class of the random variables
\[
\left\{  Y:\exists\mathcal{G\subseteq F},\;\mathcal{G}\text{ is }%
\sigma-\text{field, such that }Y\text{ is a version of\thinspace
}E(X|\mathcal{G)}\right\}
\]
is jest uniformly integrable.
\end{proposition}

\begin{proof}
Let us take $\varepsilon>0$ and select $\delta>0$ such that $\forall
F\in\mathcal{F}:$%
\[
P(F)<\delta\text{ implies, that }E(\left\vert X\right\vert I(F))<\varepsilon.
\]
The fact that it can be done follows basic properties of Lebesgue integrals
(see, e.g. \cite{lojasiewicz}). Let us select $K$ so that $E\left\vert
X\right\vert /K<\delta$. Let $Y$ be version of $E(X|\mathcal{G)}$ for some
$\sigma-$field $\mathcal{G}$. By the Jensen's inequality we have
\begin{equation}
\left\vert Y\right\vert \leq E(\left\vert X\right\vert |\mathcal{G}),\;a.s.\,.
\label{UI1}%
\end{equation}
Hence $E\left\vert Y\right\vert \leq E\left\vert X\right\vert $. Moreover, we
have:
\[
KP(\left\vert Y\right\vert >K)\leq E\left\vert Y\right\vert \leq E\left\vert
X\right\vert ,
\]

that
\[
P(\left\vert Y\right\vert >K)\leq\delta.
\]
We have $\left\{  \left\vert Y\right\vert >K\right\}  \in\mathcal{G}$, hence
by the property (\ref{UI1}) we have:
\[
\left\vert Y\right\vert I(\left\vert Y\right\vert >K)\leq E\left(  \left\vert
X\right\vert I\left(  \left\vert Y\right\vert >K\right)  |\mathcal{G}\right)
,
\]
thus:
\[
E(\left\vert Y\right\vert I(\left\vert Y\right\vert >K))\leq E(\left\vert
X\right\vert I(\left\vert Y\right\vert >K))<\varepsilon.
\]

\end{proof}

We have two very important theorems connected with the notion of uniform
integrability. The first one is a version known Lebesgue Theorem on bounded
passage to the limit under the integrals.

\begin{theorem}
\label{UI2}Let $\{X_{n}\}_{n\geq1}$ be a sequence random variables and $X$ a
random variable such that $X_{n}\rightarrow X$ in probability, as
$n\rightarrow\infty$. Moreover, let us suppose, that
\[
\forall n\geq1:\left\vert X_{n}\right\vert <K
\]
for some positive constant $K$. Then
\[
E\left\vert X_{n}-X\right\vert \rightarrow0,\,n\rightarrow\infty.
\]

\end{theorem}

\begin{proof}
We have for $k\in%
%TCIMACRO{\U{2115} }%
%BeginExpansion
\mathbb{N}
%EndExpansion
$,
\[
\forall n\geq1:P(\left\vert X\right\vert >K+1/k)\leq P(\left\vert
X_{n}-X\right\vert >1/k),
\]
hence $P(\left\vert X\right\vert >K+1/k)=0.$Thus,
\[
P(\left\vert X\right\vert >K)=P\left(  \bigcup_{k=1}^{\infty}\left\{
\left\vert X\right\vert >K+1/k\right\}  \right)  =0.
\]
Let us select $\varepsilon>0$ and $n_{0}$ so that:
\[
P(\left\vert X_{n}-X\right\vert >\varepsilon/3)<\frac{\varepsilon}%
{3K}\,\text{for }n\geq n_{0}.
\]
For $n\geq n_{0}$ we have:
\begin{align*}
E\left\vert X_{n}-X\right\vert  &  =E\left(  \left\vert X_{n}-X\right\vert
I(\left\vert X_{n}-X\right\vert >\varepsilon/3)\right)  +E\left(  \left\vert
X_{n}-X\right\vert I(\left\vert X_{n}-X\right\vert \leq\varepsilon/3)\right)
\\
&  \leq2KP(\left\vert X_{n}-X\right\vert >\varepsilon/3)+\varepsilon
/3\leq\varepsilon,
\end{align*}
since of course by the inequality $\left\vert X\right\vert \leq K$ we have
$\left\vert X_{n}-X\right\vert \leq2K.$
\end{proof}

\begin{theorem}
[on convergence in $L_{1}$]Let $\{X_{n}\}_{n\geq1}$, $X\in L_{1}$. Then
$E\left\vert X_{n}-X\right\vert \rightarrow0$, as $n\rightarrow\infty$ (that
is $X_{n}\rightarrow X$ in $L_{1})$ if and only if,when $i)$ $X_{n}\rightarrow
X$ in probability, $ii)$ the family $\{X_{n}\}_{n\geq1}$ is uniformly integrable.
\end{theorem}

\begin{proof}
$\Leftarrow$ Let us assume that the conditions $i)$ and $ii)$ are satisfied.
Let us fix $K>0$ and let us introduce function:
\[
\varphi_{K}(x)=\left\{
\begin{array}
[c]{ccc}%
K & gdy & x>K\\
x & gdy & \left\vert x\right\vert \leq K\\
-K & gdy & x<-K
\end{array}
\right.  .
\]
Let us select $\varepsilon>0$. From the properties of uniform integrability we
can select $K$ so that
\begin{equation}
E\left\vert \varphi_{k}(X_{n})-X_{n}\right\vert <\varepsilon/3\;\,\text{and
}E\left\vert \varphi_{K}(X)-X\right\vert <\varepsilon/3, \label{UI3}%
\end{equation}
since of course $\left\vert \varphi_{k}(X_{n})-X_{n}\right\vert =\left\vert
X_{n}\right\vert I(\left\vert X_{n}\right\vert >K)$. Moreover, from inequality
$\left\vert \varphi_{K}(x)-\varphi_{K}(y)\right\vert \leq\left\vert
x-y\right\vert $ it follows that $\varphi_{K}(X_{n})\rightarrow\varphi_{K}(X)$
in probability. On the base of Theorem \ref{UI2} we can select $n_{0}$ so that
for $n\geq n_{0}$,
\[
E\left\vert \varphi_{K}(X_{n})-\varphi_{K}(X)\right\vert <\varepsilon/3.
\]
Inequalities (\ref{UI3}), the last inequality and triangle inequality give:
\[
E\left\vert X_{n}-X\right\vert <\varepsilon.
\]

$\Rightarrow$ Let $X_{n}\rightarrow X$, as $n\rightarrow\infty$ in $L_{1}$.
Let us select $\varepsilon>0$ and $n_{0}$ so that for $n\geq n_{0}:$%
\[
E\left\vert X_{n}-X\right\vert <\varepsilon/2.
\]
The properties of Lebesgue integrals imply that one can select $\delta>0$, so
that if only $P(F)<\delta$, then
\[
E\left\vert XI(F)\right\vert <\varepsilon/2,\;E\left\vert X_{i}I(F)\right\vert
<\varepsilon,\;i=1,\ldots,n_{0}.
\]
Since, the sequence $\left\{  X_{n}\right\}  _{n\geq1}$ is bounded in $L_{1}$,
we can select $K$ so that
\[
\underset{n}{\sup}E\left\vert X_{n}\right\vert <\delta K.
\]
Then for $n\geq n_{0}$ we have $P(\left\vert X_{n}\right\vert >K)<\delta$ and
\begin{gather*}
E(\left\vert X_{n}\right\vert I(\left\vert X_{n}\right\vert >K))\leq\\
\leq E(\left\vert X\right\vert I(\left\vert X_{n}\right\vert >K))+E(\left\vert
X_{n}-X\right\vert I(\left\vert X_{n}\right\vert >K))\leq\varepsilon.
\end{gather*}
For $i=1,\ldots n_{0}$ we have $P(\left\vert X_{i}\right\vert >K)<\delta$ and
$E(\left\vert X_{i}I(\left\vert X_{i}\right\vert >K)\right\vert )<\varepsilon
$. Hence, $\{X_{n}\}_{n\geq1}$ is uniformly integrable family. Moreover, we
have
\[
\epsilon P(\left\vert X_{n}-X\right\vert >\epsilon)\leq E\left\vert
X_{n}-X\right\vert \rightarrow0,\;\text{when }n\rightarrow\infty.
\]
Since the number $\epsilon$ was selected to be positive, we deduce that
$X_{n}\rightarrow X$ in probability.
\end{proof}

\begin{proof}
[Proof of Proposition \ref{alfa_calk}]Condition $\underset{n}{\sup}E\left\vert
X_{n}\right\vert ^{\alpha}<\infty$ implies, that the family $\left\{
X_{n}\right\}  $ is uniformly integrable (compare example \ref{alpha}), which
together with the assumed convergence in probability on the basis of the above
mentioned theorems gives assertion.
\end{proof}

\section{Discrete time martingales\label{martyngaly}}

\subsection{Basic definitions and properties}

\begin{proof}
Let $\left(  \Omega,\mathcal{F},P\right)  $ be probability space. Let us
assume we are given also an \emph{increasing} sequence of $\sigma
-$fields\ $\left\{  \mathcal{G}_{i}\right\}  _{i\geq1}$, that is such that
$\mathcal{G}_{i}\subset\mathcal{G}_{i+1}\subset\mathcal{F}$. Such sequence of
$\sigma-$fields\ is called \emph{filtration.}
\end{proof}

\begin{definition}%
\index{Martingale}%
\label{def_mart}Sequence of integrable random variables $\left\{
X_{i}\right\}  _{\geq1}$ is called \emph{martingale} with respect to
filtration $\left\{  \mathcal{G}_{i}\right\}  _{i\geq1}$, if
\[
\forall i\geq1:X_{i}\;\text{is }\mathcal{G}_{i}\;\text{measurable and
}E(X_{i+1}|\mathcal{G}_{i})=X_{i}\text{ \ a.s.}%
\]
The sequence of the random variables $\left\{  Y_{i}\right\}  _{i\geq1}$ is
called sequence of martingale differences\emph{%
\index{Martingale difference}%
%TCIMACRO{\U{b3}}%
%BeginExpansion
${{}^3}$%
%EndExpansion
owa@r\'{o}\.{z}nica martynga\l owa} with respect to $\left\{  \mathcal{G}%
_{i}\right\}  _{i\geq1}$ if
\[
\forall i\geq1:Y_{i}\text{ is }\mathcal{G}_{i}\;\text{measurable and
}E(Y_{i+1}|\mathcal{G}_{i})=0.
\]

\end{definition}

\begin{remark}
Let us notice that if a sequence $\left\{  X_{i}\right\}  _{i\geq1}$ is
martingale, then the sequence \newline$\left\{  X_{i+1}-X_{i}\right\}
_{i\geq1}$ is a sequence of martingale differences. Similarly, if a sequence
$\left\{  Y_{i}\right\}  _{i\geq1}$ is a sequence of martingale differences,
then the sequence $\left\{  \sum_{j=1}^{i}Y_{j}\right\}  _{i\geq1}$ is a martingale.
\end{remark}

\begin{remark}
If $\{X_{n}\}_{n\geq1}$ are martingale differences, then for $i\neq j$
\newline$\operatorname*{cov}(X_{i},X_{j})=0.$
\end{remark}

\begin{example}
Let sequence $\left\{  \xi_{i}\right\}  _{i\geq1}$ consist of independent
random variables and let us assume, that $E\xi_{1}=0$. Let $\mathcal{G}%
_{i}=\sigma(\xi_{1},\ldots,\xi_{i})$. Then, $\left\{  \xi_{i}\right\}
_{i\geq1}$ is a sequence of martingale differences (of course with respect to
filtration $\left\{  \mathcal{G}_{i}\right\}  _{i\geq1})$.
\end{example}

\begin{example}
Let $\left\{  \mathcal{G}_{i}\right\}  _{i\geq1}$ will be increasing sequence
of $\sigma-$fields, $X$ integrable random variable. Then sequence
$Y_{i}=E(X|\mathcal{G}_{i})$ is a martingale.
\end{example}

\begin{definition}
The sequence$\left\{  X_{i}\right\}  _{i\geq1}$ is called super(sub)martingale
with respect to $\left\{  \mathcal{G}_{i}\right\}  _{i\geq1}$, if
\[
E(X_{i+1}|\mathcal{G}_{i})\leq(\geq)X_{i}\,\;a.s.
\]

\end{definition}

\begin{remark}
If $\left\{  X_{i}\right\}  _{i\geq1}$ is supermartingale, to $\left\{
-X_{i}\right\}  $ is submartingale.
\end{remark}

\begin{example}
Let $\left\{  X_{i}\right\}  _{i\geq1}$ will be a martingale, such that
$\forall i\geq1:E\left\vert X_{i}\right\vert ^{\alpha}<\infty$, for some
$\alpha\geq1$. Then $Y_{i}=\left\vert X_{i}\right\vert ^{\alpha}$ is
submartingale, since from Jensen's inequality we have $E\left(  \left\vert
X_{i}\right\vert ^{\alpha}|\mathcal{G}_{i-1}\right)  \allowbreak\geq\left\vert
E\left(  X_{i}|\mathcal{G}_{i-1}\right)  \right\vert ^{\alpha}\allowbreak
=\left\vert X_{i-1}\right\vert ^{\alpha}$ a.s.. Of course, every martingale is
simultaneously super and submartingale.
\end{example}

For submartingales we have the following very important theorem.

\begin{theorem}
[Doob]%
\index{Theorem!Doob}%
\label{zb_mart}Every nonnegative supermartingale\ $\left\{  X_{i}\right\}
_{i\geq0}$ converges (as $i\,\rightarrow\infty)$ with probability $1$ do
finite, nonnegative random variable.
\end{theorem}

Proof of this theorem is based on the following inequalities of combinatorial
nature. Let us consider the finite number sequence $\mathcal{X}=\{x_{1}%
,\ldots,x_{N}\}$. Let $a<b$ be two fixed reals. The greatest number $k\leq N$,
such that it is possible to find a sequence of indices
\[
1\leq s_{1}<t_{1}<s_{2}<t_{2}<s_{k}<t_{k}\leq N
\]
such that
\[
x_{s_{i}}<a,\;\;\text{and\ \ \ }x_{t_{i}}>b,\;\;1\leq i\leq k
\]
is called a number of upcrossings from below of the segment $[a,b]$ by the
sequence $\mathcal{X}.$

\begin{lemma}
[upcrossing lemma]\label{upcrossing}%
\index{Lemma!about upcrossings}%
Let $\mathcal{X}=\{x_{1},\ldots,x_{N}\}$ be any number sequence Let $b>a$ be
two reals, and $H_{a,b}^{(N)}$ let denote the number of upcrossings from below
of the segment $[a;b]$ by the sequence $\mathcal{X}$. Then:
\begin{equation}
(b-a)H_{a,b}^{(N)}\leq(a-x_{N})^{+}+\sum_{i=1}^{N}I(i)(x_{i+1}-x_{i})
\label{upcross}%
\end{equation}
wherein $I(i)$ takes only two values $0$ and $1$, and Moreover, the value
$I(i)$ is defined only by values $x_{1},\ldots,x_{i}.$
\end{lemma}

\begin{proof}
Let us denote by $\tau_{1}$ the first moment, when $\mathcal{X}$ assumes a
value less than $a$, by $\tau_{2}$ the first after $\tau_{1}$ moment, when
$\mathcal{X}$ assumes a value greater than $b$, by $\tau_{3}$ the first after
$\tau_{2}$ moment, when $\mathcal{X}$ assumes a value smaller than $a$ and so
on. Let us notice that on the segment ($\tau_{2H},N)$ the sequence
$\mathcal{X}$ will not upcross the segment $[a,b]$ from below (since it would
have to reach a value below $a$ before). Besides, we have:
\[
(b-a)H_{a,b}^{(N)}\leq\sum_{i=1}^{H}(x_{\tau_{2i}}-x_{\tau_{2i-1}}).
\]
The situation is illustrated below:%

%TCIMACRO{\FRAME{dtbpF}{3.3442in}{1.9164in}{0pt}{}{}{upcros.eps}%
%width 3.3442in;  height 1.9164in;  depth 0pt;  original-width 5.847in;
%original-height 3.3382in;  cropleft "0";  croptop "1";  cropright "1";
%cropbottom "0";  filename 'upcros.eps';file-properties "XNPEU";}}}%
%BeginExpansion
\begin{center}
\includegraphics[
height=1.9164in,
width=3.3442in
]%
{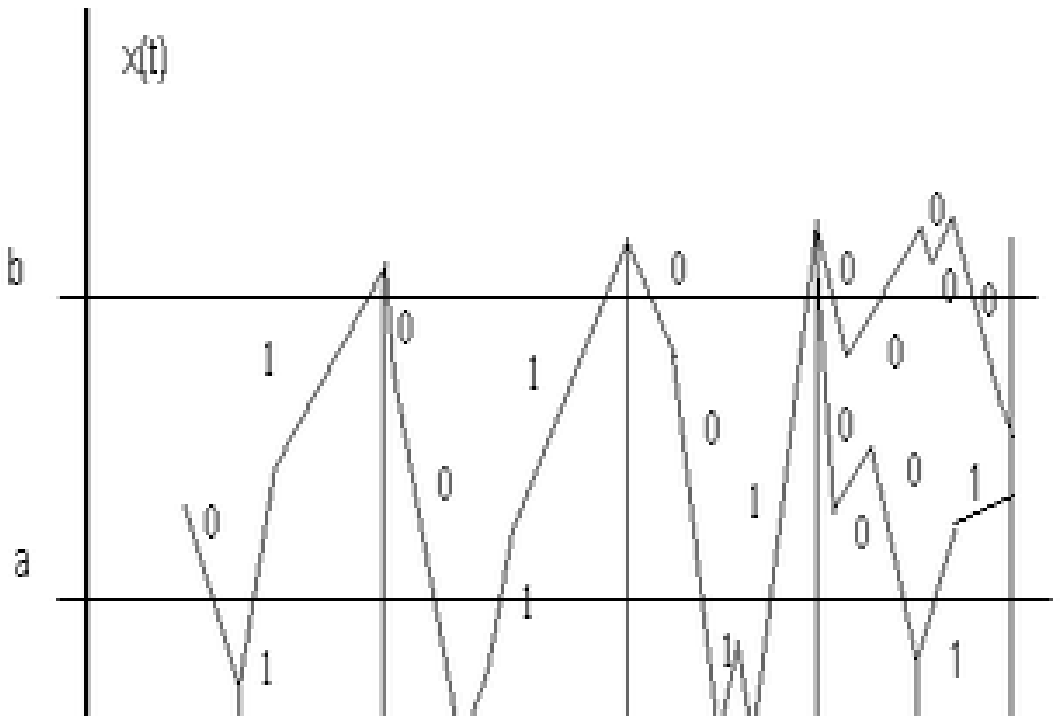}%
\end{center}
%EndExpansion
Quantity $I(t)$ is defined so that it changes its value from $0$ to $1$ or
conversely only at points $\tau_{i}$ $i=1,\ldots,H$. We define $I(\tau_{1}%
)=1$. Hence, for example $I(1)=0$, if $\tau_{1}>1$ and $I(1)=1$, when
$\tau_{1}=1$. These values are exposed on the figure above. To define value
$I(N)$ let us consider two situations (also exposed on the figure above):

1) if on the segment $(\tau_{2H},N)$ sequence $\mathcal{X}$ not once will not
assume values less than $a$, then $I(N)=0,$

2) if there exists after $\tau_{2H}$ moment $\tau_{2H+1}$, at which
$\mathcal{X}$ will assume value less than $a$, to $I(N)=1.$

Inequality \ref{upcross} is true, if $H_{a,b}^{(N)}=0$. Suppose, that
$H_{a,b}^{(N)}\geq1$. Of course, values of functions $I(t)$ at a fixed time
instant $t$ is fully defined by values $x_{1},\ldots,x_{t}$, wherein true is
equality:
\[
\sum_{t=1}^{N-1}I(t)(x_{t+1}-x_{t})=\left\{
\begin{array}
[c]{lll}%
\sum_{i=1}^{H}(x_{\tau_{2i}}-x_{\tau_{2i-1}}) & \text{in case} & 1\\
\sum_{i=1}^{H}(x_{\tau_{2i}}-x_{\tau_{2i-1}})+x_{N}-x_{\tau_{2H+1}} & \text{in
case} & 2
\end{array}
\right.  .
\]
Hence we always have:
\[
\sum_{t=1}^{N-1}I(t)(x_{t+1}-x_{t})\geq\left\{
\begin{array}
[c]{lll}%
(b-a)H_{a,b}^{(N)} & \text{in case} & 1\\
(b-a)H_{a,b}^{(N)}-x_{\tau_{2H+1}}+x_{N} & \text{in case} & 2
\end{array}
\right.
\]
Since we always have $x_{\tau_{2H+1}}<a$, we get inequality (\ref{upcross}).
\end{proof}

From of this lemma, it follows immediately the following corollary.

\begin{corollary}
\label{o przejsciach}Let $\mathcal{X}=\left\{  X_{i}\right\}  _{i=1}^{N}$ will
be finite supermartingale with respect to filtration $\left\{  \mathcal{G}%
_{i}\right\}  _{i=1}^{N}$. Let further $H_{a,b}^{(N)}$ be a random variable,
denoting number of upcrossing from below of the segment $[a,b]$ by the
supermartingale\ $\mathcal{X}$ during the time interval $[0;T]$. Then:
\[
EH_{a,b}^{(N)}\leq\frac{E(a-X_{N})^{+}}{(b-a)}.
\]

\end{corollary}

\begin{proof}
Let us notice that for the supermartingale $\mathcal{X}$ we have:
$E(I(t)(X_{t+1}-X_{t})|\mathcal{G}_{t})=I(t)(E(X_{t+1}|\mathcal{G}_{t}%
)-X_{t})\leq0$ a.s., since $I(t)$ is $\mathcal{G}_{t}-$measurable random
variable. Hence,
\[
(b-a)EH_{a,b}^{(N)}\leq E(a-X_{N})^{+}.
\]

\end{proof}

\begin{proof}
[Proof of Theorem \ref{zb_mart}]Let us denote by $A_{1}$ an event that
$\lim\inf_{t\rightarrow\infty}X_{t}=\infty$, and by $A_{2}$ an event that
$\lim\inf_{t\longrightarrow\infty}X_{t}<\lim\sup_{t\longrightarrow\infty}%
X_{t}$. It is clear, that $A_{1}\cup A_{2}$ is the event that the
supermartingale\ will not have a finite limit. But from the condition
$\lim_{t\longrightarrow\infty}EX_{t}<\infty$ and Fatou's Lemma (Lemma
\ref{Fatou}, p. \pageref{Fatou}) it follows that $P(A_{1})\allowbreak
=\allowbreak0$. Let us denote by $Q$ the of all rationals. The event $A_{2}$
one can present in the following form:
\[
A_{2}=\bigcup_{p<q;p,q\in Q}\{\omega:\underset{t\rightarrow\infty}{\lim\inf
}X_{t}(\omega)<p<q<\underset{t\rightarrow\infty}{\lim\inf}\,X_{t}(\omega)\}.
\]
Hence it is easy to deduce, that
\[
A_{2}\subseteq\bigcup_{p<q;p,q\in Q}\{\omega:H_{p,q}^{\infty}(\omega
)=\infty\}.
\]
Of course $H_{p,q}^{\infty}=\lim_{N\longrightarrow\infty}H_{p,q}^{(N)}$. From
corollary \ref{o przejsciach} and from the fact, that $\{X_{n}\}_{n\geq1}$ is
a nonnegative sequence it follows that
\[
EH_{p,q}^{N}\leq\frac{p}{q-p},
\]
since $E(a-X_{N})^{+}\leq p$. Hence, on the basis of Fatou's Lemma we have
\[
EH_{p,q}^{\infty}\leq\frac{p}{q-p},
\]
thus immediately, it follows that $P(H_{p,q}^{\infty}=\infty)=0$. and further,
that $P(A_{2})=0.$
\end{proof}

\begin{remark}
Let us notice that not necessarily $\underset{i\rightarrow\infty}{\lim}%
EX_{i}=E\underset{i\rightarrow\infty}{\lim}X_{i}.$
\end{remark}

\subsection{Theorem about convergence}

\begin{theorem}%
\index{Theorem!Martingale convergence}%
Let $\left\{  X_{i}\right\}  _{i\geq1}$ be supermartingale with respect to
filtration $\left\{  \mathcal{G}_{i}\right\}  _{i\geq1},$such that quantity
$EX_{i}^{-}=E\max(-X_{i},0)$ is bounded. Then with probability $1$ it has (as
$i\,\rightarrow\infty)$ finite limit. If additionally for $\alpha>1$
$\underset{i\rightarrow\infty}{\lim}E\left\vert X_{i}\right\vert ^{\alpha
}<\infty$, to we also have $E\underset{i\rightarrow\infty}{\lim}%
X_{i}=\underset{i\rightarrow\infty}{\lim}EX_{i}$ and $\underset{i\rightarrow
\infty}{\lim}E\left\vert X_{i}-X\right\vert =0.$
\end{theorem}

\begin{proof}
From condition $\underset{i}{\sup}EX_{i}^{-}<\infty$ it follows on the basis
of Fatou's Lemma (Lemma \ref{Fatou}, p. \pageref{Fatou}), that
$E\underset{i\,\longrightarrow\infty}{\lim\inf}\,X_{i}^{-}<\infty$. Now we
argue as in the proof Doob's Theorem \ref{zb_mart} considering additionally an
event $\underset{n\,\longrightarrow\infty}{\lim\inf}\,X_{n}=-\infty$, whose
probability is equal to zero, since $E\underset{i\,\longrightarrow\infty
}{\lim\inf}\,X_{i}^{-}<\infty$. It remains to show, that if
$\underset{i\rightarrow\infty}{\lim}E\left\vert X_{i}\right\vert ^{\alpha
}<\infty$ for some $\alpha>1$, then $E\underset{i\rightarrow\infty}{\lim}%
X_{i}=\underset{i\rightarrow\infty}{\lim}EX_{i}$. Having proven convergence of
the sequence $\{X_{n}\}_{n\geq1}$ to a finite limit it is enough now to recall
Proposition \ref{alfa_calk}.
\end{proof}

\begin{remark}
Condition of finiteness of $EX_{i}^{-}$ one can substitute by the condition of
finiteness of $E\left\vert X_{i}\right\vert ^{\alpha}$ for some $\alpha\geq1.$
\end{remark}

\begin{remark}
Since, if the sequence $\{X_{n}\}_{n\geq1}$ is a supermartingale with respect
to some filtration, then the sequence $\{-X_{n}\}_{n\geq1}$ is a
submartingale, hence from the previous theorems it follows that if
$\underset{n}{\sup}EX_{n}^{+}<\infty$, then submartingale\ $\{X_{n}\}_{n\geq
1}$ is convergent with probability $1.$
\end{remark}

Submartingales have yet another, interesting property namely they satisfy the
so-called maximal inequality. This inequality we will use in the proof of the
law of iterated logarithm in the Appendix \ref{PIL}.

\subsection{Maximal inequality}

\begin{theorem}
[Maximal inequality]%
\index{Theorem!Maximal inequality}%
Let $\{X_{n}\}_{n\geq0}$ be nonnegative submartingale with respect to
filtration $\left\{  \mathcal{G}_{i}\right\}  _{i\geq0}$. Then for any $c>0:$%
\begin{equation}
cP(\underset{k\leq n}{\sup}X_{k}\geq c)\leq EX_{n}. \label{nier_max}%
\end{equation}

\end{theorem}

\begin{proof}
Let us denote $F\allowbreak=\allowbreak\left\{  \underset{k\leq n}{\sup}%
X_{k}\geq c\right\}  $, $F_{0}=\left\{  X_{0}\geq c\right\}  $, $F_{1}%
=\allowbreak\left\{  X_{0}<c,X_{1}\geq c\right\}  \allowbreak,\ldots
,\allowbreak F_{n}=\allowbreak\bigcap_{j=0}^{n-1}\left\{  X_{j}<c\right\}
\allowbreak\cap\allowbreak\left\{  X_{n}\geq c\right\}  $. Of course, we have
$F=\bigcup_{i=0}^{n}F_{i}\,$.The events $\left\{  F_{i}\right\}  $ are
disjoint and $F_{i}\in\mathcal{G}_{i}$ for every $i=0,1,\ldots,n$. Moreover,
we have:
\begin{equation}
E(X_{n}I(F_{k}))=E\left(  E\left(  X_{n}|\mathcal{G}_{k}\right)
I(F_{k})\right)  \geq E(X_{k}I(F_{k}))\geq cEI(F_{k})=cP(F_{k}), \label{pom3}%
\end{equation}
since on the event $F_{k}$ we have $X_{k}\geq c$. By $I(F_{k})$ we denoted
here random variable, that is equal to $1$, when the event $F_{k}$ is
satisfied and $0$ in the opposite case.

Inequalities (\ref{pom3}) we add side by side and we use the fact, that
$I(F)\allowbreak=\allowbreak\sum_{i=0}^{n}I(F_{i})$. We use the fact, that
$X_{n}$ is a nonnegative random variable, hence, that $EX_{n}\allowbreak
\geq\allowbreak EX_{n}I(F).$
\end{proof}

At the end let us consider the sequence $\{X_{n}\}_{n\leq-1}$ having the
following properties: for some filtration $\left\{  \mathcal{G}_{n}\right\}
_{n\leq-1}$ we have: $X_{n}$ is a $\mathcal{G}_{n}$-measurable random variable
and $E(X_{n+1}|\mathcal{G}_{n})\allowbreak=\allowbreak X_{n}$ for
$n=\ldots-3,-2,-1$. Such sequence is called sometimes reversed martingale. Let
us notice that if we consider this sequence for $n=-N,\ldots,-1$ for some
$N>1$, then the \textquotedblright upcrossing lemma\textquotedblright, i.e.
Lemma \ref{upcrossing} is true and we have
\[
EH_{a,b}^{(N)}\leq\frac{E(a-X_{-1})^{+}}{(b-a)},
\]
for the number of upcrossings from below of the segment $[a,b]$ by the
sequence $X_{-N},\ldots,X_{-1}$. Hence, it is easy to show (as in the proof
Doob's Theorem ), that $P(\underset{n\,\longrightarrow\infty}{\lim\inf
}\,X_{-n}<\underset{n\rightarrow\infty}{\,\lim\sup\,}X_{-n})\allowbreak
=\allowbreak0$. Hence, we impose conditions that $P(\underset{n\rightarrow
\infty}{\lim}X_{-n}=-\infty\cup\underset{n\rightarrow\infty}{\lim}X_{n}%
=\infty)\allowbreak=\allowbreak0$, then we are able to state, that the
sequence $\{X_{n}\}_{n\leq-1}$ converges as $n\rightarrow-\infty$ to a finite
limit. Hence, we have;

\begin{theorem}
\label{odwrotny_martyngal}Let $\{X_{n}\}_{n\leq-1}$ will be a reversed
martingale. If $\underset{n}{\sup}E\left\vert X_{-n}\right\vert <\infty$, then
$\underset{n\rightarrow\infty}{\lim}X_{-n}$ exists and is finite with
probability 1..
\end{theorem}

True is also the following fact.

\begin{theorem}
Let be given two filtrations $\left\{  \mathcal{G}_{n}\right\}  _{n\geq1}$ and
$\left\{  \mathcal{H}_{n}\right\}  _{n\leq-1}$. Let further $Z$ will be
integrable random variable. Then the families of the random variables
\[
\left\{  E(Z|\mathcal{G}_{n})\right\}  _{n\geq1}\;\text{and }\left\{
E(Z|\mathcal{H}_{n})\right\}  _{n\leq-1}%
\]
are respectively\ martingale and reversed martingale with respect to
respective filtrations. Moreover, we have:
\begin{align*}
\underset{n\rightarrow\infty}{\lim}E(Z|\mathcal{G}_{n})  &  =E(Z|\mathcal{G}%
_{\infty})\;a.s.\;\text{and in }L_{1},\\
\underset{n\rightarrow-\infty}{\lim}E(Z|\mathcal{H}_{n})  &  =E(Z|\mathcal{H}%
_{-\infty})\;a.s.\,\text{and in }L_{1}.
\end{align*}

\end{theorem}

\begin{proof}
To check that these sequences are indeed martingales are trivial. These
sequences form a uniformly integrable family (see Appendix \ref{UnInt}).
Hence, they are convergent almost surely and in $L_{1}.$
\end{proof}

\section{Stopping times\label{czas}%
\index{Stopping time}%
}

Let $\left\{  \mathcal{G}_{i}\right\}  _{i\geq1}$ be a filtration.

\begin{definition}
Nonnegative integer valued random variable $T$ is called \emph{stopping time
}with respect to filtration $\left\{  \mathcal{G}_{i}\right\}  _{i\geq1}$,
if:
\begin{equation}
\forall i\in%
%TCIMACRO{\U{2115} }%
%BeginExpansion
\mathbb{N}
%EndExpansion
:\left\{  T\leq i\right\}  \in\mathcal{G}_{i}. \label{ST1}%
\end{equation}

\end{definition}

\begin{remark}
Let us notice that $\left\{  T=i\right\}  \in\mathcal{G}_{i}$ and $\left\{
T\geq i\right\}  =\left\{  T\leq i-1\right\}  ^{c}\in\mathcal{G}_{i-1}$.
Hence, of course $\left\{  T<i\right\}  \in\mathcal{G}_{i-1}.$
\end{remark}

Let $\{X_{n}\}_{n\geq1}$ be martingale (or submartingale), a random variable
$T$ a stopping time with respect to filtration $\left\{  \mathcal{G}%
_{i}\right\}  _{i\geq1}$. Let us denote:
\[
\forall n\in%
%TCIMACRO{\U{2115} }%
%BeginExpansion
\mathbb{N}
%EndExpansion
:X_{n}^{(T)}=X_{\min(T,n)}=\left\{
\begin{array}
[c]{ccc}%
X_{n} & gdy & T\geq n\\
X_{T} & gdy & T<n
\end{array}
\right.  .
\]

\begin{proposition}
i) If $\{X_{n}\}_{n\geq1}$ is a martingale, then also $\{X_{n}^{(T)}%
\}_{n\geq1}$ is a martingale. Moreover, $\forall n\in%
%TCIMACRO{\U{2115} }%
%BeginExpansion
\mathbb{N}
%EndExpansion
:$ $EX_{n}^{(T)}=EX_{1}.$

ii) If $\{X_{n}\}_{n\geq1}$ is a submartingale, then also $\{X_{n}%
^{(T)}\}_{n\geq1}$ is a submartingale.
\end{proposition}

\begin{proof}
We will prove only assertion $i)$. Proof of $ii)$ is almost identical.
Firstly, let us notice that $X_{n}^{(T)}$ is a $\mathcal{G}_{n}-$ measurable
random variable. It follows from the fact, that we have \ $X_{n}^{(T)}%
=X_{n}I(T\geq n)+X_{T}I(T<n)$. Moreover, denoting $X_{0}\allowbreak
=\allowbreak0$ we have:
\[
\left\vert X_{n}^{(T)}\right\vert =\left\vert \sum_{i=1}^{n}(X_{i}%
-X_{i-1})I(T\geq i)\right\vert \leq\sum_{i=1}^{n}\left\vert X_{i}%
-X_{i-1}\right\vert .
\]
Hence $E\left\vert X_{n}^{(T)}\right\vert <\infty$ for every $n$. Besides we have:%

\begin{align*}
E(X_{n}^{(T)}|\mathcal{G}_{n-1})  &  =E(X_{n}I(T\geq n)|\mathcal{G}%
_{n-1})+E(X_{T}I(T<n)|\mathcal{G}_{n-1})=\\
&  =I(T>n-1)X_{n-1}+X_{T}I(T\leq n-1)=X_{n-1}^{(T)},
\end{align*}
since $I(T>n-1)\allowbreak=\allowbreak I(T\geq n)\in\mathcal{G}_{n-1}$,
$I(T\leq n-1)\allowbreak=\allowbreak I(T<n)$ and random variable $X_{T}I(T<n)
$ is $\mathcal{G}_{n-1}$ measurable. It remained to show, that $\forall n\in%
%TCIMACRO{\U{2115} }%
%BeginExpansion
\mathbb{N}
%EndExpansion
:$ $EX_{n}^{(T)}=EX_{1}$. We have however $EX_{n}^{(T)}=EX_{n-1}^{(T)}%
=\ldots=EX_{1}.$
\end{proof}

\section{Proof of the law of iterated logarithm\label{PIL}}

Proof of this theorem is very complex and will not present it here in full
generality. We will prove this result under additional assumption, that the
respective random variables have a normal distribution. Since, that value of
$\sigma^{2}$ is not essential we will assume, that $X_{1}\sim N(0,1)$. Let us
denote
\[
S_{n}=\sum_{i=1}^{n}X_{i},\;h(n)=\sqrt{2n\log\log n};n\geq3.
\]
We start from easy relationship for normal random variables stating that:
\[
E\exp(\eta X_{1})=\exp(\frac{1}{2}\eta^{2}).
\]
Hence for $S_{n}$ we have:
\begin{equation}
E\exp(\eta S_{n})=\exp(\frac{1}{2}\eta^{2}n). \label{normalne}%
\end{equation}
Let us now notice, that the function $x\longmapsto\exp(\eta x)$ is convex,
hence the sequence $Y_{n}=\exp(\eta S_{n})$ is a positive submartingale i.e.
we have:
\[
E(Y_{n+1}|X_{1},\ldots,X_{n})\geq Y_{n}.
\]
For nonnegative submartingales we have the following Doob's maximal inequality
(see Appendix \ref{martyngaly} formula(\ref{nier_max})). Applying i to the
sequence $\left\{  Y_{n}\right\}  $ and putting $\gamma\allowbreak
=\allowbreak\exp(\eta c)>0$ and utilizing relationship (\ref{normalne}) we
get
\[
P(\underset{k\leq n}{\sup}S_{k}\geq c)=P(\underset{k\leq n}{\sup}\exp(\eta
S_{n})\geq\exp(\eta c))\leq\exp(-c\eta)\exp(\frac{1}{2}\eta^{2}n).
\]
Let us select now $\eta$ equal to be $c/n$. We get then:
\[
P(\underset{k\leq n}{\sup}S_{k}\geq c)\leq\exp(-\frac{1}{2}c^{2}/n).
\]
Let us now set $n=K^{j}$ and $c_{j}=Kh(K^{j-1})$ for some $K>1$. We have then
after simple algebra using a definition of logarithm:
\[
P(\underset{k\leq K^{j}}{\sup}S_{k}\geq c_{j})\leq\exp(-\frac{1}{2}c_{j}%
^{2}/K^{j})=\frac{1}{(j-1)^{K}(\log K)^{K}}.
\]
Let us notice that $\sum_{j\geq2}\frac{1}{(j-1)^{K}(\log K)^{K}}<\infty$.
Hence, on the basis of the Borel-Cantelli Lemma (see Lemma
\ref{Borel-Cantelli}) the events $\left\{  \underset{k\leq K^{j}}{\sup}%
S_{k}\geq c_{j}\right\}  $ happen only a finite number of times, hence
starting from some large $j$ we have for $m\in\lbrack K^{j-1},K^{j}]$
\[
S_{m}\leq\underset{m\leq K^{j}}{\sup}S_{m}\leq c_{j}=Kh(K^{j-1})\leq Kh(m),
\]
since function $h$ is non-decreasing. Hence, $\underset{n\rightarrow
\infty}{\lim\inf}\,\frac{S_{n}}{h(n)}\leq K$, which, considering freedom of
$K>1$, gives
\[
\underset{n\rightarrow\infty}{\lim\inf}\,\frac{S_{n}}{h(n)}\leq1.
\]
It remained to show, that $\underset{n\rightarrow\infty}{\lim\inf}%
\,\frac{S_{n}}{h(n)}\geq1$ since, the limit $\underset{n\rightarrow
\infty}{\lim\inf}\,\frac{S_{n}}{h(n)}\allowbreak=\allowbreak-1$ we will get
considering sequence $-S_{n}$. Hence, to show $\underset{n\rightarrow
\infty}{\lim\inf}\,\frac{S_{n}}{h(n)}\geq1$ let us consider an event:
\[
F_{n}=\left\{  S_{N^{n+1}}-S_{N^{n}}>(1-\varepsilon)h(N^{n+1}-N^{n})\right\}
,
\]
for some $\varepsilon\in(0,1)$ and $N>1$. To calculate the probability of
$F_{n}$ we have to use the following lemma about normal random variables:

\begin{lemma}%
\index{Lemma!About cdf of Normal distribution}%
\label{o normalnych}Let $Y\sim N(0,1)$. Then
\begin{align*}
P(Y  &  >x)\leq\frac{\phi(x)}{x},\\
P(Y  &  >x)\geq(x+\frac{1}{x})^{-1}\phi(x),
\end{align*}
where $\phi(x)$ is the density function of distribution $N(0,1).$
\end{lemma}

\begin{proof}
We have for the density $\phi$ : $\phi^{\prime}(x)=-x\phi(x)$. Hence,
\[
\phi(x)=\int_{x}^{\infty}y\phi(y)dy\geq x\int_{x}^{\infty}\phi(y)dy=xP(Y>x),
\]
that is we have the first assertion.\ Further, let us notice that $\left(
\frac{\phi(y)}{y}\right)  ^{\prime}=-(1+\frac{1}{y^{2}})\phi(y)$. \newline
Hence:
\[
\frac{\phi(y)}{y}=\int_{y}^{\infty}(1+\frac{1}{x^{2}})\phi(x)dx\leq(1+\frac
{1}{y^{2}})\int_{y}^{\infty}\phi(x)dx=(1+\frac{1}{y^{2}})P(Y>y).
\]
Thus, we have the second assertion.

Returning to the theorem's proof, we have on the basis of the second assertion
of the just proved lemma :
\[
P(F_{n})=1-\phi(a)\geq\frac{1}{\sqrt{2\pi}}(a+\frac{1}{a})^{-1}\exp
(-\frac{a^{2}}{2}),
\]
where we denoted: $a=(1-\varepsilon)\sqrt{2\log\log(N^{n+1}-N^{n})}$. But
$(a+a^{-1})\geq2$ for $a\geq0$ and.
\[
\exp(-\frac{a^{2}}{2})\cong\frac{1}{(n\log N)^{(1-\varepsilon)^{2}}}.
\]
Hence
\[
\sum_{n\geq1}P(F_{n})=\infty.
\]
Random variables $\{X_{n}\}_{n\geq1}$ are independent and events $\left\{
F_{n}\right\}  _{n\geq1}$ are independent. On the base of the second part of
Borel-Cantelli Lemma \ref{Borel-Cantelli}, we deduce that infinitely many
times we will have:
\[
S_{N^{n+1}}\geq(1-\varepsilon)h(N^{n+1}-N^{n})+S_{N^{n}}.
\]
Using just proved inequality
\[
\underset{n\rightarrow\infty}{\lim\inf}\,\frac{S_{n}}{h(n)}\leq1
\]
applied to $-S_{n}$ gives
\[
S_{N^{n}}\geq-h(N^{n})
\]
for sufficiently large $n$. Hence, we have:
\[
S_{N^{n+1}}\geq(1-\varepsilon)h(N^{n+1}-N^{n})-h(N^{n}).
\]
Thus, we have:
\[
\underset{n\rightarrow\infty}{\lim\inf}\,\frac{S_{n}}{h(n)}\geq
\underset{n\rightarrow\infty}{\lim\inf}\,\frac{S_{N^{n+1}}}{h(N^{n+1})}%
\geq(1-\varepsilon)\sqrt{1-\frac{1}{N}}-\frac{1}{\sqrt{N}}%
\]
since $\sqrt{\frac{\log\log(N^{n+1}-N^{n})}{\log\log N^{n+1}}}\cong1$ for
sufficiently large $n$. Now taking sufficiently large $N$ and small
$\varepsilon$ we see that indeed $\underset{n\rightarrow\infty}{\lim\inf
}\,\frac{S_{n}}{h(n)}\geq1.$
\end{proof}

\section{Symmetrization\label{symetryzacja}}

\begin{definition}
We say, that random variable $Y$ has symmetric distribution, when:
\[
\forall x\in%
%TCIMACRO{\U{211d} }%
%BeginExpansion
\mathbb{R}
%EndExpansion
:P(Y<x)=P(-Y<x).
\]

\end{definition}

Let $Y$ be symmetric random variable. Let us take any positive number $M$ and
let us denote $Y^{<M}=YI(\left\vert Y\right\vert \leq M)$ and $Y^{>M}%
=YI(\left\vert Y\right\vert >M)$. Of course, $Y=Y^{<M}+Y^{>M}$. We have the
following simple fact:

\begin{proposition}
Random variables $Y$ and $Y^{>M}-Y^{<M}$ have the same distribution.
\end{proposition}

\begin{proof}
Let us denote $Z=Y^{>M}-Y^{<M}$. We have
\[
P(Z<x)=\left\{
\begin{array}
[c]{ccc}%
P(Y<x) & \text{when} & x\leq-M\\
P(Y\leq-M)+P(-x<Y<M) & \text{when} & -M<x\leq M\\
P(Y<x) & \text{when} & M<x
\end{array}
\right.  .
\]
Let us notice now, that $P(-x<Y<M)=P(x>-Y>-M)=P(-Y<x)-P(-Y\leq-M)$. However
taking into account symmetry of the random variable $Y$ we have $P(-Y\leq
-M)=P(Y\leq-M)$ and $P(-Y<x)=P(Y<x)$. Hence, indeed $P(Z<x)=P(Y<x).$
\end{proof}

Let $X$ be a random variable. Let $X^{\prime}$ be random variable independent
of $X$ and having the same as $X$ distribution. Let us consider random
variable:
\[
X^{s}=X-X^{\prime}.
\]
It is called symmetrization of the random variable $X.$

\begin{proposition}
Random variable $X^{s}$ has symmetric distribution.
\end{proposition}

\begin{proof}
Of course we have for any $x$ $P(X^{s}<x)=P(X-X^{\prime}<x)=P(X^{\prime
}-X<x)=P(-X^{s}<x).$
\end{proof}

\section{0-1 laws\label{prawo01}}

Let $\{X_{n}\}_{n\geq1}$ be a sequence of independent random variables. Let us
denote
\[
\mathcal{G}_{n}=\sigma(X_{n},X_{n+1},\ldots),\;\;\mathcal{G}_{\infty}%
=\bigcap_{n=1}^{\infty}\mathcal{G}_{n}.
\]
Kolmogorov's Theorem states.

\begin{theorem}
[ Kolmogorov's 0-1 law]%
\index{Law!!Kolmogorov 0-1}%
Let $A\in\mathcal{G}_{\infty}$. Then either $P(A)\allowbreak=\allowbreak0$ or
$P(A)\allowbreak=\allowbreak1.$
\end{theorem}

\begin{proof}
Let $I(A)(\omega)=\left\{
\begin{array}
[c]{ccc}%
1 & gdy & \omega\in A\\
0 & gdy & \omega\notin A
\end{array}
\right.  $. Notice that the sequence of the random variables $Z_{n}%
=E(I(A)|X_{1},\ldots,X_{n})$ is a martingale with respect to filtration
$\left\{  \sigma(X_{1},\ldots,X_{n})\right\}  _{n\geq1}$. From one side,
considering the independence of $\sigma-$fields\ $\sigma(X_{1},\ldots X_{n})$
and $\mathcal{G}_{\infty}$ for every $n$ we have $Z_{n}=EI(A)=P(A)$. But on
the other hand, considering martingale convergence theorem, we get:
\[
\underset{n\rightarrow\infty}{\lim}Z_{n}=E(I(A)|\mathcal{B}_{\infty
})=I(A)\;a.s.\,,
\]
where by $\mathcal{B}_{\infty}$ we denoted $\sigma(\bigcup_{i=1}^{n}%
\sigma(X_{1},\ldots,X_{i}))$. It is obvious, that $\mathcal{G}_{\infty}%
\subset\mathcal{B}_{\infty}$. Hence, $I(A)\allowbreak=\allowbreak P(A)$. That
is $P(A)=0$ or $1$.
\end{proof}

To formulate Hewitt-Savage law one has to define the notion of the symmetric
event. Let us denote $%
%TCIMACRO{\U{211d} }%
%BeginExpansion
\mathbb{R}
%EndExpansion
^{\infty}$ set of sequences of reals. $\sigma$-field $\mathcal{B}_{\infty} $
of Borel subsets of $%
%TCIMACRO{\U{211d} }%
%BeginExpansion
\mathbb{R}
%EndExpansion
^{\infty}$ is of the form $\mathcal{B}_{\infty}=\sigma(\bigcup_{n=1}^{\infty
}\mathcal{B}_{n})$, where $\mathcal{B}_{n}$ denotes $\sigma$-field of Borel
subsets of $%
%TCIMACRO{\U{211d} }%
%BeginExpansion
\mathbb{R}
%EndExpansion
^{n}$ .

\begin{definition}
\label{sym3}Let $\{X_{n}\}_{n\geq1}$ be a sequence of the random variables. An
event $B\in\sigma(X_{n},n\geq1)$ is called symmetric, if there exists a Borel
set $C_{\infty}\in\mathcal{B}_{\infty}$ such that for every $m\geq1$ and every
permutation $\left\{  i_{1},\ldots,i_{m}\right\}  $ of the set $\left\{
1,\ldots,m\right\}  $ we have
\begin{align*}
B  &  =\left\{  \omega:(X_{1}(\omega),\ldots,X_{m}(\omega),\ldots)\in
C_{\infty}\right\}  =\\
&  \left\{  \omega:(X_{i_{1}},\ldots,X_{i_{m}},\ldots)\in C_{\infty}\right\}
.
\end{align*}

\end{definition}

\begin{theorem}
[prawo 0-1 Hewitta-Savege'a]%
\index{Law!Hewitt -Savage 0-1!}%
Let $\{X_{n}\}_{n\geq1}$ be a sequence of independent random variables having
identical distributions. Then every event symmetric, belonging to
$\sigma(X_{n},n\geq1)$ has probability $0$ or $1$.
\end{theorem}

\begin{proof}
Let $B$ be a symmetric event. By the properties of measure it follows that
there exists a sequence of events $B_{n}\in\sigma(X_{1},\ldots,X_{n})$,
$n\geq1$ such that
\begin{equation}
P(B_{n}\vartriangle B)\rightarrow0,\,n\rightarrow\infty. \label{H-S1}%
\end{equation}
Let $C_{n}$ and $C_{\infty}$ be such Borel subsets of respectively $%
%TCIMACRO{\U{211d} }%
%BeginExpansion
\mathbb{R}
%EndExpansion
^{n}$ and $%
%TCIMACRO{\U{211d} }%
%BeginExpansion
\mathbb{R}
%EndExpansion
^{\infty}$, that
\[
B_{n}=\left\{  \omega:(X_{1}(\omega),\ldots,X_{n}(\omega))\in C_{n}\right\}
\text{ and }B=\left\{  \omega:(X_{1}(\omega),\ldots,X_{n}(\omega),\ldots)\in
C_{\infty}\right\}  .
\]
Let us define also the following event:
\[
B_{n}^{^{\prime}}=\left\{  \omega:(X_{n+1}(\omega),\ldots,X_{2n}(\omega))\in
C_{n}\right\}
\]
and
\[
B^{^{\prime}}=\left\{  \omega:(X_{n+1}(\omega),\ldots,X_{2n}(\omega
),X_{1},\ldots,X_{n},X_{2n+1},\ldots)\in C_{\infty}\right\}  .
\]
Taking into account independence and identity of distributions of elements of
the sequence $\{X_{n}\}_{n\geq1}$ we have: $P(B_{n})=P(B_{n}^{^{\prime}})$
and
\begin{equation}
P(B_{n}^{^{\prime}}\cap B_{n})=P(B_{n})P(B_{n}^{^{\prime}})\rightarrow
P^{2}(B),\;n\rightarrow\infty. \label{H-S2}%
\end{equation}
Moreover, further arguing in the same way we have:
\begin{equation}
P(B_{n}^{^{\prime}}\vartriangle B^{^{\prime}})=P(B_{n}\vartriangle
B),\;n\geq1. \label{H-S3}%
\end{equation}
Symmetry of $B$ implies:
\[
P(B_{n}^{^{\prime}}\vartriangle B^{^{\prime}})=P(B_{n}^{^{\prime}}\vartriangle
B),\;n\geq1.
\]
Taking into account (\ref{H-S1}), and (\ref{H-S3}) we see that:
\[
P(B_{n}^{^{\prime}}\vartriangle B)\rightarrow0,\;n\rightarrow\infty.
\]
this convergence confronted with (\ref{H-S1}) gives:
\[
P(B_{n}\cap B_{n}^{^{\prime}})\rightarrow P(B).
\]
On its sides the above mentioned convergence combined with (\ref{H-S2}) gives
equality $P(B)=P^{2}(B)$ from which immediately follows assertion.
\end{proof}

\section{Proof Strassen's Theorem\label{strassen}}

\begin{proof}
The event described by (\ref{odwr_pil}) belongs to the so-called tail
$\sigma-\,$field. By the Kolmogorov's 0-1 law (see Appendix \ref{prawo01})
such $\sigma$-field contains only events whose probability are equal either
$0$ or $1$. Hence, this event is satisfied in fact for almost all elementary
events. In other words, we have to have:
\[
\underset{n\,\rightarrow\infty}{\lim}\frac{\sum_{i=1}^{n}X_{i}}{n}=0,
\]
with probability $1$. Hence, on the basis of Kolmogorov's Theorem discussed in
chapter \ref{simpwl} we must have $X_{1}\in L_{1}$ and $EX_{1}=0$. Thus, it
remained to show, that the condition (\ref{odwr_pil}) implies, that $X_{1}\in
L_{2}$. Let us notice that, without loss of generality, that we can assume,
that random variables $X_{i}$ are symmetric (possibly symmetrizing them). Let
us assume, that $EX^{2}=\infty$. We will show that then
$\underset{n\rightarrow\infty}{\lim\inf}\,\frac{\left\vert \sum_{i=1}^{n}%
X_{i}\right\vert }{\sqrt{2n\log\log n}}=\infty.$

Let us select $M>0$. Let us denote $X^{<M}=XI(\left\vert X\right\vert \leq
c_{M})$. We will select constant $c_{M}$ so that $E\left(  X^{<M}\right)
^{2}\geq M$. Let us denote for brevity $S_{n}=\sum_{i=1}^{n}X_{i}$,
$S_{n}^{^{\prime}}=\sum_{i=1}^{n}X^{<M}$ and $S_{n}^{^{\prime\prime}}%
=\sum_{i=1}^{n}X^{>M}$. In the part of the Appendix dedicated to
symmetrization we have shown, that random variables $X$ and $X^{<M}-X^{>M}$
have the same distributions. Thus, random variables $S_{n}$ and $S_{n}%
^{^{\prime}}-S_{n}^{^{\prime\prime}}$ have the same distributions. Hence, the
following two events
\begin{equation}
\bigcap_{i=1}^{\infty}\bigcup_{n\geq i}^{\infty}\left\{  S_{n}^{^{\prime}%
}>(Mn\log\log n)^{1/2},S_{n}-S_{n}^{^{\prime}}\geq0\right\}  \label{sym1}%
\end{equation}
and
\begin{equation}
\bigcap_{i=1}^{\infty}\bigcup_{n\geq i}^{\infty}\left\{  S_{n}^{^{\prime}%
}>(Mn\log\log n)^{1/2},S_{n}-S_{n}^{^{\prime}}\leq0\right\}  , \label{sym2}%
\end{equation}
have the same probability. It follows then from the fact, that $S_{n}%
-S_{n}^{^{\prime}}=S_{n}^{^{\prime\prime}}$, from $0-1$ Hewitt-Savage law, we
deduce that probability of every one of the two is either $0$ or $1$.
Moreover, by the LIL applied to random variables $X^{<M}$ that have variances
we get :
\[
P(\bigcap_{i=1}^{\infty}\bigcup_{n\geq i}^{\infty}\left\{  S_{n}^{^{\prime}%
}>(Mn\log\log n)^{1/2}\right\}  )=1.
\]
The event $\bigcap_{i=1}^{\infty}\bigcup_{n\geq i}^{\infty}\left\{
S_{n}^{^{\prime}}>(Mn\log\log n)^{1/2}\right\}  $ is a sum of events
(\ref{sym1}) and (\ref{sym2}), hence
\[
P(\bigcap_{i=1}^{\infty}\bigcup_{n\geq i}^{\infty}\left\{  S_{n}^{^{\prime}%
}>(Mn\log\log n)^{1/2},S_{n}-S_{n}^{^{\prime}}\geq0\right\}  )=1.
\]
Thus, we have
\[
P(\bigcap_{i=1}^{\infty}\bigcup_{n\geq i}^{\infty}\left\{  S_{n}>(Mn\log\log
n)^{1/2}\right\}  )=1,
\]
for any $M>0$. It simply means, that $\underset{n\rightarrow\infty}{\lim\sup
}\frac{\left\vert \sum_{i=1}^{n}X_{i}\right\vert }{\sqrt{2n\log\log n}}%
=\infty.$
\end{proof}

\section{Scheff\'{e}'s Lemma\label{scheffe}}

\begin{lemma}%
\index{Lemma!Scheffe}%
Let $\left\{  f_{n}\right\}  _{n\geq1}$ be a sequence of nonnegative,
integrable functions defined on some measure space $\left(  S,\mathcal{S}%
,\mu\right)  $, that are convergent almost surely to a function $f$. Then
$\int_{S}\left\vert f_{n}-f\right\vert d\mu\underset{n\rightarrow
\infty}{\longrightarrow}0$ if and only if, when $\int_{S}f_{n}d\mu
\underset{n\rightarrow\infty}{\longrightarrow}\int_{S}fd\mu.$
\end{lemma}

\begin{proof}
(\cite{Williams95}) Implication $\int_{S}\left\vert f_{n}-f\right\vert
d\mu\underset{n\rightarrow\infty}{\longrightarrow}0$ $\allowbreak
\Longrightarrow$ $\int_{S}f_{n}d\mu\allowbreak\underset{n\rightarrow
\infty}{\longrightarrow}\allowbreak\int_{S}fd\mu$ is obvious. Hence, let us
consider the reverse one. To this end let us assume, that $\int_{S}f_{n}%
d\mu\underset{n\rightarrow\infty}{\longrightarrow}\int_{S}fd\mu$. Since we
have $\left(  f_{n}-f\right)  ^{-}\leq f$, hence by the Lebesgue Theorem on
"dominated passage to the limit under the sign of the integral" we have as
$n\,\longrightarrow\infty:$
\begin{equation}
\underset{n\rightarrow\infty}{\lim}\int_{S}(f_{n}-f)^{-}d\mu=\int%
_{S}\underset{n\rightarrow\infty}{\lim}(f_{n}-f)^{-}d\mu=0. \label{dla_-}%
\end{equation}
Next we have
\[
\int_{S}(f_{n}-f)^{+}d\mu=\int_{f_{n}\geq f}(f_{n}-f)=\int f_{n}d\mu-\int
fd\mu-\int_{f_{n}<f}(f_{n}-f)d\mu.
\]
Moreover, we have:
\[
\int_{f_{n}<f}(f_{n}-f)d\mu\leq\int_{S}(f_{n}-f)^{-}d\mu\underset{n\rightarrow
\infty\infty}{\longrightarrow}0.
\]
Hence:
\[
\int_{S}(f_{n}-f)^{+}d\mu\underset{n\rightarrow\infty}{\longrightarrow}0.
\]
This and (\ref{dla_-}) imply already the assertion.
\end{proof}

\chapter{Digression on the theory of distributions\label{dystrybucje}}

\section{Space of sample functions}

For simplicity and clarity of exposition, we will consider only functions of
one variable having complex values. The closure of the set $\left\{  x\in%
%TCIMACRO{\U{211d} }%
%BeginExpansion
\mathbb{R}
%EndExpansion
:f(x)\neq0\right\}  $ is called support of the function $f$. Sample functions
are all functions that are infinitely many times differentiable and having
compact supports. The set of all sample functions will be denoted by
$\mathcal{D}$.

The most important example of a sample function is the following:
$f(x)=\left\{
\begin{array}
[c]{ccc}%
\exp(\frac{1}{x^{2}-1}) & ,gdy & \left\vert x\right\vert <1\\
0 & ,gdy & \left\vert x\right\vert \geq1
\end{array}
\right.  $ that has the following plot:
%TCIMACRO{\FRAME{dtbpF}{2.4846in}{1.6518in}{0pt}{}{}{obaan800.eps}%
%width 2.4846in;  height 1.6518in;  depth 0pt;  original-width 3.1384in;
%original-height 2.0781in;  cropleft "0";  croptop "1";  cropright "1";
%cropbottom "0";  filename 'OBAAN800.eps';file-properties "XNPEU";}}}%
%BeginExpansion
\begin{center}
\includegraphics[
height=1.6518in,
width=2.4846in
]%
{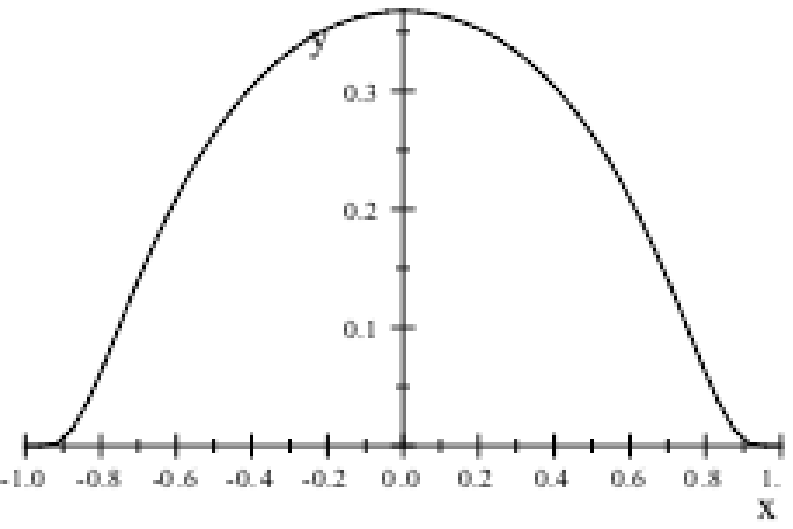}%
\end{center}
%EndExpansion

We have the following theorem:

\begin{theorem}
\label{tw_aproksym}For any $\varepsilon>0$ every continuous function $f$
having bounded support $K$ can be uniformly approximated with accuracy
$\varepsilon$ by the function $\varphi\in\mathcal{D}$. $\varphi$ can be
selected in such a way that its support can be contained in any neighborhood
of the support of the function $f$.
\end{theorem}

The following topology is introduced in the set of sample functions
$\mathcal{D}$ : \newline sequence of sample functions $\left\{  \varphi
_{i}\right\}  _{i\geq1}$ converges to a sample function $\varphi$, if
\newline\emph{i) }there\emph{\ }exists a compact set $K$ such that
$\operatorname{supp}\varphi_{i},\varphi\subset K\allowbreak,i=1,2,\ldots$
\newline\emph{ii)} derivatives of any order of functions\emph{\ }$\varphi_{i}$
converge uniformly to respective derivatives of $\varphi.$

\section{Distributions}

\begin{definition}%
\index{Distribution}%
Distribution $T$ is called any complex valued, linear functional, that is also
continuous in $\mathcal{D}$. Values of this functional will be denoted either
as $T(\varphi)$ or as $<T,\varphi>$.
\end{definition}

In other words, we have:

\begin{enumerate}
\item $T(\alpha\varphi_{1}+\beta\varphi_{2})=\alpha T(\varphi_{1})+\beta
T(\varphi_{2})$

\item If $\left\{  \varphi_{i}\right\}  $ has the limit $\varphi$, then
$T(\varphi_{i})\underset{i\rightarrow\infty}{\longrightarrow}T(\varphi)$
\end{enumerate}

\begin{remark}
Distributions are the elements of the vector space $\mathcal{D}^{^{\prime}}$.
The sum and the product by the scalar are defined in the following way:
\newline$\emph{a)}$ $(T_{1}+T_{2})(\varphi)=T_{1}(\varphi)+T_{2}(\varphi
),$\newline$\emph{b)}$ $(\lambda T)(\varphi)=\lambda T(\varphi)$
\end{remark}

\begin{example}
Let $f$ be a locally integrable function i.e. integrable on every bounded
measurable set. We define distribution $T_{f}$ in the following way:
\[
<T_{f},\varphi>\overset{df}{=}\int_{%
%TCIMACRO{\U{211d} }%
%BeginExpansion
\mathbb{R}
%EndExpansion
}f(x)\varphi(x)dx,
\]
for $\varphi\in\mathcal{D}$. This integral has sense, since function $\varphi$
and its support are bounded! It is easy to check continuity of this functional.
\end{example}

Moreover, it turns out, that :\emph{\ }Two functions\emph{\ }$f$ and $g$
define the same functional\emph{\ (}$T_{f}=T_{g})$ if and only if, they are
equal almost everywhere. Keeping this in mind it is reasonable to identify
locally integrable functions if they are equal almost everywhere, and
Moreover, it is also reasonable to identify distribution $T_{f}$ with locally
integrable function $f$.

\begin{example}
Distribution $\delta$ defined by the equality:
\[
<\delta,\varphi>=\varphi(0)
\]
we call Dirac's delta distribution. One considers also distributions
$\delta_{(a)}$ defined by:
\[
<\delta_{(a)},\varphi>=\varphi(a).
\]

\end{example}

\begin{example}
Heaviside's distribution $H$ is defined in the following way:
\[
<H,\varphi>=\int_{0}^{\infty}\varphi(x)dx.
\]
On can identify it with the Heaviside's function
\[
H(x)=\left\{
\begin{array}
[c]{lll}%
1 & ,when & x\geq0\\
0 & ,when & x<0
\end{array}
\right.  .
\]
Sometimes we talk about unit jump.
\end{example}

\begin{example}
Let $\mu$ will be any measure on $%
%TCIMACRO{\U{211d} }%
%BeginExpansion
\mathbb{R}
%EndExpansion
$. Then the integral
\[
\int_{%
%TCIMACRO{\U{211d} }%
%BeginExpansion
\mathbb{R}
%EndExpansion
}\varphi(x)d\mu(x)
\]
defines distribution $T_{\mu}$. Hence, measures are distributions!
\end{example}

\begin{definition}
By the support of the distribution $T$ ($\operatorname{supp}T$) we mean closed
set $D$ is having such property that for every function $\varphi$ having
compact support contained in the complement of $D$ we have $T(\varphi)=0$.
\end{definition}

\begin{remark}
It is easy to show, that $\operatorname{supp}\delta=\left\{  0\right\}  $ and
$\operatorname{supp}H=[0,\infty).$
\end{remark}

\begin{definition}
By the derivative of the distribution $T$ we mean distribution $DT$ defined by
the formula:
\[
<DT,\varphi>=-<T,D\varphi>.
\]

\end{definition}

\begin{remark}
Every distribution has a derivative of any order! For example, one can perform
the following calculation:
\begin{align*}
&  <DH,\varphi>=-<H,D\varphi>\\
&  =-\int_{0}^{\infty}D\varphi(x)dx=-\varphi(x)|_{x=0}^{\infty}\\
&  =\varphi(0)=<\delta,\varphi>,
\end{align*}
hence $DH=\delta$!
\end{remark}

\begin{definition}
We say, that a sequence of distributions $\left\{  T_{i}\right\}  _{i\geq1}$is
convergent to distribution $T$, if $\forall\varphi\in\mathcal{D}%
:<T_{i},\varphi>\underset{i\rightarrow\infty}{\longrightarrow}<T,\varphi>$.
\end{definition}

\begin{remark}
From the Lebesgue Theorem about the passage to the limit under the sign of the
integrals, one can deduce, that if a sequence of locally integrable functions
$\left\{  f_{j}\right\}  _{j\geq1}$converges almost surely to $f$ on every
bounded set and moreover , they are bounded by the locally integrable function
$g\geq0$, then the sequence of distributions $T_{f_{i}}$ converges to
distribution $T_{f}$. Unfortunately, , not the other way around. Since we have:
\end{remark}

\begin{theorem}
\label{zbiez_dystrybucji}If a sequence of distributions $T_{f_{i}}$ defined by
an absolutely integrable on $%
%TCIMACRO{\U{211d} }%
%BeginExpansion
\mathbb{R}
%EndExpansion
$ functions $f_{i}$, and such that $\underset{i}{\sup}\int_{%
%TCIMACRO{\U{211d} }%
%BeginExpansion
\mathbb{R}
%EndExpansion
}\left\vert f_{i}\right\vert dx<\infty$ converges to distribution $T_{f}$
defined by the function $f$ absolutely integrable, then on every Borel set $A$
we have:
\[
\int_{A}f_{i}(x)dx\underset{i\rightarrow\infty}{\longrightarrow}\int%
_{A}f(x)dx.
\]

\end{theorem}

\begin{proof}
We will prove this theorem for bounded set and not directly. Hence, suppose
that there exists a Borel set $A$, number $\varepsilon>0$ and a subsequence
$\left\{  i_{j}\right\}  _{j\geq1}$such that $\forall j\geq1$ $\left\vert
\int_{A}f_{i_{j}}(x)dx-\int_{A}f(x)dx\right\vert \geq\varepsilon$. Firstly,
let us notice that set $A$ has nonzero Lebesgue's measure. Secondly, keeping
this in mind, we have $\int_{\delta A}\left\vert f(x)\right\vert dx=0$ and
$\int_{\delta A}\left\vert f_{i_{j}}(x)\right\vert dx=0;j\geq1$ where $\delta
A$ denotes the boundary of the set $A$. Hence, one can assume, that the set
$A$ is closed. Hence, the characteristic function of the set $A$ i.e.
\thinspace$I(A)$ is continuous on $A$. From the approximation Theorem about
approximation \ref{tw_aproksym} we see that one can approximate $I(A)$ by
$\phi\in\mathcal{D}$ so that $\left\vert I(A)-\phi\right\vert \leq\epsilon$.
Hence, we have
\begin{align*}
\varepsilon &  \leq\left\vert \int_{A}\left(  f_{i_{j}}(x)-f(x)\right)
dx\right\vert \leq\\
&  \leq\left\vert \int\phi(x)\left(  f_{i_{j}}(x)-f(x)\right)  dx\right\vert
+\left\vert \int_{A}\left(  f_{i_{j}}(x)-f(x)\right)  (1-\phi(x))dx\right\vert
\\
&  \leq\left\vert \int\phi(x)\left(  f_{i_{j}}(x)-f(x)\right)  dx\right\vert
+\epsilon(\sup\int_{%
%TCIMACRO{\U{211d} }%
%BeginExpansion
\mathbb{R}
%EndExpansion
}\left\vert f_{i_{j}}\right\vert +\int_{%
%TCIMACRO{\U{211d} }%
%BeginExpansion
\mathbb{R}
%EndExpansion
}\left\vert f\right\vert ).
\end{align*}
The first summand converges following the assumption to zero. The second one
can be made sufficiently small, hence the contradiction.
\end{proof}

\section{Tempered distributions%
\index{Distribution!Tempered}%
}

In order to be able to introduce Fourier transform one has to confine a bit
the notion of distribution and consider functionals. on slightly larger space
of sample functions. Namely, we introduce space $\mathcal{S}$ of $C^{\infty}(%
%TCIMACRO{\U{211d} }%
%BeginExpansion
\mathbb{R}
%EndExpansion
\mathbb{)}$ class functions that quickly converge to zero. Namely,
$\mathcal{S}$ consists of functions $f$ such that for any $i,k\in%
%TCIMACRO{\U{2115} }%
%BeginExpansion
\mathbb{N}
%EndExpansion
$ we have $\underset{\left\vert x\right\vert \rightarrow\infty}{\lim
}\left\vert x\right\vert ^{k}\left\vert D^{(j)}f(x)\right\vert =0.$

For example, functions $\exp(-\alpha\left\vert x\right\vert ^{2})$ for
$\alpha>0$ belong to $\mathcal{S}$. Of course, $\mathcal{D}\subset\mathcal{S}%
$. A linear functional on $\mathcal{D}$ that can be continuously extended to
$\mathcal{S}$ is called tempered distribution.

\begin{remark}
Examples of tempered distributions are integrable functions, bounded
functions, and also functions, increasing to infinity not quicker than some polynomial.
\end{remark}

\begin{remark}
Every distribution that has bounded support is tempered.
\end{remark}

\begin{remark}
If a distribution is tempered, then all its derivatives are tempered.
\end{remark}

$\mathcal{S}^{^{\prime}}$ will denote set of all tempered distributions.

\section{Fourier transform\label{fourier}}

Let $\phi\in\mathcal{S}$. A Fourier transform $\mathcal{F}\phi$ of the
function $\phi$ is a function defined by:
\[
\hat{\phi}(t)=\int_{%
%TCIMACRO{\U{211d} }%
%BeginExpansion
\mathbb{R}
%EndExpansion
}\phi(x)\exp(-itx)dx.
\]

\begin{remark}
It turns out that $\hat{\phi}\in\mathcal{S}$ \thinspace hence $\mathcal{F}%
:\mathcal{S}\rightarrow\mathcal{S}$ $\,$and the transform is mutually unique.
It is linear and continuous. Similarly the inverse transform defined by the
formulae:
\[
\widetilde{\phi}(t)=\frac{1}{2\pi}\int_{%
%TCIMACRO{\U{211d} }%
%BeginExpansion
\mathbb{R}
%EndExpansion
}\phi(x)\exp(itx)dx
\]

\end{remark}

Let $T\in\mathcal{S}_{x}^{^{\prime}}$ and $\phi\in\mathcal{S}_{\lambda}$. Then
the Fourier transform $\mathcal{F}T$ of distribution $T$ is defined by the
formula:
\[
<\mathcal{F}T,\phi.>=<T,\mathcal{F}\phi>.
\]

\begin{remark}
It turns out that $\mathcal{F}T\in\mathcal{S}^{^{\prime}}$ that is
$\mathcal{F}:\mathcal{S}^{^{\prime}}\rightarrow\mathcal{S}^{^{\prime}}$.
Moreover, this mapping is continuous and linear and mutually unique. Thus,
there exists an inverse transform having similar properties.
\end{remark}

\begin{remark}
Of course, one can insert space $\mathcal{S}$ \thinspace in the space
$\mathcal{S}^{^{\prime}}$. It turns out also that we have the following
inclusion:
\[
\mathcal{S}\subset L_{2}\subset\mathcal{S}^{^{\prime}},
\]
and moreover, that $\mathcal{F}(L_{2})\allowbreak=\allowbreak L_{2}.$
\end{remark}

\begin{remark}
Since, the Fourier transform is a continuous mapping of the relative spaces in
themselves, one can deduce from the convergence of Fourier transform the
distributive convergence of respective distributions. We used this fact when
e.g. deriving formula (\ref{zbiez_dystr}) on page \pageref{zbiez_dystr}.
\end{remark}


\begin{thebibliography}{999999999999}                                                                                     %


\bibitem[AL76]{Ahmad76}Ahmad, Ibrahim A.; Lin, Pi-Erh. Nonparametric
sequential estimation of a multiple regression function. \emph{Bull. Math.
Statist.} \textbf{17} (1976), no. 1-2, 63--75. MR0436439

\bibitem[Ale61]{Alexits}Alexits, G. Convergence problems of orthogonal series.
Translated from the German by I. F\"{o}lder. \emph{International Series of
Monographs in Pure and Applied Mathematics}, Vol. \textbf{20} Pergamon Press,
New York-Oxford-Paris 1961 ix+350 pp. MR0218827

\bibitem[Azz81]{Azzalini81}Azzalini, A. A note on the estimation of a
distribution function and quantiles by a kernel method. \emph{Biometrika}
\textbf{68} (1981), no. 1, 326--328. MR0614972

\bibitem[BJ83]{BJ70}George E.~P. Box and Gwilym~M. Jenkins, \emph{Analiza
szereg\'{o}w czasowych}, PWN, Warszawa, 1983, t\l umaczenie z angielskiego
orygina\l u wydanego przez Holden-Day w 1976 r.

\bibitem[Bor81]{Borwein81}David Borwein, \emph{Matrix transformations of
weakly multiplicative sequences of the random variables}, J. London Math.
Soc.(2) \textbf{23} (1981), 363--371.

\bibitem[Bro88]{Brosamler88}A.~Brosamler, G., \emph{An almost everywhere
central limit theorem}, Math. Proc. of Cambridge Philosophical Soc.
\textbf{104} (1988), 561--574.

\bibitem[Car76]{Carrol76}R.~J. Carrol, \emph{On sequential density
estimation}, Zeitsch. Warsch. \textbf{36} (1976), 136--151.

\bibitem[CFR93]{Csaki93}E.~Csaki, A.~F\H{o}ldes, and P.~R{\'{e}}v{\'{e}}sz,
\emph{On almost sure local and global central limit theorems}, Probability
Theory and Related Fields \textbf{97} (1993), 321--337.

\bibitem[Cra46]{Cramer46}H.~Cramer, \emph{A contribution to the theory of
statistical estimation}, Skans. Akt. \textbf{29} (1946), 85--94.

\bibitem[Dav73]{Davies73}H.~I. Davies, \emph{Strong consitency of sequential
estimator of probability density function}, Bull. Math. Stat. \textbf{15}
(1973), 49--53.

\bibitem[Deh74]{Deheuvels74}P.~Deheuvels, \emph{Conditions n\'{e}cessaires et
suffiantes de convergence punctuell presque s\^{u}re e uniforme presque
s\^{u}re des estimateurs de la densit\'{e}.}, Comptes Rendus de l'Ac. Sc. de
Paris \textbf{278} (1974), 1217--1220.

\bibitem[Dev79]{Devroye79}L.~Devroy{\'{e}}, \emph{On the pointwise and
integral convergence of recursive kernel estimates of probablity densities},
Util. Math. \textbf{15} (1979), 113--128.

\bibitem[DG88]{devroy}L.~Devroye and L.~Gy\H{o}rfi, \emph{Nonparametric
estimation. the l\_1 view.}, Mir, Moskwa, 1988.

\bibitem[DM65]{Demidowicz65}F.~Demidowicz, B. and A.~Maron, I., \emph{Metody
numeryczne}, PWN, Warszawa, 1965.

\bibitem[Dur62]{Durbin62}J.~Durbin, \emph{Estimation of parameters in the time
series regression models}, J. R. Statist. Soc. B (1962).

\bibitem[Dvo56]{Dvoretzky56}A.~Dvoretzky, \emph{On stochastic approximation.},
Proc. Of the Third Berkeley Symp. On Math Stat. and Probab. T. (Berkeley),
1956, pp.~39--55.

\bibitem[EKR97]{ErmoRuszc97}Yuri-M. Ermoliev, Arkadii-V. Kryazhimskii, and
Andrzej Ruszczynski, \emph{Constraint aggregation principle in convex
optimization.}, Math.-Programming \textbf{76} (1997), no.~3, 353--372.

\bibitem[ER96]{ErmRuszc96}Y.-M. Ermoliev and A.~Ruszczynski, \emph{Convex
optimization by radial search.}, J.-Optim.-Theory-Appl. \textbf{91} (1996),
no.~3, 731--738.

\bibitem[Fab60]{Fabian60}V.~Fabian, \emph{Stochastic approximation methods},
Czech. Math. J. \textbf{10} (1960), 123--159.

\bibitem[Fab67]{Fabian67}\bysame, \emph{Stochastic approximation of minima
with improved asymptotic speed.}, Ann. Math. Statist. \textbf{38} (1967), 191--200.

\bibitem[Far62]{Farrell62}R.~H. Farrell, \emph{Bounded length confidence
intervals for the zero of regression function}, Ann. Math. Statist.
\textbf{33} (1962), 237--247.

\bibitem[Fel69]{feller1}W.~Feller, \emph{Wst\c{e}p do rachunku
prawdopodobie\'{n}stwa}, PWN, Warszawa, 1969.

\bibitem[Fih64]{Fihtenholtz64}G.~M. Fihtenholtz, \emph{Rachunek
r\'{o}\.{z}niczkowy i ca{\l }kowy}, PWN, Warszawa, 1964.

\bibitem[FSW77]{Findeisen77}W{\l }adys{\l }aw Findeisen, Jacek Szymanowski,
and Andrzej Wierzbicki, \emph{Teoria i metody obliczeniowe optymalizacji},
PWN, Warszawa, 1977.

\bibitem[Gap67]{Gaposkin67}F.~Gapo{\v{s}}kin, V., \emph{Remark on paper of p.
revesz concerning multiplicative systems of functions.}, Matematicheskije
Zamietki (in Russian) \textbf{1(6)} (1967), 653--656.

\bibitem[Gap72]{Gaposhkin72}\bysame, \emph{On the convergence of series in
weakly multiplikative systems of functions}, Mat. Sbor. \textbf{89} (1972), 355--365.

\bibitem[GC87]{Godambe87}V.~P. Godambe and Heyde~C. C., \emph{Quasi-likelihoo
and optimal estimation.}, Internat. Statist. Rev. \textbf{55} (1987), 231--244.

\bibitem[GK54]{Gnedenko54}V.~I. Gnedenko and A.~N. Kolmogorov, \emph{Limit
distributions for sums of independent random variables}, Addison -Wesley Publ.
Co., Cambridge, U.S., 1954.

\bibitem[GK91]{Godambe91}V.~P. Godambe and B.~K. Kale, \emph{Estimating
functions: An overview.}, Oxford Statistical Science Series \textbf{7} (1991), 3--20.

\bibitem[Gli74]{Glick74}N.~Glick, \emph{Consistency conditions for probability
estimators and integrals of density estimators}, Utilitas Mathematica
\textbf{6} (1974), 61--74.

\bibitem[GT89]{Godambe89}V.~P. Godambe and M.~E. Thompson, \emph{An extension
of quasi-likelihood estimation}, J. Statist. Planning and Inference
\textbf{22} (1989), 137--152.

\bibitem[Hee97]{Heermann97}W.~Heermann, Dieter, \emph{Podstawy symulacji
komputerowych w fizyce}, WNT, Warszawa, 1997.

\bibitem[HJ94]{hoffman}J.~Hoffmann-Jorgensen, \emph{Probability with a view
toward statistics}, Chapman and Hall, New York, 1994.

\bibitem[HL56]{Hodges56}J.~L. Hodges and E.~L. Lehman, \emph{The efficiency of
some nonparametric competitors of the t-test}, Ann. Math. Statist. \textbf{27}
(1956), 324--335.

\bibitem[HL97]{Heyde97}C.~C. Heyde and Y.~X. Lin, \emph{On spaces of
estimating functions}, J. Statis. Planning and Inference \textbf{63} (1997),
no.~2, 255--264.

\bibitem[HW41]{Hartman41}P.~Hartman and A.~Wintner, \emph{On the law of the
iterated logarithm}, Amer. J. Math. \textbf{63} (1941), 169--176.

\bibitem[J.72]{Kusher72b}Kushner~H. J., \emph{Stochastic approximation type
algorithms for the optimization of the constrained and multimode stochastic
problems.}, Tech. Report~72, Brown University., Providence, R. I., 1972.

\bibitem[JMS96]{Jones96}M.-C. Jones, J.-S. Marron, and S.-J. Sheather, \emph{A
brief survey of bandwidth selection for density estimation.},
J.-Amer.-Statist.-Assoc. \textbf{91} (1996), no.~433, 401--407.

\bibitem[JOP65]{Jamison65}Benton Jamison, Steven Orey, and William Pruitt,
\emph{Convergence of weighted averages of independent random variables}, Z.
Wahrscheinlichtkeitstheorie verw. Geb. \textbf{4} (1965), 40--44.

\bibitem[KC78]{KushnerClark78}H.~J. Kushner and D.~S. Clark, \emph{Stochastic
approximation methods for constrained and unconstrained systems.},
Springer-Verlag., New York, 1978.

\bibitem[KG73]{KushnerGavin73}H.~J. Kushner and T.~Gavin, \emph{A versatile
method for the monte carlo optimization of stochastic systems.}, Int. J.
Control, \textbf{18} (1973), 963--975.

\bibitem[KG74]{KushnerGavin74}\bysame, \emph{Stochastic approximation type
methods for constrained systems algorithms and numerical results.}, IEE Trans.
on Autom. Control. \textbf{AC-19} (1974), 349--357.

\bibitem[Kie52]{Kiefer52}J.~Wolfowitz~J. Kiefer, \emph{Stochastic estimation
of the maximum of a regression function}, Ann. Math. Statist. \textbf{23}
(1952), 462--466.

\bibitem[KL94]{Koronacki94}J.~Koronacki and U.~Luboinska, \emph{Estimating the
density of a functional of several random variables.},
Comput.-Statist.-Data-Anal. \textbf{18} (1994), no.~3, 317--330.

\bibitem[Kor80]{Koronacki80}J.~Koronacki, \emph{Some remarks on stochastic
approximation methods}, Numerical Techniques for Stchastic Systems (Amsterdam)
(F.~Archetti and M.~Cugiani, eds.), North-Holland, 1980, pp.~395--496.

\bibitem[Kor89]{Koronacki89}\bysame, \emph{Aproksymacja stochastyczna. metody
optymalizacji w warunkach losowych.}, Wydawnictwa Naukowo Techniczne,
Warszawa, 1989.

\bibitem[KS74]{KushnerSanvincente74}H.~J. Kushner and E.~Sanvincente,
\emph{Penalty function methods for constrained stochastic approximation.}, J.
Math. Anal. Appl. \textbf{46} (1974), 499--512.

\bibitem[KS84]{Kushner84}H.~J. Kushner and A.~Schwarz, \emph{An invariant
measure approach to the convergence of stochastic approximations with state
dependcnt noise.}, SIAM J. on Control and Optimiz. \textbf{22} (1984), 13--27.

\bibitem[Kus72]{Kushner72}H.~J. Kushner, \emph{Stochastic approximation
algorithms for local optimization of function with non-unique stationary
points.}, IEEE Trans. Autom. Control AC \textbf{17} (1972), 646--654.

\bibitem[Lan69]{Lankaster69}P.~Lankaster, \emph{Theory of matrices}, Academic
Press, New York -London, 1969.

\bibitem[Lo{\'{e}}55]{Loeve55}M.~Lo{\'{e}}ve, \emph{Probability theory}, Van
Nostrand, New York, 1955.

\bibitem[{\L }oj73]{lojasiewicz}St. {\L }ojasiewicz, \emph{Wst\c{e}p do teorii
funkcji rzeczywistych}, PWN, Warszawa, 1973.

\bibitem[LS78]{Longnecker78}M.~Longnecker and J.~Serfling, R., \emph{Moment
inequalities for $s\_n$ under general dependence restrictions, with
applications}, Z. Wahrsch. Verw. Gebiete \textbf{43} (1978), 1--21.

\bibitem[Lue73]{Luenberger73}D.~G. Luenberger, \emph{Introduction to linear
and nonlinear programming}, Adisson-Wesley, Reading Mass., 1973.

\bibitem[MM97]{Mammen97a}E.~Mammen and J.-S. Marron, \emph{Mass recentred
kernel smoothers.}, Biometrika \textbf{84} (1997), no.~4, 765--777.

\bibitem[M{\'{o}}r76]{Moricz76}Ferenc M{\'{o}}ricz, \emph{Moment inequalities
and the strong laws of large numbers}, Z. Wahrsch. Gebiete \textbf{35} (1976), 299--314.

\bibitem[M{\'{o}}r83a]{Moricz83b}\bysame, \emph{Cesaro means of orthogonal
sequences}, Ann. Probab. \textbf{11} (1983), no.~4, 827--832.

\bibitem[M{\'{o}}r83b]{Moricz83}\bysame, \emph{Matrix transformations of
weakly dependent random variables.}, J.-London-Math.-Soc. (2) \textbf{27(1)}
(1983), 185--192.

\bibitem[M{\'{o}}r85]{Moricz85c}F.~M{\'{o}}ricz, \emph{Almost sure behaviour
of the first and second arithmetic means for dependent random variables},
Math. Nachr. \textbf{129} (1985), 81--90.

\bibitem[MT94]{Moricz94}F.~M{\'{o}}ricz and K.~Tandori, \emph{Almost
everywhere covergence of orthogonal series revisited}, Journal of Math.
Analysis and Appl. \textbf{182} (1994), 637--653.

\bibitem[MT96]{Moricz96}Ferenc M{\'{o}}ricz and K{\'{a}}roly Tandori, \emph{An
improved menshov-rademacher theorem}, Proc. of the AMS \textbf{124} (1996),
no.~3, 877--885.

\bibitem[NC72]{Nevelson72}M.~B. Nevelson and R.~Z. Chasminskij,
\emph{Stochasticzeskaja approksimacja i rekurentne oceniwanije}, Nauka,
Moskwa, 1972.

\bibitem[NPR98]{NorRuszcz98}Vladimir-I. Norkin, Georg-Ch. Pflug, and Andrzej
Ruszczynski, \emph{A branch and bound method for stochastic global
optimization.}, Math.-Programming \textbf{83} (1998), no.~3,ser A, 425--450.

\bibitem[OT81]{Okuyama81}Y.~Okuyama and T.~Tsuchkura, \emph{Absolute riesz
summability of orthogonal series}, Anal. Math. \textbf{7} (1981), 199--208.

\bibitem[Par62]{Parzen62}E.~Parzen, \emph{On the estimation of probability
density function and mode}, Ann. of Statist. \textbf{33} (1962), 1065--1076.

\bibitem[Par98]{Park98}Hyo-Il Park, \emph{Minimum distance estimation based on
the kernels for $u$-statistics.}, J.-Korean-Statist.-Soc. \textbf{27} (1998),
no.~1, 113--132.

\bibitem[PRS98]{PflRuszcz98}Georg-Ch. Pflug, Andrzej Ruszczy\'{n}ski, and
Rudiger Schultz, \emph{On the glivenko-cantelli problem in stochastic
programming: linear recourse and extensions.}, Math.-Oper.-Res. \textbf{23}
(1998), no.~1, 204--220.

\bibitem[Ral75]{Ralston75}Anthony Ralston, \emph{Wstep do analizy
numerycznej}, PWN, Warszawa, 1975.

\bibitem[R{\'{e}}v66]{Revesz66}P.~R{\'{e}}v{\'{e}}sz, \emph{A convergence of
orthogonal series}, Acta. Sci. Math. Hun. \textbf{27} (1966), 253--260.

\bibitem[R{\'{e}}v67]{Revesz67}P{\'{a}}l R{\'{e}}v{\'{e}}sz, \emph{The laws of
large numbers}, Akad\'{e}miai Kiad\'{o}, Budapest, 1967.

\bibitem[RM51]{Robbins51}H.~Robbins and S.~Monro, \emph{A stochastic
approximation method.}, Ann. Math. Statist. \textbf{22} (1951), 400--407.

\bibitem[Ros56]{Rosenblatt56}M.~Rosenblatt, \emph{Remarks on some
nonparametric estimates of density function}, Ann. of Statist \textbf{27}
(1956), 832--837.

\bibitem[RS86a]{RuszcSys86}Andrzej Ruszczynski and Wojciech Syski, \emph{A
method of aggregate stochastic subgradients with on-line stepsize rules for
convex stochastic programming problems.}, Math.-Programming-Stud. (1986),
no.~28, 113--131.

\bibitem[RS86b]{RuszcSys86a}\bysame, \emph{On convergence of the stochastic
subgradient method with on-line stepsize rules.}, J.-Math.-Anal.-Appl.
\textbf{114} (1986), no.~2, 512--527.

\bibitem[Rup82]{Ruppert82}D.~Ruppert, \emph{Almost sure approximations to the
robbins-monro and kiefer-wolfowitz processes with dependent noise.}, Ann.
Prob. \textbf{10} (1982), 178--187.

\bibitem[Rus80]{Ruszc80}Andrzej Ruszczynski, \emph{Stochastic feasible
direction methods for nonsmooth stochastic optimization problems.},
Control-Cybernet. \textbf{9} (1980), no.~4, 173--187.

\bibitem[Rus84]{Ruszc84}\bysame, \emph{A recursive quadratic programming
algorithm for constrained stochastic programming problems.}, Control-Cybernet.
\textbf{13} (1984), no.~1-2, 59--72.

\bibitem[Rus97]{Ruszcz97}\bysame, \emph{Decomposition methods in stochastic
programming.}, Math.-Programming \textbf{79} (1997), no.~1-3 Ser.B, 333--353.

\bibitem[Sch58]{Schmetterer58}L.~Schmetterer, \emph{Sur l'iteration
stochastique}, La Calcul de Probabilites et Ses Applications, vol.~87, Colloq.
Intern. Centre Nat. Sci., 1958, pp.~166--176.

\bibitem[Sch91]{Schatte91}P.~Schatte, \emph{On the central limit theorem with
almost sure convergence}, Probability and Mathematical Statistics \textbf{11}
(1991), 237--246.

\bibitem[SHD94]{SheaHettDo94}Simon-J. Sheather, Thomas-P. Hettmansperger, and
Margaret-R. Donald, \emph{Data-based bandwidth selection for kernel estimators
of the integral of $f\sp2(x)$.}, Scand.-J.-Statist. \textbf{21} (1994), no.~3, 265--275.

\bibitem[Sie73]{Sielkien73}L.~Sielkien, R., \emph{Stopping rules for
stochastic approximation procedures}, Zeit. Wahr. u. verw. Geb. \textbf{26}
(1973), 67--75.

\bibitem[Sil86]{Silverman86}B.~W. Silverman, \emph{Density estimation for
statistics and data analysis}, Chapman and Hall, London, 1986.

\bibitem[Str65a]{Strass65}V.~Strassen, \emph{Almost sure behaviour of sums of
independent random variables and martingales}, Proc. Symp. Math. Statist. and
Probability, 5th Berkeley \textbf{2} (1965), 315--344.

\bibitem[Str65b]{Strassen65}\bysame, \emph{A converse to the law of the
iterated logarithm}, Z. Wahrscheinlichtkeitstheorie verw. Geb. \textbf{4}
(1965), 265--268.

\bibitem[Sza75]{[SZA2]}Pawe{\l }~J. Szab{\l }owski, \emph{Identification of
the parameters of a discrete stochastic process by the method of estimating
functions.}, Arch. Automat. i Telemech. \textbf{20} (1975), no.~3, 351--367.

\bibitem[Szabl77]{Szabl77}\bysame Generalized stochastic approximation and its
application to parameter identification of discrete stochastic processes.
\emph{Differential games and control theory, II }(Proc. 2nd Conf., Univ. Rhode
Island, Kingston, R.I., 1976), pp. 461--470. Lecture Notes in Pure and Appl.
Math., \textbf{30}. Dekker, New York, 1977. MR0688492

\bibitem[Szabl79(2)]{Sza872}\bysame Generalized laws of large numbers and
auxiliary results concerning stochastic approximation with dependent
disturbances. II. \emph{Comput. Math. Appl.} \textbf{13} (1987), no. 12,
973--987. MR0898944 (88k:62160b)

\bibitem[Szabl79]{Szab79}\bysame Application of the generalized strong laws of
large numbers to the proper choice of the amplification coefficients in
stochastic approximation with correlated d. \emph{Differential games and
control theory} (Proc. Third Kingston Conf., Part A, Univ. Rhode Island,
Kingston, R.I., 1978), pp. 223--244, Lecture Notes in Pure and Appl. Math.,
\textbf{44}, Dekker, New York, 1979. MR0546826

\bibitem[Sza79]{[SZA22]}\bysame, \emph{Some remarks concerning the dual
control problem.}, Proceedings of the Third Kingston Conference on
Differential Games and Control Theory (E.~Et~Al. Roxin, ed.), Lecture Notes in
Pure and Appl. Math., no.~44, Marcel Dekker, New York, 1979.

\bibitem[Sza87]{[SZA320]}\bysame, Stochastic approximation with dependent
disturbances. I. \emph{Comput. Math. Appl.} \textbf{13} (1987), no. 12,
951--972. MR0898943

\bibitem[Sza88a]{[SZA312]}\bysame, \emph{Optimal, recursive procedures of
identification.}, Comput. Math. Appl. \textbf{16} (1988), no.~3, 229--246.

\bibitem[Sza91]{szab4}\bysame, \emph{A few remarks on riesz summability of
orthogonal series.}, Proc. of the AMS \textbf{113} (1991), no.~1, 65--75.

\bibitem[Sza97]{Szablowski972}\bysame, \emph{Some remarks on almost sure local
and global central limit theorem}, Advances in the Theory and Practice of
Statistics: A Volume to Honor of Samuel Kotz (Norman~L. Johnson and
N.~Balakrishnan, eds.), John Wiley and Sons, Inc., 1997, pp.~551--559.

\bibitem[Tan65]{Tandori65}Karoly Tandori, \emph{Bemerkung zur konvergenz der
orthogonalreichen}, Acta Sci. Math. (Szeged) \textbf{26} (1965), 249--251.

\bibitem[Tan72]{Tandori72}K.~Tandori, \emph{Bemerkungen zum gesetz der grossen
zahlen.}, Periodica Math. Hungar. \textbf{2} (1972), 33--39.

\bibitem[TS80]{Terrell80}George-R. Terrell and David-W. Scott, \emph{On
improving convergence rates for nonnegative kernel density estimators.},
Ann.-Statist. \textbf{8} (1980), no.~5, 1160--1163.

\bibitem[TS92]{Terrell92}\bysame, \emph{Variable kernel density estimation.},
Ann.-Statist. \textbf{20} (1992), no.~3, 1236--1265.

\bibitem[Ven67]{Venter67}J.~H. Venter, \emph{An extension of the robbins-monro
procedure}, Ann. Math. Statist. \textbf{38} (1967), 181--190.

\bibitem[WD79]{Wengman79}E.~J. Wengman and H.~I. Devies, \emph{Remarks on some
recursive estimators of probability density}, Ann. Stat. \textbf{37} (1979), 316--327.

\bibitem[Wil95]{Williams95}David Williams, \emph{Probability with
martingales}, Cambridge University Press, Cambridge, 1995.

\bibitem[WJ69]{Wolverton69}C.~T. Wolverton and Wagner~T. J.,
\emph{Asymptotically optimal discriminant functions for pattern
classifications}, IEEE Trans. on Information Theory \textbf{IT-15} (1969), 258--265.

\bibitem[WJ95]{Wand95}M.~P. Wand and M.~C. Jones, \emph{Kernel smoothing},
Chapman and Hall, N. York, 1995.

\bibitem[Yam71]{Yamato71}H.~Yamato, \emph{Sequential estimation of a
continuous probabbility density function and the mode}, Bull. of Math.
Statistics \textbf{14} (1971), 1--12.

\bibitem[Yin90]{Yin90}G.~Yin, \emph{A stopping rule for the robbins-monro
method.}, J.-Optim.-Theory-Appl. \textbf{19} (1990), no.~1, 151--173.

\bibitem[Yin92]{Yin92}\bysame, \emph{On $h$-valued stochastic approximation:
finite-dimensional projections.}, Stochastic-Anal.-Appl. \textbf{10} (1992),
no.~3, 363--377.

\bibitem[Zie70]{Zielinski75}Ryszard Zieli\'{n}ski, \emph{Metody monte carlo},
WNT, Warszawa, 1970.
\end{thebibliography}
\end{document}